\pdfoutput=1
\documentclass[reqno]{amsart}
\usepackage[unicode=true,pdfusetitle,
 bookmarks=true,bookmarksnumbered=false,bookmarksopen=false,
 breaklinks=false,pdfborder={0 0 0},pdfborderstyle={},backref=false,colorlinks=true]
 {hyperref}
\hypersetup{
    colorlinks, linkcolor=purple,
    citecolor=blue, urlcolor=magenta
}

\usepackage{amssymb, mathtools, enumerate}
\usepackage{amsmath}
\usepackage{tikz}
\usepackage{tikz-cd}
\usepackage[scr]{rsfso}
\usepackage{bm}
\usepackage{bbm}
\usepackage[margin=3cm]{geometry}
\usetikzlibrary{matrix,arrows,decorations.pathmorphing, decorations.markings, cd}
\usepackage[capitalise]{cleveref}
\usepackage{stmaryrd, rotate}
\usepackage[alphabetic]{amsrefs}
\usepackage{microtype}

\pgfdeclarelayer{bg}
\pgfsetlayers{bg,main}
\usetikzlibrary{calc}

\newtheorem{theorem}{Theorem}[section]

\newtheorem{corollary}[theorem]{Corollary}
\newtheorem{lemma}[theorem]{Lemma}
\newtheorem{proposition}[theorem]{Proposition}

\newtheorem{introtheorem}{Theorem}
\newtheorem{introcorollary}[introtheorem]{Corollary}
\newtheorem{introproposition}[introtheorem]{Proposition}

\theoremstyle{definition}
\newtheorem*{claim*}{Claim}
\newtheorem{construction}[theorem]{Construction}

\newtheorem{convention}[theorem]{Convention}
\newtheorem{definition}[theorem]{Definition}
\newtheorem{introdefinition}[introtheorem]{Definition}

\newtheorem{example}[theorem]{Example}

\newtheorem{notation}[theorem]{Notation}
\newtheorem{observation}[theorem]{Observation}

\newtheorem{remark}[theorem]{Remark}
\newtheorem*{remark*}{Remark}
\newtheorem{warning}[theorem]{Warning}

\newcommand{\lc}[1]{\smash{#1}^\uparrow}
\newcommand{\rc}[1]{\smash{#1}^\downarrow}
\newcommand{\lrc}[1]{\smash{#1}^\updownarrow}
\newcommand{\wlc}[1]{\smash{#1}^{\uparrowlabel{w}}}
\newcommand{\wrc}[1]{\smash{#1}^{\downarrowlabel{w}}}

\newcommand{\trc}[1]{\smash{#1}^{\downarrowlabel{t}}}

\newcommand{\wtc}[1]{\smash{#1}^{\biarrowlabel{w}{t}}}

\newcommand{\K}{\widehat{\mathcal{K}}}
\newcommand{\Kw}{\mathcal{K}}
\newcommand{\wb}{wb}
\newcommand{\tb}{tb}

\newcommand{\alt}[1]{\textcolor{gray}{\textup{(}#1\textup{)}}}

\usepackage{macros}

\title[Unbounded Weight Structures: (Re)construction and Completion]{Unbounded Weight Structures: \\(Re)construction and Completion}
\author{Thomas Nikolaus}
\address{T.N.: FB Mathematik und Informatik, Universität Münster, Einsteinstraße 62, 48149 Münster, Germany}

\author{Phil Pützstück}
\address{P.P.: FB Mathematik und Informatik, Universität Münster, Einsteinstraße 62, 48149 Münster, Germany}

\begin{document}
\begin{abstract}
    We develop a theory of completeness for weight structures on stable categories,
    dual to the theory of complete $t$-structures.
    As in the bounded case, we show that complete weight structures are determined by their weight heart,
    giving rise to a universal construction $\cA \mapsto \Kw(\cA)$ that assigns a complete weight category
    to an additive category and recovers classical examples such as homotopy categories of chain complexes.

    We also give a general construction of weight structures on presentable stable categories generated by a
    small set of objects, generalizing a result of Bondarko and Pauksztello.
    This recovers the standard weight structure on spectra and an exotic one related to Anderson duality.
    We identify their completions with modules over the (spectral) integral Steenrod algebra.

  To treat naturally occurring examples - such as derived categories of abelian categories and module categories over ring spectra - which are often only partially weight complete, we introduce the notion of weak $t$-structures. Within this framework, we prove that any stable category equipped with compatible weight and weak $t$-structures, and satisfying left weight completeness and right $t$-completeness, can be reconstructed from its heart via a two-step completion process $\cA \mapsto \K(\cA)$.
\end{abstract}

\begingroup\parskip=0pt
\maketitle
\tableofcontents
\endgroup

\section{Introduction}

Many constructions in homotopy theory and algebra proceed by building
objects from simpler pieces. Classical examples include CW decompositions
of spaces or spectra and projective resolutions of chain complexes.
Such cell structures provide filtrations whose successive pieces may be
viewed as ``pure'' building blocks and are an indispensable tool in many constructions and computations.

Weight structures provide an abstract framework for organizing such
decompositions in stable categories. Introduced by Bondarko
\cite{Bondarko10} (and independently by Pauksztello \cite{Pauksztello08} under the name of co-$t$-structures),
they can be viewed as a dual analogue of
$t$-structures on triangulated categories or stable $\infty$-categories. Although
the formal definition appeared only relatively recently, many of the
underlying ideas were already present earlier in the literature, notably in \cite[Section 1.7]{Waldhausen-Ktheoryofspaces}, \cite[Section 4]{BGT}. This presents some important applications of weight structures
to algebraic $K$-theory. They have also been used recently, see e.g.~\cite{Sosnilo}, \cite{Ishan-Vova}, \cite{Hebestreit-Steimle}.

Concretely, a weight structure on a stable category $\cC$ is determined
by a full subcategory $\cC_{\ge 0} \subseteq \cC$ closed under retracts, extensions and finite colimits
such that every object $X \in \cC$ admits a
fiber sequence
\[
    X' \to X \to X'',
\]
called a \emph{weight decomposition}, where $X'' \in \cC_{\ge 0}$ and
$X' \in \cC_{w \leq -1} \coloneqq {}^{\perp}(\cC_{\ge 0})$\footnote{Here
${}^{\perp}(\cC_{\ge 0})$ denotes the left orthogonal subcategory
consisting of objects that admit no nontrivial morphisms into objects of
$\cC_{\ge 0}$.}. Intuitively, the objects $X'$ and $X''$ describe the parts
of $X$ of weight $\leq -1$ and $\ge 0$, respectively.
Generally we set $\cC_{\geq n} = \Sigma^n \cC_{\geq 0}$ and $\cC_{w \leq n} = \Sigma^n \cC_{w\leq 0}$ for any $n \in \Z$.
The \emph{weight heart}
\[
    \cC^{w\heartsuit} \coloneqq \cC_{\ge 0} \cap \cC_{w \leq 0},
\]
consists of objects of pure weight $0$.

In many situations where weight structures arise—such as in algebraic
$K$-theory—the ambient category $\cC$ is small and the weight structure
is \emph{bounded}.
This means that every object lies in
$\cC_{[a,b]} \coloneqq \cC_{\ge a} \cap \cC_{w \le b}$ for some integers
$a \le b$, or equivalently, that every object has a finite filtration with associated graded in (shifts) of the heart  $\cC^{w\heartsuit}$.
In this case the category $\cC$ can be recovered from the heart as the
minimal stable category generated by $\cC^{w\heartsuit}$ \cite[Proposition 3.3 and Corollary 3.4]{Sosnilo}, see also Theorem \ref{thm:vova}.
In this sense, bounded weight structures describe stable categories
that are minimally generated by their weight heart.

While the bounded case is by now well understood, considerably less is known about unbounded weight structures. This leads to the following natural question:
\begin{quote}
\emph{To what extent can a stable category be reconstructed from its weight heart in the absence of boundedness?}
\end{quote}
The present paper is devoted to this problem: we study unbounded weight structures and, in particular, focus on the opposite extreme to the bounded case: those that describe stable categories which are, in a precise sense, maximal among those admitting a given weight heart.
Our guiding principle is that a weight structure encodes how objects are assembled from the heart, while completeness determines how far this assembly process can be extended beyond the bounded setting. \\

Concretely, we introduce completeness conditions for weight structures,
analogous to those familiar from the theory of $t$-structures \cite{BBD}, \cite{Neeman-Triangulated}, \cite[Section 1.2.1]{HA}. Completeness
expresses that constructions from the heart can be carried out
indefinitely in a controlled way, and we show that complete weight
structures are again determined by their weight heart
$\cC^{w\heartsuit}$.
Completeness for weight
structures, similarly to the case of $t$-structures, comes in two
directions: left and right. Since right completeness for
$\cC$ is equivalent to left completeness for the opposite category
$\cC^{\op}$, we concentrate on the left version here. More precisely, we have:

\begin{introdefinition}[Definition \ref{def:weight-left-complete}]\label{def_weightcomp}
A weight structure $(\cC, \cC_{\geq 0})$ is called left complete if
    \begin{enumerate}
        \item $\cC_{\geq \infty} \coloneqq \bigcap_{n \geq 0} \cC_{\geq n} = 0$.
        \item $\cC$ admits sequential colimits growing in connectivity, i.e. colimits of sequential diagrams where the successive cofibers become arbitrarily connective.
        \item $\cC_{\geq 0} \subseteq \cC$ is closed under such sequential colimits growing in connectivity.
    \end{enumerate}
\end{introdefinition}

We will see that by the Dold-Kan correspondence, (2) and (3) combined
are equivalent to the statement that $\cC_{\ge 0}$ admits geometric realizations and that the inclusion $\cC_{\ge 0} \subseteq \cC$ preserves them (Lemma \ref{lem:colim-comp-via-geom-real}).
Let us also mention that if $(\cC,\cC_{\geq 0})$ instead determines a $t$-structure,
and we replace colimits with limits in (2) and (3), then the above definition
precisely recovers left $t$-completeness (Lemma \ref{lem:left-t-comp-is-conn-comp}).

We shall study many examples and see for instance that the usual weight structures
on spectra and the derived category of a ring are left complete, whereas those on finite spectra and the perfect derived category are not.
The main structural result of the paper is the following.

\begin{introtheorem}[Proposition \ref{prop:weight-pdelta}, Theorem \ref{thm:upw-comp}, Corollary \ref{cor:upw-comp-idem}]\label{Theorem_intro}
Let $\cC$ be a stable category with a weight structure.
\begin{enumerate}
    \item There exists a ``left completion'' $\lc{\cC}$ of $\cC$, i.e.~a universal weight exact functor
        \[
            \cC \to \lc{\cC}
        \]
        to a left complete weight category which is an equivalence
        if and only if $\cC$ is left complete.


  \item The induced functor
        $\cC^{w\heartsuit} \to (\lc{\cC})^{w\heartsuit}$ is an idempotent completion and $\cC_{w<\infty} \to (\lc{\cC})_{w<\infty}$ is a fully faithful dense inclusion.

    \item If $\cC$ is left complete, then the subcategory $\cC_{\ge 0}$ is idempotent complete
        and obtained from the weight heart by freely adjoining geometric realizations,
        i.e.~$\cC_{\geq 0} = \cP^{\Delta^\op}(\cC^{w\heartsuit})$.
\end{enumerate}
\end{introtheorem}

In particular, if $\cC$ is both left and right complete, then the subcategories $\cC_{\ge 0}$ and $\cC_{w \le 0}$ admit universal characterizations determined by the weight heart $\cC^{w\heartsuit}$. In this case, the entire category $\cC$ can be recovered as the left and right completion of the bounded weight structure with the same heart, i.e. the stable envelope by Sosnilo's result (Theorem \ref{thm:vova}).

More generally, we denote by $\Kw(\mathcal{A})$ the $\infty$-category obtained from an additive $\infty$-category $\mathcal{A}$ as the left and right weight completion of its stable envelope. The notation is inspired by a well-known classical construction: for an ordinary additive category $\mathcal{A}$, we show (Theorem \ref{thm:chains}) that the $\infty$-category $\Kw(\mathcal{A})$ is equivalent to the $\infty$-category which has as objects chain complexes in $\mathcal{A}$ and as mapping anima the underlying anima of the hom-complexes. In particular, its homotopy category is the classical homotopy category of chain complexes. This  homotopy category is usually denoted by $K(\cA)$ \cite{Verdier}, and we use $\Kw(\mathcal{A})$ for the $\infty$-categorical analogue.
As a consequence of Theorem \ref{Theorem_intro} we obtain the following.

\begin{introcorollary}[Theorem \ref{thm_adjoint}]\label{cor:intro-kw}
The category $\Kw(\cA)$ only depends on the idempotent completion of $\cA$. For idempotent complete additive categories the functor $\cA \mapsto \Kw(\cA)$ is right adjoint to the functor $\cC \mapsto (\cC^{w\heartsuit})^\idem$ that assigns to a weight category the idempotent completion of its weight heart.
Moreover, this adjunction restricts to an equivalence
\[
    (-)^{w\heart} \colon \{\text{complete weight categories}\} \simeq \Cat^{\add,\idem} \noloc \cK(-)
\]
between (left and right) complete weight categories and idempotent complete additive categories.
\end{introcorollary}

In particular, if $f \colon \cC \to \cD$ is a weight exact functor between complete weighted categories
which restricts to an equivalence on weight hearts $\cC^{w\heart} \simeq \cD^{w\heart}$,
then $f$ is already an equivalence.
Note that the analogous statement for $t$-structures fails,
as the example of the forgetful functor $\cD(\Z) \to \Sp$ shows;
this functor is $t$-exact, preserves all limits and colimits,
and restricts to an equivalence on $t$-hearts (both of which are the category of abelian groups $\Ab$).
Moreover, the $t$-structures on both source and target of this functor are complete.
Nevertheless, the functor is far from an equivalence.

\subsection*{Examples and Applications}

Our main motivation for this theory comes from a number of important examples, some of which we now discuss in order of increasing complexity.

The first example is the derived category of the integers $\cD(\Z)$ with the weight structure induced from usual ($t$-)connective objects. The weight coconnectives $\cD(\Z)_{w \leq 0}$ are given by those objects $C$ with $H_i(C) = 0$ for $i > 0$ and $H_0(C)$ a free abelian group.
This is left and right weight complete as one easily checks. The weight heart consists of the category of free abelian groups $\mathrm{Proj}_\Z$ (not necessarily finitely generated) so that we recover the well-known equivalence
\[
    \cD(\Z) = \Kw(\mathrm{Proj}_\Z) \ .
\]
The subcategory $\cD(\Z)^\omega$ of perfect complexes inherits a weight structure with the restriction of connectives and coconnectives. It is neither left nor right complete.  The left and right weight completion $\Kw(\mathrm{Proj}^{\mathrm{fg}}_\Z)$  is given  by the full subcategory $\cD(\Z)^{\mathrm{fg}} \subseteq \cD(\Z)$ of those objects $C$ such that $H_i(C)$ is finitely generated for each $i$.
Another example is obtained by taking the category of compact spectra $\Sp^\omega$, which is neither left nor right complete. The left completion is given by the category of bounded below spectra with finitely generated homotopy groups (Example \ref{ex:left-completion}).

In order to deal with further examples and construct new weight structures, we establish a new general construction principle for weight structures on large categories, which we believe is of independent interest.

\begin{introtheorem}[Theorem \ref{thm:gen-weight}]\label{introthm:gen-weight}
Let $\cC$ be a cocomplete stable category, $\kappa$ a regular cardinal,
and $S \subseteq \cC^\kappa$ a small set of $\kappa$-compact objects.\footnote{For example any small set of objects in a presentable category.} Then
\[
\cC_{\ge 0}
=
\bigcap_{n \ge 1} S[-n]^{\perp}
=
\{ X \in \cC \mid \hom(s,X) \text{ is connective for all } s \in S \}
\]
defines a weight structure on $\cC$.
\end{introtheorem}

Theorem \ref{introthm:gen-weight} recovers the standard weight structure on spectra by
taking $S=\{\mathbb{S}\} \subseteq \Sp$. In this case the weight
connective objects are the connective spectra, and the heart consists
of direct sums of the sphere spectrum. The coconnectives $\Sp_{w \leq 0}$ are those spectra $X$ with $H_*(X;\Z) = 0$ for $* > 0$ and $H_0(X;\Z)$ a free abelian group.

This example already appeared in work of Bondarko (\cite[Section~4.6]{Bondarko10}),
who proved Theorem~\ref{introthm:gen-weight} under the additional assumption
$\kappa=\omega$ (\cite[Theorem~2.3.4]{Bondarko22}); the same result
was independently obtained by Pauksztello
(\cite[Theorem~5, Example~8]{Pauksztello}).
The resulting weight structure on $\Sp$
is left complete but not right complete as one can easily verify. We will identify its right weight completion as the category of modules over the (spectral) integral Steenrod algebra $A = \hom_{\Sp}(\Z,\Z)$.

\begin{introproposition}[Proposition \ref{prop:right-comp-steenrod}]\label{prop_intro_standard}
The functor $\Sp \to \LMod(A)$ that sends a spectrum $X$ to its homology $\Z \otimes X$ considered as a left module over $A$ exhibits the category of left modules $\LMod(A)$ as the weight completion of $\Sp$.
Here $\LMod(A)$ carries a weight structure in which the connectives are precisely the underlying connectives.
\end{introproposition}

This result clarifies the extent to which a spectrum can be recovered
from its homology, regarded as a highly structured Steenrod module: namely precisely as an object in the right weight completion, and thus bounded below spectra can be completely recovered.
This is reflected in the (integral) Adams spectral
sequence, which computes the map
\[
\hom_\Sp(X,Y) \longrightarrow \hom_A(\Z \tensor X,\, \Z \tensor Y),
\]
whose target admits a Ext-type spectral sequence with $E_2$-page
$\Ext_{A_*}\big(H_*(X;\Z),\, H_*(Y;\Z)\big).$

Another interesting weight structure on spectra arises by applying Theorem \ref{thm:gen-weight} with $S=\{\mathbb{Z}\} \subseteq \Sp$ (this requires the case $\kappa = \omega_1$ which was not known before).
In this case the weight coconnective spectra are those spectra whose
positive homotopy groups vanish and whose $\pi_0$ is free. We refer to this weight structure as the Anderson weight structure, because of the following surprising result:

\begin{introproposition}[Proposition \ref{prop:anderson-weight}, Proposition \ref{prop_andersoncompletion}]\label{prop_intro_anderson}
The weight heart of the Anderson weight structure is given by the subcategory $\{I_\Z\}^{\oplus} \subseteq \Sp$  spanned by arbitrary direct sums of the
Anderson dual of the sphere $I_\Z \in \Sp$. This weight structure
is right complete but not left complete. The functor
\[
    \Sp \to \RMod(A),\qquad X \mapsto \hom_\Sp(\Z, X)
\]
exhibits the category of right $A$-modules $\RMod(A)$ as its left weight completion.
\end{introproposition}

We note that the integral Steenrod algebra $A$ is a ring spectrum with involution
$A \simeq A^{\op}$ induced by Anderson duality. Consequently, there is an equivalence of module categories $\LMod(A) \simeq \RMod(A)$. It follows that the completions of the two weight structures on spectra (Propositions \ref{prop_intro_standard} and \ref{prop_intro_anderson}) are equivalent.
This should be regarded as the correctly weight-completed refinement of the fact that the functor
\[
    \Sp \to \Sp, \qquad X \mapsto I_\Z \tensor X
\]
is an equivalence on weight bounded objects (Corollary \ref{cor:sp-st-vs-and}), where the source is equipped with the standard weight structure and the target with the Anderson weight structure.
However, this functor is not weight-exact on its entire domain; while it preserves weight coconnective objects, it does not preserve weight connective ones (Remark \ref{remark_bounded} and Remark \ref{remark_weight}).

\subsection*{Weak Weight and $t$-completion}

The examples discussed above suggest that weight structures arising in practice are often only left complete or only right complete. In many situations, however, a given weight structure is accompanied by a $t$-structure with the same connective objects; such pairs are called \emph{adjacent}. A typical phenomenon is that a category $\cC$ is left complete with respect to its weight structure and right complete with respect to the adjacent $t$-structure.
This occurs, for instance, for the standard weight structure on spectra.

At first glance, this asymmetry may seem surprising. However, it reflects a fundamental difference in the nature of the two notions of completeness: left weight completeness and right $t$-completeness are governed by colimit conditions, whereas right weight completeness and left $t$-completeness are governed by limit conditions. In practice, stable categories are far more often well-behaved with respect to colimits, or even generated under colimits by a controlled class of objects. This explains why left weight completeness and right $t$-completeness frequently appear together.
Another perspective on this asymmetry is that in the connective part of a $t$-structure,
the Postnikov tower gives a descending filtration, so here completeness should be a limit condition,
whereas in the connective part of a weight structure, weight complexes give ascending filtrations,
where completeness should be a colimit condition.

A natural question is whether such categories can still be recovered from their heart. In the bounded case, this is achieved via the stable envelope. In the present setting, one is led to consider a two-step process: first passing to the left weight completion, and then forming a right $t$-completion. Formally, for an additive category $\cA$, we define
\[
    \K(\cA) \coloneqq \trc{(\Kw(\cA)_{> - \infty})}
\]
as the right $t$-completion of $\Kw(\cA)_{> -\infty}$, which itself is the left weight completion of the stable envelope of $\cA$. We also write $\K^{\op}(\cA) \coloneqq \K(\cA^\op)^\op$ which is dual in the sense that it performs right weight completion and left $t$-completion.

At this point, however, a serious difficulty arises. In general, the category $\Kw(\cA)$ need not carry a $t$-structure adjacent to the weight structure, so that $t$-completion is a priori not defined.
One of the key insights of this paper is that this obstruction can be overcome by working in a more flexible framework. We introduce suitable weakenings of $t$-structures and show that the theory of completeness and completion for $t$-structures extends to this setting. In particular, $t$-completion can be carried out in a substantially more general context than previously available, without requiring the existence of finite limits or a genuine adjacent $t$-structure, while retaining its essential formal properties.

To make this precise, we consider weak $t$-structures (and similarly weak weight structures), which still allow for a meaningful notion of completeness and completion.

\begin{introdefinition}[Definition \ref{def:weak-t}, Remark \ref{rem:t-saturated}]
A weak $t$-structure on a stable category $\cC$ consists of a full subcategory $\cC_{\geq 0}$ closed under suspensions, retracts, and extensions such that for every $X \in \cC$ there is a fiber sequence
\[
X' \to X \to X''
\]
with $X' \in \cC_{> -\infty}$ and $X'' \in \cC_{t \leq 0} \coloneqq (\cC_{\geq 1})^\perp$.
\end{introdefinition}

One advantage of this flexibility is that it allows us, for instance, to define spectrum objects in prestable categories without assuming the existence of finite limits, generalizing the treatment of Lurie in \cite{SAG} (see Theorem \ref{thm:prestable-spectrum}).
In particular, we have $\Sp(\cC_{\geq 0}) \coloneqq \trc{(\cC_{> - \infty})},$
so that in this language we get an alternative description $\K(\cA) \simeq \Sp(\cP^{\Delta^\op}(\cA))$.

Similar to weak $t$-structures, we also introduce the notion of weak weight structures (Definition \ref{def:weak-w}), and extend the notions of weight completeness and completion to this setting (i.e. generalize Definition \ref{def_weightcomp} and Theorem \ref{Theorem_intro}). In fact, the entire paper is developed in the generality of weak weight structures. This additional flexibility is essential in practice: for instance, on the perfect derived category of a qcqs scheme $X$, one typically obtains only a weak weight structure (see Example \ref{ex:scheme-weak-weight}). \\

Our main theorem  shows that categories equipped with compatible weight and $t$-structures can still be recovered from their heart via this generalized completion procedure.

\begin{introtheorem}[Theorem \ref{thm:k}]\label{thm_recovery}
    A stable $\infty$-category $\cC$ with a full subcategory $\cC_{\geq 0} \subseteq \cC$ closed under finite colimits and extensions is equivalent to $\K(\cC^{w \heartsuit})$ (compatibly with the connectives)
    if and only if
\begin{enumerate}
    \item $(\cC, \cC_{\geq 0})$ is a right complete weak $t$-structure, and
    \item $(\cC_{> -\infty}, \cC_{\geq 0})$ is a left complete weight structure.
\end{enumerate}
\end{introtheorem}

As a consequence of Theorem \ref{thm_recovery}, we obtain equivalences expressing familiar stable categories as being built from additive ones:
\[
\Sp \simeq \K(\mathrm{Proj}_\mathbb{S}), \qquad
\Sp^{\mathrm{fg}} \simeq \K(\mathrm{Proj}_\mathbb{S}^{\mathrm{fg}}), \qquad
\Sp \simeq \K^\op(\{I_\Z\}^{\oplus}).
\]
More generally, for any connective $\mathbb{E}_1$-ring $R$, there is an equivalence
\[
\LMod(R) \simeq \K(\mathrm{Proj}_R)  \tag*{\text{(Example \ref{example_modules}).}}
\]
Moreover, if $\cA$ is a Grothendieck abelian category with enough projectives, then
\[
\cD(\cA) \simeq \K(\cA_{\proj})  \tag*{\text{(Corollary \ref{abelian}).}}
\]
Finally, we construct a weight structure on the category of Tate objects $\Tate(\cC)$ (in the sense of Hennion \cite{Hennion}) associated to a weighted category, and give an explicit description of its heart (Theorem \ref{thm:tate-weight}) as the full subcategory of $\Ind\Pro(\cC^{w\heartsuit})$ spanned by retracts of objects of the form
\[
\bigoplus_I a_i \oplus \prod_J b_j \qquad a_i, b_j \in \cC^{w\heartsuit}
\]
generalizing the classical notion of Tate vector spaces.
Under suitable hypotheses on $\cC$, we further show that there is a `completion' functor
\[
\Tate(\cC) \to \K\bigl(\Tate(\cC)^{w\heartsuit})
\]
with the key input being the existence of a weak $t$-structure (Lemma \ref{lem:tate-upw-comp} and Corollary \ref{coro_weak}).
As another motivation for the formalism of weak $t$-structures,
we see in Example \ref{ex:only-weak-t} that even in the nice case $\cC = \Sp^\omega$, there exists an adjacent \emph{weak} $t$-structure on $\Tate(\Sp^\omega)$, but not an adjacent $t$-structure.
In forthcoming work \cite{CondDual}, we apply these results on Tate categories to algebraic examples such as $\cD(\mathbb{Z})^\omega$, and relate these constructions to a full subcategory of condensed $\Ind$-objects in $\cC$.

\subsection*{Outline of the Sections}

Section \ref{sec:conn} recalls the notion of connectivity structure from \cite{Lawson} and introduces the notions of (co)limit completeness that underlie all the completeness conditions considered in this paper.
Both (weak) weight and (weak) $t$-structures admit underlying stable (co)connectivity structures,
and the various definitions of completeness can already be formulated on this level.
This allows us to give a general theory of functors commuting with certain sequential (co)limits
for (co)connectivity reasons that is crucial for the rest of the paper,
which we highlight via a number of examples.

Section \ref{sec:weight} introduces weakly weighted categories, a weakening of the notion of weight structure in which the weight decompositions are allowed to have bounded ``defect''.
In the case this defect is bounded by some $0 \leq \ell < \infty$,
this coincides up to reformulation with a notion studied by Levy--Sosnilo \cite{Ishan-Vova}.
After developing the basic formalism and introducing the related category $\wWCat$
of weakly weighted categories and exact functors of bounded weight amplitude,
we show that weight complexes still exist in this generality.

Section \ref{sec:upw-comp} is the technical heart of the paper.
We define left weight completeness of weakly weighted categories $\cC$
and prove Theorem \ref{Theorem_intro} in this generality.
The left completion $\lc{\cC}$ is identified with a canonical full subcategory
of $\Ind(\cC_{w< \infty})$,
and we highlight important consequences
such as the existence of a weight complex functor valued in $\Ind$-objects (Corollary \ref{cor:weight-complex}).

Section \ref{sec:chains} defines $\Kw(\cA)$ and proves Corollary \ref{cor:intro-kw}.
We show in Theorem \ref{thm:chains}
that the notation is warranted by proving that if $\cA$ is an (ordinary) additive category,
there is a weight exact equivalence between $\Kw(\cA)$ and an $\infty$-category of chain complexes in $\cA$
whose homotopy category recovers the classical $K(\cA)$, one of the motivating examples for weight structures.
Combining this with Theorem \ref{Theorem_intro}(3), we recover in Corollary \ref{cor:d-}
Lurie's universal property of the bounded below derived category of an abelian category with enough projectives.

Section \ref{sec:gen} proves Theorem \ref{introthm:gen-weight},
which crucially relies on the results of Appendix \ref{sec:app}.
We apply this to define the standard weight structure on $\LMod(R)$
for a connective $\E_1$-ring $R$ and study its completeness properties (Proposition \ref{prop:lmod-weight}).

Section \ref{sec:t} develops the theory of weak $t$-structures and
their right completions, paralleling Sections \ref{sec:weight} and
\ref{sec:upw-comp}.
The main result, Theorem \ref{thm:t-dw-comp}, is the direct analogue of
Theorem \ref{Theorem_intro} and recovers the classical right completion
when applied to an ordinary $t$-structure.

Section \ref{sec:prestable} applies the completion theory to prestable categories.
We define two kinds of colimit completions,
one ``weight-like'' and one ``$t$-like''.
The latter lets us define spectrum objects for prestable categories,
and the main result of the section Theorem \ref{thm:prestable-spectrum}
is on the properties of this construction.

Section \ref{sec:adj} introduces the functor
$\cA \mapsto \K(\cA)$ on additive categories and proves
Theorem \ref{thm_recovery}.
We apply this to $\LMod(R)$ for a connective $\E_1$-ring $R$
and to the derived category $\cD(\cA)$ of a Grothendieck abelian
category with enough projectives (Proposition \ref{prop:da-as-k}).

Section \ref{sec:anderson} studies weight structures on $\Sp$.
We construct the Anderson weight structure
on spectra mentioned in the introduction
and prove Propositions \ref{prop_intro_standard} and \ref{prop_andersoncompletion}.
In particular, we show how passing from spectra to modules
over the Steenrod algebra can be seen as a weight completion,
and passing backwards as a $t$-completion.

Section \ref{sec:tate} constructs weight structures on Tate categories
under minimal assumptions (Theorem \ref{thm:tate-weight}).
Moreover, we investigate sufficient conditions
for the existence of adjacent weak $t$-structures (Proposition \ref{prop:adj-t-on-tate}, Corollary \ref{coro_weak}),
and show that even in very nice examples such as that of finite spectra,
one can in general only expect a \emph{weak} adjacent $t$-structure
(Example \ref{ex:only-weak-t}).

Appendix \ref{sec:app} proves a crucial tool required
for the proof of Theorem \ref{introthm:gen-weight}.
Concretely, for a given ordinal $\gamma$,
we give concrete conditions
under which the limit of an inverse system of connective spectra
indexed by $\gamma^\op$ is again connective.
For $\gamma = \omega$ this recovers the usual Mittag--Leffler criterion.

\subsection*{Acknowledgements}

We would like to thank Thorger Geiß and Maxime Ramzi for many helpful conversations. We are especially grateful to Maxime Ramzi for allowing us to include Lemma \ref{lem:mittag-leffler-an}, which we learned from him. We also thank Vova Sosnilo for a discussion that led to Proposition \ref{prop_intro_anderson}. We are grateful to the homotopy theory group in Bielefeld for a joint research seminar, and in particular to Victor Saunier and Fabian Hebestreit for helpful discussions related to weight structures.
Finally, we would like to thank Peter J\o rgensen, Jiacheng Liang, Marius Nielsen, Kabeer Manali Rahul, and Victor Saunier
for helpful feedback on an earlier version of this article.

Both authors were funded by the Deutsche Forschungsgemeinschaft (DFG, German Research Foundation) – Project-ID 427320536 – SFB 1442, as well as under Germany’s Excellence Strategy EXC2044/2–390685587, Mathematics Münster: Dynamics–Geometry–Structure.

\subsection*{Notational Conventions}\label{sec:notation}

\begin{itemize}
    \item From now on, i.e.~in the body of the paper, ``category'' means ``$\infty$-category''.

    \item We denote mapping anima by $\map$, and hom spectra by $\hom$.

    \item Given a category $\cC$ we denote its idempotent completion by $\cC^\idem$.
        If $\cC$ is additive, then $\cC^\widem$ denotes its weak idempotent completion,
        which is the smallest full subcategory of $\cC^\idem$ containing $\cC$
        and closed under cofibers of split monomorphisms.

    \item By a dense subcategory $\cC$ of $\cD$
        we mean a full subcategory $\cC \subseteq \cD$
        so that every object in $\cD$ is a retract of an object in $\cC$.

    \item A $t$-category is a stable category with $t$-structure.
        A weighted category is a stable category with weight structure.

    \item If $\cC$ is a semiadditive category admitting arbitrary sums,
        then for a full subcategory $S \subseteq \cC$ we denote by $S^{\oplus} \subseteq \cC$
        the full subcategory on sums of objects in $S$.

    \item If $\cC$ is stable and $\cD \subseteq \cC$ a full subcategory, then
        we let $\lperp{\cD} \subseteq \cC$ be the full subcategory on objects $c \in \cC$
        for which all maps $c \to d$ for $d \in \cD$ are nullhomotopic. Dually for $\rperp{\cD}$.

    \item We will always use homological indexing conventions.

    \item Occasionally, instead of repeating an entire sentence with minimal changes,
        we will highlight the alternatives like \alt{this}.
        For example, we might write ``Suppose that $\cC$ is a category admitting countable coproducts \alt{countable products} and pushouts \alt{pullbacks}. Then $\cC$ admits sequential colimits \alt{sequential limits}.''
\end{itemize}

\section{Connectivity structures}\label{sec:conn}

The aim of this section is to recall the notion of
connectivity structure from \cite{Lawson}.
The idea is that this abstracts the notion of connective
maps of spaces to a general category.
The authors of \cite{Lawson} use this mainly to define a notion of skeleta
and cellular approximations which, although this is not mentioned
by the authors, essentially recovers the notion of weight structures in the setting of stable categories.
Our aim instead is to investigate
how to use the notions of connectivity
to prove that functors preserve certain kinds of (co)limits
for (co)connectivity reasons.
Concretely, we would like to cover the following examples
(which we will reprove below using the developed formalism).

\begin{example}[Lemmas \ref{lem:ex-group}, \ref{lem:ex-tensor}, \ref{lem:ex-t-struc} and \ref{lem:ex-suspension}]\label{ex:enum}
    ~
\begin{enumerate}
    \item For a finite group $G$, the functors $(-)^{hG},(-)_{hG},(-)^{tG} \colon \Sp^{BG} \to \Sp$ preserve the (co)limits of the Whitehead and Postnikov towers.

    \item If $X \in \Sp$ is bounded below with finitely generated homotopy (equivalently, homology) groups, then $X \tensor - \colon \Sp \to \Sp$
        preserves uniformly bounded below
        products, and hence uniformly bounded below sequential limits.
        Moreover, $\hom_\Sp(X,-) \colon \Sp \to \Sp$
        preserves uniformly bounded above filtered colimits.

    \item Suppose $A$ is a finite type anima (i.e.~a CW complex of finite type),
        and that $f \colon \cC \to \cD$ is an exact functor of left complete $t$-categories.
        Then $\cC$ and $\cD$ admit $A$-indexed colimits,
        and if $f$ restricts to $\cC_{t \geq 0} \to \cD_{t \geq -n}$
        for some $n$, then $f$ preserves them.

    \item For $X \in \An$ we have $\Sigma_+^\infty X \xto{\simeq} \lim_n \Sigma_+^\infty\tau_{\leq n}X$ in $\Sp$.
\end{enumerate}
\end{example}

Let us recall the definition of connectivity structures from \cite{Lawson}.
Their basic notion is that of \(k\)-connected maps. In the setting of topoi,
these coincide with \((k+1)\)-connective maps, that is, maps whose fibers are
\((k+1)\)-connective. In what follows, we will instead take connective maps as
the primary notion.

\begin{definition}[{{\cite[Definition 2.1]{Lawson}}}]\label{def:conn-struc}
    Let $\cC$ be a category admitting pullbacks.
    A connectivity structure on $\cC$ consists
    of a collection of $k$-connective maps for every $k \in \Z$
    satisfying the following axioms:
    \begin{enumerate}
        \item $k$-connective maps are $(k-1)$-connective.

        \item Equivalences are $\infty$-connective,
            meaning $k$-connective for all $k$.
            For convenience, we say that all morphisms
            are $-\infty$-connective.

        \item $k$-connective maps are stable under pullback.

        \item If $f$ and $g$ are composable morphisms, then:
            \begin{enumerate}
                \item If $f$ and $g$ are $k$-connective, so is $gf$.
                \item If $gf$ is $k$-connective and $f$ is $(k-1)$-connective,
                    then $g$ is $k$-connective.
                \item If $gf$ is $k$-connective and $g$ is $(k+1)$-connective,
                    then $f$ is $k$-connective.
            \end{enumerate}
    \end{enumerate}
    Given another category $\cD$ with connectivity structure
    we say that a functor $f \colon \cC \to \cD$ has amplitude $\geq a$
    for $-\infty \leq a \leq \infty$
    if for large enough $k$ it sends $k$-connective maps
    to $(k+a)$-connective maps.\footnote{This is slightly weaker than the notion ``adding at least $a$ to connectivity'' of \cite[Definition 2.3]{Lawson}.}
    We say $f$ has bounded below amplitude if there exists such an $a \in \Z$.
    As these are closed under composition,
    we obtain a category $\Cat^{\cn}$ of categories with connectivity structure and functors of bounded below amplitude.
\end{definition}

\begin{example}[{{\cite[Example 2.14]{Lawson}}}]\label{ex:conn-topos}
    Let $\cX$ be a topos.
    Then $\cX$ admits a connectivity structure
    where a map is $k$-connective
    precisely if it is $k$-connective in the sense of \cite[6.5.1.10]{HTT}.
    For $\cX = \An$,
    this recovers the usual notion of connectivity;
    a map $f \colon X \to Y$ of anima
    is $k$-connective precisely if for all choices of basepoints,
    $\pi_j(f)$ is an isomorphism for $j < k$ and $\pi_k(f)$ is a surjection.
    Equivalently, $f$ is $k$-connective precisely if all the fibers are $k$-connective for all choices of basepoints,
    where a $0$-connective anima is a nonempty one
    and $X \in \An$ is $k$-connective for $k \geq 1$
    if $\pi_\ell(X) \cong \pi_\ell(*)$ for $\ell \leq k-1$.
\end{example}

\begin{definition}[{{\cite[Definition 2.25]{Lawson}}}]
    Let $\cC$ be a pointed category equipped with a connectivity
    structure. We say an object $X \in \cC$ is $k$-connective
    if the map $X \to *$ is $k$-connective.\footnote{By \cite[Proposition 2.16]{Lawson} this is equivalent to $* \to X$ being $(k-1)$-connective.}
    We write $\cC_{\geq k} \subseteq \cC$ for the full subcategory of $k$-connective objects, and $\cC_{\geq \infty} = \bigcap_{k \geq 0} \cC_{\geq k}$.
\end{definition}

\begin{definition}[{{\cite[Definition 2.24]{Lawson}}}]
    Let $\cC$ be a stable category.
    A connectivity structure on $\cC$ is stable
    if any given map $f$ is $k$-connective if and only if $\Omega f$ is $(k-1)$-connective.
    In particular, we have that $\cC_{\geq k} = \{\Sigma^k X \mid X \in \cC_{\geq 0}\}$.
\end{definition}

The following Lemma shows that the definition
of a stable connectivity structure simplifies considerably,
essentially being completely determined by $\cC_{\geq 0} \subseteq \cC$.

\begin{lemma}\label{lem:stable-conn-struc}
    Let $\cC$ be a stable category.
    \begin{enumerate}
        \item If $\cC_{\geq 0} \subseteq \cC$
            is a full prestable subcategory in the sense of \cite[Appendix C]{SAG}, i.e.~closed under finite colimits and extensions\footnote{Equivalently: contains 0, and is closed under extensions and cofibers.}
            then this determines a stable connectivity structure
            on $\cC$ by setting $\cC_{\geq k} = \{\Sigma^k X \mid X \in \cC_{\geq 0}\}$
            and defining a map to be $k$-connective if its fiber lies in $\cC_{\geq k}$.

        \item If $\cC$ is equipped with a stable connectivity structure,
            then $\cC_{\geq 0} \subseteq \cC$ is closed under finite colimits
            and extensions, and $\cC_{\geq n} \subseteq \cC_{>-\infty}$ is closed under retracts.
            Moreover, the connectivity structure defined in (1) agrees with the given one on $\cC$.

        \item If $\cD$ is a prestable category,
            then it embeds fully faithfully into its Spanier-Whitehead stabilization
            $\SW(\cD) \coloneqq \colim_n (\cD \xto{\Sigma} \cD \xto{\Sigma} \cD \to \cdots)$ (colimit taken in $\Cat$),
            which is the initial stable category receiving a right exact functor from $\cD$.
            This determines a stable connectivity structure $\SW(\cD)_{\geq 0} = \cD$
            which is closed under retracts in $\SW(\cD)$.

        \item If $(\cC,\cC_{\geq 0})$ is a stable connectivity structure,
            then the canonical map $\SW(\cC_{\geq 0}) \to \cC$ is fully faithful
            with essential image $\cC_{>-\infty} = \bigcup_{n \geq 0} \cC_{\geq -n}$.
    \end{enumerate}
\end{lemma}
\begin{proof}
    Except the claim about retracts,
    the first two points were noted in \cite[Remark 2.30]{Lawson}.
    By shifting and induction, it is enough to show that if $X \in \cC_{\geq 0}$ is a retract of $Y \in \cC_{\geq 1}$,
    then also $X \in \cC_{\geq 1}$. Let $C = \cofib(X \to Y) \in \cC_{\geq 0}$, so that $Y \simeq X \oplus C$.
    Then $X$ sits in an extension $Y \to X \to \Sigma C$
    where the outer terms lie in $\cC_{\geq 1}$, so we conclude.
    Point (3) follows from (1), (2) and the universal property of $\SW(-)$ shown in \cite[Proposition C.1.1.7(b)]{SAG}.
    For (4), we have a cone
    \[\begin{tikzcd}
        {\cC_{\geq 0}} && {\cC_{\geq 0}} && {\cC_{\geq 0}} & \cdots \\
        &&& \cC
        \arrow["\Sigma", hook, from=1-1, to=1-3]
        \arrow["\inc"', curve={height=12pt}, hook, from=1-1, to=2-4]
        \arrow["\Sigma", hook, from=1-3, to=1-5]
        \arrow["{\Omega\inc}"{description}, hook, from=1-3, to=2-4]
        \arrow[from=1-5, to=1-6]
        \arrow["{\Omega^2\inc}"{description}, hook', from=1-5, to=2-4]
    \end{tikzcd}\]
    which induces a comparison functor $f \colon \SW(\cC_{\geq 0}) \to \cC$
    that clearly has essential image $\cC_{>-\infty}$.
    Since all the functors in this cone are fully faithful,
    it follows that $f$ is fully faithful.
\end{proof}

\begin{notation}
    Because of the above Lemma, we will denote a stable connectivity structure by $(\cC,\cC_{\geq 0})$.
\end{notation}

\begin{remark}\label{rem:prestable-conn}
    If $\cD$ is a prestable category,
    then the connectivity structure on $\SW(\cD)$
    induces a good notion of connectivity on $\cD$
    by $\cD_{\geq n} = \SW(\cD)_{\geq n} = \{\Sigma^n X \mid X \in \cD\} \subseteq \cD$
    for $n \geq 0$, so that $\cD_{\geq 0} = \cD$.
    Since $\cD$ generally does not admit pullbacks, this is not a connectivity structure in the strict sense
    of Definition \ref{def:conn-struc},
    but it still makes sense to refer to it as such since
    it essentially behaves like one,
    and we can perform all arguments in $\SW(\cD)$ where it constitutes a genuine connectivity structure.
    We will come back to this in Section \ref{sec:prestable}.
\end{remark}

\begin{convention}
    From now on, unless specified otherwise, we will always implicitly assume
    that a connectivity structure on a stable category is stable.
\end{convention}

\begin{remark}\label{rem:cocon}
    By passing to opposite categories,
    everything in this section dualizes to a notion of coconnectivity structure
    satisfying the dual results.
    Concretely, we say that a coconnectivity structure on $\cC$ is the data
    of a connectivity structure on $\cC^\op$.
    For example, if $\cC$ is stable
    there is a dual of Lemma \ref{lem:stable-conn-struc}
    which says that a coconnectivity structure on $\cC$
    is the same data as a full subcategory $\cC_{\leq 0} \subseteq \cC$
    closed under finite limits and extensions.
    A functor between stable coconnectivity structures
    has amplitude $\leq n$ precisely if it restricts to $\cC_{\leq 0} \to \cD_{\leq n}$.
\end{remark}

\begin{observation}\label{obs:conn}
    Let $\cC$ be a category with connectivity structure.
    \begin{enumerate}
        \item For any small category $K$,
            the category $\Fun(K,\cC)$ inherits a pointwise connectivity structure.
            Similarly, given $X \in \cC$ we can declare a map in $\cC_{X/}$ or $\cC_{/X}$
            to be $k$-connective precisely if it is so in $\cC$,
            and this defines a connectivity structure on the slices $\cC_{/X}$ and $\cC_{X/}$.

        \item If $\cC$ and its connectivity structure
            are stable, and $\cD$ is another such category,
            then an exact functor $f \colon \cC \to \cD$
            has amplitude $\geq a$ if and only if it restricts
            to $\cC_{\geq 0} \to \cD_{\geq a}$.

        \item If $\cC$ is a stable category with $t$-structure
            $(\cC_{t \geq 0},\cC_{t \leq 0})$
            \alt{weight structure $(\cC_{w \geq 0},\cC_{w \leq 0})$}
            then $\cC_{t \geq 0}$ \alt{$\cC_{w \geq 0}$}
            determines a stable connectivity structure on $\cC$
            via Lemma \ref{lem:stable-conn-struc}.
    \end{enumerate}
\end{observation}

\begin{definition}\label{def:grow-in-conn}
    Let $\cC$ be a category with connectivity structure.
    A direct sequential diagram $X_\bullet \colon \N \to \cC$
    is growing in connectivity if for all $n \geq 0$
    the map $X_k \to X_{k+1}$ is $n$-connective for large enough $k$.
    Applying this definition to the slice category $\cC_{/X}$
    for some $X$, this defines the notion of a sequential
    cone $X_\bullet \Rightarrow \const X$ growing in connectivity;
    namely that for every $n$ the map $X_k \to X_{k+1}$ (equivalently, the map $X_k \to X$)
    is $n$-connective for large enough $k$.
    We also consider the analogously defined notions for inverse sequential diagrams $\N^\op \to \cC$ (and their cones).
\end{definition}

\begin{observation}\label{obs:pres-amp}
    If $f \colon \cC \to \cD$ is a functor of bounded below amplitude
    between categories with connectivity structures,
    then composition with $f$ preserves sequential diagrams growing in connectivity.
\end{observation}

As we will see below (Lemma \ref{lem:cn-cone-colim}), the point of the definition of these kinds of sequential diagrams
will be that under certain assumptions on the target $\cD$ a functor of bounded below amplitude
automatically sends sequential cones growing in connectivity to (co)limiting cones.

\begin{example}
    Endow $\An$ with the canonical connectivity structure
    from Example \ref{ex:conn-topos}.
    The cell structure of a CW complex $X$
    defines a sequential cone $X_\bullet \colon \N \to \An_{/X}$
    growing in connectivity in $\An$.
    Similarly, if $\cC$ is a stable category with weight structure
    and $X \in \cC$, there exists an associated
    weight complex $X_{\leq 0} \to X_{\leq 1} \to X_{\leq 2} \to \cdots$
    over $X$ which defines a sequential cone growing in (weight) connectivity
    (cf.~Lemma \ref{lem:weight-cplx}).
\end{example}

\begin{example}
    Given $X \in \An$, the Postnikov tower
    yields an inverse sequential cone $\const X \Rightarrow \tau_{\leq \bullet}X$
    growing in connectivity.
\end{example}

\begin{example}\label{ex:dold-kan}
    Let $(\cC,\cC_{\geq 0})$ be a stable category with connectivity structure
    and $X \colon \Delta^\op \to \cC$ a simplicial object.
    We recall properties of the Dold-Kan correspondence
    from \cite[Theorem 1.2.4.1, Remarks 1.2.4.2 and 1.2.4.3]{HA}
    (compare also the prestable version of \cite[Theorem C.1.3.1]{SAG}):
    \begin{enumerate}
        \item There is an equivalence $\cC^{\Delta^\op} \simeq \cC^{\N}$ which associates to a simplicial object $X$
            the sequential diagram $|X|_0 \to |X|_1 \to |X|_2 \to \cdots$
            where $|X|_n \coloneqq \colim_{\Delta^\op_{\leq n}}X$.
        \item If the colimits exist, then
            $|X| \coloneqq \colim_{\Delta^\op}X \simeq \colim_n |X|_n$.

        \item $\cofib(|X|_n \to |X|_{n+1})$ is a direct summand of $\Sigma^{n+1}X_{n+1}$,
            hence is $(n+k+1)$-connective whenever $X_{n+1}$ is $k$-connective
            (since $\cC_{\geq 0}$ is closed under retracts in $\cC_{>-\infty}$
             by Lemma \ref{lem:stable-conn-struc}).

        \item In particular, if $X$ is uniformly bounded below,
            i.e.~factors through $\cC_{\geq -k}$ for some $k$,
            then $(|X|_n)_n$ is growing in connectivity.
    \end{enumerate}
    Thus uniformly bounded below geometric realizations are examples
    of sequential colimits growing in connectivity.
    In fact, point (3) shows that the equivalence in (1)
    restricts to an equivalence between $(\cC_{\geq 0})^{\Delta^\op}$
    and the full subcategory of $(\cC_{\geq 0})^{\N}$
    on sequential diagrams $X_\bullet$ where $X_{n+1}/X_n \in \cC_{\geq n+1}$
    for $n \geq 0$.
    We found the description via sequential diagrams to be more convenient for most of our presentation, with the exception of Proposition \ref{prop:weight-pdelta} below.
\end{example}

\begin{definition}
    We say that a connectivity structure on a category $\cC$ is colimit \alt{limit} complete if
    \begin{enumerate}
        \item All $\infty$-connective maps are equivalences.

        \item $\cC$ admits colimits \alt{limits} of sequential diagrams growing in connectivity.

        \item For $k \in \Z$ and a sequential diagram growing in connectivity
            $X_\bullet \colon \N \to \cC$ \alt{$X_\bullet \colon \N^\op \to \cC$}
            all of whose maps are $k$-connective,
            also $X_0 \to \colim_n X_n$ \alt{$\lim_n X_n \to X_0$} is $k$-connective.\footnote{Equivalently,
                we can ask that for each $k \in \Z$, the full subcategory of $\Ar(\cC)$ on $k$-connective
                maps is closed under colimits \alt{limits} of sequential diagrams growing in connectivity
                (with respect to the pointwise connectivity structure).}
    \end{enumerate}
\end{definition}

\begin{observation}\label{obs:colim-comp}
    If $\cC$ (and its connectivity structure) is stable, the above definition simplifies to:
    \begin{enumerate}
        \item $\cC_{\geq \infty} \coloneqq \bigcap_{n \geq 0} \cC_{\geq n} = 0$.
        \item $\cC$ admits sequential colimits \alt{limits} growing in connectivity.
        \item $\cC_{\geq 0}$ is closed under sequential colimits \alt{limits} growing in connectivity.
    \end{enumerate}
\end{observation}

\begin{warning}
    In \cite{Lawson}, a connectivity structure is called left complete if it satisfies axiom (1), namely that all $\infty$-connective maps are equivalences. In this paper, we will instead refer to this property as ``left separated'' or ``hypercomplete''. Our use of the term completeness is motivated by its compatibility with existing notions in the literature, for instance recovering the notion of a left complete $t$-structure (see Lemma \ref{lem:left-t-comp-is-conn-comp} below).
\end{warning}

Using Example \ref{ex:dold-kan} we can reformulate (2)+(3)
of Observation \ref{obs:colim-comp} using geometric realizations.

\begin{lemma}\label{lem:colim-comp-via-geom-real}
    Let $(\cC,\cC_{\geq 0})$ be a stable connectivity structure.
    The following are equivalent:
    \begin{enumerate}
        \item $\cC$ admits sequential colimits growing in connectivity, and
            $\cC_{\geq 0} \subseteq \cC$ is closed under them.
        \item $\cC_{\geq 0}$ admits geometric realizations
            and $\cC_{\geq 0} \subseteq \cC$ preserves them.
    \end{enumerate}
\end{lemma}
\begin{proof}
    The implication (1) $\Rightarrow$ (2) is immediate
    from Example \ref{ex:dold-kan}.
    For the converse, let $X_\bullet \colon \N \to \cC$
    be growing in connectivity.
    By cofinality, we may assume that $X_{n+1}/X_n \in \cC_{\geq n+1}$ for $n \geq 0$.
    We need to show that $X_\bullet$ admits a colimit in $\cC$
    which lies in $\cC_{\geq 0}$ if $X_\bullet$ is pointwise connective.
    By replacing $X_\bullet$ with $X_\bullet/X_0$
    we can assume that $X_\bullet$ is pointwise connective
    and linearly growing in connectivity.
    Now Example \ref{ex:dold-kan} shows that there is some
    $Y_\bullet \in (\cC_{\geq 0})^{\Delta^\op}$
    whose associated sequential object
    under the Dold-Kan correspondence is $X_\bullet$.
    In particular, $X_\bullet$ has colimit $|Y_\bullet| \in \cC_{\geq 0}$,
    and this is also a colimit in $\cC$ by assumption.
\end{proof}

\begin{lemma}\label{lem:cn-cone-colim}
    Let $f \colon \cC \to \cD$ be a map of categories with connectivity structures
    and suppose that $\cD$ is colimit \alt{limit} complete.
    If $f$ has bounded below amplitude,
    then it sends direct \alt{inverse} sequential cones growing in connectivity to colimit \alt{limit} cones.
\end{lemma}
\begin{proof}
    We consider the colimit complete case, the limit complete one being analogous.
    Let $n \geq 0$. It suffices to check that the comparison map $\colim_k f(X_k) \to f(X)$ is $n$-connective.
    We can write this map as the colimit in $\Ar(\cD)$ of the maps $f(X_k) \to f(X)$.
    Since $f$ has bounded below amplitude, by Observation \ref{obs:pres-amp}
    this sequential diagram in $\Ar(\cD)$ is still growing in connectivity
    and eventually takes values in $n$-connective maps,
    so its colimit is $n$-connective by colimit completeness of $\cD$.
\end{proof}

\begin{corollary}\label{cor:bdd-amp-pres}
    If both $\cC$ and $\cD$ are colimit complete \alt{limit complete},
    then an exact functor of bounded below amplitude $f \colon \cC \to \cD$
    preserves colimits \alt{limits} of sequential diagrams growing in connectivity.
\end{corollary}

\begin{lemma}\label{lem:conn-comp-inv-under-eqv}
    (Co)limit completeness is invariant under equivalence in $\Cat^\cn$.
    Concretely, let $\cC$ and $\cD$ be categories with connectivity structures
    and $F \colon \cC \simeq \cD$ an equivalence so that $F$ and $F^{-1}$ have bounded below amplitude.
    Then $\cC$ is (co)limit complete if and only if $\cD$ is.
\end{lemma}

\begin{remark}\label{rem:conn-comp-inv-under-eqv}
    In particular, if $(\cC,\cC_{\geq 0})$ and $(\cC,\cC'_{\geq 0})$
    are two stable connectivity structures on $\cC$ so that $\cC_{\geq 0} \subseteq \cC'_{\geq -a} \subseteq \cC_{\geq -(a+b)}$ for some $a,b\in \Z$, then $(\cC,\cC_{\geq 0})$ is (co)limit complete if and only if $(\cC,\cC'_{\geq 0})$ is.
\end{remark}

\begin{proof}
    Suppose that $\cD$ is colimit complete (the limit case being similar).
    If $f$ is an $\infty$-connective map in $\cC$, then $Ff$ is $\infty$-connective in $\cD$,
    hence an equivalence. Since $F$ is conservative, also $f$ is an equivalence.
    If $X_\bullet \colon \N \to \cC$ is a sequential diagram growing in connectivity,
    then so is $FX_\bullet$, and hence the latter admits colimit in $\cD$.
    But then one checks that $F^{-1}(\colim_n FX_n)$ is also the colimit of $X_\bullet$ in $\cC$.
    Finally, let $F$ resp.~$F^{-1}$ have amplitude $\geq -a$ resp.~$\geq -b$
    and $X_\bullet \colon \N \to \cC$ be a sequential diagram growing in connectivity
    where all maps are connective (the same argument works for $k$-connective maps).
    Let $X = \colim_n X_n$ be the colimit.
    Now pick $n$ large enough so that all maps $X_k \to X_{k+1}$ are $(a+b)$-connective for $k \geq n$.
    Using colimit completeness of $\cD$, we see that $FX_n \to \colim_{k \geq n} FX_k$ is $b$-connective,
    and hence $X_n \to X$ is connective. Since also $X_0 \to X_n$ is connective,
    we deduce that $X_0 \to X$ is connective, as desired.
\end{proof}

\begin{lemma}\label{lem:left-t-comp-is-conn-comp}
    Let $\cC$ be a stable category with $t$-structure.
    The following are equivalent:
    \begin{enumerate}
        \item The $t$-structure on $\cC$ is left complete.
        \item The connectivity structure $(\cC,\cC_{t \geq 0})$ is limit complete.
        \item The connectivity structure $(\cC,\cC_{t \geq 0})$ is both limit and colimit complete.
    \end{enumerate}
\end{lemma}

\begin{remark}
    As the proof shows, this crucially relies on the functoriality
    of truncations in $t$-structures, and therefore
    has no direct weight analogue.
\end{remark}

\begin{proof}
    Clearly (3) implies (2), and the implication (2) $\Rightarrow$ (1)
    follows from \cite[Prop.~1.2.1.19]{HA},
    noting that by the discussion above Proposition 1.2.1.17 of op.~cit.,
    the proof actually only needs the existence
    of sequential limits growing in connectivity
    and the closure of $\cC_{\geq 0}$ under them,
    as opposed to general countable products.
    Specifically, we can identify $\what{\cC} \coloneqq \lim_n \cC_{\leq n}$
    with the full subcategory of $\Fun(\N^\op,\cC)$ on those diagrams
    $(X_n)_n$ where $X_n \in \cC_{\leq n}$ and $X_{n+1} \to X_n$
    induces an equivalence on $\tau_{\leq n}$.
    Note that such sequential diagrams are in particular growing in connectivity,
    so limit completeness guarantees that the canonical map
    $\cC \to \what{\cC}$ admits a right adjoint given by taking the limit.
    We want to show that unit and counit of this adjunction form an equivalence.
    The counit at an object $(X_n)_n \in \what{\cC}$ is given by the map
    $\tau_{\leq n}\lim_k X_k \to X_n$ in the $n$-th component.
    To see that this is an equivalence, we show that $\fib(\lim_k X_k \to X_n) = \lim_{k \geq n} \fib(X_k \to X_n)$ lies in $\cC_{\geq n+1}$.
    But $(\fib(X_k \to X_n))_{k \geq n}$ is a sequential diagram growing
    in connectivity and taking values in $\cC_{\geq n+1}$, so this follows
    from limit completeness. Thus the counit is an equivalence.
    The unit $X \to \lim_n \tau_{\leq n}X$ is also clearly an equivalence,
    since by similar arguments its fiber lies in $\cC_{\geq \infty} = 0$.
    This concludes the proof of (2) $\Rightarrow$ (1).

    It remains to show that (1) implies (3).
    If $X \in \cC_{\geq \infty}$, then $\tau_{\leq n}X = 0$ for all $n$, and hence $X = 0$ by left completeness.
    Next we show that $\cC$ admits limits and colimits of sequential diagrams growing in connectivity.
    By left completeness, the functors  $\tau_{\leq n} \colon \cC \to \cC_{\leq n}$
    respectively $\tau_{\leq n} \colon \cC_{\leq n+1} \to \cC_{\leq n}$
    induce an equivalence $\cC \simeq \lim_n \cC_{\leq n}$.

    The crucial point now is that given a sequential diagram $X_\bullet \colon \N \to \cC$
    (analogously for $\N^\op \to \cC$) growing in connectivity,
    each $\tau_{\leq n}X_\bullet$ is eventually constant, so admits a (co)limit.
    Moreover, any functor preserves (co)limits of eventually constant sequential diagrams.
    It follows that $\lim_n \cC_{\leq n}\simeq \cC$ admits (co)limits of
    sequential diagrams growing in connectivity,
    see e.g.~\cite[\href{https://kerodon.net/tag/06B5}{Tag 06B5}]{Kerodon}.

    Finally, we need to argue that the connectives are closed under such (co)limits.
    As in any $t$-structure, the connectives are closed under all existing colimits,
    so consider $X_\bullet \colon \N^\op \to \cC$ growing in connectivity.
    Up to cofinality, we may assume that $\fib(X_{n+1} \to X_n) \in \cC_{\geq n+1}$.
    Hence $X_{n+1} \to X_n$ is inverted by $\tau_{\leq n}$,
    and we see that $\lim_k \tau_{\leq n}X_k \simeq \tau_{\leq n}X_n$.
    Under the equivalence $\cC \simeq \lim_n \cC_{\leq n}$, we can thus identify the limit of $X_\bullet$
    with the tower $(\tau_{\leq n}X_n)_n$ with transition maps $\tau_{\leq n+1}X_{n+1} \to \tau_{\leq n}X_{n+1} \simeq \tau_{\leq n}X_n$.
    In other words, this tower identifies with the Postnikov tower of $\lim X_\bullet$.
    From this description it is clear that if $X_\bullet$ is pointwise connective,
    then also $\tau_{\leq 0}(\lim X_\bullet) \simeq \tau_{\leq 0}X_0$ is connective, and hence $\lim X_\bullet$ is too.
\end{proof}

\begin{warning}
    If $\cC$ is stable with $t$-structure
    whose underlying stable connectivity structure is colimit complete,
    then it does not have to be left complete.
    Indeed, if $\cC$ admits sequential colimits growing in connectivity
    and the $t$-structure is left separated (so that $\cC_{\geq \infty}=0$),
    then it is already colimit complete since for any $t$-structure
    the connectives are closed under all colimits.
    A concrete example of a stable cocomplete category
    with left separated but not left complete $t$-structure
    is $\cC = \Shv^\hyp(X;\cD(\Z))$ where $X=\prod_{n \geq 1}S^n$,
    which we learned from \cite[Section 2.3]{Haine}.
    The canonical $t$-structure is left separated by hypercompleteness,
    and clearly we have all colimits.
    However for $\cF = \bigoplus_{n=1}^\infty \Z[n]$
    (where $\Z[n]$ is the constant sheaf on $\Z[n] \in \cD(\Z)$)
    we have that $\cF \to \lim_n \tau_{\leq n}\cF = \prod_{n=1}^\infty \Z[n]$
    is not an equivalence by \cite[Examples 2.15,2.16]{Haine}.
\end{warning}

\begin{example}\label{ex:sp-conn-comp}
    The $t$-structure on $\Sp$ is complete,
    and hence $(\Sp,\Sp_{\geq 0})$ is both limit and colimit complete.
\end{example}

\begin{example}
    Let $\cX$ be a topos, and equip it with the connectivity
    structure from Example \ref{ex:conn-topos}.
    Analogously to Lemma \ref{lem:left-t-comp-is-conn-comp}
    one can show that the following are equivalent:
    \begin{enumerate}
        \item Postnikov towers converge in $\cX$.
        \item The connectivity structure on $\cX$ is limit complete.
        \item The connectivity structure on $\cX$ is limit and colimit complete.
    \end{enumerate}
    We warn the reader that it does not seem to be known
    whether generally, convergence of Postnikov towers implies
    Postnikov-completeness of topoi, cf.~\cite[Warning 1.23]{Haine}.
    In the following, we will only use that the canonical connectivity
    structure on $\An$ is limit complete, which also easily follows
    from the Milnor sequence on homotopy groups,
    so we leave the proof of the above equivalences to the interested reader.
\end{example}

As a next step, we want to show that colimit completeness
implies a certain idempotent completeness in the setting
of stable connectivity structures.
To prove this, we will need a strengthening
of \cite[Proposition 1.3.3.10(1)]{HA},
which says that a category admitting finite coproducts
and geometric realizations admits all coproducts.

\begin{lemma}\label{lem:fintype-colim}
    Let $\cC$ be a category admitting finite coproducts
    and geometric realizations.
    If $K$ is a quasi-category such that each set of $n$-simplices $K_n$ is finite
    then $\cC$ admits $K$-indexed colimits.
    More generally, if $\cD$ is a ($\infty$-)category so that $\map(\Delta^n,\cD) \in \An^\omega$
    for all $n \geq 0$, then $\cC$ admits $\cD$-indexed colimits.
    In particular, $\cC$ is idempotent complete.
\end{lemma}
\begin{proof}
    The classical ``coproduct version'' of the Bousfield-Kan formula (see \cite[Corollary 12.3]{Shah})
    immediately shows the statement for $K$-indexed colimits.
    In this way we also recover \cite[Proposition 1.3.3.10(1)]{HA},
    that $\cC$ admits finite colimits (since it only remains to check the existence of pushouts).

    Moreover, once we know that $\cC$ is idempotent complete,
    it follows from \cite[Lemma 21.1.2.14]{SAG} that if $\cC$ admits $K$-indexed
    colimits and $L$ is a retract of $K$, then $\cC$ admits $L$-indexed colimits.
    In particular, it will follow that $\cC$ admits colimits indexed by compact (as opposed to only finite)
    anima, and then we can conclude via the $\infty$-categorical Bousfield-Kan formula (see \cite[Corollary 12.5]{Shah}),
    which precisely shows that if $\cC$ admits geometric realizations and $\map(\Delta^n,\cD)$-indexed
    colimits for all $n \geq 0$, then it also admits $\cD$-indexed colimits (and provides a formula for this colimit).

    Thus it remains to show idempotent completeness of $\cC$.
    Now $\Idem = N(\mathrm{Idem})$ is the nerve of the 1-category $\mathrm{Idem}$
    which has one object $x$ and two endomorphisms $\id_x$ and $e$, subject to the relation $e^2 = e$.
    The nerve $\Idem$ has a particularly simple description;
    it is the simplicial set obtained from the terminal semi-simplicial set by freely adjoining degeneracies.
    Concretely, for each $n \geq 0$, the set of $n$-simplices is in bijection with the set of surjective,
    order-preserving maps
    $
        d \colon [n] \twoheadrightarrow [m]
    $
    in $\Delta$. In particular, the set of $n$-simplices in this nerve is finite for all $n$,
    and we conclude by applying the first part.
   \footnote{Given $\Idem \to \cC$ classifying an idempotent endomorphism
        $e \colon X \to X$, the (cofinal) semi-simplicial subobject of the simplicial object
        provided by the Bousfield-Kan formula looks as follows: each object is $X$,
        the first face map is always $e$, and all other face maps are identities.
        The colimit of this (semi-)simplicial diagram is the retract classified by $e$.}
\end{proof}

\begin{corollary}\label{cor:grow-conn-idem}
    Let $\cC$ be a prestable category (e.g.~the connective
    part of a stable connectivity structure).
    Then $\cC$ admits geometric realizations if and only if it admits
    sequential colimits growing in connectivity (with respect
    to the canonical notion of connectivity of Remark \ref{rem:prestable-conn}).
    In this case, $\cC$ and hence also $\SW(\cC)$ are idempotent complete.
    In particular, if $(\cC,\cC_{\geq 0})$ is a colimit complete stable connectivity structure,
    then $\cC_{\geq 0}$ and $\cC_{>-\infty}$ are idempotent complete.
\end{corollary}
\begin{proof}
    This follows immediately
    from Lemma \ref{lem:fintype-colim}
    and the prestable Dold-Kan correspondence of \cite[Theorem C.1.3.1]{SAG},
    which allows us to view geometric realizations
    as sequential colimits growing in connectivity
    analogously to the stable case we recalled in Example \ref{ex:dold-kan}.
\end{proof}

If $(\cC,\cC_{\geq 0})$ is a limit complete stable connectivity structure,
there is generally no reason for it to admit totalizations (or geometric realizations).
However, the idempotent completeness still holds.

\begin{lemma}\label{lem:lim-comp-conn-idem}
    Let $(\cC,\cC_{\geq 0})$ be a limit complete stable connectivity structure.
    Then $\SW(\cC_{\geq 0}) \simeq \cC_{>-\infty}$ and $\cC_{\geq 0}$ are idempotent complete.
\end{lemma}
\begin{proof}
    This is analogous to the proof of \cite[Lemma 2.4]{Balmer-Schlichting}.
    By Thomason's theorem \cite{Thomason} the zeroth (connective) algebraic $K$-theory functor
    $K_0 \colon \Cat_\st \to \Ab$ induces for any given stable category $\cD$ a bijection
    \[
        \{\text{dense stable subcategories }\cD_0 \subseteq \cD\} \leftrightarrow \{\text{subgroups } H \leq K_0(\cD)\}.
    \]
    In particular, if we can show that $K_0(\cD^\idem) = 0$, then it follows that $\cD = \cD^\idem$.
    So let $\cD = \SW(\cC_{\geq 0}) \simeq \cC_{>-\infty}$.
    We establish that $K_0(\cD^\idem) = 0$ by an Eilenberg Swindle.
    Note that if $X \in \cC_{>-\infty}$, then $\cdots \to \Sigma^4X \oplus \Sigma^2X \oplus X \to \Sigma^2X \oplus X \to X$
    is a sequential diagram growing in connectivity, so its limit $\prod_{n \geq 0} \Sigma^{2n}X$ exists.
    This defines an endofunctor $S \colon \cD \to \cD,\ X \mapsto \prod_{n \geq 0}\Sigma^{2n}X$.
    Note also that $S \simeq \Sigma^2S \oplus \id$.
    Now $S$ extends to an endofunctor on idempotent completions $S' \colon \cD^\idem \to \cD^\idem$
    which still satisfies $S' \simeq \Sigma^2S' \oplus \id$.
    Since $\Sigma^2$ induces the identity on $K_0$, it follows that $K_0(S') \oplus K_0(\id) = K_0(S')$,
    and hence that $K_0(\id) = 0$, i.e.~$K_0(\cD^\idem) = 0$, as desired.
    Thus $\cD$ is idempotent complete.
    The same follows for $\cC_{\geq 0}$ as it is closed under retracts in $\cC_{>-\infty}$ by Lemma \ref{lem:stable-conn-struc}.
\end{proof}

We now discuss the examples from the introduction of this section.
They are well known to experts, but we want to highlight
how they are easily intuited and proved via the formalism developed here.

\begin{lemma}\label{lem:ex-group}
    For a finite group $G$, the functors $(-)^{hG},(-)_{hG},(-)^{tG} \colon \Sp^{BG} \to \Sp$ preserve the (co)limits of the Whitehead and Postnikov towers.
\end{lemma}
\begin{proof}
    The category $\Sp$ is presentable stable with complete $t$-structure,
    so the underlying (stable) connectivity structure is both limit and colimit complete (Lemma \ref{lem:left-t-comp-is-conn-comp}).
    Clearly $(-)_{hG}$ preserves colimits and $(-)^{hG}$ preserves limits,
    so due to the canonical fiber sequence $(-)_{hG} \to (-)^{hG} \to (-)^{tG}$
    it suffices to prove that $(-)_{hG}$ preserves sequential limits growing in ($t$-)connectivity
    and $(-)^{hG}$ preserves sequential colimits growing in $t$-coconnectivity,
    where we equip $\Sp^{BG}$ with the pointwise $t$-structure.
    Since $\Sp_{\geq 0}$ is closed under colimits, $(-)_{hG}$ has amplitude $\geq 0$,
    and thus preserves such limits by Corollary \ref{cor:bdd-amp-pres}.
    Dually (cf.~Remark \ref{rem:cocon}) $(-)^{hG}$ has coconnectivity amplitude $\leq 0$, so preserves such colimits.
\end{proof}

\begin{lemma}\label{lem:ex-tensor}
    Let $X \in \Sp$ be bounded below with finitely generated
    homotopy (equivalently, homology) groups.
    Then:
    \begin{enumerate}
        \item $X \tensor -  \colon \Sp \to \Sp$
            preserves uniformly bounded below products
            and hence uniformly bounded below sequential limits.
            For example, this shows $\S_{p}^\land \tensor \Z \simeq \Z_p^\land$.

        \item $\hom_{\Sp}(X,-) \colon \Sp \to \Sp$
            preserves uniformly bounded above filtered colimits.
    \end{enumerate}
\end{lemma}
\begin{proof}
    We only prove (1), the argument for point (2) being analogous.
    Recall that the $t$-structure on $\Sp$ is compatible with the monoidal structure,
    so that $\Sp_{\geq 0}$ is closed under tensor products.
    Moreover, products and coproducts / filtered colimits are $t$-exact.
    It follows that $\lim \colon \Sp^{\N^\op} \to \Sp$ has amplitude $\geq -1$;
    a sequential inverse limit of connective spectra is $(-1)$-connective.
    Now let $X \in \Sp_{\geq 0}^\fg$ be a connective spectrum
    with finitely generated homology groups.
    It is well known that then there exists a cell structure $X = \colim_n X_n$
    where each $X_n$ is finite and each $\cofib(X_n \to X_{n+1})$
    is a finite sum of $\Sigma^{n+1} \S$,
    so that in particular $X_\bullet$ is growing in connectivity
    (for example, this follows from the weight structure on $\Sp^\fg_{>-\infty}$
    which is restricted from the standard weight structure on $\Sp$).
    We show that $X \tensor -$ preserves uniformly bounded below products.
    Indeed, for a collection of (without loss of generality)
    connective spectra $(Y_i)_{i \in I}$ we have a commutative diagram
    \[\begin{tikzcd}
        {X \tensor \prod_i Y_i} & {\prod_i X \tensor Y_i} \\
        {\colim_n X_n \tensor \prod_i Y_i} & {\colim_n \prod_i X_n \tensor Y_i}
        \arrow[from=1-1, to=1-2]
        \arrow["\simeq", from=2-1, to=1-1]
        \arrow["\simeq"', from=2-1, to=2-2]
        \arrow[from=2-2, to=1-2]
    \end{tikzcd}\]
    where in the bottom we use that each $X_n$ is finite. Finally, products are $t$-exact,
    so $\prod_i \colon \Sp^{I} \to \Sp$ has bounded amplitude,
    and $(X_\bullet \tensor Y_i)_i$ is a sequential diagram in $\Sp^{I}$ growing in connectivity,
    so that its colimit is preserved by $\prod_i$.
    So the remaining two maps in the square are equivalences, as desired.
\end{proof}

\begin{lemma}\label{lem:ex-t-struc}
    Suppose $A$ is a finite type anima (i.e.~a CW complex of finite type),
    and that $f \colon \cC \to \cD$ is an exact functor of left complete $t$-categories.
    Then $\cC$ and $\cD$ admit colimits of uniformly bounded below $A$-indexed diagrams,
    and if $f$ restricts to $\cC_{t \geq 0} \to \cD_{t \geq -n}$
    for some $n$, then $f$ preserves them.
    Dually, if $\cC$ and $\cD$ are right $t$-complete
    then they admit limits of uniformly bounded above $A$-indexed diagrams
    and functors of bounded above $t$-amplitude preserve them.
\end{lemma}
\begin{proof}
Let $A \in \An$ be some space, and suppose we have a cell structure $A = \colim_n A_n$
where $A_0$ is a set and $A_n \to A_{n+1}$ is obtained by attaching $(n+1)$-cells,
i.e.~we have pushouts
\[\begin{tikzcd}
    {\coprod_{i \in I_{n+1}}S^n} & {\coprod_{i \in I_{n+1}} *} \\
    {A_n} & {A_{n+1}}
    \arrow[from=1-1, to=1-2]
    \arrow["{(\ell_i)_i}"', from=1-1, to=2-1]
    \arrow["{(a_i)_i}", from=1-2, to=2-2]
    \arrow[from=2-1, to=2-2]
    \arrow["\lrcorner"{anchor=center, pos=0.125, rotate=180}, draw=none, from=2-2, to=1-1]
\end{tikzcd}\]
Now let $\cC$ be a stable category with $t$-structure and enough (co)limits
and $\phi \colon A \to \cC$ be some diagram. Then we have the following
pushout / pullback squares in $\cC$:
\[\begin{tikzcd}[column sep = tiny]
	{\bigoplus_{i}\colim_{S^n}\phi\ell_i} & {\colim_{A_n} \phi|_{A_n}} & 0 & {\prod_{i}\Omega^{n+1}\phi(a_i)} & {\lim_{A_{n+1}}\phi|_{A_{n+1}}} & {\prod_{i}\phi(a_i)} \\
	{\bigoplus_{i}\phi(a_i)} & {\colim_{A_{n+1}}\phi|_{A_{n+1}}} & {\bigoplus_{i}\Sigma^{n+1}\phi(a_i)} & 0 & {\lim_{A_n}\phi|_{A_n}} & {\prod_{i}\lim_{S^n}\phi\ell_i}
	\arrow[from=1-1, to=1-2]
	\arrow[from=1-1, to=2-1]
	\arrow[from=1-2, to=1-3]
	\arrow[from=1-2, to=2-2]
	\arrow[from=1-3, to=2-3]
	\arrow[from=1-4, to=1-5]
	\arrow[from=1-4, to=2-4]
	\arrow["\lrcorner"{anchor=center, pos=0.125}, draw=none, from=1-4, to=2-5]
	\arrow[from=1-5, to=1-6]
	\arrow[from=1-5, to=2-5]
	\arrow["\lrcorner"{anchor=center, pos=0.125}, draw=none, from=1-5, to=2-6]
	\arrow[from=1-6, to=2-6]
	\arrow[from=2-1, to=2-2]
	\arrow["\lrcorner"{anchor=center, pos=0.125, rotate=180}, draw=none, from=2-2, to=1-1]
	\arrow[from=2-2, to=2-3]
	\arrow["\lrcorner"{anchor=center, pos=0.125, rotate=180}, draw=none, from=2-3, to=1-2]
	\arrow[from=2-4, to=2-5]
	\arrow[from=2-5, to=2-6]
\end{tikzcd}\]
Note also that generally $\colim_n \colim_{A_n} \phi|_{A_n} \simeq \colim_A \phi$
and analogously for limits. We thus make the following observations:
\begin{enumerate}
    \item If $\phi$ takes uniformly bounded below values, then
        $\colim_{A_0} \phi \to \colim_{A_1}\phi \to \colim_{A_2}\phi \to \cdots$
        is a sequential diagram growing in $t$-connectivity.

    \item If $\phi$ takes uniformly bounded above values, then
        $\cdots \to \lim_{A_2}\phi \to \lim_{A_1}\phi \to \lim_{A_0}\phi$
        is a sequential diagram growing in $t$-coconnectivity.
\end{enumerate}
Now assume that $A$ is of finite type,
so that each $A_n$ is a finite\footnote{In fact, as we argued in the proof of Lemma \ref{lem:fintype-colim}, $A_n$ being compact is enough here.}
anima.
If $\cC$ and $\cD$ are stable categories with left complete
$t$-structure, then by Lemma \ref{lem:left-t-comp-is-conn-comp}
they are colimit complete, and thus admit colimits of uniformly bounded below
$A$-indexed diagrams by the above analysis.
Moreover, if $f \colon \cC \to \cD$ is an exact functor restricting
to $\cC_{\geq 0} \to \cD_{\geq -n}$ for some $n$,
then $f$ has bounded amplitude and thus preserves such colimits.
Dually, if $\cC$ and $\cD$ are right $t$-complete,
they admit $A$-indexed limits of uniformly bounded above diagrams
by the dual of Lemma \ref{lem:left-t-comp-is-conn-comp},
and if $f \colon \cC \to \cD$ is an exact functor
which restricts to $\cC_{\leq 0} \to \cD_{\leq n}$ for some $n$,
then it preserves those limits.
\end{proof}

\begin{corollary}
    The $t$-structure on $\Sp$ is complete and compatible with filtered colimits and products,
    so that if $I$ is a filtered category and $J$ a set
    the functors $\colim \colon \Sp^I \to \Sp$ and $\prod_J \colon \Sp^J \to \Sp$
    are $t$-exact for the pointwise $t$-structures.
    The above shows that if $A$ is a finite type anima,
    then these functors preserve colimits resp.~limits of uniformly bounded below resp.~above $A$-indexed diagrams.
    Equivalently, $\colim_A \colon \Sp^A \to \Sp$ preserves uniformly bounded below products,
    and $\lim_A \colon \Sp^A \to \Sp$ preserves uniformly bounded above filtered colimits.
\end{corollary}

\begin{lemma}\label{lem:ex-suspension}
    For $X \in \An$ we have $\Sigma_+^\infty X \xto{\simeq} \lim_n \Sigma_+^\infty\tau_{\leq n}X$ in $\Sp$.
\end{lemma}
\begin{proof}
    The functor $\Sigma^\infty_+ \colon \An \to \Sp$ is a functor
    of bounded below amplitude between limit complete connectivity structures
    (cf.~Example \ref{ex:conn-topos}).
    Thus it preserves sequential inverse limits growing in connectivity,
    of which Postnikov towers are the canonical example.
\end{proof}

\section{Weak weight structures}\label{sec:weight}

It turns out that the theory of completions of weight structures developed in this paper—specifically in Section \ref{sec:upw-comp} - also extends to a natural weakening of the definition of weight structures. In most cases, this weaker notion coincides with one considered by Levy-Sosnilo in \cite{Ishan-Vova}.
In this section, we begin by introducing this variant,
and then show in subsection \ref{subsec:weight-cplx}, as in the setting of weighted categories,
that weight complexes still exist.

\begin{definition}\label{def:weak-w}
    Let $\ell \in \N_0 \cup \{\infty\}$.
    A weakly $\ell$-weighted category is a stable category $\cC$
    together with full subcategories $\cC_{\geq 0},\cC_{\leq 0} \subseteq \cC$ such that
    \begin{enumerate}
        \item $\cC_{\geq 0}$ is closed under suspensions,
            $\cC_{\leq 0}$ under loops, and both under extensions and retracts.\footnote{In particular, $\cC_{\geq 0}$ is closed under finite colimits and $\cC_{\leq 0}$ under finite limits.}
        \item $\hom_\cC(X_{\leq 0},Y_{\geq 0}) \in \Sp_{\geq 0}$
            for $X_{\leq 0} \in \cC_{\leq 0}$ and $Y_{\geq0} \in \cC_{\geq 0}$.

        \item For every $X \in \cC$ there is a cofiber sequence
            $X_{< \ell} \to X \to X_{\geq 0}$
            with $X_{<\ell} \in \cC_{< \ell}$ and $X_{\geq 0} \in \cC_{\geq 0}$.
            We refer to this as a weight decomposition of $X$.
    \end{enumerate}
    Here $\cC_{\geq n} = \Sigma^n \cC_{\geq 0}$ and $\cC_{<n+1} = \cC_{\leq n} = \Sigma^n \cC_{\leq 0}$
    for $n \in \Z$, and $\cC_{<\infty} = \bigcup_{n \geq 0} \cC_{\leq n}$. In the case $\ell = \infty$, we simply refer to such a structure as weakly weighted.
\end{definition}

\begin{warning}\label{warn:wwcat-non-dual}
    Suppose that $(\cC^\op,(\cC_{\leq 0})^\op,(\cC_{\geq 0})^\op)$ is weakly $\ell$-weighted.
    For $\ell < \infty$,
    it is straightforward to check that this is the same
    as $(\cC,\cC_{\geq 0},\cC_{\leq 0})$ being weakly $\ell$-weighted.
    However for $\ell = \infty$, the decomposition axiom (3)
    slightly changes to require the existence of decompositions
    $X_{\leq 0} \to X \to X_{>-\infty}$ in $\cC$.
    In this case, we call $\cC$ weakly co-weighted,
    and weakly bi-weighted if both $(\cC,\cC_{\geq 0},\cC_{\leq 0})$
    and $(\cC^\op,(\cC_{\leq 0})^\op,(\cC_{\geq 0})^\op)$ are weakly weighted.
\end{warning}

Note that if $0 \leq \ell \leq \ell' \leq \infty$, then any weakly $\ell$-weighted structure $(\cC,\cC_{\geq 0},\cC_{\leq 0})$ is also weakly $\ell'$-weighted. In particular, every weakly $\ell$-weighted category is weakly weighted.

\begin{remark}\label{rem:bounded-is-weakly-weighted}
    Suppose that $(\cC,\cC_{\geq 0},\cC_{\leq 0})$ satisfies axioms (1) and (2) of Definition \ref{def:weak-w}.
    If $\cC_{<\infty} = \cC$, then $\cC$ is weakly weighted via the trivial decompositions $X \xto{\id} X \to 0$ for all $X \in \cC$.
    Dually, if $\cC_{>-\infty} = \cC$, then $\cC$ is weakly co-weighted via the trivial decompositions $0 \to X \xto{\id} X$.
    So if $\cC_{<\infty} = \cC = \cC_{>-\infty}$, then $\cC$ is weakly bi-weighted.
    In particular, if $(\cC,\cC_{\geq 0})$ is a stable connectivity structure
    we can define $\cC_{\leq 0} = \lperp{(\cC_{\geq 1})}$.
    Then $(\cC_{>-\infty},\cC_{\geq 0}, \cC_{(-\infty,0])})$ is always weakly co-weighted
    since $\cC_{\geq 0} \subseteq \cC_{>-\infty}$ is closed under retracts (Lemma \ref{lem:stable-conn-struc}),
    and $(\cC_{<\infty},\cC_{[0,\infty)}, \cC_{\leq 0})$ is weakly weighted
    if $\cC_{[0,\infty)} \subseteq \cC_{<\infty}$ is closed under retracts.

    In this sense, the notions of weakly $\infty$-(co-)weighted categories and their associated decompositions are only genuinely interesting for unbounded objects. Nevertheless, even when the original category $\cC$ is bounded, its completions studied in the next section remain nontrivial and worthy of study. Thus, the observation may also be interpreted as saying that the bounded case fits into the completion framework without requiring any additional hypotheses.
\end{remark}

\begin{remark}\label{rem:weak-weight-reformulation}
    A priori, there is another way to reasonably weaken the axioms
    of a weighted category;
    for $n,\ell \in \N_0$, let us temporarily define an $(n,\ell)$-weighted category $\cC$
    as a stable category $\cC$ with two full subcategories
    $\cC_{\geq 0},\cC_{\leq 0}$ still satisfying axioms (1) and (3) as above,
    but with the following instead of (2):
    \begin{enumerate}
        \item[(2')] $\hom_\cC(X_{\leq 0},Y_{\geq 0}) \in \Sp_{\geq -n}$
            for $X_{\leq 0} \in \cC_{\leq 0}$ and $Y_{\geq 0} \in \cC_{\geq 0}$.
    \end{enumerate}
    Given an $(n,\ell)$-weighted category $(\cC,\cC_{\geq 0},\cC_{\leq 0})$,
    it is straightforward from the definitions to show that
    \begin{itemize}
        \item $(\cC,\cC_{\geq k},\cC_{\leq k})$ is $(n,\ell)$-weighted for all $k \in \Z$.
        \item $(\cC,\cC_{\geq 0}, \cC_{\leq -n})$ is $(0,n+\ell)$-weighted.
        \item $(\cC,\cC_{\geq -\ell},\cC_{\leq 0})$ is $(n+\ell,0)$-weighted.
    \end{itemize}
    In \cite{Ishan-Vova} the $(\ell,0)$-weighted categories which are bounded
    (i.e.~$\cC = \cC_{<\infty} = \cC_{>-\infty}$)
    were called heart categories of cohomological dimension $\leq \ell$.\footnote{A priori, their definition does not include that $\cC_{\geq 0}$ and $\cC_{\leq 0}$ are closed under retracts in $\cC$. However, given the other assumptions, this actually follows from boundedness of $\cC$, cf.~Lemma \ref{lem:stable-conn-struc}.}
    In particular, this notion is structurally equivalent
    to that of $(0,\ell)$-weighted categories,
    i.e.~our notion of weakly $\ell$-weighted categories,
    but we will adopt the latter formulation for our purposes.
\end{remark}

\begin{example}\label{ex:perf-weakly-weighted}
    In view of the above Remark, \cite[Example 2.18]{Ishan-Vova} shows that if $R$ is any $(-1)$-connective
    ring spectrum, then $\Perf(R)$ is weakly $1$-weighted
    for $\Perf(R)_{\leq 0}$ resp.~$\Perf(R)_{\geq 0}$ the retract closures of the extension closures
    of the categories $\{\Sigma^{-n}R \mid n \geq 0\}$ resp.~$\{\Sigma^n R \mid n \geq 1\}$. Note that we need $n \geq 1$ in the connectives due to the shift in translating from weakly $(1,0)$-weighted
    to weakly $(0,1)$-weighted category in the notation of Remark \ref{rem:weak-weight-reformulation}.
    More generally, we can consider a $(-\ell)$-connective ring spectrum $R$ for some $0 \leq \ell < \infty$
    and define $\Perf(R)_{\leq 0}$ as before and $\Perf(R)_{\geq 0}$ as the retract closure
    of the extension closure of $\{\Sigma^n R \mid n \geq \ell\}$.
    Then $(\Perf(R),\Perf(R)_{\geq 0},\Perf(R)_{\leq 0})$ satisfies axioms (1) and (2)
    of Definition \ref{def:weak-w},
    however by \cite[Remark 2.19]{Ishan-Vova} it is generally not weakly $\ell$-weighted.
    On the other hand, it is trivially weakly $\infty$-(co-)weighted by Remark \ref{rem:bounded-is-weakly-weighted}
    since $\Perf(R) = \Perf(R)_{<\infty} = \Perf(R)_{>-\infty}$.
\end{example}

\begin{definition}\label{def:weight-amplitude}
    An exact functor $f \colon \cC \to \cD$ between weakly weighted categories
    has (weight) amplitude $[a,b]$ for $-\infty \leq a \leq b \leq \infty$
    if the underlying functor of connectivity structures
    has connectivity amplitude $\geq a$ and coconnectivity amplitude $\leq b$.
    Equivalently (cf.~Observation \ref{obs:conn}(2)),
    if it restricts to $\cC_{\geq 0} \to \cD_{\geq a}$ and $\cC_{\leq 0} \to \cD_{\leq b}$.
    Here $\cD_{\geq -\infty} = \cD = \cD_{\leq \infty}$.
    Let us also say that $f$ has amplitude in $(a,b)$
    if it restricts to $\cC_{\geq 0} \to \cD_{>a}$ and $\cC_{\leq0} \to \cD_{<b}$
    and similarly for amplitude $[a,b)$ or $(a,b]$.
    We say that $f$ has bounded below (weight) amplitude if $a \in \Z$,
    and bounded amplitude if $a,b \in \Z$.
\end{definition}

\begin{example}\label{ex:corep-bdd}
    Let $X \in \cC_{\leq b}$ for some $b \in \Z$.
    Then $\hom_\cC(X,-) \colon \cC \to \Sp$ has weight amplitude $[-b,\infty]$
    by the axioms of a weakly weighted category,
    where we equip $\Sp$ with the standard weight structure (whose connectives are the $t$-connectives).
    In particular, by Example \ref{ex:sp-conn-comp} and Corollary \ref{cor:bdd-amp-pres}
    this functor sends sequential cones growing in (weight) connectivity to colimiting cones.
\end{example}

\begin{definition}\label{def:wwcat}
    Given two weakly weighted categories $\cC,\cD$ and $-\infty \leq a \leq b \leq \infty$ we let
    \[
        \Fun^\Ex_{[a,b]}(\cC,\cD) \subseteq \Fun^\Ex(\cC,\cD)
    \]
    be the full subcategory on exact functors of amplitude $[a,b]$,
    and similarly for (half-)open intervals.
    Let $\wWCat$ denote the category of weakly weighted categories and functors of bounded below amplitude.
    Similarly, we let $\wWCat_{=0}$ resp.~$\wWCat_{\geq 0}$ denote the category of weakly weighted
    categories and exact functors that are weight exact resp.~preserve connectives.
    Note that these are not subcategories of $\wWCat$ (or each other)
    since they do not share the same equivalences.
\end{definition}

\begin{definition}\label{def:saturated}
    Let $\cC$ be a weakly weighted category.
    We say that $\cC$ is left saturated if $\cC_{\leq 0} = \lperp{(\cC_{\geq 1})}$
    and right saturated if $\cC_{\geq 0} = \rperp{(\cC_{\leq -1})}$.
    We call $\cC$ saturated if it is both left and right saturated.
\end{definition}

We now present three lemmas on the main advantages of saturation.
Namely, by virtue of being defined as orthogonal complements,
the (co)connectives in a saturated weakly weighted category
enjoy many more closure properties than required in the definition
(Lemma \ref{lem:weight-closure}).
Moreover, weight exact functors between saturated
weakly weighted categories behave more as one expects from
the theory of weighted categories (Lemma \ref{lem:saturated}).
Finally, we show that in the case $\ell < \infty$,
up to equivalence in $\wWCat$ (though not up to weight-exact equivalence!),
every weakly $\ell$-weighted category is saturated.

\begin{lemma}\label{lem:weight-closure}
    Let $\cC$ be a weakly weighted category.
    If $\cC$ is left \alt{right} saturated, then $\cC_{\leq 0}$ \alt{$\cC_{\geq 0}$}
        is closed under the following operations, provided they exist:
    \begin{enumerate}
        \item retracts and extensions.
        \item finite limits \alt{finite colimits}.
        \item arbitrary direct sums \alt{arbitrary products}.
        \item colimits \alt{limits} of diagrams $X_\bullet \colon \gamma \to \cC$ \alt{$X_\bullet \colon \gamma^\op \to \cC$}
            for some ordinal $\gamma$ satisfying
            \begin{enumerate}
                \item For a successor ordinal $\alpha+1 < \gamma$,
                    the map $X_{\alpha} \to X_{\alpha+1}$ has cofiber in $\cC_{\leq 0}$
                    \alt{the map $X_{\alpha+1} \to X_\alpha$ has fiber in $\cC_{\geq 0}$}.
                \item For a limit ordinal $\lambda < \gamma$,
                    the map $\colim_{\alpha < \lambda} X_\alpha \to X_\lambda$
                    has cofiber in $\cC_{\leq 0}$
                    \alt{the map $X_\lambda \to \lim_{\alpha < \lambda} X_\alpha$
                    has fiber in $\cC_{\geq 0}$}.
            \end{enumerate}
            Note this includes sequential colimits \alt{limits} growing in weight coconnectivity
            \alt{connectivity}.
    \end{enumerate}
\end{lemma}
\begin{proof}
    This is a special case of Lemma \ref{lem:orth-closure}.
\end{proof}

\begin{remark}\label{rem:saturated}
    If $\cC$ is a weighted category
    then it satisfies the above saturation properties,
    see e.g.~\cite[Lemma 3.1.2]{Hebestreit-Steimle}.
    As a consequence, $\cC_{\leq 0}$ respectively $\cC_{\geq 0}$
    is automatically closed under finite limits
    respectively finite colimits,
    so a weakly $0$-weighted category is just a weighted category.
\end{remark}

\begin{lemma}\label{lem:saturated}
    The following are convenient consequences of saturation:
    \begin{enumerate}
        \item Let $f \colon \cC \to \cD$ be a fully faithful weight exact functor of weakly weighted categories.
            Let $X \in \cC$ and suppose that $f(X)$ is connective \alt{coconnective}.
            If $\cC$ is right \alt{left} saturated, then $X$ is connective \alt{coconnective}.

        \item Let $L \colon \cC \rightleftarrows \cD \noloc R$
            be an adjunction between weakly weighted categories.
            If $\cD$ is left saturated and $R$ preserves connectives,
            then $L$ preserves coconnectives.
            If $\cC$ is right saturated and $L$ preserves coconnectives,
            then $R$ preserves connectives.

        \item In particular, if $f \colon \cC \to \cD$ is a weight
            exact functor of saturated weakly weighted categories which is
            an underlying equivalence (i.e.~in $\Cat$),
            then also $f^{-1}$ is weight exact.
    \end{enumerate}
\end{lemma}
\begin{proof}
    For (1) the ``only if'' follows from weight exactness.
    If $f(X)$ is connective (the coconnective case being analogous),
    then $\hom(Y,X) \simeq \hom(f(Y),f(X)) \in \Sp_{\geq 0}$
    for all $Y \in \cC_{\leq 0}$, so saturation of $\cC$ implies $X \in \cC_{\geq 0}$.
    The proof for (2) is similarly based on saturation of both $\cC$ and $\cD$,
    and (3) is a direct consequence of (2) and the fact that $f^{-1}$
    will be both left and right adjoint to $f$.
\end{proof}

\begin{remark}
    All of these can go wrong in the non-saturated setting.
    For example, if $(\cC,\cC_{\geq 0},\cC_{\leq 0})$
    is weakly weighted but not saturated
    then $(\cC,\cC_{\geq 0},\cC_{\leq -1})$
    is also weakly weighted,
    and $\Sigma \colon (\cC,\cC_{\geq 0},\cC_{\leq -1}) \to (\cC,\cC_{\geq 0},\cC_{\leq 0})$ is a weight exact functor which is an underlying
    equivalence, but its inverse is not weight exact.
\end{remark}

\begin{lemma}\label{lem:saturation}
    Let $\cC$ be weakly $\ell$-weighted
    and define $\cC'_{\leq 0} \coloneqq \lperp{(\cC_{\geq 1})}$
    and $\cC'_{\geq 0} \coloneqq \rperp{(\cC'_{\leq -1})}$.
    Then:
    \begin{enumerate}
        \item $\cD = (\cC,\cC_{\geq 0}, \cC'_{\leq 0})$
            and $\cE = (\cC,\cC'_{\geq 0}, \cC'_{\leq 0})$ are still weakly $\ell$-weighted.

        \item The identity functors $\cC \to \cD \to \cE$ are weight exact,
            and the identity $\cE \to \cD$ has amplitude $(-(\ell+1),0]$ and $\cD \to \cC$ has amplitude $[0,\ell+1)$.

        \item In particular, $\cC \to \cD$ is an equivalence in $\wWCat_{\geq 0}$,
            showing that any weakly weighted category is equivalent to a left saturated one
            in a way which preserves the weight connectives.

        \item If $\ell < \infty$, then $\cE \to \cC$ has bounded amplitude $[-\ell,\ell]$
            so yields an equivalence in $\wWCat$, showing that any weakly $\ell$-weighted category
            is equivalent in $\wWCat$ to a saturated one.\footnote{Observe that this also proves that (left) saturation is not invariant under equivalence in $\wWCat$, though of course it is in $\wWCat_{=0}$.}

        \item If $\cD$ is weakly $\ell'$-weighted and $\ell,\ell' < \infty$
            and $f \colon \cC \to \cD$ is a map of bounded below \alt{bounded} amplitude
            which is an equivalence on underlying categories,
            then $f^{-1}$ also has bounded below \alt{bounded} amplitude, so $f$ is an equivalence in $\wWCat$.
    \end{enumerate}
    Analogous results hold if we swap the order of definition
    as $\cC'_{\geq 0} = \rperp{(\cC_{\leq -1})}$ and $\cC'_{\leq 0} = \lperp{(\cC'_{\geq 1})}$.
\end{lemma}
\begin{proof}
    By definition the orthogonality axiom still holds in $\cD$ and $\cE$.
    Lemma \ref{lem:weight-closure} shows
    that $\cC'_{\geq 0}$ and $\cC'_{\leq 0}$ have the required closure properties in $\cC$.
    Because $\cC_{\leq 0} \subseteq \cC'_{\leq 0}$ and $\cC_{\geq 0} \subseteq \cC'_{\geq 0}$
    the weight decompositions of $\cC$ also give ones for $\cD$ and $\cE$,
    and this also proves that the functors $\cC \to \cD \to \cE$ are weight exact.

    Now suppose that $X \in \cC'_{\geq 0}$.
    Picking a weight decomposition $X_{\leq -1} \to X \to X_{>-(\ell+1)}$ in $\cC$,
    the first map vanishes, exhibiting $X$ as retract of the last term,
    thus proving $X \in \cC_{>-(\ell+1)}$.
    Similarly, if $Y \in \cC'_{\leq 0}$, we can pick a weight decomposition $Y_{<\ell+1} \to Y \to Y_{\geq 1}$
    in $\cC$ and again the map $Y \to Y_{\geq 1}$ vanishes since $\cC_{\geq 1} \subseteq \cC'_{\geq 1}$.
    Thus $Y \in \cC_{<\ell+1}$.
    This shows that $\cC'_{\geq 0} \subseteq \cC_{>-(\ell+1)}$ and $\cC'_{\leq 0} \subseteq \cC_{<\ell+1}$,
    which yields (2).
    Now (3) and (4) are immediate from the definition of maps in $\wWCat$,
    and (5) follows from Lemma \ref{lem:saturated}(3).
\end{proof}

\begin{example}\label{ex:scheme-weak-weight}
    Let $X$ be a qcqs scheme and $\Perf(X)$ its stable category of perfect complexes.
    Then there exists some $\ell < \infty$ and a left saturated weak $\ell$-weight structure on $\Perf(X)$
    where $\Perf(X)_{w \geq 0}$ consists of those perfect complexes of Tor-amplitude $\geq 0$.
    Under the monoidal duality $\hom(-,\cO_X) \colon \Perf(X)^\op \simeq \Perf(X)$
    this corresponds to a right saturated weak $\ell$-weight structure with coconnectives
    those complexes of Tor-amplitude $\leq 0$.

    To see this, note that by \cite[Theorem 3.2(iv)]{Neeman-Qcoh},
    for any qcqs scheme $X$ the category $\QCoh(X)$
    with its standard $t$-structure is weakly approximable in the sense of Neeman \cite{Neeman-Approx},
    which is recalled in \cite[Appendix A]{Ishan-Vova} in a way that matches our notation.
    Concretely, Neeman shows that there exists a single perfect complex $G \in \Perf(X)$
    and some integer $A > 0$ so that
    \begin{enumerate}
        \item $G \in \QCoh(X)_{t \leq A} \cap \QCoh(X)_{w \leq A}$ where $\QCoh(X)_{w \leq 0} \coloneqq \lperp{(\QCoh(X)_{t \geq 1})}$\footnote{In fact, we even have $G \in \QCoh(X)_{w \leq A-1}$, but for notational convenience we will use the slightly weaker bound.},
        \item $G$ generates $\QCoh(X)$,
        \item for every $F \in \QCoh(X)_{t \geq 0}$ there exists a fiber sequence $E \to F \to D$
            where $D \in \QCoh(X)_{t \geq 1}$ and $E$ lies in the closure of $\{\Sigma^n G \mid -A \leq n \leq A\}$
            under retracts, extensions and coproducts, hence $E \in \QCoh(X)_{w \leq 2A}$.
    \end{enumerate}
    Note that (2) implies that $\QCoh(X)^\omega = \Perf(X)$.
    As detailed in \cite[Corollary A.8]{Ishan-Vova}, this implies
    that $\Perf(X)$ admits a $(2A-1,0)$-weight structure in the sense of Remark \ref{rem:weak-weight-reformulation}
    with connectives $\Perf(X) \cap \QCoh(X)_{t \geq 0}$
    and coconnectives the closure of $\{\Sigma^n G \mid n \leq A\}$ under retracts and extensions.

    Translating to our setting as in the remark,
    we see that $\Perf(X)$ admits a weakly $2A$-weighted structure
    with connectives $\Perf(X)_{w \geq 0} = \Perf(X) \cap \QCoh(X)_{t \geq 0}$
    and coconnectives the closure of $\{\Sigma^n G \mid n \leq -A\}$
    under extensions and retracts.
    Note that $\Perf(X)_{w \geq 0}$ consists precisely of those perfect complexes of Tor-amplitude $\geq 0$,
    and the coconnectives are contained in $\QCoh(X)_{t \leq 0}$.
    We can now apply Lemma \ref{lem:saturation} to see that we can saturate the coconnectives
    without any issue, and hence obtain the desired weak weight structure on $\Perf(X)$.
    Finally, we remark that the duality $\hom(-,\cO_X) \colon \Perf(X)^\op \simeq \Perf(X)$
    precisely interchanges the Tor-amplitude $\geq 0$ and Tor-amplitude $\leq 0$
    perfect complexes. Thus we formally obtain a weak $2A$-weight structure on $\Perf(X)$
    which is right saturated and has coconnectives the perfect complexes of Tor-amplitude $\leq 0$.
\end{example}

\subsection{Weight complexes in weakly weighted categories}\label{subsec:weight-cplx}

Much of the theory of weight structures carries over to this weakened setting.
Importantly, we have a good notion of weight complexes,
which we want to introduce in slightly larger generality.

\begin{definition}\label{def:upw-dense}
    Let $\cC$ be a weakly weighted category.
    We call a full stable subcategory $\cU \subseteq \cC_{<\infty}$
    left dense in $\cC$ if for every $X \in \cC$ and $n \geq 0$ there is a map
    $f \colon U \to X$ with $U \in \cU$ and $\cofib(f) \in \cC_{\geq n}$.
    Note that $\cU = \cC_{w < \infty}$ is left dense in $\cC$.

    Given an object $X \in \cC$, a $\cU$-weight complex for $X$
    is a diagram $U_\bullet \colon \N \to \cU_{/X} = \cU \times_\cC \cC_{/X}$
    so that $\N \to \cC_{/X}$ is a cone growing in connectivity
    in the sense of Definition \ref{def:grow-in-conn},
    meaning that for every $n \geq 0$ the cofibers of $U_k \to X$
    (and hence of $U_k \to U_{k+1}$) are eventually $n$-connective.
\end{definition}

\begin{remark}\label{rem:upw-dense}
    A left dense subcategory $\cU \subseteq \cC_{<\infty}$ is always dense.
    In fact, given $X \in \cC_{\leq n}$
    we can choose a cofiber sequence $U \to X \to X_{w \geq n+1}$
    with $U \in \cU$. Since $X \to X_{w \geq n+1}$ vanishes,
    this exhibits $X$ as retract of $U$.
    Note also that if $X \in \cC_{w \geq 0}$ then $U \in \cC_{w \geq 0}$.
    Dealing with retract / idempotent completeness issues is precisely the purpose of this definition.
\end{remark}

\begin{lemma}\label{lem:weight-cplx}
    Let $\cC$ be a weakly weighted category
    and $\cU \subseteq \cC_{< \infty}$ be left dense in $\cC$.
    Given $X \in \cC$, there exists a strictly increasing $f \colon \N \to \N$
    and a commutative diagram
    \[\begin{tikzcd}
        {U_{\leq f(0)}} & {U_{\leq f(1)}} & {U_{\leq f(2)}} & \cdots \\
        X & X & X & \cdots \\
        {X_{\geq 0}} & {X_{\geq f(0)+1}} & {X_{\geq f(1)+1}} & \cdots
        \arrow[from=1-1, to=1-2]
        \arrow[from=1-1, to=2-1]
        \arrow[from=1-2, to=1-3]
        \arrow[from=1-2, to=2-2]
        \arrow[from=1-3, to=1-4]
        \arrow[from=1-3, to=2-3]
        \arrow[Rightarrow, no head, from=2-1, to=2-2]
        \arrow[from=2-1, to=3-1]
        \arrow[Rightarrow, no head, from=2-2, to=2-3]
        \arrow[from=2-2, to=3-2]
        \arrow[Rightarrow, no head, from=2-3, to=2-4]
        \arrow[from=2-3, to=3-3]
        \arrow[from=3-1, to=3-2]
        \arrow[from=3-2, to=3-3]
        \arrow[from=3-3, to=3-4]
    \end{tikzcd}\]
    where each column is a cofiber sequence, with $X_{\geq n} \in \cC_{\geq n}$
    and $U_{\leq f(n)} \in \cU \cap \cC_{\leq f(n)}$.
    In particular, the first two rows define a $\cU$-weight complex for $X$.
    Moreover, if $U_\bullet \colon \N \to \cU_{/X}$ is any $\cU$-weight complex
    for $X$, then $U_\bullet \colon \N \to \cU_{/X}$ is cofinal.
\end{lemma}
\begin{proof}
    We construct the diagram inductively. To begin, choose a cofiber
    sequence $U \to X \to X_{\geq 0}$. Since $\cU \subseteq \cC_{< \infty}$,
    we define $f(0)$ so that $U \in \cC_{\leq f(0)}$
    and let $U_{\leq f(0)} = U$.
    For notational convenience, let $f(-1) \coloneqq -1$.
    Suppose by induction we have constructed the diagram up to the $n$-th column.
    Again we can find $f(n+1) \geq f(n)$
    and some $U_{\leq f(n+1)} \in \cU \cap \cC_{\leq f(n+1)}$
    with a cofiber sequence $U_{\leq f(n+1)} \to X \to X_{\geq f(n)+1}$.
    Since the composite $U_{\leq f(n)} \to X_{\geq f(n)+1}$ vanishes,
    we obtain factorizations
    \[\begin{tikzcd}
        & {U_{\leq f(n+1)}} \\
        {U_{\leq f(n)}} & X & {X_{\geq f(n-1)+1}} \\
                        & {X_{\geq f(n)+1}}
                        \arrow[from=1-2, to=2-2]
                        \arrow[dashed, from=2-1, to=1-2]
                        \arrow[from=2-1, to=2-2]
                        \arrow["0"', from=2-1, to=3-2]
                        \arrow[from=2-2, to=2-3]
                        \arrow[from=2-2, to=3-2]
                        \arrow[dashed, from=2-3, to=3-2]
    \end{tikzcd}\]
    which give us the desired maps to extend the diagram to the $(n+1)$-st column.
    By induction, we obtain the desired diagram.

    For the addendum, we argue using Quillen's Theorem A.
    Given $U \in \cU_{/X}$ we have a pullback
    \[\begin{tikzcd}
        {\N_{U/}} & {(\cU_{/X})_{U/}} \\
        \N & {\cU_{/X}}
        \arrow[from=1-1, to=1-2]
        \arrow[from=1-1, to=2-1]
        \arrow["\lrcorner"{anchor=center, pos=0.125}, draw=none, from=1-1, to=2-2]
        \arrow[from=1-2, to=2-2]
        \arrow["{U_{\bullet}}", from=2-1, to=2-2]
    \end{tikzcd}\]
    We need to show that $\N_{U/}$ is weakly contractible, i.e.~$|\N_{U/}| \simeq *$.
    Note that $\N_{U/} \to \N$ is a left fibration classifying the functor
    $\N \to \An, n \mapsto \map_{\cU_{/X}}(U,U_{n})$.
    Therefore
    \[
        |\N_{U/}|
        \simeq \colim_n \map_{\cU_{/X}}(U, U_{n})
        \simeq \colim_n \map_{\cC}(U,U_{n}) \times_{\map_{\cC}(U,X)} *
    \]
    Since $U \in \cC_{< \infty}$, Example \ref{ex:corep-bdd}
    yields $\colim_n \hom_\cC(U,U_n) \xto{\simeq} \hom_\cC(U,X)$.
    Using that $\Omega^\infty \colon \Sp \to \An$ preserves sequential colimits,
    we conclude that $\N_{U/}$ is weakly contractible, so $U_{\bullet}$
    is cofinal.
\end{proof}

\begin{example}\label{ex:standard-weight-complex}
    In the case $\cU = \cC_{< \infty}$, and $\ell = 0$,
    we can pick $f(n) = n$ and obtain the classical weight complex,
    with $n$-th column $X_{\leq n} \to X \to X_{\geq n+1}$,
    and where $X_{\leq n} \to X_{\leq n+1}$ has cofiber in $\cC_{=n+1}$.
    Similarly, for $1 \leq \ell < \infty$ we can pick $f(n) = \ell n$,
    where the $n$-th column looks like $X_{\leq \ell n} \to X \to X_{\geq 1+\ell(n-1)}$,
    and the cofiber of $X_{\leq \ell n} \to X_{\leq \ell (n+1)}$
    lies in $\cC_{[1+\ell (n-1),\,\ell(n+1)]}$.
\end{example}

The following lemma is an analogue of the classical cellular approximation theorem.

\begin{lemma}
    Let $\cC$ be a weakly $\ell$-weighted category for $1 \leq \ell < \infty$
    and $X_{\leq \ell \bullet}$ a $\cC_{<\infty}$-weight complex for $X \in \cC$
    as in the above example. Given a map $f \colon X \to Y$,
    there exists a $\cC_{<\infty}$-weight complex $Y_{\leq \ell\bullet}$
    as above and a commutative diagram
    \[\begin{tikzcd}[ampersand replacement=\&]
        {X_{\leq 0}} \& \cdots \& {X_{\leq \ell n}} \& {X_{\leq \ell(n+1)}} \& \cdots \& X \\
        {Y_{\leq \ell}} \& \cdots \& {Y_{\leq \ell (n+1)}} \& {Y_{\leq \ell(n+2)}} \& \cdots \& Y
        \arrow[from=1-1, to=1-2]
        \arrow["{f_0}", from=1-1, to=2-1]
        \arrow[from=1-2, to=1-3]
        \arrow[from=1-3, to=1-4]
        \arrow["{f_n}", from=1-3, to=2-3]
        \arrow[from=1-4, to=1-5]
        \arrow["{f_{n+1}}", from=1-4, to=2-4]
        \arrow[from=1-5, to=1-6]
        \arrow["f", from=1-6, to=2-6]
        \arrow[from=2-1, to=2-2]
        \arrow[from=2-2, to=2-3]
        \arrow[from=2-3, to=2-4]
        \arrow[from=2-4, to=2-5]
        \arrow[from=2-5, to=2-6]
    \end{tikzcd}\]
    In the case $\ell=0$ where $\cC$ is weighted, we can even choose $Y_{\leq \bullet}$
    first and then find maps $X_{\leq n} \to Y_{\leq n}$ making a commutative diagram
    \[\begin{tikzcd}
        {X_{\leq 0}} & \cdots & {X_{\leq n}} & {X_{\leq (n+1)}} & \cdots & X \\
        {Y_{\leq 0}} & \cdots & {Y_{\leq n}} & {Y_{\leq (n+1)}} & \cdots & Y
        \arrow[from=1-1, to=1-2]
        \arrow["{f_0}", from=1-1, to=2-1]
        \arrow[from=1-2, to=1-3]
        \arrow[from=1-3, to=1-4]
        \arrow["{f_n}", from=1-3, to=2-3]
        \arrow[from=1-4, to=1-5]
        \arrow["{f_{n+1}}", from=1-4, to=2-4]
        \arrow[from=1-5, to=1-6]
        \arrow["f", from=1-6, to=2-6]
        \arrow[from=2-1, to=2-2]
        \arrow[from=2-2, to=2-3]
        \arrow[from=2-3, to=2-4]
        \arrow[from=2-4, to=2-5]
        \arrow[from=2-5, to=2-6]
    \end{tikzcd}\]
\end{lemma}
\begin{proof}
    We begin with the case $\ell \geq 1$.
    Picking a weight decomposition $Y_{\leq \ell} \to Y \to Y_{\geq 1}$,
    we note that $X_{\leq 0} \to Y_{\geq 1}$ vanishes,
    and hence we obtain a factorization
    \[\begin{tikzcd}
        {X_{\leq 0}} & X & {X_{\geq 1-\ell}} \\
        {Y_{\leq \ell}} & Y & {Y_{\geq 1}}
        \arrow[from=1-1, to=1-2]
        \arrow["{f_0}"', dashed, from=1-1, to=2-1]
        \arrow[from=1-2, to=1-3]
        \arrow["f", from=1-2, to=2-2]
        \arrow[dashed, from=1-3, to=2-3]
        \arrow[from=2-1, to=2-2]
        \arrow[from=2-2, to=2-3]
    \end{tikzcd}\]
    The left square yields the base case of our induction on $n$.
    Suppose now we have already constructed the diagram up to $f_n$,
    and we wish to construct $f_{n+1}$ making the squares with $f_n$
    and $f$ commute.
    Consider the following commutative diagram, where $P$ is defined as the pushout
    \[\begin{tikzcd}
        {X_{\leq\ell n}} & {X_{\leq \ell(n+1)}} & X \\
        {Y_{\leq \ell(n+1)}} & P & Y
        \arrow[from=1-1, to=1-2]
        \arrow[from=1-1, to=2-1]
        \arrow[from=1-2, to=1-3]
        \arrow[from=1-2, to=2-2]
        \arrow[from=1-3, to=2-3]
        \arrow[from=2-1, to=2-2]
        \arrow["\lrcorner"{anchor=center, pos=0.125, rotate=180}, draw=none, from=2-2, to=1-1]
        \arrow[from=2-2, to=2-3]
    \end{tikzcd}\]
    Since $\ell \geq 1$, we have $\cofib(Y_{\leq \ell(n+1)} \to P) = \cofib(X_{\leq \ell n} \to X_{\leq \ell(n+1)}) \in \cC_{\leq \ell(n+1)}$,
    and hence also $P \in \cC_{\leq \ell (n+1)}$.
    Now pick a weight decomposition $Y_{\leq \ell(n+2)} \to Y \to Y_{\geq 1+\ell(n+1)}$,
    which yields a factorization of $P \to Y$ through $Y_{\leq \ell(n+2)} \to Y$,
    ultimately giving the commutative diagram
    \[\begin{tikzcd}
        {X_{\leq\ell n}} & {X_{\leq \ell(n+1)}} && X \\
        {Y_{\leq \ell(n+1)}} & P & {Y_{\leq \ell(n+2)}} & Y
        \arrow[from=1-1, to=1-2]
        \arrow[from=1-1, to=2-1]
        \arrow[from=1-2, to=1-4]
        \arrow[from=1-2, to=2-2]
        \arrow[from=1-4, to=2-4]
        \arrow[from=2-1, to=2-2]
        \arrow["\lrcorner"{anchor=center, pos=0.125, rotate=180}, draw=none, from=2-2, to=1-1]
        \arrow[from=2-2, to=2-3]
        \arrow[from=2-3, to=2-4]
    \end{tikzcd}\]
    By induction, we can thus build the desired diagram,
    and it remains to show that the sequential diagram $Y_{\leq \ell \bullet}$
    constructed in this way is indeed a $\cC_{<\infty}$-weight complex for $Y$.
    To this end, it suffices to check that the cofiber of the composite
    $Y_{\leq \ell(n+1)} \to P \to Y_{\leq \ell(n+2)}$ grows in connectivity
    with $n$. Using that $\cofib(Y_{\leq \ell(n+1)} \to P) = \cofib(X_{\leq \ell n} \to X_{\leq \ell(n+1)}) \in \cC_{\geq 1+\ell(n-1)}$
    and that we can similarly control the connectivity
    of the cofibers of the maps $Y_{\leq \ell(n+1)} \to Y$
    and $Y_{\leq \ell(n+2)} \to Y$,
    a straightforward diagram chase shows that
    the cofiber of the above composite is indeed in $\cC_{\geq 1+\ell(n-1)}$.

    For the case $\ell = 0$, the construction is essentially the same;
    in the inductive step,
    we again define $P = X_{\leq n+1} \amalg_{X_{\leq n}} Y_{\leq n} \in \cC_{\leq n+1}$
    and obtain a commutative diagram
    \[\begin{tikzcd}[ampersand replacement=\&]
        {X_{\leq n}} \& {X_{\leq n+1}} \& X \\
        {Y_{\leq n}} \& P \& Y
        \arrow[from=1-1, to=1-2]
        \arrow[from=1-1, to=2-1]
        \arrow[from=1-2, to=1-3]
        \arrow[from=1-2, to=2-2]
        \arrow[from=1-3, to=2-3]
        \arrow[from=2-1, to=2-2]
        \arrow["\lrcorner"{anchor=center, pos=0.125, rotate=180}, draw=none, from=2-2, to=1-1]
        \arrow[from=2-2, to=2-3]
    \end{tikzcd}\]
    We again note that since $P \in \cC_{\leq n+1}$ and $Y/Y_{\leq n+1} \in \cC_{\geq n+2}$
    we have a factorization of $P \to Y$ through $Y_{\leq n+1} \to Y$.
    We now need to argue that the composite $c \colon Y_{\leq n} \to P \to Y_{\leq n+1}$
    is homotopic to the map $i \colon Y_{\leq n} \to Y_{\leq n+1}$
    that is part of the given weight complex $Y_{\leq \bullet}$.
    By construction, both maps agree upon postcomposing with $Y_{\leq n+1} \to Y$,
    and hence $i-c \colon Y_{\leq n} \to Y_{\leq n+1}$ lifts to $\fib(Y_{\leq n+1} \to Y) = \Omega Y_{\geq n+2} \in \cC_{\geq n+1}$. By orthogonality we see that $i-c \simeq 0$, as desired.
    This agreement of maps is the part which fails for $\ell \geq 1$,
    and is hence the reason we cannot choose $Y_{\leq \ell \bullet}$ freely in that case.
\end{proof}

\section{Left completion of weak weight structures}\label{sec:upw-comp}

We now come to the heart of this paper.
The section is divided into three subsections.
In subsection \ref{subsec:left-completeness}
we define left completeness of weak weight structures
in terms of the completeness conditions we saw in Section \ref{sec:conn},
and draw some consequences from the theory developed there.
We conclude by showing that the connective part of a left complete
weighted category is freely generated by its weight heart under geometric realizations
(Proposition \ref{prop:weight-pdelta}).

Next, in subsection \ref{subsec:left-completion}
we define left completions
and state our main theorem (Theorem \ref{thm:upw-comp}) on their existence and properties.
This identifies the left completion $\lc{\cC}$ of $\cC$ as a certain full subcategory
of $\Ind(\cC_{<\infty})$ and discusses functoriality of the association $\cC \mapsto \lc{\cC}$.
We also draw many useful consequences of this Theorem,
such as the functoriality of weight complexes when considering them as $\Ind$-objects
(Corollary \ref{cor:weight-complex})
and that left completion interacts well with the dual notion of right completion
(Proposition \ref{prop:lr-complete}).

Finally, subsection \ref{subsec:proof-thm-upw-comp}
is devoted to the proof of Theorem \ref{thm:upw-comp}.
Along the way, we will see a criterion
for recognizing left completions (Corollary \ref{cor:left-comp-via-embed})
and how to generate a weight structure on $\Ind(\cC)$ for any weakly weighted $\cC$
(Proposition \ref{prop:ind-weight}).

\subsection{Left completeness of weak weight structures}\label{subsec:left-completeness}

\begin{definition}\label{def:weight-left-complete}
    A weakly weighted category $\cC$ is left (weight) complete
    if its underlying connectivity structure is colimit complete (cf.~Observation \ref{obs:colim-comp}), meaning:
    \begin{enumerate}
        \item $\cC_{\geq \infty} \coloneqq \bigcap_{n \geq 0} \cC_{\geq n} = 0$.
        \item $\cC$ admits sequential colimits growing in (weight) connectivity.
        \item $\cC_{\geq 0} \subseteq \cC$ is closed under such sequential colimits.
    \end{enumerate}
    We let $\wWCatl \subseteq \wWCat$ be the full subcategory on left complete weakly weighted categories,
    and similarly for $\wWCatl_{\geq 0},\wWCatl_{=0}$ from Definition \ref{def:wwcat}.
\end{definition}

\begin{remark}
    As mentioned in Observation \ref{obs:colim-comp},
    parts (2) and (3) together are equivalent to $\cC_{\geq 0}$ admitting
    and $\cC_{\geq 0} \subseteq \cC$ preserving geometric realizations.
\end{remark}

\begin{observation}\label{obs:left-comp}
    The property of a weakly weighted category $\cC$ being left complete
    does not depend on $\cC_{\leq 0}$. In particular, if $(\cC,\cC_{\geq 0},\cC'_{\leq 0})$
    is another weak weight structure on $\cC$ with the same connectives (e.g.~$\cC'_{\leq 0} = {}^\perp(\cC_{\geq 1})$), then it is also left complete.
    Moreover, if $\cC$ is left complete then $\cC_{>-\infty}$ is left complete and idempotent complete by Corollary \ref{cor:grow-conn-idem},
    and $\cC_{\geq 0}$ admits the ``finite type'' colimits described in Lemma \ref{lem:fintype-colim}.
    In particular, also $\cC_{\geq 0}$ and $\cC_{=0}$ are idempotent complete.
    Moreover, if $\cU$ is left dense in $\cC$,
    then any $\cU$-weight complex $X_\bullet$ of $X$ converges, i.e.~$\colim_n X_n \simeq X$.
    In this sense, any object in $\cC$ may be obtained as the colimit
    of a sequential diagram in $\cC_{<\infty}$ growing in connectivity.
\end{observation}

\begin{lemma}\label{lem:t-vs-w-comp}
    Let $\cC$ be a weakly weighted category which also admits a $t$-structure with $\cC_{w\geq 0} = \cC_{t \geq 0}$.
    \begin{enumerate}
        \item If the $t$-structure is left complete, then also the weight structure is left complete.
        \item Suppose that
            \begin{enumerate}
                \item $\cC_{\geq \infty} = 0$.
                \item $\cC$ admits sequential limits growing in connectivity.
                \item $\cC$ is right saturated in the sense of Definition \ref{def:saturated} (e.g.~it is a non-weak weight structure).
            \end{enumerate}
            Then the $t$-structure and hence also the weight structure is left complete.
    \end{enumerate}
\end{lemma}
\begin{proof}
    Point (1) is a direct consequence of Lemma \ref{lem:left-t-comp-is-conn-comp}.
    For (2), by the same lemma it suffices to show that the underlying connectivity
    structure is limit complete. The first two requirements are exactly the hypotheses (a) and (b),
    and the closure of $\cC_{\geq 0}$ under sequential limits growing in connectivity
    follows from the right saturation and Lemma \ref{lem:weight-closure}.
\end{proof}

\begin{example}
    Since the usual $t$-structure on $\Sp$ is complete, also the standard weight structure on $\Sp$ is left complete.
\end{example}

\begin{lemma}\label{lem:upw-comp-eqv}~
    \begin{enumerate}
        \item Let $f \colon \cC \to \cD$ be of amplitude $[a,b)$
            for $-\infty < a \leq b \leq \infty$
            between left complete weakly weighted categories.
            If its restriction $\cC_{<\infty} \to \cD_{<\infty}$
            is fully faithful \alt{an equivalence}, then so is $f$.

        \item Let $f \colon \cC \to \cD$ be a weight exact functor of left complete weakly weighted categories.
            If the restriction $\cC_{=0} \to \cD_{=0}$ is fully faithful \alt{an equivalence},
            then so is the restriction $\cC_{>-\infty} \to \cD_{>-\infty}$.
    \end{enumerate}
\end{lemma}
\begin{proof}
    For (1), we note that $f$ has bounded below amplitude
    and thus preserves sequential colimits growing in connectivity by Lemma \ref{lem:cn-cone-colim}.
    By the existence of convergent $\cD_{<\infty}$-weight complexes in $\cD$,
    essential surjectivity of the restriction implies essential surjectivity of $f$.
    Now suppose the restriction is fully faithful and let $X,Y \in \cC$ have $\cC_{<\infty}$-weight complexes $X_\bullet$ and $Y_\bullet$.
    Then $\hom(X,Y) \to \hom(fX,fY)$ identifies with $\lim_k \colim_n \hom(X_k,Y_n) \simeq \lim_k \colim_n \hom(fX_k,fY_n)$, so we conclude.
    For (2), exactness of $f$ implies that it restricts to a fully faithful functor
    on weight bounded objects which is essentially surjective if $f|_{\cC_{=0}}$ is.
    We noted in Observation \ref{obs:left-comp} that $\cC_{>-\infty},\cD_{>-\infty}$ are still left complete,
    so the claim now follows from (1).
\end{proof}

We now want to present another description
of the bounded below part of a left complete weighted category (i.e.~the case $\ell=0$).
This expresses that in a left complete weighted
category $\cC$, the category $\cC_{\geq 0}$ is in a sense
``freely generated'' by the colimits of weight complexes.
To make this formal, it is convenient to use the Dold-Kan correspondence
to recast sequential diagrams growing linearly in connectivity
into geometric realizations of diagrams landing in the weight heart.
To this end, recall the properties of the Dold-Kan correspondence
mentioned in Example \ref{ex:dold-kan},
which lets us make our above intuition precise in the following way.
Let $\cP^{\Delta^\op}(-)$ denote the operation of freely adding geometric realizations
in the sense of \cite[Prop.~5.3.6.2]{HTT}.

\begin{proposition}\label{prop:weight-pdelta}
    If $\cC$ is a left complete weighted category,
    then the inclusion $\cC_{=0} \subseteq \cC_{\geq 0}$
    induces equivalences $\cP^{\Delta^\op}(\cC_{=0}) \xto{\simeq} \cC_{\geq 0}$
    and $\SW(\cP^{\Delta^\op}(\cC_{=0})) \simeq \cC_{>-\infty}$.
    Moreover, every object in $\cP^{\Delta^\op}(\cC_{=0})$ is a geometric realization
    of a simplicial object in $\cC_{=0}$.
\end{proposition}
\begin{proof}~
    For this proof, we let $\cP'(\cC_{=0}) \subseteq \cP^{\Delta^\op}(\cC_{=0})$
    denote the full subcategory
    on geometric realizations of simplicial objects in $\cC_{=0}$
    (the point being that we don't allow iterated geometric realizations).
    We will first show that $\cP'(\cC_{=0}) \simeq \cC_{\geq 0}$,
    and then conclude that actually $\cP'(\cC_{=0}) \simeq \cP^{\Delta^\op}(\cC_{=0})$.

    By Example \ref{ex:dold-kan}, the category $\cC$
    admits geometric realizations of uniformly bounded below diagrams.
    In particular, $\cC_{\geq 0}$ admits geometric realizations,
    and we obtain a comparison functor
    $\cP^{\Delta^\op}(\cC_{=0}) \to \cC_{\geq 0}$ taking the colimit
    and restricting to the inclusion on $\cC_{=0}$.
    Let $f \colon \cP'(\cC_{=0}) \to \cC_{\geq 0}$ denote its restriction.
    We claim that $f$ is an equivalence.
    To see that it is essentially surjective,
    let $X \in \cC_{\geq 0}$ and consider an associated weight complex
    $X_{\leq \bullet}$ constructed as in Lemma \ref{lem:weight-cplx},
    so that $X_{\leq n} \in \cC_{[0,n]}$.
    The sequential diagram $X_{\leq 0} \to X_{\leq 1} \to \cdots$
    corresponds under Dold-Kan to a simplicial object $X'$
    with $|X'| = \colim_n X_{\leq n} = X$.
    Moreover, $X'$ is really a simplicial object in $\cC_{=0}$,
    since $X'_{n+1}$ is a direct summand of $\Omega^{n+1}\cofib(X_{\leq n} \to X_{\leq n+1}) \in \cC_{=0}$.
    Hence $f$ is essentially surjective.

    For $X_\bullet,Y_\bullet \in \cC_{=0}^{\Delta^\op}$ we have a commutative diagram
    \[\hspace{-2em}\begin{tikzcd}[ampersand replacement=\&, column sep = small]
        {\map_{\cP'(\cC_{=0})}(\indcolim{m} X_m, \indcolim{n} Y_n)} \& {\lim_m \map_{\cP'(\cC_{=0})}(X_m, \indcolim{n} Y_n)} \& {\lim_m \colim_n \map_{\cC_{=0}}(X_m,Y_n)} \\
        {\map_{\cC_{\geq 0}}(\colim_m X_m, \colim_n Y_n)} \& {\lim_m \map_{\cC_{\geq 0}}(X_m, \colim_n Y_n)} \& {\lim_m \colim_n \map_{\cC_{\geq 0}}(X_m,Y_n)}
        \arrow["\simeq", from=1-1, to=1-2]
        \arrow[from=1-1, to=2-1]
        \arrow[from=1-2, to=2-2]
        \arrow["\simeq"', from=1-3, to=1-2]
        \arrow["\simeq", from=1-3, to=2-3]
        \arrow["\simeq"', from=2-1, to=2-2]
        \arrow[from=2-3, to=2-2]
    \end{tikzcd}\]
    which reduces proving fully faithfulness of $f$ to showing
    that if $X \in \cC_{=0}$, then the colimit interchange morphism
    $
        |\map_{\cC_{\geq 0}}(X, Y_\bullet)| \to \map_{\cC_{\geq 0}}(X, |Y|)
    $
    is an equivalence.
    Since $X$ is weight coconnective and $|Y|$ and all the $Y_n$
    are weight connective, we may identify the above mapping spaces
    (which are canonically connective spectra by additivity
of $\cC_{\geq 0}$) with the corresponding mapping spectra in $\cC$.
    Then we note that $|Y| = \colim_n |Y|_n$ is a sequential colimit growing in connectivity,
    and that $\hom_\cC(X,-)$ preserves finite colimits and sequential colimits growing in connectivity
    by Example \ref{ex:corep-bdd}.

    This proves that $f$ is an equivalence.
    But then, since this is true in $\cC_{\geq 0}$,
    this implies that $\cP'(\cC_{=0})$ admits geometric realizations (since $\cC_{\geq 0}$ does)
    and for $X \in \cC_{=0}$
    the functor $\hom_{\cP'(\cC_{=0})}(X,-)$ preserves them.
    From this it follows that the inclusion $\cP'(\cC_{=0}) \subseteq \cP^{\Delta^\op}(\cC_{=0})$
    preserves geometric realizations, and hence, by the universal property of $\cP^{\Delta^\op}$,
    that the inclusion is an equivalence.
    The equivalence $\SW(\cP^{\Delta^\op}(\cC_{=0})) \simeq \cC_{>-\infty}$
    is then an instance of Lemma \ref{lem:stable-conn-struc}.
\end{proof}

\subsection{Left completion of weak weight structures and consequences}\label{subsec:left-completion}

Now that we have discussed the notion of left completeness,
the next goal will be to define and prove existence of a left completion.

\begin{definition}\label{def:left-completion}
    A weight exact functor of weakly weighted categories
    $\eta \colon \cC \to \lc{\cC}$
    exhibits $\lc{\cC}$ as the left completion of $\cC$
    if $\lc{\cC}$ is left complete and
    for any other left complete weakly weighted category $\cD$,
    precomposition by $\eta$ induces an equivalence
    \begin{equation}\label{eq:upw-comp}
        \eta^* \colon \Fun^\Ex_{[0,0]}(\lc{\cC},\cD) \xto{\simeq} \Fun^{\Ex}_{[0,0]}(\cC,\cD).
    \end{equation}
    In other words, $\eta$ exhibits $\lc{\cC}$ as the left adjoint object to $\cC$
    under the inclusion $\wWCatl_{=0} \subseteq \wWCat_{=0}$.
\end{definition}

To state the main theorem of this section,
let us mention that we will see in Proposition \ref{prop:ind-weight} below
that if $\cC$ is weakly weighted, then $\Ind(\cC)$ always admits a weight structure
where the Yoneda embedding $j \colon \cC \hookrightarrow \Ind(\cC)$ is weight exact.

\begin{theorem}\label{thm:upw-comp}
    Any weakly $\ell$-weighted category $\cC$ admits a weakly $\ell$-weighted left completion $\lc{\cC}$.
    This yields a left Bousfield localization $\lc{(-)} \colon \wWCat_{=0} \to \wWCatl_{=0}$ such that
    \begin{enumerate}
        \setlength{\itemsep}{.2em}
        \item There exists a fully faithful weight exact functor $\lc{\cC} \hookrightarrow \Ind(\cC_{<\infty})$
            which lets us identify
            \begin{align*}
                \qquad\qquad\lc{\cC}
                &= \{X \in \Ind(\cC_{<\infty}) \mid \forall n \geq 0\text{ there is $c \in \cC_{<\infty}$ and a map
                $jc \to X$ with $n$-connective cofiber}\}\\
                &= \{\indcolim{n} c_n \in \Ind(\cC_{<\infty}) \mid c_\bullet \colon \N \to \cC_{<\infty} \text{ growing in connectivity}\}\\
                &= \{X \in \Ind(\cC_{<\infty}) \mid X \text{ admits a $\cC_{<\infty}$-weight complex}\}
            \end{align*}
            which is the largest full stable subcategory of $\Ind(\cC_{<\infty})$
            in which $\cC_{<\infty}$ is left dense,
            and also the smallest full stable subcategory of $\Ind(\cC_{<\infty})$
            containing $\cC_{<\infty}$ and closed under sequential colimits growing in connectivity.
            Moreover, we have
            \[
                \lc{\cC}_{\geq 0}
                = \{\indcolim{n} c_n \mid c_\bullet \colon \N \to \cC_{[0,\infty)}\text{ growing in connectivity}\}.
            \]

        \item The inclusion induces a weight exact equivalence $\lc{(\cC_{<\infty})} \xto{\simeq} \lc{\cC}$. Moreover:
            \begin{enumerate}
                \setlength{\itemsep}{.2em}
                \item The unit $\eta_\cC \colon \cC \to \lc{\cC}$
                    is weight exact and the left Kan extension
                    of $\eta_{\cC_{<\infty}}$ along $\cC_{<\infty} \subseteq \cC$.

                \item $\eta_{\cC_{<\infty}}$
                    agrees with the Yoneda embedding
                    $j \colon \cC_{< \infty} \hookrightarrow \lc{\cC} \subseteq \Ind(\cC_{<\infty})$, so $\eta_{\cC_{<\infty}}$ always fully faithful.
                    In fact, for $X \in \cC_{<\infty}$ we have $\hom_\cC(X,-) \xto{\simeq} \hom_{\lc{\cC}}(\eta X, \eta -)$.

                \item $\cC_{<\infty} \subseteq \lc{\cC}$
                    is left dense, $\cC_{\leq 0} \subseteq \lc{\cC}_{\leq 0}$ is dense,
                    and we have $\cC_{[0,\infty)} = \cC_{<\infty} \cap \lc{\cC}_{\geq 0}$.

                \item $\eta_\cC$ agrees with the Verdier projection
                    $\cC \to \cC/\cC_{\geq \infty}$ whenever $\cC$ already
                    admits sequential colimits growing in connectivity
                    and $\cC_{\geq 0}$ is closed under them.
            \end{enumerate}

        \item For any left complete weakly weighted category $\cD$,
            $a \in \Z$ and $a \leq b \leq \infty$,
            restriction along $\eta$ induces an equivalence
            \[
                \eta_\cC^* \colon \Fun^\Ex_{[a,b]}(\lc{\cC},\cD) \xto{\simeq} \Fun^\Ex_{[a,b]}(\cC,\cD)
            \]
            with inverse given by left Kan extension along $\eta_\cC$.
            In particular, if $f \colon \cC \to \cD$ is a functor in $\wWCat$ of amplitude $[a,b]$,
            then $\lc{f} = (\eta_{\cC})_!(\eta_\cD f) \colon \lc{\cC} \to \lc{\cD}$
            again has amplitude $[a,b]$.
            Analogously for $[a,b)$ instead of $[a,b]$.

        \item In particular, left completion also defines a left adjoint
            to the inclusions $\wWCatl_{\geq 0} \subseteq \wWCat_{\geq 0}$
            respectively $\wWCatl \subseteq \wWCat$.
    \end{enumerate}
\end{theorem}

\begin{corollary}\label{cor:weight-complex}
    There is a ``weight complex functor''
    $w = \eta \colon \cC \to \lc{\cC} \subseteq \Ind(\cC_{<\infty})$ so that
    \begin{enumerate}
        \item $w$ is weight exact and
            restricts to the Yoneda embedding on $\cC_{<\infty}$.
        \item if $X_\bullet$ is any $\cC_{<\infty}$-weight complex
            for $X$, then $wX \simeq \indcolim{n} X_n$.
        \item In particular, $\cC_{<\infty}$-weight complexes are unique when viewed as $\Ind$-objects.
        \item If $\cC$ is stably symmetric monoidal and the weak weight structure
            is compatible with the monoidal structure
            in the sense that both $\cC_{\leq 0},\cC_{\geq 0} \subseteq \cC$
            are symmetric monoidal subcategories, then $w \colon \cC \to \Ind(\cC_{<\infty})$
            is unital and lax symmetric monoidal,
            and the structure map $wX \tensor wY \to w(X \tensor Y)$ is an equivalence
            when $X$ and $Y$ are either both bounded above or both bounded below.
            This is an analogue of Aoki's monoidal weight complex functor \cite{Aoki-Weight}.
    \end{enumerate}
\end{corollary}
\begin{proof}
    Points (1)-(3) are immediate consequences of the theorem.
    For (4), we equip $\Ind(\cC_{<\infty})$ with the canonical induced presentably symmetric monoidal structure
    making the Yoneda embedding $j$ symmetric monoidal.
    Now since $i \colon \cC_{<\infty} \subseteq \cC$ is symmetric monoidal,
    also the left Kan extension $i_! \colon \Fun(\cC_{<\infty},\Ind(\cC_{<\infty}))
    \to \Fun(\cC,\Ind(\cC_{<\infty}))$ obtains a symmetric monoidal structure
    when we equip both source and target with the Day convolution monoidal structure
    (for target category $\An$ this holds by symmetric monoidality of $\cP(-) \colon \An \to \Pr^L$
    \cite[Section 4.8.1]{HA}, and then we can tensor in $\CAlg(\Pr^L)$ with $\Ind(\cC_{<\infty})$).
    Commutative algebras in the day convolution are equivalently lax symmetric monoidal functors
    \cite[Example 2.2.6.9]{HA},
    hence $\eta = i_! j$ is lax symmetric monoidal.
    It restricts to $j$ hence is symmetric monoidal on $\cC_{<\infty}$.
    If $X,Y$ are bounded below with weight complexes $X_\bullet$ respectively $Y_\bullet$,
    then it is easy to check that also $(X_n \tensor Y_n)_n$ is a weight complex for $X \tensor Y$.
    Thus $\colim_n w(X_n \tensor Y_n) \simeq w(X \tensor Y)$, and the canonical structure map
    $wX \tensor wY \to w(X \tensor Y)$ is the map induced on colimits by the equivalences
    $wX_n \tensor wY_n \xto{\simeq} w(X_n \tensor Y_n)$.
\end{proof}

\begin{corollary}\label{cor:upw-comp-idem}
    Let $\cC$ be a weakly weighted category. Then:
    \begin{enumerate}
        \item $\cC_{\wb} \hookrightarrow (\lc{\cC})_{\wb}$
            is an idempotent completion.
            In particular, we recover a special
            case of \cite[Theorem 0.1.(2)]{BondarkoSosnilo},
            that bounded weight structures always extend to the idempotent completion.

        \item $\lc{\cC}_{<\infty} = \{X \in \Ind(\cC_{<\infty}) \mid \text{ there is some bounded below $Y$ with }X \oplus Y \in \cC_{<\infty}\}$.

        \item $\cC_{<\infty} = (\lc{\cC})_{<\infty}$
            if and only if $\cC_{\wb}$ is idempotent complete.
            In particular, if $\cC_{<\infty}$ is idempotent complete
            then also $(\lc{\cC})_{<\infty}$ is.
    \end{enumerate}
\end{corollary}
\begin{proof}
    For (1), note that $\lc{\cC}_{\wb}$ is idempotent complete by Observation \ref{obs:left-comp},
    so it remains to show that $\cC_{\wb} \hookrightarrow (\lc{\cC})_{\wb}$ is dense.
    By Theorem \ref{thm:upw-comp}(2c)
    the inclusion $\cC_{<\infty} \subseteq \lc{\cC}$ is left dense,
    so $\cC_{<\infty} \subseteq \lc{\cC}_{<\infty}$ is dense.
    Now let $X \in \lc{\cC}_{[a,b]}$ for some $a,b \in \Z$.
    Then there is some $c \in \cC_{<\infty}$ and a map $c \to X$ with cofiber in $\lc{\cC}_{\geq b+1}$.
    It follows that $X$ is a retract of $c \in \cC_{<\infty} \cap \lc{\cC}_{\geq a} = \cC_{[a,\infty)} \subseteq \cC_{\wb}$, as desired.

    To see (2), note that by density we have an inclusion $(\lc{\cC})_{<\infty} \subseteq (\cC_{<\infty})^\idem$.
    Now suppose that $X \in (\lc{\cC})_{\leq n}$.
    We find $c \in \cC_{<\infty}$ and a map $jc \to X$
    with cofiber $C$ in $\lc{\cC}_{\geq n+1}$.
    So with $Y = \Omega C$ we get a splitting $jc \simeq X \oplus Y$.

    Conversely, suppose that $X,Y \in \Ind(\cC_{<\infty})$
    with $X \oplus Y \simeq jc$, where $c \in \cC_{<\infty}$
    and $Y$ is bounded below. We show that then $Y \in \lc{\cC}$ (so also
    $X \in \lc{\cC}$) using the explicit description of $\lc{\cC}$ from Theorem \ref{thm:upw-comp}(1).
    We have the cofiber sequences $jc \to Y \to \Sigma X$
    and $\Sigma jc \to \Sigma X \to \Sigma^2 Y$,
    which together yield cofiber sequences
    $jd \to Y \to \Sigma^2 Y$ and $c \to d \to \Sigma c$ so that $d \in \cC_{<\infty}$.
    Continuing inductively, we obtain cofiber sequences
    $je_n \to Y \to \Sigma^{2n} Y$ for some $e_n \in \cC_{<\infty}$.
    Since $Y$ was assumed bounded below, this shows $Y \in \lc{\cC}$.

    Finally, let us show (3). If $\cC_{\wb}$ is idempotent complete,
    then combining (1) and (2) shows that for any $X \in (\lc{\cC})_{< \infty}$
    there is some $Y \in \cC_{\wb}$ for which $X \oplus Y \in \cC_{<\infty}$.
    But then clearly also $X \in \cC_{<\infty}$.
    Conversely, suppose that $\cC_{<\infty} = (\lc{\cC})_{< \infty}$
    and let $X \in (\cC_{\wb})^\idem \subseteq (\cC_{<\infty})^\idem$.
    Then there is some $Y \in (\cC_{\wb})^\idem$ such that $X \oplus Y \in \cC_{\wb}$,
    and hence $X \in (\lc{\cC})_{<\infty} = \cC_{<\infty}$ by (1).
\end{proof}

\begin{remark}\label{rem:w-idem-bs}
    Even if $\cC$ is left complete,
    the question of idempotent completeness of the entire
    category $\cC$ is a lot more subtle.
    We will see a concrete example of a left (and right!)
    complete weighted category
    whose non-idempotent-completeness we learned from \cite{BondarkoSosnilo}
    in Remark \ref{rem:ch-idem} below.
    Their paper discusses the interaction of weight structures and idempotent completions, and their methods show that if $\cC$ is a left complete (weakly) weighted category, then the following are equivalent:
    \begin{enumerate}
        \item $\cC$ is idempotent complete.
        \item There exists a (weak) weight structure on $\cC^\idem$ so that
            the inclusion $\cC \subseteq \cC^\idem$ is weight exact.
        \item $\cC^\idem$ is generated as a stable category by $(\cC_{<\infty})^\idem$ and $(\cC_{>-\infty})^\idem$ ($=\cC_{>-\infty}$).
    \end{enumerate}
    Indeed, it is not hard to see a priori that if $\cC$ is left complete
    and the weak weight structure extends to $\cC^\idem$,
    then the latter is also left complete, by noting that $\cC \subseteq \cC^\idem \subseteq \Ind(\cC_{<\infty})$.
    Moreover, by \cite[\href{https://kerodon.net/tag/03YQ}{Tag 03YQ}]{Kerodon}
    if $\cD'$ is any category and $\cD \subseteq \cD'$ a dense full subcategory,
    then any functor $\cD' \to \cE$ is both left and right Kan extended from $\cD$.
    In particular, since the definition of the left completion of a weakly weighted category $\cC$
    is essentially ``the maximal left complete weakly weighted category so that every map of bounded below amplitude
    is left Kan extended from $\cC$'',
    this yields the equivalence of (1) and (2).
    In the notation of \cite{BondarkoSosnilo},
    we always have $\lc{\cC} = \mathrm{Kar}^w_\mathsf{max}(\lc{\cC})$.
    The equivalence with point (3) follows from \cite[Theorem 2.2.2(I.2)]{BondarkoSosnilo}.
\end{remark}

\begin{corollary}\label{cor:left-comp-inverts-idem}
    Let $\cC,\cD$ be weakly weighted
    and $f \colon \cC \to \cD$ of amplitude $[a,\infty)$ for $a \in \Z$.
    \begin{enumerate}
        \item $\lc{f} \colon \lc{\cC}_{>-\infty} \to \lc{\cD}_{>-\infty}$
            is fully faithful \alt{an equivalence}
            exactly if $f^\idem \colon (\cC_{\wb})^\idem \to (\cD_{\wb})^\idem$ is.

        \item $\lc{f} \colon \lc{\cC} \to \lc{\cD}$ is fully faithful
            exactly if $f \colon \cC_{<\infty} \to \cD_{<\infty}$ is.

        \item $\lc{f} \colon \lc{\cC} \to \lc{\cD}$ is an equivalence
            exactly if $f \colon \cC_{<\infty} \to \cD_{<\infty}$ is a fully faithful left dense inclusion.
    \end{enumerate}
    Here `equivalence' means `equivalence of underlying categories'
    as opposed to `equivalence in $\wWCat$',
    though in the saturated case we do get an equivalence in $\wWCat$
    by Lemma \ref{lem:saturated}(3),
    and for $\ell < \infty$ it also follows from Lemma \ref{lem:saturation}.
\end{corollary}
\begin{proof}
    By Theorem \ref{thm:upw-comp} also $\lc{f}$ has amplitude $[a,\infty)$.
    For (1), note that if $\lc{f} \colon \lc{\cC}_{>-\infty} \to \lc{\cD}_{>-\infty}$
    is an equivalence, then also its restriction to weight bounded objects is an equivalence,
    so by Corollary \ref{cor:upw-comp-idem}(1) also $f^\idem$ is.
    For the converse assume that $f^\idem \colon (\cC_{\wb})^\idem \to (\cD_{\wb})^\idem$
    is an equivalence. By Observation \ref{obs:left-comp} and Lemma \ref{lem:upw-comp-eqv} it suffices to show
    that $\lc{f}$ restricts to an equivalence on weight bounded objects.
    We have a commutative square
    \[\begin{tikzcd}[ampersand replacement=\&]
        {\cC_{\wb}} \& {(\lc{\cC})_{\wb}} \\
        {\cD_{\wb}} \& {(\lc{\cD})_{\wb}}
        \arrow["{\eta_{\cC}}", from=1-1, to=1-2]
        \arrow["f"', from=1-1, to=2-1]
        \arrow["{\lc{f}}", from=1-2, to=2-2]
        \arrow["{\eta_{\cD}}", from=2-1, to=2-2]
    \end{tikzcd}\]
    and by Corollary \ref{cor:upw-comp-idem}(1)
    the horizontal maps are idempotent completions,
    which allows us to identify $\lc{f}|_{(\lc{\cC})_{\wb}}$
    with the equivalence $(f|_{\cC_{\wb}})^\idem$.
    This shows (1).

    Point (2) is a consequence of Lemma \ref{lem:upw-comp-eqv},
    the natural fully faithful inclusions
    $\cC_{<\infty} \subseteq \lc{\cC}_{<\infty} \subseteq (\cC_{<\infty})^\idem$,
    cancellation properties of fully faithful functors,
    and that $(-)^\idem$ preserves fully faithful functors.

    For (3), suppose first that $f \colon \cC_{<\infty} \to \cD_{<\infty}$ is fully faithful left dense.
    Since also $\cD_{<\infty} \subseteq \lc{\cD}$ is fully faithful left dense,
    one easily checks that the composite $\cC_{<\infty} \to \lc{\cD}$ is left dense as well.
    In particular, if $X \in \lc{\cD}$, it admits a $\cC_{<\infty}$-weight complex $c_\bullet \colon \N \to (\cC_{<\infty})_{/X}$, which converges to $X$ in $\lc{\cD}$ by left completeness.
    Since $c_\bullet$ is growing in connectivity, it admits a colimit in $\lc{\cC}$
    and $\lc{f}$ preserves this colimit.
    This proves that $\lc{f}$ is essentially surjective,
    and fully faithfulness follows from (2).

    For the converse, suppose that $\lc{f}$ is an equivalence.
    Then $f \colon \cC_{<\infty} \to \cD_{<\infty}$ is fully faithful.
    Let $d \in \cD_{<\infty}$ and $n \geq 0$.
    There exists $X \in \lc{\cC}$ with $\lc{f}(X) = \eta_\cD(d) \in \lc{\cD}$.
    Now by left density of $\cC_{<\infty} \subseteq \lc{\cC}$,
    there is some $c \in \cC_{<\infty}$ and a map $\eta_\cC c \to X$
    with cofiber in $\lc{\cC}_{\geq n-a}$.
    Hence $\eta_{\cD}(fc) = \lc{f}(\eta_\cC c) \to \lc{f}(X) = \eta_{\cD}(d)$ has cofiber in $\lc{\cD}_{\geq n}$.
    By fully faithfulness of $\eta_{\cD}$ on bounded above objects,
    we get that there is a map $fc \to d$ in $\cD_{<\infty}$ whose cofiber is in $\cD_{<\infty} \cap \lc{\cD}_{\geq n} = \cD_{[n,\infty)}$ (cf.~Theorem \ref{thm:upw-comp}(2c)).
    This proves that $f \colon \cC_{<\infty} \hookrightarrow \cD_{<\infty}$ is also left dense.
\end{proof}

\begin{corollary}\label{cor:left-comp-included}
    Let $\cC,\cD$ be weakly weighted categories, and suppose that $\cD$ is left complete
    and we have a fully faithful
    weight exact $f \colon \cC_{<\infty} \hookrightarrow \cD$.
    Then the induced $F = (\eta_\cC)_!f \colon \lc{\cC} \to \cD$ is fully faithful.
\end{corollary}
\begin{proof}
    Since $\cC_{<\infty} \subseteq \lc{\cC}_{<\infty}$ is dense,
    this follows from Lemma \ref{lem:upw-comp-eqv}(1).
\end{proof}

\begin{remark}\label{rem:dw-comp}
    Let $\cC$ be weakly $\ell$-co-weighted in the sense of Warning \ref{warn:wwcat-non-dual},
    so that $\cC^\op$ is weakly $\ell$-weighted.
    We say that $\cC$ is right complete if $\cC^\op$ is left complete,
    i.e.~if its underlying coconnectivity structure $(\cC,\cC_{\leq 0})$ is limit complete.
    The results in this section
    then give an analogous good theory of right (weight) completion
    on the category $\wWCat'$ of weakly co-weighted categories and exact functors
    of bounded above amplitude.
    For example, Theorem \ref{thm:upw-comp}(1)
    dualizes to exhibit the right completion $\rc{\cC}$
    of some weakly co-weighted category $\cC$ as the full subcategory
    of $\Pro(\cC_{>-\infty})$ on those objects $X$
    where for each $n \geq 0$ there exists some $c \in \cC_{>-\infty}$
    and a map $X \to c$ with $(-n)$-coconnective fiber
    (with respect to the weight structure induced by the dual of Proposition \ref{prop:ind-weight}).
\end{remark}

We now show that the procedure of left completing weakly bi-weighted categories
(cf.~Warning \ref{warn:wwcat-non-dual}) commutes with the dual procedure of right completion
(cf.~Remark \ref{rem:dw-comp}).

\begin{proposition}\label{prop:lr-complete}
    Let $\cC$ be a weakly bi-weighted category
    (e.g.~$\cC$ is weakly $\ell$-weighted for some $\ell < \infty$).
    \begin{enumerate}
        \item If $\cC$ is right \alt{left} complete
            then $\lc{\cC}$ \alt{$\rc{\cC}$} is both left and right complete.

        \item There is an equivalence $\rc{(\lc{\cC})} \simeq \lc{(\rc{\cC})}$ in $(\wWCat_{=0})_{\cC/}$.
    \end{enumerate}
\end{proposition}
\begin{proof}
    We begin with (1).
    Let $\cC$ be right complete, so that $\cC_{\leq -\infty} = 0$,
    $\cC$ admits sequential limits growing in coconnectivity
    and $\cC_{\leq 0}$ is closed under such limits.
    Then $\cC_{<\infty}$ is also right complete, and in particular idempotent complete by the dual of Corollary \ref{cor:grow-conn-idem}.
    Thus also $\cC_{\leq 0}$ is idempotent complete,
    and the inclusion $\cC_{\leq 0} \subseteq (\lc{\cC})_{\leq 0}$
    is an equivalence by density (Theorem \ref{thm:upw-comp}(2c)).
    Therefore $(\lc{\cC})_{\leq -\infty} \simeq \cC_{\leq -\infty} = 0$.

    Now let $X_\bullet \colon \N^\op \to \lc{\cC}$ be a sequential diagram
    growing in coconnectivity.
    By cofinality we can assume that $\fib(X_{n+1} \to X_n) \in \lc{\cC}_{\leq -(n+1)}$ for all $n \geq 0$.
    Now let $X = \lim_n X_n$ computed in $\Ind(\cC_{<\infty})$,
    and $k \geq 0$. We need to show that there exists a $c \in \cC_{<\infty}$
    and a map $c \to X$ with $k$-connective cofiber.
    There is a weight decomposition $X_{0,\leq b} \to X_0 \to X_{0, \geq k}$
    in $\lc{\cC}$ for some $0 \leq b < \infty$ and an extension
    $\fib(X_1 \to X_0) \to \fib(X_1 \to X_{0,\geq k}) \to X_{0,\leq b}$
    which shows that also
    $X_{1,\leq b} \coloneqq \fib(X_1 \to X_{0,\geq k})$ lies in $(\lc{\cC})_{\leq b}$.
    Continuing inductively, we can build the sequential diagram of vertical fiber sequences in $\lc{\cC}$
    \[\begin{tikzcd}[ampersand replacement=\&]
        \cdots \& {X_{2,\leq b}} \& {X_{1,\leq b}} \& {X_{0,\leq b}} \\
        \cdots \& {X_2} \& {X_1} \& {X_0} \\
        \cdots \& {X_{0,\geq k}} \& {X_{0,\geq k}} \& {X_{0,\geq k}}
        \arrow[from=1-1, to=1-2]
        \arrow[from=1-2, to=1-3]
        \arrow[from=1-2, to=2-2]
        \arrow[from=1-3, to=1-4]
        \arrow[from=1-3, to=2-3]
        \arrow[from=1-4, to=2-4]
        \arrow[from=2-1, to=2-2]
        \arrow[from=2-2, to=2-3]
        \arrow[from=2-2, to=3-2]
        \arrow[from=2-3, to=2-4]
        \arrow[from=2-3, to=3-3]
        \arrow[from=2-4, to=3-4]
        \arrow[Rightarrow, no head, from=3-1, to=3-2]
        \arrow[Rightarrow, no head, from=3-2, to=3-3]
        \arrow[Rightarrow, no head, from=3-3, to=3-4]
    \end{tikzcd}\]
    Note also that $\fib(X_{n+1,\leq b} \to X_{n,\leq b}) = \fib(X_{n+1} \to X_n)$,
    so that the top row is also a sequential diagram growing in coconnectivity.
    Since it also lives entirely in $\lc{\cC}_{\leq b} = \cC_{\leq b}$
    and the latter admits limits of such diagrams
    (and the inclusion into $\Ind(\cC_{<\infty})$ preserves them),
    we see that passing to the limit yields a fiber sequence
    $X_{\leq b} \to X \to X_{0,\geq k}$ with $X_{\leq b} \in \cC_{\leq b}$,
    as desired.
    Thus $\lc{\cC}$ admits limits of sequential diagrams growing in coconnectivity.
    Since $\cC_{\leq 0} \simeq \lc{\cC}_{\leq 0}$ also admits such limits
    and the inclusion $\cC_{<\infty} \to \lc{\cC}$ preserves limits
    we deduce that $\lc{\cC}$ is still right complete.
    Applying this result to $\cC^\op$,
    we also see that if $\cC$ is left complete, then $\rc{\cC}$ is both left and right complete.

    We now show (2).
    Let $\cB$ be a (left and right) complete weakly (co-)weighted category.
    Denote by $\eta_\cC \colon \cC \to \lc{\cC}$ the left completion
    and by $\eps_\cC \colon \cC \to \rc{\cC}$ the right completion.
    Then
    \[
        \Fun^\Ex_{[0,0]}(\lc{(\rc{\cC})},\cB)
        \xto[\simeq]{\eta_{\rc{\cC}}^*} \Fun^\Ex_{[0,0]}(\rc{\cC},\cB)
        \xto[\simeq]{\eps_\cC^*} \Fun^\Ex_{[0,0]}(\cC,\cB)
    \]
    and similarly $\Fun^\Ex_{[0,0]}(\rc{(\lc{\cC})},\cB) \simeq \Fun^\Ex_{[0,0]}(\cC,\cB)$.
    By Yoneda (applied to the full subcategory of $\wWCat_{=0}$ on the complete weakly (co-)weighted categories),
    this yields an equivalence $\lc{(\rc{\cC})} \simeq \rc{(\lc{\cC})}$ in $(\wWCat_{=0})_{\cC/}$.
\end{proof}

\begin{notation}\label{not:lr-comp}
    If $\cC$ is weakly bi-weighted we will simply write $\lrc{\cC} \coloneqq \rc{(\lc{\cC})}$.
\end{notation}

\subsection{The proof of Theorem \ref{thm:upw-comp}}\label{subsec:proof-thm-upw-comp}

We now work towards the proof of Theorem \ref{thm:upw-comp},
and begin by showing that functors $f \colon \cC \to \cD$ in $\wWCat$
where $\cD$ is left complete are uniquely determined by their restriction to a left dense subcategory.
It will be convenient to state this in the slightly larger generality of a colimit complete
connectivity structure $(\cD,\cD_{\geq 0})$ as target.

\begin{lemma}\label{lem:weight-comp-lkan}
    Let $\cC$ be a weakly weighted category and $(\cD,\cD_{\geq 0})$
    a colimit complete stable connectivity structure.
    Suppose that $i \colon \cU \subseteq \cC$ is left dense
    and equip $\cU$ with connectives $\cU_{\geq 0} = \cU \cap \cC_{\geq 0}$.
    For $a \in \Z$, Kan extension and restriction along $i$ induce an equivalence
    \[
        i_! \colon \Fun^\Ex_{\geq a}(\cU,\cD) \simeq \Fun^\Ex_{\geq a}(\cC,\cD) \noloc i^*
    \]
    where $\Fun^\Ex_{\geq a} \subseteq \Fun^\Ex$ is the full subcategory
    of functors of amplitude $\geq a$, i.e.~restricting to $\cC_{\geq 0} \to \cD_{\geq a}$.
\end{lemma}
\begin{proof}
    Let $f \colon \cU \to \cD$ be an exact functor of amplitude $\geq a$.
    We begin by showing that $i_!f$ exists.
    Using the pointwise formula for left Kan extensions,
    we need to show that for $X \in \cC$, the colimit of the diagram $\cU_{/X} \to \cU \xto{f} \cD$ exists.
    Pick a $\cU$-weight complex $U_\bullet \colon \N \to \cU_{/X}$ for $X$.
    This is cofinal by Lemma \ref{lem:weight-cplx},
    so it suffices to check that the colimit
    of the composite diagram $f(U_\bullet) \colon \N \to \cD$ exists.
    But by assumption on $f$ this is again a sequential diagram growing in connectivity,
    hence admits a colimit by colimit completeness of $\cD$.
    Thus the left Kan extension $i_!f$ exists.

    Now let $X \in \cC_{\geq 0}$.
    Picking a $\cU$-weight complex $U_\bullet$ for $X$ as in Lemma \ref{lem:weight-cplx},
    we see that after (without loss of generality) discarding the first term,
    we have $U_n \in \cC_{\geq 0}$.
    Since by definition $\cU_{\geq 0} = \cU \cap \cC_{\geq 0}$,
    it follows that also $U_n \in \cU_{\geq 0}$.
    But then $(i_!f)(X) = \colim_n f(U_n)$ is a colimit of a sequential diagram
    growing in connectivity in $\cD_{\geq a}$.
    By colimit completeness of $\cD$ the colimit is still $a$-connective.
    Thus $i_!f$ sends connectives to $a$-connectives,
    so has amplitude $\geq a$ again.

    Hence $i_!$ exists and is fully faithful since $i$ is fully faithful.
    Clearly $i^*$ also exists and is right adjoint to $i_!$.
    We now check that $i_!$ is also essentially surjective.
    Suppose that $f \colon \cC \to \cD$ has amplitude $\geq a$.
    Then $i_!i^*f$ exists and is of amplitude $\geq a$ by the above,
    and we have a counit $\eps \colon i_!i^*f \Rightarrow f$.
    For $X \in \cC$ with weight complex $U_\bullet \colon \N \to \cU_{/X}$,
    the above shows that $\eps_X$ is precisely map $\colim_n f(U_n) \to f(X)$,
    which is an equivalence by Lemma \ref{lem:cn-cone-colim}.
\end{proof}

\begin{proposition}\label{prop:weight-comp-lkan}
    Let $\cU,\cC,\cD$ be weakly weighted categories.
    Suppose that
    \begin{enumerate}
        \item $\cD$ is left complete.
        \item There is a fully faithful weight exact inclusion $\cU \subseteq \cC_{<\infty}$
            so that $i \colon \cU \subseteq \cC$ is left dense.
        \item $\cU_{\leq 0} \subseteq \cC_{\leq 0}$ is dense.
        \item $\cU_{\geq 0} = \cU \cap \cC_{\geq 0}$.\footnote{For example, by Lemma \ref{lem:saturated} this holds if $\cU$ is right saturated (e.g.~$\ell=0$).}
    \end{enumerate}
    Then for $a \in \Z$ and $a \leq b \leq \infty$
    the equivalence of Lemma \ref{lem:weight-comp-lkan} restricts to
    \[
        i_! \colon \Fun^\Ex_{[a,b]}(\cU,\cD) \simeq \Fun^\Ex_{[a,b]}(\cC,\cD) \noloc i^*
    \]
    and similarly for $[a,b)$.
\end{proposition}
\begin{proof}
    By Lemma \ref{lem:weight-comp-lkan}
    it remains to check that if $f \colon \cU \to \cD$ has amplitude $[-\infty,b]$,
    then also $i_!f$ does.
    But by density of $\cU_{\leq 0} \subseteq \cC_{\leq 0}$ any $X \in \cC_{\leq 0}$
    is a retract of some $U \in \cU_{\leq 0}$.
    Hence $i_!f(X)$ is a retract of $i_!f(U) = f(U) \in \cD_{\leq b}$,
    showing that $i_!f$ has amplitude $[-\infty,b]$ again.
    The same argument works for $[-\infty,b)$.
\end{proof}

\begin{corollary}\label{cor:left-comp-via-embed}
    Let $\cC,\cD$ be a weakly weighted categories.
    Suppose that
    \begin{enumerate}
        \item $\cD$ is left complete.
        \item There is a fully faithful weight exact $j \colon \cC_{<\infty} \hookrightarrow \cD$
            which is left dense,
        \item $\cC_{\leq 0} \subseteq \cD_{\leq 0}$ is dense.
        \item $\cC_{[0,\infty)} = \cC_{<\infty} \cap \cD_{\geq 0}$
            (e.g.~$\cC_{<\infty}$ is right saturated, e.g.~weighted).
    \end{enumerate}
    Then $j$ and the left Kan extension $i_!j \colon \cC \to \cD$
    of $j$ along $i \colon \cC_{<\infty} \subseteq \cC$ are left completions.
\end{corollary}

\begin{remark}
    In particular, if $\cC$ is left complete, then $\cC_{<\infty} \subseteq \cC$ is a left completion.
\end{remark}

\begin{proof}
    That $j$ is a left completion is immediate from Proposition \ref{prop:weight-comp-lkan}.
    Similarly, since $i \colon \cC_{<\infty} \subseteq \cC$ also satisfies these assumptions,
    the same result shows that $i_!j \colon \cC \to \cD$ is again weight exact,
    and that for any left complete weakly weighted $\cE$
    also $i^*$ is an equivalence in the following commutative triangle
    \[\begin{tikzcd}[ampersand replacement=\&]
        {\Fun^\Ex_{[0,0]}(\cD,\cE)} \& {\Fun^\Ex_{[0,0]}(\cC,\cE)} \\
        \& {\Fun^\Ex_{[0,0]}(\cC_{<\infty},\cE)}
        \arrow["{(i_!j)^*}", from=1-1, to=1-2]
        \arrow["{j^*}"', from=1-1, to=2-2]
        \arrow["\simeq", draw=none, from=1-1, to=2-2]
        \arrow["{i^*}", from=1-2, to=2-2]
        \arrow["\simeq"', draw=none, from=1-2, to=2-2]
    \end{tikzcd}\]
    Thus by 2-out-of-3 we see that also $i_!j$ has the correct universal property.
\end{proof}

\begin{corollary}\label{cor:easy-left-comp}
    Suppose there is a weight exact functor $f \colon \cC \to \cD$
    of weighted categories so that $\cD$ is left complete
    and $f$ restricts to an equivalence $f \colon \cC_{w < \infty} \simeq \cD_{w < \infty}$. Then $f$ is a left completion.
\end{corollary}
\begin{proof}
    This follows by applying Corollary \ref{cor:left-comp-via-embed},
    noting that assumption (4) holds since $\cC$ is weighted.
\end{proof}

\begin{example}\label{ex:left-completion}
    Consider $\Sp^\omega$ with its standard weight structure (restricted from $\Sp$).
    Applying Corollary \ref{cor:left-comp-via-embed},
    we recognize $\lc{(\Sp^\omega)} \simeq \Sp^\fg_{>-\infty}$ as the full subcategory of $\Sp$
    consisting of those spectra which are bounded
    below and have finitely generated homotopy groups.
    We will see more examples in Sections \ref{sec:chains} and \ref{sec:anderson}.
\end{example}

\begin{lemma}\label{lem:upw-comp-basic}
    Let $\cC$ be weakly $\ell$-weighted and $\cD$ weakly weighted.
    Suppose that:
    \begin{enumerate}
        \item[(i)] $\cD$ admits sequential colimits growing in connectivity and $\cD_{\geq 0}$ is closed under them.
        \item[(ii)] There is a fully faithful weight exact embedding $j \colon \cC_{<\infty} \hookrightarrow \cD$.
    \end{enumerate}
    Denote by $\langle \cC_{<\infty}\rangle \subseteq \cD$
    the smallest full subcategory of $\cD$ which contains $\cC_{<\infty}$
    and is closed under sequential colimits growing in connectivity.
    Then:
    \begin{enumerate}
        \setlength{\itemsep}{.5em}
        \item If $c_\bullet,d_\bullet \colon \N \to\cC_{<\infty}$ are growing in connectivity
            then any map $\colim c_\bullet \to \colim d_\bullet$ in $\cD$
            lifts to a map of diagrams $c_\bullet \Rightarrow d_{f(\bullet)}$ for some
            strictly increasing $f \colon \N \to \N$.

        \item As full subcategories of $\cD$, we have
            \begin{align*}
                \langle \cC_{<\infty} \rangle
                    =\ & \{\colim_n c_n \mid c_\bullet \colon \N \to \cC_{<\infty} \text{ growing in connectivity}\},\\
                \langle \cC_{<\infty} \rangle_{\geq 0}
                    \coloneqq\ &
                    \{\colim_n c_n \mid c_\bullet \colon \N \to \cC_{[0,\infty)} \text{ growing in connectivity}\}
                    \subseteq \langle \cC_{<\infty}\rangle \cap \cD_{\geq 0},\\
                \langle \cC_{<\infty} \rangle \cap \cD_{\geq \infty}
                    =\ &0,\\
                \qquad\qquad\cC_{<\infty} \cap \langle \cC_{<\infty}\rangle_{\geq 0}
                    =\ & \cC_{[0,\infty)}.
            \end{align*}
            Moreover, $\langle \cC_{<\infty}\rangle_{\geq 0}$
            is closed under sequential colimits growing in connectivity.
            In case that $\cC_{[0,\infty)} = \cC_{<\infty} \cap \cD_{\geq 0}$
            (e.g.~if $\cC_{<\infty}$ is right saturated)
            we have $\langle \cC_{<\infty}\rangle_{\geq 0} = \langle \cC_{<\infty}\rangle \cap \cD_{\geq 0}$.

        \item $\langle \cC_{<\infty} \rangle$ is a stable subcategory of $\cD$
            and with $\langle \cC_{<\infty}\rangle_{\leq 0} \coloneqq \langle \cC_{<\infty}\rangle \cap \cD_{\leq 0}$ we have:
            \begin{enumerate}
                \item $\langle \cC_{<\infty}\rangle$ is a left complete weakly $\ell$-weighted category.
                \item Both inclusions $\cC_{<\infty} \subseteq \langle \cC_{<\infty}\rangle
                    \subseteq \cD$ are weight exact, and the first is left dense.
                \item If $\cC_{\leq 0} \subseteq \langle \cC_{<\infty}\rangle_{\leq 0}$
                    is dense, then both $\cC_{<\infty} \subseteq \langle \cC_{<\infty}\rangle$
                    and its Kan extension $\cC \to \langle \cC_{<\infty}\rangle$
                    are left completions.
            \end{enumerate}

        \item If $\cD_{\geq \infty} = 0$ so that $\cD$ is left complete,
            then $\langle \cC_{<\infty} \rangle$ is the largest stable subcategory
            of $\cD$ in which $\cC_{<\infty}$ is left dense, i.e.~
            \begin{align*}
                \qquad\qquad\langle \cC_{<\infty} \rangle
                =\ &\{X \in \cD \mid \text{for $n \geq 0$ there is $c \in \cC_{<\infty}$ and a map
                    $c \to X$ with cofiber in $\cD_{\geq n}$}\}\\
                =\ & \{X \in \cD \mid X\text{ admits a $\cC_{<\infty}$-weight complex}\}.
            \end{align*}
    \end{enumerate}
\end{lemma}
\begin{proof}~
\begin{enumerate}
    \item Let us first show the addendum, that we can lift maps between colimits of such diagrams.
        To this end, let $c_\bullet, d_\bullet \colon \N \to \cC_{<\infty}$ be sequential diagrams growing in connectivity
        and define $c_\infty = \colim_n c_n$ and similarly for $d_\infty$.
        We show that given a map $c_\infty \to d_\infty$,
        there is a strictly increasing functor $f \colon \N \to \N$
        and a natural transformation $c_\bullet \Rightarrow d_{f(\bullet)}$,
        which on colimits induces the given map.
        Recall from Example \ref{ex:corep-bdd} that since $c_0 \in \cC_{<\infty} \subseteq \cD_{<\infty}$
        we have $\colim_n \hom(c_0,d_n) \simeq \hom(c_0, d_\infty)$.
        It follows that $c_0 \to c_\infty \to d_\infty$ factors through some $d_{f(0)} \to d_\infty$.
        Now suppose by induction on $k \geq 0$ that we have constructed a commutative diagram
        \[\begin{tikzcd}[ampersand replacement=\&]
            {c_0} \& {c_1} \& \cdots \& {c_k} \& {c_\infty} \\
            {d_{f(0)}} \& {d_{f(1)}} \& \cdots \& {d_{f(k)}} \& {d_\infty}
            \arrow[from=1-1, to=1-2]
            \arrow[from=1-1, to=2-1]
            \arrow[from=1-2, to=1-3]
            \arrow[from=1-2, to=2-2]
            \arrow[from=1-3, to=1-4]
            \arrow[from=1-4, to=1-5]
            \arrow[from=1-4, to=2-4]
            \arrow[from=1-5, to=2-5]
            \arrow[from=2-1, to=2-2]
            \arrow[from=2-2, to=2-3]
            \arrow[from=2-3, to=2-4]
            \arrow[from=2-4, to=2-5]
        \end{tikzcd}\]
        for some strictly increasing sequence $f(0),f(1),\dots,f(k)$.
        We can then consider the commutative diagram
        \[\begin{tikzcd}[ampersand replacement=\&]
            {c_k} \& {c_{k+1}} \& {c_\infty} \\
            {d_{f(k)}} \& P \& {d_\infty}
            \arrow[from=1-1, to=1-2]
            \arrow[from=1-1, to=2-1]
            \arrow[from=1-2, to=1-3]
            \arrow[from=1-2, to=2-2]
            \arrow[from=1-3, to=2-3]
            \arrow[from=2-1, to=2-2]
            \arrow["\lrcorner"{anchor=center, pos=0.125, rotate=180}, draw=none, from=2-2, to=1-1]
            \arrow[from=2-2, to=2-3]
        \end{tikzcd}\]
        where $P$ is defined as the pushout of the left square, and thus lies in $\cD_{<\infty}$.
        Hence the same argument as for $c_0$ above shows that $P \to d_\infty$
        factors through some $d_{f(k+1)} \to d_\infty$.
        By induction, this proves the claim that we can lift the given map $c_\infty \to d_\infty$
        to a map of diagrams.

    \item For the first equality,
        we prove that the right side already admits sequential colimits growing in connectivity.
        So suppose that we have a sequential diagram growing in connectivity
        $c_\infty^\bullet \colon \N \to \cD$ where for each $n \geq 0$ there is some sequential
        diagram $c^n_\bullet \colon \N \to \cC_{<\infty}$ and an equivalence $c^n_\infty \simeq \colim_k c_k^n$.
        We can first pass to a subsequence of $c_\infty^\bullet$ to assume
        that $\cofib(c_\infty^n \to c_\infty^{n+1}) \in \cD_{\geq n+1}$ for all $n \geq 0$.
        By the claim we just showed, we can moreover inductively pass to suitable subsequences
        and in this way lift this to a diagram $c^\bullet_\bullet \colon \N \times \N \to \cC_{<\infty}$
        where also for each $n,k \geq 0$ we have that $\cofib(c^n_k \to c^n_{k+1}) \in \cC_{[k+1,\infty)}$.
        Since the diagonal $\N \to \N \times \N$ is cofinal,
        we see that $\colim_n c_\infty^n = \colim_n c_n^n$,
        so it remains to check that $(c_n^n)_n$ is growing in connectivity.
        Now we have an extension $\cofib(c^n_n \to c^{n+1}_n) \to \cofib(c_n^n \to c_{n+1}^{n+1}) \to \cofib(c^{n+1}_n \to c^{n+1}_{n+1})$
        where the last term lies in $\cC_{\geq n+2}$, so it suffices to show that the first term
        lies in $\cC_{\geq n+1}$.
        But we can form the diagram of cofiber sequences
        \[\begin{tikzcd}[ampersand replacement=\&]
            {c_n^n} \& {c_\infty^n} \& {A_{\geq n+1}} \\
            {c^{n+1}_n} \& {c^{n+1}_{\infty}} \& {B_{\geq n+2}} \\
            {E_{\geq n+1}} \& {D_{\geq n+1}} \& {C_{\geq n+2}}
            \arrow[from=1-1, to=1-2]
            \arrow[from=1-1, to=2-1]
            \arrow[from=1-2, to=1-3]
            \arrow[from=1-2, to=2-2]
            \arrow[from=1-3, to=2-3]
            \arrow[from=2-1, to=2-2]
            \arrow[from=2-1, to=3-1]
            \arrow[from=2-2, to=2-3]
            \arrow[from=2-2, to=3-2]
            \arrow[from=2-3, to=3-3]
            \arrow[from=3-1, to=3-2]
            \arrow[from=3-2, to=3-3]
        \end{tikzcd}\]
        where $A_{\geq n+1}, B_{\geq n+2}$ and $D_{\geq n+1}$ have indicated weight connectivity
        by construction of the diagram $c_\bullet^\bullet$.
        But then the claimed weight connectivity of $C_{\geq n+2}$ and $E_{\geq n+1}$ follows.

        The inclusion $\langle \cC_{<\infty}\rangle_{\geq 0} \subseteq \langle \cC_{<\infty}\rangle \cap \cD_{\geq 0}$
        is clear by the assumption on $\cD$.
        Let us show that the reverse inclusion also holds under the assumption
        that $\cC_{[0,\infty)} = \cC_{<\infty} \cap \cD_{\geq 0}$.
        Indeed, given $X \in \langle \cC_{<\infty}\rangle \cap \cD_{\geq 0}$,
        we can write $X = \colim_n c_n$ for $c_\bullet \colon \N \to \cC_{<\infty}$
        satisfying $\cofib(c_n \to c_{n+1}) \in \cC_{\geq n+1}$ for all $n \geq 0$.
        Note also that by the assumption on $\cD$ we have $X/c_n = \colim_{k \geq n}c_k/c_n \in \cD_{\geq n+1}$.
        Now we have cofiber sequences $c_n \to X \to X/c_n$ where the middle term lies in $\cD_{\geq 0}$
        and the right term in $\cD_{\geq 1}$, hence the left term also lies in $\cD_{\geq 0}$.
        Thus $c_n \in \cC_{<\infty} \cap \cD_{\geq 0} = \cC_{[0,\infty)}$, as desired.

        To see that $\langle \cC_{<\infty} \rangle \cap \cD_{\geq \infty} = 0$,
        suppose that $c_\bullet \colon \N \to \cC_{<\infty}$ is growing in connectivity
        and $c_\infty = \colim_n c_n \in \cD_{\geq \infty}$.
        Then each map $c_n \to c_\infty$ vanishes, which upon taking colimits
        shows that $0 \simeq \id_{c_\infty}$, and hence that $c_\infty = 0$.

        Next we show that $\cC_{<\infty} \cap \langle \cC_{<\infty}\rangle_{\geq 0} = \cC_{[0,\infty)}$.
        The inclusion $\supseteq$ is clear,
        so suppose that $c \in \cC_{\leq b}$ and there exists some $c_\bullet \colon \N \to \cC_{[0,\infty)}$
        so that $\colim_n c_n = c$ and $\cofib(c_n \to c_{n+1}) \in \cC_{\geq n+1}$ for all $n \geq 0$.
        Then $c/c_b \in \cD_{\geq b+1}$
        so the map $c_b \to c$ admits a retraction.
        Since $\cC_{<\infty}$ is a full subcategory, $c$ is a retract of $c_b$ in $\cC_{<\infty}$,
        and since $\cC_{[0,\infty)}$ is closed under retracts, we deduce $c \in \cC_{[0,\infty)}$.

        Finally, to see that $\langle \cC_{<\infty}\rangle_{\geq 0}$ is closed under sequential
        colimits growing in connectivity,
        we can use the same argument as in the proof of the first equality
        to lift a sequential diagram growing in connectivity in $\langle \cC_{<\infty}\rangle_{\geq 0}$
        to a $\N \times \N$ indexed diagram in $\cC_{[0,\infty)}$, and show that the resulting diagonal
        is still growing in connectivity. Taking the colimit, we deduce the result from the assumption on $\cD$.

    \item To see that $\langle \cC_{<\infty} \rangle$ is a stable subcategory of $\cD$,
        note that the equivalent description from (2)
        is closed under finite sums and shifts,
        so it remains to check closure under cofibers.
        Since we can lift maps on colimits to maps of diagrams by (1),
        this reduces to checking that the (pointwise) cofiber of a map $c_\bullet \Rightarrow d_\bullet$
        of sequential diagrams growing in connectivity is itself still growing in connectivity,
        but this is straightforward.
        From this description of computing finite colimits
        it is also clear that $\langle \cC_{<\infty}\rangle_{\geq 0}$ is closed under finite colimits and extensions.
        Moreover, since it admits sequential colimits growing in connectivity,
        it is idempotent complete by Corollary \ref{cor:grow-conn-idem}, hence closed under retracts.
        On the other hand, $\langle \cC_{<\infty}\rangle_{\leq 0}$ is closed under finite limits,
        extensions, and retracts because $\cD_{\leq 0} \subseteq \cD$ is.
        It is also clear that the inclusion $\cC_{<\infty} \subseteq \langle \cC_{<\infty}\rangle$ is weight exact.
        The orthogonality follows from $\langle \cC_{<\infty} \rangle_{\geq 0} \subseteq \cD_{\geq 0}$.

        Next, we provide a weight decomposition of $X \in \langle \cC_{<\infty}\rangle$.
        We can choose $c_\bullet \colon \N \to \cC_{<\infty}$
        with $\colim_n c_n = X$ and $\cofib(c_n \to c_{n+1}) \in \cC_{\geq n+1}$ for $n \geq 0$.
        In particular, we get $X_{\geq 1} \coloneqq X/c_0 = \colim_{n} c_n/c_0 \in \langle \cC_{<\infty}\rangle_{\geq 1}$.
        Pick a weight decomposition $(c_0)_{<\ell+1} \to c_0 \to (c_0)_{\geq 1}$
        where $(c_0)_{< \ell+1} \in \cC_{<\ell+1}$ and $(c_0)_{\geq 1} \in \cC_{\geq 1}$.
        Then the cofiber of $(c_0)_{<\ell+1} \to X$ is an extension of $(c_0)_{\geq 1}$ and $X_{\geq 1}$,
        hence still lies in $\langle \cC_{<\infty}\rangle_{\geq 1}$.
        This proves that $\langle \cC_{<\infty}\rangle$ has weight decompositions of defect $\ell$.
        The argument also shows that $\cC_{<\infty} \subseteq \langle \cC_{<\infty}\rangle$ is left dense.

        This proves that $\langle \cC_{<\infty}\rangle$ is weakly $\ell$-weighted,
        and by the results in (2) it follows that it is left complete, showing (a).
        We have also already observed (b),
        and hence (c) follows from Corollary \ref{cor:left-comp-via-embed}
        in view of the last equality of (2).

    \item The two new proposed descriptions of $\langle \cC_{<\infty}\rangle$ agree
        since given an $X$ in the first one, we can build a weight complex
        as in Lemma \ref{lem:weight-cplx},
        and admitting a $\cC_{w <\infty}$-weight complex also evidently implies lying in the first one.
        Clearly every object in $\langle \cC_{<\infty}\rangle$ admits a $\cC_{<\infty}$-weight complex.
        Conversely, if $X \in \cD$ admits a $\cC_{<\infty}$-weight complex,
        then it converges by left completeness of $\cD$,
        and hence $X \in \langle \cC_{<\infty}\rangle$.
        It is also clear that the first new description is the largest stable subcategory
        of $\cD$ in which $\cC_{<\infty}$ is left dense.
\end{enumerate}
\end{proof}

In the case $\cC = \cD$ of the above Lemma there is more to be said.

\begin{lemma}\label{lem:upw-comp-quot}
    Let $\cC$ be a weakly $\ell$-weighted category
    and suppose that $\cC$ admits sequential colimits growing in connectivity and $\cC_{\geq 0}$ is closed under them.
    Let $\langle \cC_{<\infty}\rangle$ be as in Lemma \ref{lem:upw-comp-basic} ($\cC = \cD$).
    Then:
    \begin{enumerate}
        \item The full stable subcategories $\langle \cC_{<\infty}\rangle,\cC_{\geq \infty}$ determine a semiorthogonal decomposition of $\cC$
        \[\begin{tikzcd}[ampersand replacement=\&]
            {\langle \cC_{<\infty}\rangle} \&\& \cC \&\& {\cC_{\geq \infty}}
            \arrow[""{name=0, anchor=center, inner sep=0}, "j", shift left=2, hook, from=1-1, to=1-3]
            \arrow[""{name=1, anchor=center, inner sep=0}, "{j^R}", shift left=2, from=1-3, to=1-1]
            \arrow[""{name=2, anchor=center, inner sep=0}, "p", shift left=2, from=1-3, to=1-5]
            \arrow[""{name=3, anchor=center, inner sep=0}, "{p^R}", shift left=2, hook', from=1-5, to=1-3]
            \arrow["\dashv"{anchor=center, rotate=-90}, draw=none, from=0, to=1]
            \arrow["\dashv"{anchor=center, rotate=-90}, draw=none, from=2, to=3]
        \end{tikzcd}\]
            where both $j$ and $j^R$ are weight exact
            and commute with the inclusions of $\cC_{<\infty}$.

        \item Given a $\cC_{< \infty}$-weight complex
            $X_{\bullet}$ of $X \in \cC$ we have an identification
            of cofiber sequences
            \[\begin{tikzcd}[ampersand replacement=\&]
                {\colim_n X_n} \& X \& {\colim_n X/X_n} \\
                {jj^RX} \& X \& {p^RpX}
                \arrow[from=1-1, to=1-2]
                \arrow["\simeq"', from=1-1, to=2-1]
                \arrow[from=1-2, to=1-3]
                \arrow["\simeq", from=1-3, to=2-3]
                \arrow["{\eps_X^j}", from=2-1, to=2-2]
                \arrow[Rightarrow, no head, from=2-2, to=1-2]
                \arrow["{\eta^p_X}", from=2-2, to=2-3]
            \end{tikzcd}\]
            where $\eps^j \colon jj^R \Rightarrow \id$ is the counit
            and $\eta^p \colon \id \Rightarrow p^Rp$ the unit.

        \item $j^R$ agrees with the left Kan extension of $\cC_{<\infty} \subseteq \langle \cC_{<\infty}\rangle$
            along $\cC_{<\infty} \subseteq \cC$
            and exhibits $\cC/\cC_{\geq \infty} \simeq \langle \cC_{<\infty}\rangle$ as left completion of $\cC$.

        \item If $\cC$ is already left complete,
            then $j$ and $j^R$ are mutually inverse weight exact equivalences.
    \end{enumerate}
\end{lemma}
\begin{proof}
    Clearly $\langle \cC_{<\infty}\rangle \subseteq {}^\perp(\cC_{\geq \infty})$.
    Moreover, given any $\cC_{< \infty}$-weight complex $X_\bullet$ of $X \in \cC$
    there is a cofiber sequence as depicted
    in the first row of (2). The left term lies in $\langle \cC_{<\infty}\rangle$
    and the right one in $\cC_{\geq \infty}$ by stability
    of each $\cC_{\geq n}$ under sequential colimits growing in connectivity.
    From this we automatically obtain the claimed semiorthogonal decomposition
    and identification as in (2), see e.g.~\cite[Section 7.2.1]{SAG}.

    We already know from Lemma \ref{lem:upw-comp-basic} that $j$
    is weight exact and commutes with the inclusions of $\cC_{<\infty}$.
    But then the fact that $\id \simeq j^Rj$ shows that also $j^R$
    commutes with the inclusions of $\cC_{<\infty}$,
    and hence that $j^R$ preserves coconnectives.
    Now let $X \in \cC_{\geq 0}$.
    Then by Lemma \ref{lem:weight-cplx}
    we can choose the $\cC_{<\infty}$-weight complex
    $X_\bullet$ for $X$,
    and after possibly discarding some terms, we may assume that $X_n \in \cC_{\geq 0}$ for all $n$.
    This shows that $j^RX = \colim_n X_n$ is again connective
    in the weak weight structure on $\langle \cC_{<\infty}\rangle$
    defined in Lemma \ref{lem:upw-comp-basic}.
    Thus $j^R$ is also weight exact,
    and is therefore left Kan extended from its restriction by Proposition \ref{prop:weight-comp-lkan}.
    Since $\cC_{\leq 0}$ is dense in itself, the rest of (3)
    follows from Lemma \ref{lem:upw-comp-basic}(3c).
    Also (4) is clear by Lemma \ref{lem:upw-comp-basic}(4).
\end{proof}

\begin{proposition}\label{prop:ind-weight}
    Let $\cC$ be a weakly $\ell$-weighted category.
    There is a weight structure on $\Ind(\cC)$
    so that the Yoneda-embedding $j \colon \cC \hookrightarrow \Ind(\cC)$ is weight exact. Moreover:
    \begin{enumerate}
        \item The weight structure is compactly generated by $\cC_{\leq 0}$
            in the sense of Definition \ref{def:gen-weight}. Concretely:
            \begin{enumerate}
                \item $X \in \Ind(\cC)$ is weight connective
                    if and only if $\hom(c_{\leq 0},X)$
                    is connective for all $c_{\leq 0} \in \cC_{\leq 0}$.

                \item $\Ind(\cC)_{\leq 0}$ is the smallest class
                    containing $\cC_{\leq 0}$ and closed under
                    finite limits, extensions, retracts,
                    sums, and sequential colimits of maps whose
                    cofiber is already in the class.
            \end{enumerate}

        \item $\Ind(\cC)$ admits an adjacent $t$-structure
            (i.e.~$\Ind(\cC)_{t \geq 0} = \Ind(\cC)_{w \geq 0}$).
            In particular, $\Ind(\cC)_{w \geq 0}$ is closed under all colimits.

        \item We have $\Ind(\cC_{\geq 0}) \subseteq \Ind(\cC)_{\geq 0} \subseteq \Ind(\cC_{\geq -\ell})$.\footnote{For $\ell = \infty$, the second inclusion is vacuous. If $\cC$ is also a weakly co-weighted category
                (cf.~Warning \ref{warn:wwcat-non-dual}), i.e.~we also have weight decompositions
                $X_{\leq 0} \to X \to X_{>-\infty}$,
                then the same proof shows that $\Ind(\cC)_{\geq 0} \subseteq \Ind(\cC_{>-\infty})$.}
                In particular, for a weight structure on $\cC$ (i.e.~for $\ell = 0$) we have equality.

        \item If $\ell = 0$,
            then the weight heart $\Ind(\cC)_{w=0}$
            coincides with the idempotent completion
            of $(\cC_{=0})^\oplus$, the full subcategory on sums of objects in $\cC_{=0}$.

        \item There exists a stable recollement
            \[\begin{tikzcd}
                {\Ind(\cC_{< \infty})} && {\Ind(\cC)} && {\Ind(\cC)_{\geq \infty}}
                \arrow["{j_!}", shift left=3, hook', from=1-1, to=1-3]
                \arrow["{j_*}"', shift right=3, hook, from=1-1, to=1-3]
                \arrow["{j^*}"{description}, from=1-3, to=1-1]
                \arrow["p", shift left=3, from=1-3, to=1-5]
                \arrow["{p^{RR}}"', shift right=3, from=1-3, to=1-5]
                \arrow["{p^R}"{description}, hook', from=1-5, to=1-3]
            \end{tikzcd}\]
            where $j_!$ is induced by the inclusion via $\Ind(-)$,
            and $j_!,j^*$ are weight exact and $j_*$ has amplitude $[0,\infty]$,
            and $j^*$ is a left completion.
            Both the weight structure and adjacent $t$-structure
            on $\Ind(\cC_{<\infty})$ are left complete.

        \item If $\cD$ is another weakly weighted category
            and $f \colon \cC \to \cD$ a map of amplitude $[a,b)$ \alt{$[a,b]$}
            for $-\infty \leq a \leq b \leq \infty$,
            then $f_! = \Ind(f) \colon \Ind(\cC) \to \Ind(\cD)$
            is of weight amplitude $[a-\ell,b)$ \alt{$[a-\ell,b]$}.
    \end{enumerate}
\end{proposition}
\begin{proof}
    To show (1), we could just cite Theorem \ref{thm:gen-weight}
    (which in this specific case was well known before,
    cf.~Remark \ref{rem:gen-weight}),
    though we find it instructive to give the short proof here:
    The only nontrivial part is to check the existence of weight decompositions,
    which are constructed via a small object argument.
    Given $X \in \Ind(\cC)$, define $X'$ via the cofiber sequence
    $\bigoplus_{c \in (\cC_{\leq 0})_{/X}} c \to X \to X'$.
    Now let $X_0 = X$ and define $X_n \to X_{n+1}$ as $X_n \to (X_n)'$.
    The weight decomposition is then $\colim_n F_n \to X \to \colim_n X_n$
    where inductively each $F_n = \fib(X \to X_n)$
    and thus also $\colim_nF_n$
    is weight coconnective (since $\cofib(F_n \to F_{n+1})$ is a coproduct
    of objects in $\cC_{\leq 0}$).
    On the other hand, clearly $\pi_0\hom(c,\colim_n X_n) = 0$
    for $c \in \cC_{\leq 0}$, and since $\cC_{\leq 0}$ is closed under $\Omega$
    we see $\colim_n X_n$ is 1-connective, as desired.

    For (2), we note that by (1a) the subcategory $\Ind(\cC)_{w \geq 0} \subseteq \Ind(\cC)$
    is closed under colimits and extensions, so the claim follows from \cite[Proposition 1.4.4.11]{HA}.

    To see (3), note that clearly $\Ind(\cC_{\geq 0}) \subseteq \Ind(\cC)_{\geq 0}$
    since Yoneda is weight exact and the latter is closed under colimits.
    Now let $X \in \Ind(\cC)_{\geq 0}$.
    The canonical map $\colim_{c \in \cC_{/X}} c \to X$ is an equivalence
    by compact generation. Note also that $\cC_{/X}$ is filtered since
    it admits finite colimits.
    The full subcategory $(\cC_{\geq -\ell})_{/X} \subseteq \cC_{/X}$
    is still filtered for the same reason,
    and for every map $c \to X$,
    we can pick a weight decomposition $c_{\leq -1} \to c \to c_{\geq -\ell}$
    and note that $c_{\leq -1} \to X$ vanishes, so that $c \to X$
    admits a map to $c_{\geq -\ell} \to X$ in $\cC_{/X}$.
    Thus the inclusion of filtered categories $(\cC_{\geq -\ell})_{/X} \subseteq \cC_{/X}$
    is weakly cofinal in the sense of \cite[\href{https://kerodon.net/tag/06CK}{Tag 06CK}]{Kerodon},
    and thus \cite[\href{https://kerodon.net/tag/06CP}{Tag 06CP}]{Kerodon} shows that it is already cofinal.
    Thus $\colim_{c \in (\cC_{\geq -\ell})_{/X}}c \simeq X$, which proves (3).

    For (4), suppose that $\ell = 0$ and $X \in \Ind(\cC)_{w=0}$
    and consider the map $f \colon \bigoplus_{c \in (\cC_{=0})_{/X}}c \to X$.
    It suffices to check that the cofiber of $f$ is 1-connective,
    since then $f$ admits a retraction.
    We can check this by mapping in from $d \in \cC_{\leq 0}$.
    But $d \to \Sigma c$ vanishes, so any such map $d \to \cofib(f)$
    factors through $d \to X$. Picking a weight decomposition
    $d_{\leq -1} \to d \to d_{=0}$, we get that $d \to X$
    actually factors through $d \to d_{= 0} \to X$.
    But then by definition of $f$ the map $d_{=0} \to X$
    factors through $f$, and hence the entire composite
    $d \to \cofib(f)$ vanishes. This shows $\cofib(f) \in \Ind(\cC)_{\geq 1}$
    so that $f$ admits a retraction, as desired.

    For point (5) we want to apply Lemma \ref{lem:upw-comp-quot}
    to the weighted category $\Ind(\cC)$.
    To this end, we first need to argue that $\langle \Ind(\cC)_{<\infty}\rangle = \Ind(\cC_{<\infty})$
    in the notation of Lemma \ref{lem:upw-comp-basic}.
    It suffices to show that $\langle \Ind(\cC)_{<\infty}\rangle$ is closed under all colimits,
    since it contains $\cC_{<\infty}$.
    As it is a stable subcategory, it suffices to consider coproducts.
    Let $X^i \in \langle \Ind(\cC)_{<\infty}\rangle$ for $i \in I$.
    Since $\Ind(\cC)_{<\infty}$ is left dense and we are in the $\ell =0$ case,
    we can use Example \ref{ex:standard-weight-complex}
    to find $X^i_\bullet \colon \N \to \Ind(\cC)_{<\infty}$ with $\colim_n X^i_n \simeq X^i$
    and $X^i_n \in \Ind(\cC)_{\leq n}$ for all $n$.
    Since $\Ind(\cC)$ is saturated, each $\Ind(\cC)_{\leq n}$ is closed under arbitrary sums,
    and we see that $\bigoplus_i X^i_\bullet \colon \N \to \Ind(\cC)_{<\infty}$
    is still a well defined sequential diagram growing in connectivity.
    Clearly its colimit is $\bigoplus_i X^i$, so this proves that $\langle \Ind(\cC)_{<\infty}\rangle = \Ind(\cC_{<\infty})$.

    Let us also note that $\Ind(\cC_{<\infty})_{\leq 0} = \Ind(\cC)_{\leq 0}$,
    which agrees with the definition of $\langle \Ind(\cC)_{<\infty}\rangle_{\leq 0}$ of Lemma \ref{lem:upw-comp-basic}.
    But then the weight structure on $\Ind(\cC_{<\infty})$ defined here
    necessarily agrees with the weight structure on $\langle \Ind(\cC)_{<\infty}\rangle$
    defined in Lemma \ref{lem:upw-comp-basic}, since both are saturated
    and the coconnectives agree.

    Now $p^{RR}$ exists for formal reasons;
    $j_!$ is a functor of compactly generated stable categories
    that preserves colimits and compact objects,
    hence admits adjoints $j_! \dashv j^* \dashv j_*$,
    and thus the functor $p$ admits analogous adjoints
    so that each row fits into a Verdier sequence,
    see e.g.~\cite[Corollary A.2.10]{Hermitian-II}.
    Moreover, it follows from Lemma \ref{lem:saturated} that $j_*$ preserves connectives.

    Finally, for (6), note that $f_!$ preserves the operations
    under which $\cC_{\leq 0}$ generates $\Ind(\cC)_{\leq 0}$
    and restricts to $f$ on $\cC$.
    From this we see that if $f$ has amplitude $[-\infty,b)$, then also $f_!$ has.
    In the case $\ell = \infty$, there is now nothing left to do,
    so let $\ell < \infty$
    and suppose that $f$ has amplitude $[a,\infty]$ for some $a \in \Z$.
    Then by (3) the restriction of $f_!$ to $\Ind(\cC)_{\geq 0}$ factors as
    $\Ind(\cC)_{\geq 0} \subseteq \Ind(\cC_{\geq -\ell}) \to \Ind(\cD_{\geq a-\ell}) \subseteq \Ind(\cD)_{\geq a-\ell}$.
    Thus $f_!$ has amplitude $[a-\ell, \infty]$.
\end{proof}

\begin{remark}\label{rem:easy-ind-left-comp}
    It is also easy to see directly that the above weight structure on $\Ind(\cC_{<\infty})$ is left complete.
    By (2) it only remains to observe that if $X \in \Ind(\cC_{<\infty})_{\geq \infty}$
    then any map $c \to X$ vanishes for $c \in \cC_{<\infty}$, hence $X = 0$.
\end{remark}

With this, we have all the necessary ingredients to prove Theorem \ref{thm:upw-comp}.

\begin{proof}[Proof of Theorem \ref{thm:upw-comp}]\label{proof:upw-comp}
    Let us begin by showing that every weakly weighted category $\cC$
    admits a left completion.
    To this end, we apply Lemma \ref{lem:upw-comp-basic}
    to the Yoneda embedding $j \colon \cC_{<\infty} \hookrightarrow \Ind(\cC_{<\infty})$
    and let $\lc{\cC} \coloneqq \langle \cC_{<\infty}\rangle$ in the notation of said lemma.
    We need to show that $\cC_{\leq 0} \subseteq \lc{\cC}_{\leq 0}$ is dense.
    This is a consequence of the special description we have for $\Ind(\cC)_{\leq 0}$
    from Proposition \ref{prop:ind-weight}(1b).
    Indeed, by left density of $\cC_{<\infty} \subseteq \lc{\cC}$,
    it follows that a given $X \in \lc{\cC}_{\leq 0}$
    is a retract of some $c \in \cC_{<\infty}$ and hence compact in $\Ind(\cC_{<\infty})$.
    Now $\Ind(\cC)_{\leq 0}$ is generated by $\cC_{\leq 0}$
    under finite limits, extensions, retracts,
    sums and certain sequential colimits.
    In particular, since $\cC_{\leq 0}$ is already closed under finite limits and extensions,
    this means that $X \in \Ind(\cC_{\leq 0}) \subseteq \Ind(\cC_{<\infty})$.
    Since this inclusion preserves filtered colimits and $X$ is compact in $\Ind(\cC_{<\infty})$
    it is also compact in $\Ind(\cC_{\leq 0})$, and hence $X \in (\cC_{\leq 0})^\idem$,
    as desired.

    Thus Lemma \ref{lem:upw-comp-basic}(3c) now shows that the (codomain-)restricted
    Yoneda embedding $\cC_{<\infty} \subseteq \langle \cC_{<\infty}\rangle$
    as well as its left Kan extension $\cC \to \langle \cC_{<\infty}\rangle$
    are left completions.
    Now the points (1), (2a), and (2c) follow from the Lemma,
    and (2d) is just Lemma \ref{lem:upw-comp-quot}.

    The fully faithfulness in (2b) is also clear.
    For the rest of (2b),
    consider $X \in \cC_{w \leq n}$ and $Y \in \cC$
    with $\cC_{w < \infty}$-weight complex $Y_\bullet$.
    We have the following commutative square
\[\begin{tikzcd}
	{\colim_n \hom_\cC(X,Y_n)} & {\hom_\cC(X,Y)} \\
	{\colim_n\hom_{\lc{\cC}}(\eta X,\eta Y_n)} & {\hom_{\lc{\cC}}(\eta X, \eta Y)}
	\arrow[from=1-1, to=1-2]
	\arrow["\simeq", from=1-1, to=2-1]
	\arrow[from=1-2, to=2-2]
	\arrow["\simeq"', from=2-1, to=2-2]
\end{tikzcd}\]
    Here the left vertical map is an equivalence by fully faithfulness
    of $\eta$ on $\cC_{w < \infty}$,
    and the bottom map is an equivalence by definition of $\eta$
    and since $\hom_{\lc{\cC}}(\eta X, -)$ has bounded below weight amplitude by Example \ref{ex:corep-bdd}.
    The cofiber of the top map is given by $\colim_n \hom_\cC(X,Y/Y_n)$,
    which lies in $\Sp_{w \geq \infty} = 0$ by another application
    of Example \ref{ex:corep-bdd}.
    Thus also the top map and hence the right map is an equivalence,
    proving the rest of (2b).

    To see (3), note that as in the proof of Corollary \ref{cor:left-comp-via-embed} we have a commutative diagram
    \[\begin{tikzcd}[ampersand replacement=\&]
        {\Fun^\Ex_{[a,b]}(\lc{\cC},\cE)} \& {\Fun^\Ex_{[a,b]}(\cC,\cE)} \\
        \& {\Fun^\Ex_{[a,b]}(\cC_{<\infty},\cE)}
        \arrow["{(\eta_\cC)^*}", from=1-1, to=1-2]
        \arrow["{j^*}"', from=1-1, to=2-2]
        \arrow["\simeq", draw=none, from=1-1, to=2-2]
        \arrow["{i^*}", from=1-2, to=2-2]
        \arrow["\simeq"', draw=none, from=1-2, to=2-2]
    \end{tikzcd}\]
    where $i^*$ and $j^*$ are equivalences due to Proposition \ref{prop:weight-comp-lkan},
    and similarly for $[a,b)$ in place of $[a,b]$.
    Since Kan extensions compose, also the inverse
    of $(\eta_\cC)^*$ is given by left Kan extension along $\eta_\cC$.
    Finally, (4) is an immediate consequence of (3).
\end{proof}

\section{Generating complete weight structures from their hearts}\label{sec:chains}

In this section, we see what the theory of completions of (weak) weight structures
can say about one of the motivating examples of a weight structure,
the category of chain complexes in an ordinary additive category modulo chain homotopy equivalences
(or more specifically, its $\infty$-categorical incarnation).
Concretely, we will show that this category can be obtained as the (left and right) weight completion
of the category of bounded chain complexes of $\cA$ with its canonical weight structure.
We will also give an explicit description of the (unique!) left complete and right bounded
weight structure with a given heart.

To begin, let us recall the following result of Sosnilo, which tells us how to generate
bounded weight structures from their hearts.
Given an additive category $\cA$, denote by $\Stab(\cA)$ the stable envelope of $\cA$,
which is the smallest full stable subcategory of $\Fun^\times(\cA^\op,\Sp)$
containing the image of the Yoneda embedding.
This assembles into a left adjoint $\Stab(-) \colon \Cat^\add \to \Cat^\Ex$
of the forgetful functor from $\Cat^\Ex \to \Cat^\add$.
Recall also that $\cA^\widem$ denotes the weak idempotent completion (cf.~the notational conventions).

\begin{theorem}[{{\cite[Prop.~3.3 and Cor.~3.4]{Sosnilo}}}]\label{thm:vova}
    Let $\cA$ be an additive category.
    \begin{enumerate}
        \item $\Stab(\cA)$ admits a bounded weight structure
            whose weight heart identifies with $\cA^\widem$,
            the weak idempotent completion of $\cA$.

        \item Let $\WCat_{=0}$ be the category of weighted categories and weight exact functors.
            The above construction upgrades to an adjunction
            \[
                \Stab(-) \colon \Cat^\add \rightleftarrows \WCat_{=0} \noloc (-)_{=0}
            \]
            where the right adjoint $(-)_{=0}$ associates to a weighted category its weight heart.

        \item Let $\WCat^b_{=0} \subseteq \WCat_{=0}$ be the full subcategory on bounded weighted categories.
            Then the functor $(-)_{=0} \colon \WCat^b_{=0} \to \Cat^\add$ is fully faithful
            with essential image consisting of the weakly idempotent complete additive categories,
            so the above adjunction restricts to a left Bousfield localization
            $\Stab(-) \colon \Cat^\add \rightleftarrows \WCat^b_{=0} \noloc (-)_{=0}$.

        \item In particular, restricting to bounded idempotent complete weighted categories
            $\WCat^{b,\perf}_{=0} \subseteq \WCat^b_{=0}$
            and idempotent complete additive categories $\Cat^{\add,\idem} \subseteq \Cat^\add$
            the above adjunction becomes an equivalence $\Cat^{\add,\idem} \simeq \WCat^{b,\perf}_{=0}$.
    \end{enumerate}
\end{theorem}

We obtain the following as a corollary of Proposition \ref{prop:weight-pdelta}.

\begin{proposition}\label{prop:pdelta-add}
    Let $\cA$ be an additive category. Then:
    \begin{enumerate}
        \item $\cP^{\Delta^\op}(\cA)$ is idempotent complete,
            and the map $\cP^{\Delta^\op}(\cA) \to \cP^{\Delta^\op}(\cA^\idem)$ is an equivalence.
            The full subcategory $\cA^\idem \subseteq \cP^{\Delta^\op}(\cA)$
            consists precisely of the projective objects in the sense of \cite[Section 5.5.8]{HTT},
            i.e.~those $X \in \cP^{\Delta^\op}(\cA)$ where $\map(X,-) \colon \cP^{\Delta^\op}(\cA) \to \An$
            preserves geometric realizations.

        \item $\cP^{\Delta^\op}(\cA)$ is prestable,
            with $\SW(\cP^{\Delta^\op}(\cA)) \simeq \lc{\Stab(\cA)}$.
            In particular, $\SW(\cP^{\Delta^\op}(\cA))$
            always admits a left complete weight structure with heart $\cA^\idem$.

        \item Every object in $\cP^{\Delta^\op}(\cA)$
            is a geometric realization of a simplicial object that is levelwise projective
            i.e.~a simplicial object in $\cA^\idem$.
    \end{enumerate}
\end{proposition}
\begin{proof}
    Here (1) is a consequence of Lemma \ref{lem:fintype-colim}.
    By Observation \ref{obs:left-comp} and Theorem \ref{thm:vova}
    we have $(\lc{\Stab(\cA)})_{w=0} = \cA^{\idem}$,
    so Proposition \ref{prop:weight-pdelta} and (1)
    yield $\cP^{\Delta^\op}(\cA) \simeq (\lc{\Stab(\cA)})_{w \geq 0}$,
    which gives (2) by applying $\SW(-)$, see Lemma \ref{lem:stable-conn-struc}(4).
    Also (3) is a direct consequence of Proposition \ref{prop:weight-pdelta}.
\end{proof}

\begin{definition}\label{def:kw}
    Let $\cA$ be an additive category. We will denote by
    \[
        \Kw(\cA) \coloneqq \lrc{\Stab(\cA)} \simeq \rc{\SW(\cP^{\Delta^\op}(\cA))}
    \]
    the (left and right) weight completion of $\Stab(\cA)$
    (by Proposition \ref{prop:lr-complete} the order does not matter).
    Note also that $\Kw(\cA)_{\wb} = \Stab(\cA)^\idem = \Stab(\cA^\idem)$ and
    $\Kw(\cA)_{=0} = \cA^\idem$ by Corollary \ref{cor:upw-comp-idem} and its dual.
\end{definition}

We will see below that the terminology / notation of $\Kw(\cA)$ is justified
by comparing it to the classical case of chain complexes in an ordinary additive category up to chain homotopy.
For now, we deduce the following result from our theory.

\begin{theorem}\label{thm_adjoint}
    Let $\WCat^{b,\perf}_{=0} \subseteq \WCat^b_{=0} \subseteq \WCat_{=0}$ be as in Theorem \ref{thm:vova}.
    Consider also the full subcategory $\WCat^{\updownarrow}_{=0} \subseteq \WCat_{=0}$
    on (left and right) complete weighted categories.
    \begin{enumerate}
        \item We have adjunctions
            \[\begin{tikzcd}[ampersand replacement=\&]
                {\Cat^\add} \&\& {\WCat^b_{=0}} \&\& {\WCat_{=0}} \&\& {\WCat_{=0}^{\updownarrow}}
                \arrow[""{name=0, anchor=center, inner sep=0}, "{\Stab(-)}", curve={height=-12pt}, from=1-1, to=1-3]
                \arrow[""{name=1, anchor=center, inner sep=0}, "{(-)_{=0}}", curve={height=-12pt}, hook', from=1-3, to=1-1]
                \arrow[""{name=2, anchor=center, inner sep=0}, "{\inc_b}", curve={height=-12pt}, hook', from=1-3, to=1-5]
                \arrow[""{name=3, anchor=center, inner sep=0}, "{(-)_{\wb}}", curve={height=-12pt}, from=1-5, to=1-3]
                \arrow[""{name=4, anchor=center, inner sep=0}, "{\lrc{(-)}}", curve={height=-12pt}, from=1-5, to=1-7]
                \arrow[""{name=5, anchor=center, inner sep=0}, "{\inc_{\updownarrow}}", curve={height=-12pt}, hook', from=1-7, to=1-5]
                \arrow["\dashv"{anchor=center, rotate=-90}, draw=none, from=0, to=1]
                \arrow["\dashv"{anchor=center, rotate=-90}, draw=none, from=2, to=3]
                \arrow["\dashv"{anchor=center, rotate=-90}, draw=none, from=4, to=5]
            \end{tikzcd}\]

        \item By composing and restricting we obtain an adjoint equivalence
            \[
                \lrc{(-)} \colon \WCat^{b,\perf}_{=0} \simeq \WCat_{=0}^{\updownarrow} \noloc (-)_{\wb}.
            \]
            In particular, composing further with Theorem \ref{thm:vova}(4),
            we also obtain
            \[
                \Kw(-) \colon \Cat^{\add,\idem} \simeq \WCat_{=0}^{\updownarrow} \noloc (-)_{=0}.
            \]

        \item The composite functor $(-)^\idem \circ (-)_{=0} \colon \WCat_{=0} \to \Cat^{\add,\idem}$
            admits the fully faithful right adjoint $\Kw(-) \colon \Cat^{\add,\idem} \hookrightarrow \WCat_{=0}$.
            In particular, we can think of $\Kw(\cA)$ as the terminal weighted category
            whose weight heart is $\cA^\idem$.
    \end{enumerate}
\end{theorem}

\begin{proof}
    For (1) the left adjunction is Theorem \ref{thm:vova},
    the middle one is obvious, and the right one follows from Theorem \ref{thm:upw-comp}.

    For (2), we simply note that the unit and counit maps are equivalences.
    Indeed, the unit map at some $\cC \in \WCat^b_{=0}$ is given by $\eta_\cC \colon \cC \to (\lrc{\cC})_{\wb}$
    which is adjoint to the canonical map $\cC \to \lrc{\cC}$ using that $\cC$ is bounded.
    Now if $\cC$ is idempotent complete,
    it follows from Corollary \ref{cor:upw-comp-idem}(1) and its dual
    that $\eta_\cC$ is an equivalence.
    Similarly, the counit map at some $\cD \in \WCat_{=0}^{\updownarrow}$
    is given by $\eps_\cD \colon \lrc{(\cD_{\wb})} \to \cD$ which is adjoint to the inclusion $\cD_{\wb} \subseteq \cD$.
    In particular, we know that $(\eps_\cD)_{\wb}$ is an equivalence,
    so it follows from Lemma \ref{lem:upw-comp-eqv} and its dual that $\eps_\cD$ is an equivalence.

    For (3), note that under the equivalence $\Cat^{\add,\idem} \simeq \WCat^{\updownarrow}_{=0}$
    from (2), the functor $\Kw \colon \Cat^{\add,\idem} \to \WCat_{=0}$
    identifies with the right adjoint $\inc_{\updownarrow} \colon \WCat^{\updownarrow}_{=0} \subseteq \WCat_{=0}$.
    The left adjoint is then the top right composite in the following square
    \[\begin{tikzcd}[ampersand replacement=\&]
        {\WCat_{=0}} \& {\WCat^\updownarrow_{=0}} \\
        {\Cat^\add} \& {\Cat^{\add,\idem}}
        \arrow["{\lrc{(-)}}", from=1-1, to=1-2]
        \arrow["{(-)_{=0}}"', from=1-1, to=2-1]
        \arrow["\simeq"', from=1-2, to=2-2]
        \arrow["{(-)_{=0}}", draw=none, from=1-2, to=2-2]
        \arrow["{(-)^\idem}", from=2-1, to=2-2]
    \end{tikzcd}\]
    so it remains to see that the square commutes.
    In other words, we need to construct a natural transformation $(-)_{=0} \Rightarrow (-)_{=0} \circ \lrc{(-)}$
    and show it is an idempotent completion.
    The natural transformation $(-)_{\wb} \Rightarrow \id_{\WCat_{=0}}$ yields an equivalence
    $\lrc{(-)} \circ (-)_{\wb} \simeq \lrc{(-)}$ by using Corollary \ref{cor:left-comp-inverts-idem} and its dual.
    But now the unit yields a natural map $(-)_{\wb} \Rightarrow \lrc{(-)} \circ (-)_{\wb} \simeq \lrc{(-)}$
    which is an idempotent completion by Corollary \ref{cor:upw-comp-idem}.
    Postcomposing with $(-)_{=0}$ then yields the claim.
\end{proof}

Let us now discuss the example of chain complexes in more detail
and show that the above notation matches up with the classical
meaning of $K(\cA)$ for an ordinary additive category.
For an ordinary additive category $\cA$
we let $\Ch_\infty(\cA) \coloneqq N_{dg}(\Ch(\cA))$
be the dg-nerve (in the sense of \cite[Section 1.3.1]{HA})
of the ordinary category of chain complexes in $\cA$,
see also \cite[Remark 1.3.2.2]{HA}.
This is stable by \cite[Proposition 1.3.2.10]{HA},
and its homotopy category is the classical triangulated category $K(\cA)$,
whose objects are chain complexes in $\cA$ up to chain homotopy equivalence,
and whose morphisms are chain homotopy classes of chain maps.
As mentioned in the paragraph above \cite[Corollary 1.3.2.18]{HA},
any full subcategory of $\Ch(\cA)$ closed under shifts and formations of mapping cones
defines a (full) stable subcategory of $\Ch_\infty(\cA)$.
In particular, the categories $\Ch^b(\cA), \Ch^-(\cA), \Ch^+(\cA)$ of bounded, bounded below, and bounded above
chain complexes determine full stable subcategories $\Ch^b_\infty(\cA), \Ch_\infty^-(\cA),\Ch^+_\infty(\cA) \subseteq \Ch_\infty(\cA)$.

One of the motivating examples of weight structures in \cite{Bondarko10}
(see also \cite[Proposition 4.6]{Schnuerer})
is the weight structure on $K(\cA)$ where $K(\cA)_{w \geq 0}$ consists of those chain complexes that are chain homotopy equivalent to ones concentrated in non-negative degrees, and similarly for $K(\cA)_{w \leq 0}$.
The weight decompositions are achieved by the ``stupid'' / ``naive'' truncations.
Of course, this also yields a weight structure on $\Ch_\infty(\cA)$,
which turns out to restrict to $\Ch^b_\infty(\cA)$,
and yield a weight exact equivalence $\Ch^b_\infty(\cA) \simeq \Stab(\cA)$.
For the convenience of the reader, we will reprove the existence
of this weight structure on the way to the following theorem.

\begin{theorem}\label{thm:chains}
    Let $\cA$ be an ordinary additive category.
    There is a weight structure on $\Ch_\infty(\cA)$ so that
    \begin{enumerate}
        \item $\Ch_\infty(\cA)_{w \geq 0}$ is the full subcategory on those objects equivalent
            to chain complexes concentrated in non-negative degrees,
            and dually for $\Ch_\infty(\cA)_{w \leq 0}$.
            The ``naive'' truncations of a chain complex
            provide a weight decomposition.
            The weight heart is $\cA^\idem$.

        \item There is a weight exact equivalence $\Ch_\infty(\cA^\op) \simeq \Ch_\infty(\cA)^\op$.

        \item The weight structure is both left and right complete.

        \item The weight structure restricts to a bounded weight structure on $\Ch^b_\infty(\cA)$
            with heart $\cA^\widem$, so there is a weight exact equivalence $\Stab(\cA) \simeq \Ch^b_\infty(\cA)$.
            We have $\Ch^-_\infty(\cA) = \Ch_\infty(\cA)_{w >-\infty}$,
            so $\Ch^-_\infty(\cA)$ admits a left complete and right bounded weight structure with weight
            heart $\cA^\idem$, and dually for $\Ch^+_\infty(\cA)$.

        \item We have a weight exact equivalence
            $
                \Kw(\cA) \simeq \Ch_\infty(\cA),
            $
            which justifies the notation $\Kw(\cA)$.

        \item Let $\cC$ be a (left and right) complete weighted category.
            If $\cC_{w = 0}$ is an ordinary category,
            then there exists a weight exact equivalence $\Ch_\infty(\cC_{w=0}) \simeq \cC$.
    \end{enumerate}
\end{theorem}

\begin{warning}\label{warn:ch-bounded}
    We always have an inclusion $\Ch^b_\infty(\cA) \subseteq \Ch_\infty(\cA)_{\wb}$
    of the dg-nerve of the category of bounded chain complexes
    into the weight-bounded objects in $\Ch_\infty(\cA)$.
    This is generally \emph{not} an equivalence
    (note that definitionally $\Ch_\infty(\cA)_{[-n,n]}$ contains objects $X \in \Ch_\infty(\cA)$
    which can be represented by a $(-n)$-connective complex and also represented by an $n$-coconnective
    complex in $\cA$, but not necessarily by a simultaneously $(-n)$-connective and $n$-coconnective
    complex). This comes down to an idempotent completeness issue.
\end{warning}

As a consequence, we recover a Theorem of Schnürer\footnote{This is stated in the language of triangulated categories, but if $\cA$ is an additive $\infty$-category, then it is idempotent complete if and only if its homotopy category is idempotent complete, see e.g.~\cite[Proposition 2.3.2]{Sosnilo-Thesis}.} on the idempotent completeness of bounded below / above chain complexes.

\begin{corollary}[{{\cite[Theorem 1.2]{Schnuerer}}}]\label{cor:ch-idem}
    The categories $\Ch^-_\infty(\cA)= \Ch_\infty(\cA)_{w>-\infty}$ and $\Ch^+_\infty(\cA) = \Ch_\infty(\cA)_{w< \infty}$
    are idempotent complete. In particular, they are closed under retracts in $\Ch_\infty(\cA)$.
\end{corollary}
\begin{proof}
    By the above $\Ch_\infty(\cA)_{\geq 0}$ admits finite colimits and sequential colimits
    growing in connectivity, so that it is idempotent complete by Corollary \ref{cor:grow-conn-idem}.
    By shifting we deduce that $\Ch_\infty(\cA)_{>-\infty}$ is idempotent complete,
    and dually $\Ch_\infty(\cA)_{< \infty}$ is idempotent complete.
\end{proof}

\begin{remark}\label{rem:ch-idem}
    Coming back to Warning \ref{warn:ch-bounded},
    the above corollary of course also implies that each $\Ch_\infty(\cA)_{w[-n,n]}$
    and hence $\Ch_\infty(\cA)_{\wb} = \bigcup_{n \geq 0} \Ch_\infty(\cA)_{w[-n,n]}$ are idempotent complete.
    On the other hand, the dg-nerve $\Ch_\infty^b(\cA)$ of the category of bounded chain complexes in $\cA$
    is generally not idempotent complete, see \cite[Remark 3.2]{Schnuerer} for a concrete example.
    However, the inclusion $\Ch_\infty^b(\cA) \subseteq \Ch_\infty(\cA)_{\wb}$ is always an idempotent completion.
    Moreover, as detailed in \cite[Section 3.1]{BondarkoSosnilo},
    any commutative ring $R$ which has $K_{-1}(R) \neq 0$
    yields an example of an ordinary additive category $\cA = \Proj^{\fg}(R)$
    for which $\Ch_\infty(\cA)$ is not idempotent complete.
    See also the discussion at Remark \ref{rem:w-idem-bs}.
\end{remark}

\begin{corollary}\label{cor:d-}
    Let $\cA$ be an abelian category with enough projective objects.
    Let $\cD^-(\cA) \coloneqq \Ch_\infty^-(\cA_\proj)$ be the bounded-below
    derived category of $\cA$ as defined in \cite[Definition 1.3.2.7]{HA}.
    Then $\cD^-(\cA)$ admits a weight structure such that
    \begin{enumerate}
        \item $X \in \cD^-(\cA)$ is weight connective
            if and only if it can be represented
            by a connective complex of projectives in $\cA$,
            if and only if $H_n(X) = 0$ for $n < 0$.
            In particular, the weight structure is adjacent
            to the left complete $t$-structure of \cite[Prop.~1.3.2.19, 1.3.3.16]{HA}.

        \item $X \in \cD^-(\cA)$ is weight coconnective
            if and only if it can be represented
            by a complex of projectives concentrated
            in non-positive (though bounded below) degrees.

        \item $\cD^-(\cA)_{w=0} = \cA_\proj$.

        \item The weight structure is left complete and right bounded.

        \item It follows that
            $
                \cD^-(\cA)_{\geq 0} \simeq \cP^{\Delta^\op}(\cA_\proj)
            $
            and we have weight exact equivalences
            \[
                \cD^-(\cA)
                \simeq \lc{\Stab(\cA_{\proj})}
                \simeq \SW(\cP^{\Delta^\op}(\cA_\proj)).
            \]
            In particular, this recovers the universal property
            of $\cD^-(\cA)$ shown in \cite[Theorem 1.3.3.8]{HA}
            (see also \cite[Lemma 1.3.3.17]{HA}).
    \end{enumerate}
\end{corollary}
\begin{proof}
    We take the weight structure provided to us by Theorem \ref{thm:chains},
    which makes (2) and (4) clear (note that (4)
    also follows from (1) via Lemma \ref{lem:left-t-comp-is-conn-comp}).
    Since abelian categories are idempotent complete
    and projectives are closed under retracts, (3) follows.
    For (1) it remains to see the homological characterization
    of weight connectives, which was shown in the proof
    of \cite[Proposition 1.3.2.19]{HA}.
    Finally (5) follows from Proposition \ref{prop:pdelta-add}.
\end{proof}

We will also discuss the unbounded derived category of an abelian category $\cA$ in Section \ref{sec:adj} below.
For now, we work towards proving Theorem \ref{thm:chains}
and begin with the naive truncations.

\begin{lemma}\label{lem:ch-trunc-cofseq}
    Let $\cA$ be an ordinary additive category
    and $X \in \Ch(\cA)$ be a chain complex in $\cA$.
    For every $k \in \Z$, the ``naive truncations'' yield a short exact sequence
    of chain complexes $\tau_{\leq k}X \to X \to \tau_{\geq k+1}X$
    \[\begin{tikzcd}[ampersand replacement=\&]
        \cdots \& 0 \& 0 \& {X_k} \& {X_{k-1}} \& \cdots \\
        \cdots \& {X_{k+2}} \& {X_{k+1}} \& {X_k} \& {X_{k-1}} \& \cdots \\
        \cdots \& {X_{k+2}} \& {X_{k+1}} \& 0 \& 0 \& \cdots
        \arrow[from=1-1, to=1-2]
        \arrow[from=1-2, to=1-3]
        \arrow[from=1-2, to=2-2]
        \arrow[from=1-3, to=1-4]
        \arrow[from=1-3, to=2-3]
        \arrow[from=1-4, to=1-5]
        \arrow[Rightarrow, no head, from=1-4, to=2-4]
        \arrow[from=1-5, to=1-6]
        \arrow[Rightarrow, no head, from=1-5, to=2-5]
        \arrow[from=2-1, to=2-2]
        \arrow[from=2-2, to=2-3]
        \arrow[Rightarrow, no head, from=2-2, to=3-2]
        \arrow[from=2-3, to=2-4]
        \arrow[Rightarrow, no head, from=2-3, to=3-3]
        \arrow[from=2-4, to=2-5]
        \arrow[from=2-4, to=3-4]
        \arrow[from=2-5, to=2-6]
        \arrow[from=2-5, to=3-5]
        \arrow[from=3-1, to=3-2]
        \arrow[from=3-2, to=3-3]
        \arrow[from=3-3, to=3-4]
        \arrow[from=3-4, to=3-5]
        \arrow[from=3-5, to=3-6]
    \end{tikzcd}\]
    This is a cofiber sequence in $\Ch_\infty(\cA)$.\footnote{Note that this is not a \emph{split} short exact sequence, in which case the claim would be trivial.}
\end{lemma}

\begin{warning}
    Despite the notation, the naive truncations are \emph{not} functorial on $\Ch_\infty(\cA)$.
\end{warning}

\begin{proof}
    By \cite[Remark 1.3.2.17]{HA}, the cofiber of the map $f \colon \tau_{\leq k}X \to X$ in $\Ch_\infty(\cA)$
    is computed by the mapping cone $C(f)$.
    It thus suffices to show that there is a chain homotopy equivalence from $C(f)$ to $\tau_{\geq k+1}X$.
    This follows from the above sequence of chain complexes being termwise split exact,
    see e.g.~\cite[\href{https://stacks.math.columbia.edu/tag/014L}{Tag 014L}]{stacks-project}.
\end{proof}

\begin{remark}\label{rem:ch-dual}
    Note that the classical equivalence $\Ch(\cA^\op) \simeq \Ch(\cA)^\op$
    induces an equivalence $\Ch_\infty(\cA^\op) \simeq \Ch_\infty(\cA)^\op$
    which restricts to $\Ch_\infty(\cA^\op)_{w \geq 0} \simeq (\Ch_\infty(\cA)_{w \leq 0})^\op$.
    In other words, endowing $\Ch(\cA)^\op$ with the evident dual (co)connectivity structures,
    the equivalence $\Ch(\cA^\op) \simeq \Ch(\cA)^\op$
    has (co)connectivity amplitude $0$.
\end{remark}

\begin{lemma}\label{lem:ch-conn-comp}
    The connectivity structure $(\Ch_\infty(\cA),\Ch_\infty(\cA)_{w \geq 0})$ is colimit complete,
    and for any complex $X_*$ the naive truncations
    determine a sequential cone $\tau_{\leq \bullet}X_*$ over $X$
    which is growing in connectivity and also a colimit cone.
    By the duality of Remark \ref{rem:ch-dual}, it follows that $(\Ch_\infty(\cA), \Ch_\infty(\cA)_{w \leq 0})$
    is limit complete.
\end{lemma}
\begin{proof}
    That the sequential cone $\tau_{\leq \bullet}X_*$ over $X_*$ is growing in connectivity
    follows from Lemma \ref{lem:ch-trunc-cofseq}.
    To see that it is a colimit cone, we verify the universal property.
    Recall from \cite[Construction 1.3.1.13, Proposition 1.3.1.17]{HA}
    that the mapping anima in $\Ch_\infty(\cA)$ between two complexes $A_*,B_*$
    in $\cA$ is computed as the underlying anima of the object in $\cD(\Z)$ represented by
    the $\Ch(\Ab)$-enriched hom $\und{\hom}_{\Ch(\cA)}(A_\bullet,B_\bullet)$
    whose definition is recalled in \cite[Definition 1.3.2.1]{HA}.
    Important for us will only be that the $p$-th term of this complex
    is given by $\prod_{n \in \Z} \Hom_{\cA}(A_n,B_{n+p})$ for $p \in \Z$.

    Namely, given any $A_* \in \Ch(\cA)$,
    the above cone yields a (strict) limit cone
    $\und{\hom}_{\Ch(\cA)}(X_*,A_*) \xto{\cong} \und{\hom}_{\Ch(\cA)}(\tau_{\leq \bullet}X_*,A_*)$ in $\Ch(\Z)$.
    We claim that this also presents the limit in $\cD(\Z)$.
    Recall that there is a model structure on $\Ch(\Z)$ (constructed e.g.~in \cite[Theorem 2.3.11]{Hovey})
    whose fibrations are the levelwise surjections and whose weak equivalences are the quasi-isomorphisms
    so that the model structure presents $\cD(\Z)$.
    Now we see that every map $\und{\hom}_{\Ch(\cA)}(\tau_{\leq n+1}X_*,A_*) \to \und{\hom}_{\Ch(\cA)}(\tau_{\leq n}X_*,A_*)$ is levelwise just projecting away one factor.
    Thus our inverse sequential diagram in $\Ch(\Z)$ consists entirely of fibrations in this model structure, and is thus fibrant in the injective model structure for inverse sequential diagrams in $\Ch(\Z)$ (whose weak equivalences and cofibrations are pointwise).
    We conclude that the strict limit indeed models the limit in $\cD(\Z)$.
    Overall, since the forgetful functor $\cD(\Z) \to \An$ also preserves limits,
    this verifies that $\tau_{\leq n}X_* \Rightarrow \const X_*$
    is a colimit cone in $\Ch_\infty(\cA)$.

    From this it is easy to see that any $X_* \in \Ch_\infty(\cA)_{w \geq \infty}$
    has to be chain-homotopy equivalent to $0$.
    Indeed, given any other chain complex $A_*$,
    we note that $\hom_{\Ch_\infty}(\tau_{\leq n}A_*,X_*)$ is $k$-connective for all $k\geq 0$
    since $X_*$ is equivalent to an $(n+k)$-connective complex for all $k$.
    Hence the hom spectrum vanishes, and by taking the limit also $\hom(A_*,X_*) = 0$,
    which shows $X_* = 0$.

    Now let $X^0 \to X^1 \to X^2 \to \cdots$ be a sequential diagram growing in connectivity in $\Ch_\infty(\cA)$,
    where by cofinality we may assume that $\cofib(X^n \to X^{n+1}) \in \Ch_\infty(\cA)_{w\geq n+1}$
    for all $n \geq 0$.
    We claim that we can inductively choose representatives $X^n_*$ of $X^n$
    and of the maps so that for all $n \geq 0$ and $k \leq n$ the map $X^n_k \to X^{n+1}_k$ is the identity.
    Indeed, begin by choosing any representative $X^0_*$ of $X_0$.
    Having chosen representatives for the objects and maps up to $X^n_*$,
    we note that we have a cofiber sequence $F \to X^n_* \to X^{n+1}$
    with $F \in \Ch_\infty(\cA)_{w \geq n}$,
    so we can pick a model $F_*$ for $F$ with $F_k = 0$ for $k < n$,
    as well as any chain map $F_* \to X^n_*$ of the right chain homotopy class.
    We have $X^{n+1} \simeq \cofib(f)$, but the cofiber of $f$ in $\Ch_\infty(\cA)$
    can be computed by a mapping cone construction.
    Thus the explicit formula for the mapping cone yields the desired representative $X^{n+1}_* \coloneqq C(f)$.

    We now define the complex $X^\infty_*$ by $\tau_{\leq n}X^\infty_* = \tau_{\leq n}X^n_*$
    and compute for any $A_*$ that
    \[
        \lim_n \hom(X_*^n,A_*)
        \simeq \lim_{n,k} \hom(\tau_{\leq k}X_*^n,A_*)
        \simeq \lim_n \hom(\tau_{\leq n} X_*^n,A_*)
        \simeq \lim_n \hom(\tau_{\leq n}X^\infty_*,A_*)
        \simeq \hom(X^\infty_*,A_*)
    \]
    which proves $\colim_n X^n_* \simeq X^\infty_*$.
    Thus $\Ch_\infty(\cA)$ admits sequential colimits growing in connectivity,
    and it is clear from the inductive construction of the $X^n_*$ and $X^\infty_*$ that
    $\Ch_\infty(\cA)_{w \geq 0}$ is closed under such colimits.
    This concludes the proof that $(\Ch_\infty(\cA), \Ch_\infty(\cA)_{w \geq 0})$ is colimit complete.
\end{proof}

With this, we can now prove Theorem \ref{thm:chains}.

\begin{proof}[Proof of Theorem \ref{thm:chains}]
    The proposed weight connectives and coconnectives
    clearly satisfy the desired orthogonality,
    are closed under retracts by Corollary \ref{cor:ch-idem},
    and naive truncations provide us with weight decompositions
    by Lemma \ref{lem:ch-trunc-cofseq}.
    To determine the weight heart, note that $\Ch_\infty(\cA)_{w=0}$ clearly contains $\cA$
    and hence $\cA^\idem$ by the idempotent completeness of $\Ch_\infty(\cA)$
    and retract-closure of the weight heart.
    Conversely, suppose that $X \in \Ch_\infty(\cA)_{w = 0}$.
    Then $X$ is represented by
    a connective chain complex $X_*$,
    and we obtain a weight decomposition
    $X_0[0] \to X \to \tau_{\geq 1}X$.
    Since the map $X \to \tau_{\geq 1}X$ vanishes,
    this exhibits $X$ as a retract of $X_0[0] \in \cA \subseteq \Ch_\infty(\cA)$.
    This proves (1).

    Now (2) is clear from Remark \ref{rem:ch-dual},
    and (3) follows from Lemma \ref{lem:ch-conn-comp}.

    It is also clear from the construction that the weight structure restricts to $\Ch^b_\infty(\cA)$
    and $\Ch^{\pm}_\infty(\cA)$, so by Theorem \ref{thm:vova},
    for (4) it remains to identify the weight heart of $\Ch^b_\infty(\cA)$ as $\cA^\idem$.
    This was done in \cite[Example 3.1.3(1)]{Hebestreit-Steimle}.

    For (5), note that by (3) we have $\Ch_\infty(\cA) \simeq \lrc{(\Ch_\infty(\cA)_{\wb})}$.
    The inclusion $\cA \subseteq \cA^\idem \subseteq \Ch_\infty(\cA)_{\wb}$
    induces by Theorem \ref{thm:vova} a weight exact functor $f \colon \Stab(\cA) \to \Ch_\infty(\cA)_{\wb}$
    which restricts to the inclusion $\cA \subseteq \cA^\idem$ on weight hearts.
    In particular, $f$ is an idempotent completion, and we deduce from Corollary \ref{cor:left-comp-inverts-idem}
    that $\lc{f} \colon \lc{\Stab(\cA)} \to \lc{(\Ch_\infty(\cA)_{\wb})} = \Ch^-_\infty(\cA)$
    is an equivalence.
    But then of course also $\lrc{f} = \rc{(\lc{f})} \colon \Kw(\cA) \coloneqq \lrc{\Stab(\cA)} \to \lrc{(\Ch_\infty(\cA)_{\wb})} = \Ch_\infty(\cA)$ is an equivalence (recall also from Proposition \ref{prop:lr-complete} that the order of left and right completion does not matter here, since we are in the $\ell=0$ case).
    Finally, (6) follows via a similar argument as in (5).
\end{proof}

\section{$\kappa$-compactly generated weight structures}\label{sec:gen}

As we saw in the last section and specifically Theorem \ref{thm:vova},
it is well known how to generate bounded weight structures from their heart.
Another result in this vein is \cite[Theorem 4.3.2(II)]{Bondarko10}, which shows how a set of generators
for a stable category with mutually connective hom spectra determine a bounded weight structure.
The goal of this section is to instead focus on generating weight structures
on (large) cocomplete stable categories by a collection of objects.
Namely, in the proof of Proposition \ref{prop:ind-weight}
we saw how given a weakly weighted category $\cC$
we can construct a weight structure on $\Ind(\cC)$
which is in a certain sense generated by the Yoneda-image of $\cC_{\leq 0}$.
In this section, we will generalize
this to construct weight structures generated by $\kappa$-compact generators
for some regular cardinal $\kappa$,
see Theorem \ref{thm:gen-weight} below.
We also give some examples in this section and analyze their completeness properties,
although the main application of the theorem is in constructing a certain exotic ``Anderson weight structure''
in Section \ref{sec:anderson}.

We will again focus on non-weak weight structures.
Recall from Remark \ref{rem:saturated} that these are saturated,
and hence the weight (co)connectives satisfy the closure properties listed in Lemma \ref{lem:weight-closure}.
It is the fourth closure property about ordinal-indexed diagrams,
or more specifically the results in Appendix \ref{sec:app}
which enable us to generalize away from the compact case discussed in Proposition \ref{prop:ind-weight}.

\begin{definition}\label{def:gen-weight}
    Let $\cC$ be a cocomplete weighted category,
    $\kappa$ a regular cardinal, and $S \subseteq \cC^\kappa$
    a small collection of $\kappa$-compact objects.
    We say the weight structure is $\kappa$-compactly generated by $S$
    if $X \in \cC$ is weight connective
    precisely if $\hom(s,X) \in \Sp_{\geq 0}$ for all $s \in S$.
\end{definition}

\begin{remark}
    If $\cC$ admits a weight structure $\kappa$-compactly generated by $S$,
    then the same weight structure is also $\lambda$-compactly generated by $S$
    for any regular cardinal $\lambda \geq \kappa$.
    Recall also that if $\cC$ is presentable, then $\cC^\kappa$ is small for each $\kappa$.
    In this case, if $\cC$ admits a $\kappa$-compactly generated weight structure,
    then the weight structure is also $\kappa$-compactly generated by $\cC_{\leq 0} \cap \cC^\kappa$.
\end{remark}

Recall that by $S^\oplus$ we denote the full subcategory on arbitrary sums of objects in $S$.

\begin{theorem}\label{thm:gen-weight}
    Let $\cC$ be a cocomplete stable category, $\kappa$ a regular
    cardinal, and $S \subseteq \cC^\kappa$ a small collection of $\kappa$-compact objects.
    Then there exists a weight structure on $\cC$ that is $\kappa$-compactly generated by $S$. This satisfies:
    \begin{enumerate}
        \item $\cC_{\geq 0} = \{X \in \cC \mid \hom(s,X) \in \Sp_{\geq 0} \text{ for all $s \in S$}\}$.
        \item $\cC_{\leq 0}$ is the smallest class $\langle S\rangle$ containing $S$
            and closed under finite limits, arbitrary sums, retracts, and colimits over diagrams
            $X_\bullet \colon \gamma \to \cC$ such that $\gamma \leq \kappa$ is an ordinal,
            $X_\bullet$ preserves colimits\footnote{This means that the canonical map $\colim_{\alpha < \lambda}X_\alpha \to X_\lambda$
            is an equivalence for all limit ordinals $\lambda < \gamma$.},
            and $X_0$ and $\cofib(X_{\alpha} \to X_{\alpha+1})$ already lie in the class for any $\alpha+1 < \gamma$.

        \item If shifts of objects in $S$ jointly detect equivalences,
            then $\cC_{\geq \infty} = 0$.

        \item Suppose that $\kappa = \omega$. Then $\cC_{\geq 0}$ is closed under all colimits.
            Moreover, if we pick a weight decomposition $s_{\leq -1} \to s \to s_{=0}$ for all $s \in S$
            and let $S_{=0} \subseteq \cC_{=0}$ denote the full subcategory on the $s_{=0}$,
            then $\cC_{=0} = (S_{=0})^{\oplus,\idem}$ consists of retracts of sums of objects in $S_{=0}$.

        \item In particular, if $\kappa = \omega$ and the (shifts of) $s \in S$ jointly detect equivalences, then $\cC$ is left complete.
    \end{enumerate}
\end{theorem}

\begin{remark}\label{rem:gen-weight}
    If $\kappa = \omega$, this has previously appeared as \cite[Theorem 5]{Pauksztello}
    or \cite[Theorem 2.3.4]{Bondarko22}.
\end{remark}

\begin{proof}
    We define $\cC_{\geq 0}$ as in (1)
    and $\cC_{\leq 0} = \{X \in \cC \mid \hom(X,Y) \in \Sp_{\geq 0} \text{ for }Y \in \cC_{\geq 0}\}$ for now.
    Clearly both are closed under retracts and orthogonal,
    so to establish a weight structure it remains to prove the existence of weight decompositions, which we shall achieve via a small object argument.
    Let $X \in \cC$. We now use transfinite induction to construct a diagram $X_\bullet \colon \kappa+1 \to \cC$ as follows:
    \begin{enumerate}
        \item Let $X_0 = X$.
        \item Having constructed the diagram up to $X_{\alpha}$, we define $X_{\alpha+1}$ via the cofiber sequence
            \[
                \bigoplus_{n \geq 0} \bigoplus_{s \in S_{/X_\alpha[n]}} s[-n] \to X_\alpha \to X_{\alpha+1}.
            \]
        \item For limit ordinals $\lambda \leq \kappa$ we let $X_{\lambda} \simeq \colim_{\alpha < \lambda}X_\alpha$.
    \end{enumerate}
    Now consider the $(\kappa+1)$-indexed diagram of fibers
    $F_\bullet \coloneqq \fib(\const X \Rightarrow X_\bullet)$.
    Taking colimits, this yields a fiber sequence $F_\kappa \to X \to X_\kappa$,
    and we claim that $X_{\kappa} \in \cC_{\geq 1}$ and $F_\kappa \in \cC_{\leq 0}$.
    For the former, let $n \geq 0$ and note that since each $s \in S$ is $\kappa$-compact
    and $\kappa$ is $\kappa$-filtered, we have
    \[
        \pi_{-n}\hom(s,X_\kappa) \simeq \colim_{\alpha < \kappa}\pi_0\hom(s[-n],X_\alpha)
    \]
    However, for each $\alpha < \kappa$, the map $\hom(s[-n],X_\alpha) \to \hom(s[-n],X_{\alpha+1})$
    induces the zero map on $\pi_0$ by construction.
    Thus $\hom(s,X_\kappa)$ is 1-connective, showing $X_\kappa \in \cC_{\geq 1}$.

    To see that $F_{\kappa} \in \cC_{\leq 0}$,
    we let $Y \in \cC_{\geq 0}$ and show that $\hom(F_\kappa,Y) \in \Sp_{\geq 0}$.
    By construction of $F_\bullet$
    we have $F_0 = 0$, $F_\lambda = \colim_{\alpha < \lambda} F_\alpha$
    for limit ordinals $\lambda \leq \kappa$
    and also $\cofib(F_\alpha \to F_{\alpha+1}) = \bigoplus_{n \geq 0}\bigoplus_{s \in S_{/X_\alpha[n]}}s[-n]$ for $\alpha < \kappa$.
    It follows that $\hom(F_\bullet,Y) \colon \kappa^\op \to \Sp$
    satisfies the assumptions
    of Corollary \ref{cor:mittag-leffler-sp} for $n=0$,
    so that $\hom(F_\kappa,Y) = \lim_{\alpha < \kappa}\hom(F_\alpha,Y)$
    is connective.

    Next, to obtain the desired description of $\cC_{\leq 0}$
    from (2), suppose that $X \in \cC_{\leq 0}$. Taking a weight decomposition
    $F_{\kappa} \to X \to X_{\kappa}$ as above,
    the second map vanishes
    so $X$ is a retract of $F_\kappa \in \langle S\rangle$
    and hence also lies in $\langle S \rangle$, as desired.

    Everything of (3)-(5) except the description of $\cC_{=0}$ in the case $\kappa = \omega$ is clear
    from the definitions.
    For the latter, let $S_{=0}$ be as described in (4). The argument is as in the proof
    of Proposition \ref{prop:ind-weight}(4).
    Namely, given $X \in \cC_{=0}$ we consider the canonical map $f \colon \bigoplus_{s_{=0} \in (S_{=0})_{/X}}s_{=0} \to X$.
    Then every map $s \to \cofib(f)$ factors through $X \to \cofib(f)$.
    But since $X$ is connective, the map $s \to X$ factors through $s_{=0}$
    and hence lifts to $\fib(X \to \cofib(f))$,
    which shows that the original map $s \to \cofib(f)$ was null.
    This proves that $\cofib(f)$ is 1-connective,
    so that $f$ admits a retraction, as desired.
\end{proof}

We now apply Theorem \ref{thm:gen-weight} to construct a weight
structure on the category of (left-)modules over a connective $\E_1$-ring.
For $R = \S$, the resulting weight structure on spectra was originally constructed by Bondarko in \cite[Section 4.6]{Bondarko10}.

\begin{proposition}\label{prop:lmod-weight}
    Let $R$ be a connective $\E_1$-ring.
    There is a weight structure on the category $\LMod(R)$
    compactly generated by $\{R\}$, satisfying:
    \begin{enumerate}
        \item the weight structure is adjacent to the canonical complete $t$-structure
            of \cite[Proposition 7.1.13]{HA}.
            We let $\LMod(R)_{\geq 0} \coloneqq \LMod(R)_{w \geq 0} = \LMod(R)_{t \geq 0}$.

        \item If $A$ is an ordinary ring,
            this recovers the ``projective'' weight structure on $\LMod(A) \simeq \cD(A)$
            considered in \cite[Section 4.1]{Bondarko-Shamov}, for which
            \begin{enumerate}
                \item $\cD(A)_{w \leq 0}$ consists
                    of objects represented by termwise projective,
                    $K$-projective\footnote{In the sense of Spaltenstein \cite{Spaltenstein}.} complexes concentrated in degrees $\leq 0$.
                    In particular, $\cD(A)_{w(-\infty,0]} = \cD(A)_{>-\infty} \cap \cD(A)_{w \leq 0}$
                    consists of retracts of termwise projective complexes concentrated
                    in a bounded range of non-positive degrees.

                \item $\cD(A)_{w = 0}$ is the additive category of projective $A$-modules.

                \item If $M \in \cD(A)_{t \leq 0}$ has $\pi_{-n}M$ of projective dimension $\leq n$
                    for each $n \geq 0$, then $M \in \cD(A)_{w \leq 0}$.
                    In particular, if $A$ has left global dimension $\leq d$,
                    then $\cD(A)_{t \leq -d} \subseteq \cD(A)_{w \leq 0}$
            \end{enumerate}

        \item In the general case, an $R$-module $M$ is weight coconnective
            if and only if $\pi_0R \tensor_R M$ is weight coconnective in $\LMod(\pi_0R)$.
            Note that if $R$ is $n$-truncated, then $\LMod(R)_{w \leq 0} \subseteq \LMod(R)_{t \leq n}$.

        \item The weight heart $\LMod(R)^{w \heartsuit} = \LMod(R)_{w = 0}$
            is given by $\Proj(R) = \{R\}^{\oplus,\idem}$, the idempotent completion
            of the full subcategory of $\LMod(R)$ on direct sums of $R$.

        \item The weight structure is left complete.
            If $A$ is an ordinary ring of finite left global dimension,
            then the weight structure on $\LMod(A) = \cD(A)$ is also right complete.

        \item If $R \to S$ is a map of connective $\E_1$-rings,
            then $S \tensor_R - \colon \LMod(R) \to \LMod(S)$ is weight exact.

        \item If $R$ is commutative, then the weight structure
            is compatible with the monoidal structure in the sense of
            Corollary \ref{cor:weight-complex}(4).
    \end{enumerate}
\end{proposition}
\begin{proof}
    Applying Theorem \ref{thm:gen-weight}, point (1) is clear.
    We assume (2) for now and postpone its proof towards the end.

    For point (6), the basechange preserves (weight) connectives since all the rings we consider
    are connective.
    Since the basechange functor is an exact left adjoint,
    it also preserves all the operations under which $\LMod(R)_{w \leq 0}$ is generated
    from $R$. Since $S \tensor_R R = S \in \LMod(S)_{w \leq 0}$,
    we conclude that $S \tensor_R -$ is weight exact.

    For (3), let $\cC \subseteq \LMod(R)$ be the full subcategory
    on those $M$ for which $M_0 \coloneqq \pi_0R \tensor_R M \in \LMod(\pi_0R)_{w \leq 0}$.
    By (6) we have $\LMod(R)_{w \leq 0} \subseteq \cC$.
    Conversely, let $M \in \cC$ and $N \in \LMod(R)_{\geq 0}$.
    For each $k \geq 0$, we have that
    \[
        \hom_{R}(M, \Sigma^k\pi_kN) \simeq \hom_{\pi_0R}(\pi_0R \tensor_R M, \Sigma^k\pi_k N) \in \Sp_{\geq k}
    \]
    and hence an induction shows that $\hom_R(M, \tau_{\leq k}N)$
    is connective for all $k$.
    By left completeness of the $t$-structure,
    $\hom_R(M,N) \simeq \lim_n \hom_R(M, \tau_{\leq n}N)$
    is a sequential limit of connective spectra
    growing in connectivity, hence itself connective.
    This shows $\cC \subseteq \LMod(R)_{w \leq 0}$.

    Next, we show (4).
    Consider $M \in \LMod(R)_{w = 0}$.
    Then $M_0 \coloneqq \pi_0R \tensor_R M \in \LMod(\pi_0R)$ is an ordinary projective $\pi_0R$-module
    by (2) and (6). In particular, there exists another $\pi_0R$-module $P$
    so that $P \oplus M_0$ is free.
    Picking a basis, we find a map $\bigoplus_I R \to M_0$
    which is a split epimorphism on $\pi_0$. Since $\cofib(M \to M_0) \simeq \Sigma\tau_{\geq 1}R \tensor_R M$ is 2-connective and $\bigoplus_I R$ lies in the weight-heart,
    we see that this lifts to a map $f \colon \bigoplus_I R \to M$
    which is still a split epi on $\pi_0$
    since $M \to M_0$ is a $\pi_0$ isomorphism.
    But then $\cofib(f)$ is 1-connective,
    and we see that $f$ admits a retraction, as desired.

    For (5), the left completeness follows from left completeness of the $t$-structure
    and Lemma \ref{lem:left-t-comp-is-conn-comp}.
    Given an ordinary ring $A$ of left global dimension $\leq d$,
    then by (2) we have inclusions $\LMod(A)_{w \leq 0} \subseteq \LMod(A)_{t \leq 0} \subseteq \LMod(A)_{w \leq d}$.
    Since the $t$-structure on $\LMod(A)$ is right complete,
    we deduce that also the weight structure is right complete
    via Remark \ref{rem:conn-comp-inv-under-eqv} and the dual of Lemma \ref{lem:left-t-comp-is-conn-comp}.

    For (7), the only possibly non-trivial part is the closure of $\Mod(R)_{w \leq 0}$ under tensor products.
    However, $\tensor_R$ preserves the operations under which $\Mod(R)_{w \leq 0}$ is generated from $R$
    in both variables separately, and $R \tensor_R R \simeq R \in \Mod(R)_{w \leq 0}$,
    which allows us to conclude.

    It remains to show (2).
    We will use the classical projective model structure on the category $\Ch(A)$
    of (unbounded) chain complexes of $A$-modules, see \cite[Section 2.3]{Hovey}.
    This presents the $\infty$-category $\cD(A)$.
    Recall that the cofibrant objects are termwise projective
    and also $K$-projective in the sense of Spaltenstein \cite{Spaltenstein}.
    Note that $K$-projectivity, unlike cofibrancy in $\Ch(A)$, is definitionally
    invariant under chain homotopy equivalence (though of course not under quasi-isomorphism).
    Note also that bounded below termwise projective complexes are cofibrant, hence $K$-projective,
    which allows us to deduce the description of $\cD(A)_{w(-\infty,0]}$ from that of $\cD(A)_{w \leq 0}$.

    Suppose that $P_\bullet$ is a $K$-projective complex concentrated
    in degrees $\leq 0$. Let $X \in \cD(A)_{\geq 0}$, represented by some complex $Q_\bullet$
    concentrated in degrees $\geq 0$ (e.g.~using the smart truncation).
    Then $K$-projectivity of $P_\bullet$ means that we can compute $\hom_{\cD(A)}(P_\bullet, Q_\bullet)$
    as the chain complex hom, whose degree $p$ term is given by $\prod_{n \in \Z} \Hom_A(P_n, Q_{n+p})$.
    Clearly this vanishes for $p < 0$, which proves that $P_\bullet$ presents a weight-coconnective
    object in $\cD(A)$.

    Conversely, suppose that $X \in \cD(A)_{w \leq 0}$,
    which we can represent by a cofibrant (hence $K$-projective and pointwise projective) complex $P_\bullet$.
    For any ordinary $A$-module $N$,
    we then know that the Hom-complex $\und{\Hom}_A(P_\bullet, N)$ has vanishing homology in degrees $< 0$.
    So for each $m \geq 1$ we have exact sequences
    \[
        \Hom_A(P_{m-1},N) \xto{d_m^*} \Hom_A(P_m, N) \xto{d_{m+1}^*} \Hom_A(P_{m+1},N).
    \]
    Now set $C_m = \Coker(d_{m+1})$, and let $\overline{d_m} \colon C_m \to P_{m-1}$ be the induced map.
    The exactness yields that for every $N$, the induced map
    $\overline{d_m}^* \colon \Hom_A(P_{m-1},N) \to \Hom_A(C_m,N)$
    is surjective. Taking $N = C_m$, we get that $\overline{d_m}$ is split injective,
    so in particular $\Img(d_m) \subseteq P_{m-1}$ is a direct summand, hence again projective.
    Moreover, since $0 = \Ker(\overline{d_m}) \cong \Ker(d_m)/\Img(d_{m+1})$,
    we see that $P_\bullet$ is exact at $m$.
    Now the short exact sequences $0 \to \Img(d_{m+1}) \to P_m \to \Img(d_m) \to 0$
    split by projectivity of $\Img(d_m)$,
    showing that the complex $(\cdots \to P_3 \to P_2 \to P_1 \to \Img(d_1))$
    is even split exact, and hence a contractible subcomplex of $P_\bullet$.
    So $P_\bullet$ is chain homotopy equivalent to the quotient by this subcomplex.
    The resulting quotient is concentrated in degrees $\leq 0$
    and is still termwise projective since $\Img(d_1) \subseteq P_0$ is a direct summand,
    and is still $K$-projective since it is chain homotopy equivalent to the $K$-projective $P_\bullet$.
    This concludes the identification of $\cD(A)_{w \leq 0}$.

    For (2b), let $X \in \cD(A)_{w = 0} \subseteq \cD(A)_{t = 0}$.
    So we can view $X$ as an ordinary $A$-module.
    Moreover, for any other ordinary $A$-module $M$, we have
    that $\Ext_A^k(X,M) \cong \pi_{-k}\hom_{\cD(A)}(X,M) = 0$,
    which shows that $X$ is projective, as desired.

    For (2c), suppose $M \in \cD(A)$ is $t$-coconnective and $\pi_{-n}M$ has projective dimension $\leq n$
    for each $n \geq 0$. Let $N$ be an ordinary $A$-module.
    The assumption exactly yields that $\hom(\Omega^n \pi_{-n}M,N)$ is connective for each $n \geq 0$.
    An induction along the Whitehead tower then shows that also each $\hom(\tau_{\geq -n}M,N)$ is connective,
    and a Mittag-Leffler argument then shows that $\hom(M,N) \simeq \lim_n \hom(\tau_{\geq -n}M,N)$
    is connective.
    For general connective $N \in \cD(A)_{\geq 0}$,
    we can again induct our way up the Postnikov tower of $N$ to deduce that each $\hom(M,\tau_{\leq n}N)$
    is connective,
    and that $\hom(M,N) = \lim_n \hom(M,\tau_{\leq n}N)$ is a sequential limit growing in connectivity
    of connective spectra, hence connective.
    This shows that $M \in \cD(A)_{w \leq 0}$.
\end{proof}

\begin{example}\label{spectra}
    For $R = \Z$, we get a complete weight structure on $\cD(\Z)$
    which is adjacent to the complete $t$-structure.
    The weight coconnectives are precisely the $t$-coconnective complexes with free abelian $\pi_0$.

    For $R = \S$, we obtain a left complete weight structure adjacent to the complete $t$-structure on $\Sp$.
    The weight coconnective spectra are precisely those for which $\Z \tensor X$ is weight coconnective
    in $\cD(\Z)$, i.e.~those spectra whose $\Z$-homology is free in degree $0$ and vanishes in degrees $\geq 1$.
\end{example}

\begin{remark}\label{rem:sp-not-dw-comp}
    The description of weight coconnective spectra
    yields that $\Sp_{w \leq -\infty} = \{X \in \Sp \mid \Z \tensor X = 0\} \neq 0$.
    This is not the only reason that the standard weight structure on $\Sp$ is not right complete;
    $\Sp_{w \leq 0}$ is also not closed under sequential inverse limits growing in weight coconnectivity,
    so we cannot just quotient out $\Sp_{w \leq -\infty}$
    to guarantee right completeness by (the dual of) Lemma \ref{lem:upw-comp-quot}.
    Indeed, to see this, note that $\prod_{n \geq 0} \S[-n] \simeq \lim(\cdots \to \S[-1] \times \S \to \S)$
    is such a limit, where each term $\prod_{n =0}^k \S[-n] \in \Sp_{w \leq 0}$.
    However,
    \[
        \pi_k\bigg(\Z \tensor \prod_{n \geq 0}\S[-n]\bigg)_\Q
        = \pi_k\bigg(\Q \tensor \prod_{n \geq 0}\S[-n]\bigg)
        = \pi_k\bigg(\prod_{n \geq 0}\S[-n]\bigg)_\Q
        = \bigg(\prod_{n \geq 0}\pi_{k+n}(\S)\bigg)_\Q
        \neq 0
    \]
    for all $k$, since there is for example an embedding
    $\prod_{p\geq k\text{ prime}} \F_p \hookrightarrow \prod_{n \geq 0} \pi_{k+n}(\S)$,
    and the element $([1],[1],[1],\dots)$ spans an infinite cyclic group in the former,
    so that both groups have non-trivial rationalization.
    In particular, we see that $\prod_{n \geq 0} \S[-n] \notin \Sp_{w < \infty}$.
\end{remark}

The examples discussed here only required the previously known version of Theorem \ref{thm:gen-weight}
for $\kappa = \omega$.
However, in Section \ref{sec:anderson} we will see another interesting
example of a weight structure on Spectra, for whose existence we need the $\kappa = \omega_1$ version of the Theorem.

\section{Weak $t$-structures and their completions}\label{sec:t}

Unlike completions of weight structures as introduced in the previous sections,
the notion of complete $t$-structure is already well established.
In this section, we record how the right completion of a $t$-structure
can be achieved by an analogous procedure as that of the left completion of weight structures above.
In fact, we will work with the following weakening of the notion of $t$-structure,
akin to weakly weighted categories.
The entire theory parallels that of weakly weighted categories and their completions,
so we will be brief.

\begin{definition}\label{def:weak-t}
    Let $\cC$ be a stable category and $m \in \N_0 \cup \{\infty\}$.
    A weak $t$-structure of defect $m$ on $\cC$ consists of full subcategories
    $\cC_{\leq 0},\cC_{\geq 0} \subseteq \cC$ such that
    \begin{enumerate}
        \item $\cC_{\geq 0}$ is closed under suspensions,
            $\cC_{\leq 0}$ under loops, and both under retracts and extensions.
        \item $\hom_\cC(X_{\geq 0},Y_{\leq 0}) \in \Sp_{t\leq 0}$
            for $X_{\geq 0} \in \cC_{\geq 0}$ and $Y_{\leq 0} \in \cC_{\leq 0}$.
        \item For every $X \in \cC$ there is a cofiber sequence
            $X_{>-m} \to X \to X_{\leq 0}$ with $X_{> -m} \in \cC_{> -m}$ and $X_{\leq 0} \in \cC_{\leq 0}$.
            We refer to this as a $t$-decomposition of $X$.
    \end{enumerate}
    Here $\cC_{>n-1} = \cC_{\geq n} = \Sigma^n \cC_{\geq 0}$ and $\cC_{\leq n} = \Sigma^n \cC_{\leq 0}$
    for $n \in \Z$, and $\cC_{>-\infty} = \bigcup_{n \geq 0} \cC_{\geq -n}$.
    This recovers the usual notion of $t$-structure for $m = 0$.
\end{definition}

As before, a weak $t$-structure of defect $m$ is also a weak $t$-structure of any defect $m' \geq m$.
By a weak $t$-category we mean a stable category equipped with a weak $t$-structure of defect $\infty$.

\begin{warning}\label{warn:wtcat-non-dual}
    As in Warning \ref{warn:wwcat-non-dual},
    if $(\cC^\op,(\cC_{\leq 0})^\op,(\cC_{\geq 0})^\op)$ is a weak $t$-category of defect $m$
    and $m < \infty$, then $(\cC,\cC_{\geq 0},\cC_{\leq 0})$ is a weak $t$-category of the same defect.
    But if $m = \infty$, then we have $t$-decompositions $X_{\geq 0} \to X \to X_{<\infty}$ instead.
    In this case, we say $\cC$ is a weak co-$t$-structure\footnote{In the literature a weight structure is sometimes
    referred to as a co-$t$-structure. This is incompatible with the nomenclature here.},
    and a weak bi-$t$-structure if both $(\cC,\cC_{\geq 0},\cC_{\leq 0})$
    and $(\cC^\op,(\cC_{\leq 0})^\op,(\cC_{\geq 0})^\op)$ are weak $t$-structures.
\end{warning}

\begin{remark}\label{rem:bounded-is-weak-t}
    Analogously to Remark \ref{rem:bounded-is-weakly-weighted},
    we note that if $(\cC,\cC_{\geq 0},\cC_{\leq 0})$ satisfies axioms (1) and (2) of Definition \ref{def:weak-t}
    and $\cC_{>-\infty} = \cC$, then $\cC$ is a weak $t$-category via the trivial decompositions $X \xto{\id} X \to 0$
    for all $X \in \cC$.
    Dually, if $\cC_{<\infty} = \cC$, then $\cC$ is a weak co-$t$-category
    via the trivial decompositions $0 \to X \xto{\id} X$.
    So if $\cC_{<\infty} = \cC = \cC_{>-\infty}$, then $\cC$ is a weak bi-$t$-category.
    In particular, if $(\cC,\cC_{\geq 0})$ is a stable connectivity structure
    we can define $\cC_{\leq 0} = \rperp{(\cC_{\geq 1})}$.
    Then $(\cC_{>-\infty},\cC_{\geq 0}, \cC_{(-\infty,0]})$ is always a weak $t$-structure,
    and $(\cC_{<\infty}, \cC_{[0,\infty)}, \cC_{\leq 0})$ is a weak co-$t$-structure
    when $\cC_{[0,\infty)}$ is closed under retracts in $\cC_{<\infty}$.
\end{remark}

There are various equivalent reformulations of the above notion,
analogous to those of Remark \ref{rem:weak-weight-reformulation}.
The notion of $t$-amplitude of an exact functor between weak $t$-categories
is defined analogously to Definition \ref{def:weight-amplitude},
which yields categories $\wTCat,\wTCat_{\leq 0},\wTCat_{=0}$ of weak $t$-categories and exact functors of bounded above $t$-amplitude, $t$-amplitude $[-\infty,0]$, and $t$-amplitude $[0,0]$.
Dual to the case of weak weight structures, we will mostly be concerned with the \emph{right} completions of weak $t$-structures, i.e.~colimit completions of their underlying coconnectivity structures.

\begin{observation}
    Following Definition \ref{def:saturated}, we say that a weak $t$-category is left resp.~right saturated
    if $\cC_{\leq 0} = \rperp{(\cC_{\geq 1})}$ resp.~$\cC_{\geq 0} = \lperp{(\cC_{\leq -1})}$
    and saturated if it is both left and right saturated.
    If $\cC$ is a left \alt{right} saturated weak $t$-category,
    then $\cC_{\leq 0}$ \alt{$\cC_{\geq 0}$} is closed under limits \alt{colimits} in $\cC$.
\end{observation}

\begin{remark}\label{rem:t-saturated}
    In practice, the weak $t$-structures we consider will typically be left saturated in the above sense,
    because we often just start with a class of connectives and then \emph{define}
    the weight- and $t$-coconnectives as the orthogonal complements, as in the introduction.
    In particular, if we specify only a single class $\mathcal{C}_{\geq 0}$ and refer to it as defining a weak $t$-structure, we implicitly take the other class to be its saturation as above.

    In general, the analogue of Lemma \ref{lem:saturation},
    shows that given a weak $t$-structure,
    one can construct another by \emph{left saturating} $\cC$,
    i.e.~by replacing $\cC_{\leq 0}$ with the right orthogonal of $\mathcal{C}_{\geq 1}$.
    The identity will yield a functor of $t$-amplitude $[0,0]$ in one direction,
    and a functor of $t$-amplitude $[0,m]$ in the other.
\end{remark}

\begin{example}\label{ex:corep-bdd-t}
    If $X \in \cC_{\geq a}$ for some $a \in \Z$,
    then $\hom_\cC(X,-) \colon \cC \to \Sp$ has $t$-amplitude in $[-\infty,-a]$.
    Hence Corollary \ref{cor:bdd-amp-pres} implies that this functor preserves
    sequential colimits growing in coconnectivity.
\end{example}

\begin{definition}\label{def:right-dense}
    Analogously to Definition \ref{def:upw-dense} we say that a full stable
    subcategory $\cU \subseteq \cC_{>-\infty}$ is right dense in the weak $t$-category $\cC$ if for every $X \in \cC$
    and $n \geq 0$ there exists a map $U \to X$ with cofiber in $\cC_{\leq -n}$.
    The canonical example is $\cC_{>-\infty}$ itself.
    Given $X \in \cC$, a $\cU$-Whitehead tower for $X$ is a diagram $U_\bullet \colon \N \to \cU_{/X}$
    so that $\N \to \cC_{/X}$ is a cone growing in coconnectivity in the sense of (the dual of)
    Definition \ref{def:grow-in-conn}, meaning that for every $n \geq 0$
    the cofibers $U_k \to X$ (and hence of $U_k \to U_{k+1}$) are eventually $n$-coconnective.
\end{definition}

Lemma \ref{lem:weight-cplx} and its proof dualize verbatim to show the following.

\begin{lemma}
    If $\cU$ is right dense in $\cC$, then any $X \in \cC$ admits a $\cU$-Whitehead tower,
    and any such tower $U_\bullet \colon \N \to \cU_{/X}$ is cofinal.
\end{lemma}

\begin{definition}
    A weak $t$-structure on $\cC$ is right complete if
    its underlying coconnectivity structure $(\cC,\cC_{t \leq 0})$ is colimit complete, meaning:
    \begin{enumerate}
        \item $\cC_{\leq -\infty} \coloneqq \bigcap_{n \geq 0} \cC_{\leq -n} = 0$.
        \item $\cC$ admits sequential colimits growing in $t$-coconnectivity.
        \item $\cC_{\leq 0} \subseteq \cC$ is closed under such sequential colimits.
    \end{enumerate}
    We let $\wTCatr \subseteq \wTCat$ be the full subcategory on right complete weak $t$-categories.
\end{definition}

\begin{observation}\label{obs:right-comp}
    Analogously to Observation \ref{obs:left-comp},
    the property of a weak $t$-category $\cC$ being right complete
    does not depend on $\cC_{\geq 0}$. In particular, if $(\cC,\cC'_{\geq 0},\cC_{\leq 0})$
    is another weak $t$-structure on $\cC$ with the same coconnectives (e.g.~$\cC'_{\geq 0} = \lperp{(\cC_{\leq -1})}$), then it is also right complete.
    Moreover, if $\cC$ is right complete, then $\cC_{<\infty}$ is right complete and idempotent complete by the dual of Lemma \ref{lem:lim-comp-conn-idem}.
    In particular, also $\cC_{\leq 0}$ and $\cC_{=0}$ are idempotent complete.
    Moreover, if $\cU$ is right dense in $\cC$,
    then any $\cU$-Whitehead tower $X_\bullet$ of $X$ converges, i.e.~$\colim_n X_n \simeq X$.
    In this sense, any object in $\cC$ may be obtained as the colimit
    of a sequential diagram in $\cC_{>-\infty}$ growing in coconnectivity.
\end{observation}

The dual of Lemma \ref{lem:upw-comp-eqv} gives the following lemma.

\begin{lemma}\label{lem:t-dw-comp-eqv}~
    \begin{enumerate}
        \item Let $f \colon \cC \to \cD$ be of amplitude $(a,b]$
            for $-\infty \leq a \leq b < \infty$
            between right complete weak $t$-categories.
            If its restriction $\cC_{>-\infty} \to \cD_{>-\infty}$
            is fully faithful \alt{an equivalence}, then so is $f$.

        \item Let $f \colon \cC \to \cD$ be a $t$-exact functor of right complete weak $t$-categories.
            If the restriction $\cC_{=0} \to \cD_{=0}$ is fully faithful \alt{an equivalence},
            then so is the restriction $\cC_{<\infty} \to \cD_{<\infty}$.
    \end{enumerate}
\end{lemma}

\begin{definition}\label{def:right-completion}
    A $t$-exact functor of weak $t$-categories
    $\eta \colon \cC \to \rc{\cC}$
    exhibits $\rc{\cC}$ as the right completion of $\cC$
    if $\rc{\cC}$ is right complete and
    for any other right complete weak $t$-category $\cD$,
    precomposition by $\eta$ induces an equivalence
    \begin{equation}\label{eq:dw-comp}
        \eta^* \colon \Fun^\Ex_{[0,0]}(\rc{\cC},\cD) \xto{\simeq} \Fun^{\Ex}_{[0,0]}(\cC,\cD).
    \end{equation}
    In other words, $\eta$ exhibits $\rc{\cC}$ as the left adjoint object to $\cC$
    under the inclusion $\wTCatr_{=0} \subseteq \wTCat_{=0}$.
\end{definition}

As one expects,
the main theorem on left completions of weak weight structures
has an analogue for right completions of weak $t$-structures:

\begin{theorem}\label{thm:t-dw-comp}
    Any weak $t$-category of defect $m$ admits a right completion $\rc{\cC}$.
    This yields a left Bousfield localization $\rc{(-)} \colon \wTCat_{=0} \to \wTCatr_{=0}$ such that
    \begin{enumerate}
        \item There exists a fully faithful $t$-exact functor $\rc{\cC} \hookrightarrow \Ind(\cC_{>-\infty})$\footnote{Here we equip $\Ind(\cC_{>-\infty})$ with the $t$-structure from Proposition \ref{prop:ind-t} with connectives $\Ind(\cC_{\geq 0})$.}
            which lets us identify
            \begin{align*}
                \qquad\qquad\rc{\cC}
                &\simeq \{X \in \Ind(\cC_{>-\infty}) \mid \text{$\forall n \geq 0$ there is $c \in \cC_{>-\infty}$ and a map
                $jc \to X$ with $n$-coconnective cofiber}\}\\
                &= \{X \in \Ind(\cC_{>-\infty}) \mid X \text{ admits a $\cC_{>-\infty}$-Whitehead tower}\}\\
                &= \{\indcolim{n} c_n \in \Ind(\cC_{>-\infty}) \mid c_\bullet \colon \N \to \cC_{>-\infty} \text{ growing in coconnectivity}\}.
            \end{align*}
            which is the largest stable subcategory of $\Ind(\cC_{>-\infty})$
            in which $\cC_{>-\infty}$ is right dense,
            and also the smallest stable subcategory of $\Ind(\cC_{>-\infty})$
            containing $\cC_{>-\infty}$ and closed under sequential colimits growing in coconnectivity.
            Moreover, we have
            \[
                \rc{\cC}_{\leq 0}
                = \{\indcolim{n} c_n \mid c_\bullet \colon \N \to \cC_{(-\infty,0]}\text{ growing in coconnectivity}\}.
            \]

        \item The inclusion induces a $t$-exact equivalence $\rc{(\cC_{>-\infty})} \xto{\simeq} \rc{\cC}$. Moreover:
            \begin{enumerate}
                \setlength{\itemsep}{.2em}
                \item The unit $\eta_\cC \colon \cC \to \rc{\cC}$
                    is $t$-exact and the left Kan extension
                    of $\eta_{\cC_{>-\infty}}$ along $\cC_{>-\infty} \subseteq \cC$.

                \item $\eta_{\cC_{>-\infty}}$
                    agrees with the Yoneda embedding
                    $j \colon \cC_{>- \infty} \hookrightarrow \rc{\cC} \subseteq \Ind(\cC_{>-\infty})$, so is always fully faithful.
                    In fact, for $X \in \cC_{>-\infty}$ we have $\hom_\cC(X,-) \xto{\simeq} \hom_{\rc{\cC}}(\eta X, \eta -)$.

                \item $\cC_{>-\infty} \subseteq \rc{\cC}$
                    is right dense, $\cC_{\geq 0} \subseteq \rc{\cC}_{\geq 0}$ is dense,
                    and we have $\cC_{(-\infty,0]} = \cC_{>-\infty} \cap \rc{\cC}_{\leq 0}$.

                \item $\eta_\cC$ agrees with the Verdier projection
                    $\cC \to \cC/\cC_{\leq -\infty}$ whenever $\cC$ already
                    admits sequential colimits growing in coconnectivity
                    and $\cC_{\leq 0}$ is closed under them.
            \end{enumerate}

        \item For any right complete weak $t$-category $\cD$, $b \in \Z$
            and $-\infty \leq a \leq b$,
            restriction along $\eta$ induces an equivalence
            \[
                \eta_\cC^* \colon \Fun^\Ex_{[a,b]}(\rc{\cC},\cD) \xto{\simeq} \Fun^\Ex_{[a,b]}(\cC,\cD)
            \]
            with inverse given by left Kan extension along $\eta_\cC$.
            In particular, if $f \colon \cC \to \cD$ is a functor in $\wTCat$ of amplitude $[a,b]$,
            then $\rc{f} = (\eta_{\cC})_!(\eta_\cD f) \colon \rc{\cC} \to \rc{\cD}$
            again has amplitude $[a,b]$.
            Analogously for $(a,b]$ instead of $[a,b]$.

        \item In particular, right completion also defines a left adjoint
            to the inclusions $\wTCatr_{\leq 0} \subseteq \wTCat_{\leq 0}$
            respectively $\wTCatr \subseteq \wTCat$.
    \end{enumerate}
\end{theorem}

We now show that in the case $m=0$ the above right completion
recovers the usual right completion of $t$-structures.

\begin{corollary}\label{cor:t-comp-compare}
    Let $\cC$ be a stable category with $t$-structure.
    Then $\cC$ is right complete if and only if it is right complete
    in our sense. Moreover we obtain a natural $t$-exact equivalence $\rc{\cC} \simeq \lim_n \cC_{\geq -n}$ under $\cC$,
    i.e.~our definition of right $t$-completion specializes to the usual one.
\end{corollary}
\begin{proof}
    The dual of Lemma \ref{lem:left-t-comp-is-conn-comp}
    shows that the $t$-structure is right complete if and only if the underlying
    coconnectivity structure is colimit complete, which is precisely
    our definition of right completeness.

    Next, the universal property of $\rc{\cC}$ yields a $t$-exact comparison map
    $\rc{\cC} \to \what{\cC} \coloneqq \lim_n \cC_{\geq -n}$.
    Moreover, this extends to a commutative diagram
    \[\begin{tikzcd}
        \cC & {\what{\cC}} \\
        {\rc{\cC}} & {\what{\rc{\cC}}}
        \arrow[from=1-1, to=1-2]
        \arrow[from=1-1, to=2-1]
        \arrow[two heads, from=1-2, to=2-2]
        \arrow[from=2-1, to=1-2]
        \arrow["\simeq"', from=2-1, to=2-2]
    \end{tikzcd}\]
    where the bottom map is an equivalence by right completeness of $\rc{\cC}$.
    This implies that the right vertical map is essentially surjective.
    But it is also fully faithful as limit of the fully faithful
    functors $\cC_{\geq -n} \to (\rc{\cC})_{\geq -n}$, cf.~Theorem \ref{thm:t-dw-comp}(2b).
    By 2-out-of-3 also the diagonal functor in this diagram is an equivalence,
    as desired.
\end{proof}

\begin{remark}
    If $\cC$ is a $t$-category,
    then from the above it is clear that the right completion
    only depends on $\cC_{>-\infty} = \SW(\cC_{\geq 0})$
    (cf.~Lemma \ref{lem:stable-conn-struc}(4)).
    Thus by \cite[Remark C.1.2.10]{SAG}
    there is another description of the right completion of
    $\cC$ as the category $\Sp(\cC_{\geq 0}) = \lim(\cdots \to \cC_{\geq 0} \xto{\Omega} \cC_{\geq 0} \xto{\Omega} \cC_{\geq 0})$ of spectrum objects in $\cC_{\geq 0}$.
    We view the right completion of weak $t$-structures
    as a generalized way to take spectrum objects on a prestable category
    and come back to this in Section \ref{sec:prestable-spectrum}.
\end{remark}

The following are the direct analogues of Corollaries \ref{cor:upw-comp-idem} and \ref{cor:left-comp-inverts-idem}.

\begin{corollary}\label{cor:dw-comp-idem}
    Let $\cC$ be a weak $t$-category. Then:
    \begin{enumerate}
        \item $\cC_{\tb} \hookrightarrow (\rc{\cC})_{\tb}$
            is an idempotent completion.

        \item $\rc{\cC}_{>-\infty} = \{X \in \Ind(\cC_{>-\infty}) \mid \text{ there is some bounded above $Y$ with }X \oplus Y \in \cC_{>-\infty}\}$.

        \item $\cC_{>-\infty} = (\rc{\cC})_{>-\infty}$
            if and only if $\cC_{\tb}$ is idempotent complete.
            In particular, if $\cC_{>-\infty}$ is idempotent complete
            then also $(\rc{\cC})_{>-\infty}$ is.
    \end{enumerate}
\end{corollary}

\begin{corollary}\label{cor:right-comp-inverts-idem}
    Let $\cC,\cD$ be weak $t$-categories
    and $f \colon \cC \to \cD$ of amplitude $(-\infty,b]$ for $b \in \Z$.
    \begin{enumerate}
        \item $\rc{f} \colon \rc{\cC}_{<\infty} \to \rc{\cD}_{<\infty}$
            is fully faithful \alt{an equivalence}
            exactly if $f^\idem \colon (\cC_{\tb})^\idem \to (\cD_{\tb})^\idem$ is.

        \item $\rc{f} \colon \rc{\cC} \to \rc{\cD}$ is fully faithful
            exactly if $f \colon \cC_{>-\infty} \to \cD_{>-\infty}$ is.

        \item $\rc{f} \colon \rc{\cC} \to \rc{\cD}$ is an equivalence
            exactly if $f \colon \cC_{>-\infty} \to \cD_{>-\infty}$ is a fully faithful right dense inclusion.
    \end{enumerate}
    Here `equivalence' means `equivalence of underlying categories'
    as opposed to `equivalence in $\wTCat$'.\footnote{Though in the saturated case this is implied
        by the analogue of Lemma \ref{lem:saturated}(3),
        and for $m < \infty$ it also follows from the analogue of Lemma \ref{lem:saturation}.}
\end{corollary}

\begin{remark}
    If $\cC$ is a weak co-$t$-category in the sense of Warning \ref{warn:wtcat-non-dual},
    so that $\cC^\op$ is a weak $t$-category, then as in Remark \ref{rem:dw-comp}
    we can dualize the definitions and results in this section
    to obtain a good theory of left completion for weak co-$t$-structures
    (hence for weak $t$-structures of defect $m < \infty$).

    Overall, we can summarize the various notions of completeness
    introduced in this paper via the following table:
    \renewcommand{\arraystretch}{1.5}
    \begin{center}
    \begin{tabular}{c | c | c}
                         & left complete                            & right complete\\\hline
        (weak) $t$-structure    & $(\cC,\cC_{t \geq 0})$ limit complete    & $(\cC,\cC_{t \leq 0})$ colimit complete\\\hline
        (weak) weight structure & $(\cC,\cC_{w \geq 0})$ colimit complete  & $(\cC,\cC_{w \leq 0})$ limit complete
    \end{tabular}
    \end{center}
\end{remark}

The proof of Theorem \ref{thm:t-dw-comp}
is entirely analogous to the proof of Theorem \ref{thm:upw-comp} in the weight case.
For the convenience of the reader, below we spell out the analogues of the series of Lemmas and Propositions
from Section \ref{sec:upw-comp} which lead up to the proof of Theorem \ref{thm:upw-comp}; the proofs go through with minimal adaptions.

\begin{lemma}\label{lem:t-comp-lkan}
    Let $\cC$ be a weak $t$-category and $(\cD,\cD_{\leq 0})$
    a colimit complete stable coconnectivity structure.
    Suppose that $i \colon \cU \subseteq \cC$ is right dense
    and equip $\cU$ with coconnectives $\cU_{\leq 0} = \cU \cap \cC_{\leq 0}$.
    For $b \in \Z$, Kan extension and restriction along $i$ induce an equivalence
    \[
        i_! \colon \Fun^\Ex_{\leq b}(\cU,\cD) \simeq \Fun^\Ex_{\leq b}(\cC,\cD) \noloc i^*
    \]
    where $\Fun^\Ex_{\leq b} \subseteq \Fun^\Ex$ is the full subcategory
    of functors of amplitude $\leq b$, i.e.~restricting to $\cC_{\leq 0} \to \cD_{\leq b}$.
\end{lemma}
\begin{proof}
    Analogous to Lemma \ref{lem:weight-comp-lkan}.
\end{proof}

\begin{proposition}\label{prop:t-comp-lkan}
    Let $\cU,\cC,\cD$ be weak $t$-categories.
    Suppose that
    \begin{enumerate}
        \item $\cD$ is right complete.
        \item There is a fully faithful $t$-exact inclusion $\cU \subseteq \cC_{>-\infty}$ so that $i \colon \cU \subseteq \cC$ is right dense.
        \item $\cU_{\geq 0} \subseteq \cC_{\geq 0}$ is dense.
        \item $\cU_{\leq 0} = \cU \cap \cC_{\leq 0}$.\footnote{For example, by Lemma \ref{lem:saturated} this holds if $\cU$ is left saturated (e.g.~$m=0$).}
    \end{enumerate}
    Then for $b \in \Z$ and $-\infty \leq a \leq b$
    the equivalence of Lemma \ref{lem:t-comp-lkan} restricts to
    \[
        i_! \colon \Fun^\Ex_{[a,b]}(\cU,\cD) \simeq \Fun^\Ex_{[a,b]}(\cC,\cD) \noloc i^*
    \]
    and similarly for $(a,b]$.
\end{proposition}
\begin{proof}
    Analogous to Proposition \ref{prop:weight-comp-lkan}.
\end{proof}

\begin{corollary}\label{cor:right-comp-via-embed}
    Let $\cC,\cD$ be a weak $t$-categories.
    Suppose that
    \begin{enumerate}
        \item $\cD$ is right complete.
        \item There is a fully faithful $t$-exact $j \colon \cC_{>-\infty} \hookrightarrow \cD$
            which is right dense.
        \item $\cC_{\geq 0} \subseteq \cD_{\geq 0}$ is dense.
        \item $\cC_{(-\infty,0]} = \cC_{>-\infty} \cap \cD_{\leq 0}$
            (e.g.~$\cC_{>-\infty}$ is left saturated, e.g.~a $t$-category).
    \end{enumerate}
    Then $j$ and the left Kan extension $i_!j \colon \cC \to \cD$
    of $j$ along $i \colon \cC_{>-\infty} \subseteq \cC$ are right completions.
\end{corollary}
\begin{remark}
    In particular, if $\cC$ is right complete, then $\cC_{>-\infty} \subseteq \cC$ is a right completion.
\end{remark}
\begin{proof}
    Analogous to Corollary \ref{cor:left-comp-via-embed}.
\end{proof}

\begin{lemma}\label{lem:dw-comp-basic}
    Let $\cC$ be a weak $t$-category of defect $m$ and $\cD$ a weak $t$-category.
    Suppose that:
    \begin{enumerate}
        \item[(i)] $\cD$ admits sequential colimits growing in coconnectivity and $\cD_{\leq 0}$ is closed under them.
        \item[(ii)] There is a fully faithful $t$-exact embedding $j \colon \cC_{>-\infty} \hookrightarrow \cD$.
    \end{enumerate}
    Denote by $\langle \cC_{>-\infty}\rangle \subseteq \cD$
    the smallest full subcategory of $\cD$ which contains $\cC_{>-\infty}$
    and is closed under sequential colimits growing in coconnectivity.
    Then:
    \begin{enumerate}
        \setlength{\itemsep}{.3em}
        \item If $c_\bullet,d_\bullet \colon \N \to\cC_{>-\infty}$ are growing in coconnectivity,
            then any map $\colim c_\bullet \to \colim d_\bullet$ in $\cD$
            lifts to a map of diagrams $c_\bullet \Rightarrow d_{f(\bullet)}$ for some
            strictly increasing $f \colon \N \to \N$.

        \item As full subcategories of $\cD$, we have
            \begin{align*}
                \langle \cC_{>-\infty} \rangle
                    =\ & \{\colim_n c_n \mid c_\bullet \colon \N \to \cC_{>-\infty} \text{ growing in coconnectivity}\},\\
                \langle \cC_{>-\infty} \rangle_{\leq 0}
                    \coloneqq\ &
                    \{\colim_n c_n \mid c_\bullet \colon \N \to \cC_{(-\infty,0]} \text{ growing in coconnectivity}\}
                    \subseteq \langle \cC_{>-\infty}\rangle \cap \cD_{\leq 0},\\
                \langle \cC_{>-\infty} \rangle \cap \cD_{\leq -\infty}
                    =\ &0,\\
                \qquad\qquad\cC_{>-\infty} \cap \langle \cC_{>-\infty}\rangle_{\leq 0}
                    =\ & \cC_{(-\infty,0]}.
            \end{align*}
            Moreover, $\langle \cC_{>-\infty}\rangle_{\leq 0}$
            is closed under sequential colimits growing in coconnectivity.
            In case that $\cC_{(-\infty,0]} = \cC_{>-\infty} \cap \cD_{\leq 0}$
            (e.g.~if $\cC_{>-\infty}$ is left saturated)
            we have $\langle \cC_{>-\infty}\rangle_{\leq 0} = \langle \cC_{>-\infty}\rangle \cap \cD_{\leq 0}$.

        \item $\langle \cC_{>-\infty} \rangle$ is a stable subcategory of $\cD$
            and with $\langle \cC_{>-\infty}\rangle_{\geq 0} \coloneqq \langle \cC_{>-\infty}\rangle \cap \cD_{\geq 0}$ we have:
            \begin{enumerate}
                \item $\langle \cC_{>-\infty}\rangle$ is a right complete weak $t$-category of defect $m$.
                \item Both inclusions $\cC_{>-\infty} \subseteq \langle \cC_{>-\infty}\rangle
                    \subseteq \cD$ are $t$-exact, and the first is right dense.

                \item If $\cC_{\geq 0} \subseteq \langle \cC_{>-\infty}\rangle_{\geq 0}$
                    is dense, then both $\cC_{>-\infty} \subseteq \langle \cC_{>-\infty}\rangle$
                    and its Kan extension $\cC \to \langle \cC_{>-\infty}\rangle$
                    are right completions.
            \end{enumerate}

        \item If $\cD_{\leq -\infty} = 0$ so that $\cD$ is right complete,
            then $\langle \cC_{>-\infty} \rangle$ is the largest stable subcategory
            of $\cD$ in which $\cC_{>-\infty}$ is right dense, i.e.
            \begin{align*}
                \langle \cC_{>-\infty} \rangle
                =\ &\{X \in \cD \mid \text{for $n \geq 0$ there is $c \in \cC_{>-\infty}$ and a map
                    $c \to X$ with cofiber in $\cD_{\leq -n}$}\}\\
                =\ & \{X \in \cD \mid X\text{ admits a $\cC_{>-\infty}$-Whitehead tower}\}.
            \end{align*}
    \end{enumerate}
\end{lemma}
\begin{proof}
    Analogous to Lemma \ref{lem:upw-comp-basic}.
\end{proof}

\begin{lemma}\label{lem:dw-comp-quot}
    Let $\cC$ be a weak $t$-category of defect $m$
    and suppose that $\cC$ admits sequential colimits growing in coconnectivity and $\cC_{\leq 0}$ is closed under them.
    Let $\langle \cC_{>-\infty}\rangle$ be as in Lemma \ref{lem:dw-comp-basic} ($\cC = \cD$).
    Then:
    \begin{enumerate}
        \item The full stable subcategories $\langle \cC_{>-\infty}\rangle,\cC_{\leq -\infty}$ determine a semiorthogonal decomposition of $\cC$
        \[\begin{tikzcd}[ampersand replacement=\&]
            {\langle \cC_{>-\infty}\rangle} \&\& \cC \&\& {\cC_{\leq -\infty}}
            \arrow[""{name=0, anchor=center, inner sep=0}, "j", shift left=2, hook, from=1-1, to=1-3]
            \arrow[""{name=1, anchor=center, inner sep=0}, "{j^R}", shift left=2, from=1-3, to=1-1]
            \arrow[""{name=2, anchor=center, inner sep=0}, "p", shift left=2, from=1-3, to=1-5]
            \arrow[""{name=3, anchor=center, inner sep=0}, "{p^R}", shift left=2, hook', from=1-5, to=1-3]
            \arrow["\dashv"{anchor=center, rotate=-90}, draw=none, from=0, to=1]
            \arrow["\dashv"{anchor=center, rotate=-90}, draw=none, from=2, to=3]
        \end{tikzcd}\]
            where both $j$ and $j^R$ are $t$-exact
            and commute with the inclusions of $\cC_{>-\infty}$.

        \item Given a $\cC_{>- \infty}$-Whitehead tower
            $X_{\bullet}$ of $X \in \cC$ we have an identification
            of cofiber sequences
            \[\begin{tikzcd}[ampersand replacement=\&]
                {\colim_n X_n} \& X \& {\colim_n X/X_n} \\
                {jj^RX} \& X \& {p^RpX}
                \arrow[from=1-1, to=1-2]
                \arrow["\simeq"', from=1-1, to=2-1]
                \arrow[from=1-2, to=1-3]
                \arrow["\simeq", from=1-3, to=2-3]
                \arrow["{\eps_X^j}", from=2-1, to=2-2]
                \arrow[Rightarrow, no head, from=2-2, to=1-2]
                \arrow["{\eta^p_X}", from=2-2, to=2-3]
            \end{tikzcd}\]
            where $\eps^j \colon jj^R \Rightarrow \id$ is the counit
            and $\eta^p \colon \id \Rightarrow p^Rp$ the unit.

        \item $j^R$ agrees with the left Kan extension of $\cC_{>-\infty} \subseteq \langle \cC_{>-\infty}\rangle$
            along $\cC_{>-\infty} \subseteq \cC$
            and exhibits $\cC/\cC_{\leq -\infty} \simeq \langle \cC_{>-\infty}\rangle$ as right completion of $\cC$.

        \item If $\cC$ is already right complete,
            then $j$ and $j^R$ are mutually inverse $t$-exact equivalences.
    \end{enumerate}
\end{lemma}
\begin{proof}
    Analogous to Lemma \ref{lem:upw-comp-quot}.
\end{proof}

\begin{proposition}\label{prop:ind-t}
    Let $\cC$ be a weak $t$-category of defect $m$.
    There is a $t$-structure on $\Ind(\cC)$ such that the Yoneda embedding is $t$-exact and
    \begin{enumerate}
        \item $\Ind(\cC)_{\geq 0} = \Ind(\cC_{\geq 0})$.

        \item We have $\Ind(\cC_{\leq 0}) \subseteq \Ind(\cC)_{\leq 0} \subseteq \Ind(\cC_{\leq m})$.\footnote{For $m = \infty$, the second inclusion is vacuous. If $\cC$ is also a weak co-$t$-structure
                (cf.~Warning \ref{warn:wtcat-non-dual}), i.e.~we also have weight decompositions
                $X_{\geq 0} \to X \to X_{<\infty}$,
                then the same proof shows that $\Ind(\cC)_{\leq 0} \subseteq \Ind(\cC_{<\infty})$.}
                In particular, for a $t$-structure on $\cC$ (i.e.~for $m = 0$) we have equality.

        \item There exists a stable recollement
            \[\begin{tikzcd}
                {\Ind(\cC_{>- \infty})} && {\Ind(\cC)} && {\Ind(\cC)_{\leq -\infty}}
                \arrow["{j_!}", shift left=3, hook', from=1-1, to=1-3]
                \arrow["{j_*}"', shift right=3, hook, from=1-1, to=1-3]
                \arrow["{j^*}"{description}, from=1-3, to=1-1]
                \arrow["p", shift left=3, from=1-3, to=1-5]
                \arrow["{p^{RR}}"', shift right=3, from=1-3, to=1-5]
                \arrow["{p^R}"{description}, hook', from=1-5, to=1-3]
            \end{tikzcd}\]
            where $j_!$ is induced by the inclusion via $\Ind(-)$,
            and $j_!,j^*$ are $t$-exact and $j_*$ has amplitude $[-\infty,0]$,
            and $j^*$ is a right completion.
            In particular, $\Ind(\cC_{>- \infty})$ is right complete.

        \item If $\cD$ is another weak $t$-category
            and $f \colon \cC \to \cD$ a map of $t$-amplitude $[a,b]$ \alt{$(a,b]$}
            for $-\infty \leq a \leq b \leq \infty$
            then $f_! = \Ind(f) \colon \Ind(\cC) \to \Ind(\cD)$
            is of $t$-amplitude $[a,b+m]$ \alt{$(a,b+m]$}.
    \end{enumerate}
\end{proposition}
\begin{proof}
    The $t$-structure exists by \cite[1.4.4.11(1)]{HA},
    and the rest follows as in Proposition \ref{prop:ind-weight},
    with one slight caveat;
    for the proof of $\Ind(\cC)_{\leq 0} \subseteq \Ind(\cC_{\leq m})$,
    the category $(\cC_{\leq m})_{/X}$
    does generally not have finite colimits hence is not filtered for the same reason as $\cC_{/X}$.
    However, given any finite diagram in $(\cC_{\leq m})_{/X}$
    it admits a colimit $c \in (\cC_{< \infty})_{/X}$, and then picking a $t$-decomposition
    of defect $m$ for the colimit yields a factorization of $c \to X$
    through $c \to c_{\leq m}$. Thus there is a cone on the original diagram
    in $(\cC_{\leq m})_{/X}$, showing that $(\cC_{\leq m})_{/X}$
    is filtered and cofinal in $\cC_{/X}$.
\end{proof}

\begin{warning}\label{warn:two-t-on-ind}
    If $\cC$ is a weakly $\ell$-weighted category
    which also admits a weak $t$-structure,
    then the induced weight structure
    on $\Ind(\cC)$ admits an adjacent $t$-structure
    by Proposition \ref{prop:ind-weight},
    and the above Proposition shows that the weak $t$-structure
    also induces a $t$-structure on $\Ind(\cC)$.
    In general, these do not agree.
    However, if $\cC_{w \geq 0} = \cC_{t \geq 0}$
    and $\ell = 0$, then Propositions \ref{prop:ind-weight}(3) and \ref{prop:ind-t}(1) guarantee that the two $t$-structures on $\Ind(\cC)$ do in fact agree.
\end{warning}

\begin{proof}[Proof of Theorem \ref{thm:t-dw-comp}]
    Again, the proof is largely analogous to that of Theorem \ref{thm:upw-comp}.
    Concretely, one wants to apply \ref{lem:dw-comp-basic} to the $t$-exact Yoneda embedding
    $\cC_{>-\infty} \hookrightarrow \Ind(\cC_{>-\infty})$
    to deduce that $\rc{\cC} \coloneqq \langle \cC_{>-\infty}\rangle$ is the desired right completion
    of $\cC_{>-\infty}$ and $\cC$.
    To this end, we need to check that $\cC_{\geq 0} \subseteq \rc{\cC}_{\geq 0}$ is dense.
    By right density of $\cC_{>-\infty} \subseteq \rc{\cC}$ it follows that a given $X \in \rc{\cC}_{\geq 0}$
    is a retract of some $c \in \cC_{>-\infty}$ and hence compact in $\Ind(\cC_{>-\infty})$.
    Now $\Ind(\cC)_{\geq 0} = \Ind(\cC_{\geq 0})$ and the inclusion into $\Ind(\cC_{>-\infty})$
    preserves filtered colimits. So $X$ is also compact in $\Ind(\cC)_{\geq 0} = \Ind(\cC_{\geq 0})$,
    hence lies in $(\cC_{\geq 0})^\idem$, as desired.
    The rest of the proof adapts verbatim.
\end{proof}

\section{The setting of prestable categories}\label{sec:prestable}

We saw in Section \ref{sec:conn} that a prestable category $\cC$
canonically gives rise to a bounded below stable connectivity structure $(\SW(\cC),\cC)$.
In subsection \ref{sec:prestable-comp} we investigate how colimit completeness
of $\SW(\cC)$ can be expressed in terms of $\cC$,
and define a colimit completion for certain kinds of prestable categories.
Moreover, in subsection \ref{sec:prestable-spectrum} we give a construction of spectrum
objects for prestable categories which generalizes the usual one that is defined
when the prestable category admits finite limits.
We begin by recalling the following equivalent characterizations of prestability.

\begin{proposition}[{{\cite[Section C.1.2]{SAG}}}]
    Let $\cC$ be a category. The following are equivalent:
    \begin{enumerate}
        \item $\cC$ is pointed, admits finite colimits, $\Sigma \colon \cC \to \cC$ is fully faithful,
            and every morphism $X \to \Sigma Y$ extends to a bifiber sequence $F \to X \to \Sigma Y$.

        \item $\cC$ is a finite-colimit-closed and extension-closed full subcategory of a stable category.
    \end{enumerate}
    Moreover, in this case:
    \begin{enumerate}
        \item[(a)] The canonical map $\cC \to \SW(\cC)$ is fully faithful.
        \item[(b)] If $f \colon \cC \hookrightarrow \cD$ is any fully faithful
            embedding where $\cD$ is stable and $f$ preserves finite colimits,
            then the essential image of $f$ is closed under extensions in $\cD$.
    \end{enumerate}
\end{proposition}

We also saw in Lemma \ref{lem:stable-conn-struc} that a prestable category
is equivalently a choice of connectivity structure on a stable category.
By its universal property, $\SW(\cC)$ is the initial stable category into which $\cC$ embeds
in a finite-colimit preserving way, and this defines a connectivity structure $\SW(\cC)_{\geq 0} = \cC$
which is closed under retracts in $\SW(\cC)$ (cf.~Lemma \ref{lem:stable-conn-struc}).
This yields a reasonable notion of connectivity on $\cC$,
and generally we let $\cC_{\geq n} = \SW(\cC)_{\geq n}= \{\Sigma^n X \mid X \in \cC\}$
for $n \geq 0$ (which is then also closed under retracts in $\cC$).
We will always mean this notion of connectivity when talking about, say, sequential diagrams
growing in connectivity in a prestable category.
There is more to be said when $\cC$ admits finite limits:

\begin{proposition}[{{\cite[Lemma C.1.2.9]{SAG}}}]\label{prop:prestable-finlim}
    Let $\cC$ be a prestable category. The following are equivalent:
    \begin{enumerate}
        \item $\cC$ admits finite limits.
        \item $\SW(\cC)_{\geq 0} = \cC$ is the connective part of a $t$-structure on $\SW(\cC)$.
    \end{enumerate}
    In this case $\cC$ admits categorical truncations
    by restricting the $t$-truncations on $\SW(\cC)$
    (cf.~\cite[Warning 1.2.1.9]{HA}).
\end{proposition}

\subsection{Completeness}\label{sec:prestable-comp}

We start by formulating colimit completeness for prestable categories.

\begin{definition}
    A prestable category $\cC$ is colimit \alt{limit} complete if
    \begin{enumerate}
        \item $\cC_{\geq \infty} \coloneqq \bigcap_{n \geq 0} \cC_{\geq n} = 0$.
        \item $\cC$ admits sequential colimits \alt{limits} growing in connectivity.
    \end{enumerate}
\end{definition}

By Lurie's prestable Dold-Kan correspondence \cite[Theorem C.1.3.1]{SAG}
we obtain an analogue of Example \ref{ex:dold-kan}. In particular, we get:

\begin{lemma}
    A prestable category $\cC$ admits sequential colimits growing in connectivity
    if and only if it admits geometric realizations.
\end{lemma}

\begin{example}
    A non-trivial stable category $\cC$, considered as a prestable category,
    can never be (co)limit complete, since $\cC_{\geq n} = \cC$ for each $n$.
\end{example}

\begin{lemma}\label{lem:conn-comp-prestable}
    Let $(\cC,\cC_{\geq 0})$ be a stable connectivity structure.
    Then this is colimit \alt{limit} complete if and only if the prestable category
    $\cC_{\geq 0}$ is colimit \alt{limit} complete and the inclusion $\cC_{\geq 0} \subseteq \cC$
    preserves colimits \alt{limits} of sequential diagrams growing in connectivity.\footnote{In the colimit case, this is of course equivalent to the inclusion preserving geometric realizations.}
\end{lemma}
\begin{proof}
    Clearly if the stable connectivity structure is colimit \alt{limit} complete,
    then also $\cC_{\geq 0}$ is colimit \alt{limit} complete.
    For the converse, it remains to show that $\cC$ admits colimits \alt{limits}
    of sequential diagrams growing in connectivity.
    We will deal with the case of colimits, as the case of limits is dual.
    So let $X_\bullet \colon \N \to \cC$ be a sequential diagram growing in connectivity.
    By cofinality we may assume that $\cofib(X_n \to X_{n+1}) \in \cC_{\geq n+1}$.
    In particular, also $\cofib(X_0 \to X_n)$ is connective for all $n$.
    Then the diagram $\cofib(X_0 \to X_\bullet)$ lies in $\cC_{\geq 0}$
    and is still growing in connectivity, hence admits a colimit $C \in \cC_{\geq 0}$.
    By assumption, this is then also a colimit diagram in $\cC$,
    and one checks that $\fib(C \to \Sigma X_0)$ is the colimit of the original $X_\bullet$.
    We also see that if each $X_n$ is connective, then the colimit lies in $\cC_{\geq 0}$.
\end{proof}

\begin{corollary}\label{cor:sw-complete}
    A category $\cC$ is a colimit \alt{limit} complete prestable category precisely if it is contained as the connectives in a stable category with colimit \alt{limit} complete connectivity structure.
    Concretely, we can embed $\cC$ into the Spanier-Whitehead category $\SW(\cC)$
    which is colimit \alt{limit} complete if and only if $\cC$ is.
\end{corollary}
\begin{proof}
    By Lemma \ref{lem:conn-comp-prestable},
    it only remains to check that if $\cC$ is a colimit \alt{limit} complete prestable category,
    then the inclusion $\cC \subseteq \SW(\cC)$ preserves colimits \alt{limits} of sequential
    diagrams growing in connectivity. But this follows from the fact that the connectivity
    structure on $\SW(\cC)$ is bounded below; if $X_\bullet \colon \N \to \cC$
    is a sequential diagram growing in connectivity with colimit $X$,
    and $Y \in \SW(\cC)$; then $\Sigma^nY \in \cC$ for some $n \geq 0$,
    and hence $\hom(X,Y) \to \lim_n \hom(X_n, Y)$ is a shift of the equivalence
    $\hom(X,\Sigma^nY) \simeq \lim_n \hom(X_n,\Sigma^nY)$.
    Analogously for limits of sequential diagrams growing in connectivity.
\end{proof}

\begin{corollary}\label{cor:prest-idem}
    A colimit complete prestable category is idempotent complete
    and admits the kinds of colimits described in Lemma \ref{lem:fintype-colim}.
    Also a limit complete prestable category is idempotent complete.
\end{corollary}
\begin{proof}
    This is a consequence of Corollary \ref{cor:sw-complete} and Lemmas \ref{lem:fintype-colim} and \ref{lem:lim-comp-conn-idem}.
\end{proof}

In the case where $\cC$ even admits finite limits,
we get an analogue of Lemma \ref{lem:left-t-comp-is-conn-comp}.

\begin{lemma}
    Let $\cC$ be a prestable category admitting finite limits.
    \begin{enumerate}
        \item For $n \geq 0$, we have $X \in \cC_{\geq n+1}$
            if and only if $\tau_{\leq n}X = 0$
            (cf.~Proposition \ref{prop:prestable-finlim}.)

        \item The following are equivalent:
            \begin{enumerate}
                \item $\cC$ is limit complete.
                \item $\cC$ is limit and colimit complete.
                \item $\cC$ is (Postnikov) complete in the sense of \cite[Definition C.1.2.12]{SAG}, i.e.~if $\tau_{\leq n}X =0$ for all $n$ then $X =0$,
                    and $\cC = \lim(\cdots \to \tau_{\leq 2}\cC\xto{\tau_{\leq 1}} \tau_{\leq 1}\cC \xto{\tau_{\leq 0}} \tau_{\leq 0}\cC)$.
                \item The $t$-structure on $\SW(\cC)$ is left complete.
            \end{enumerate}
    \end{enumerate}
\end{lemma}
\begin{proof}
    If $X$ is an object in a stable category
    with $t$-structure, then $\tau_{\leq n}X = 0$ if and only if $X$ is $(n+1)$-$t$-connective. Hence (1) is immediate from Proposition \ref{prop:prestable-finlim}.

    For (2) the implication (b) $\Rightarrow$ (a) is clear.
    Next, (a) implies by Corollary \ref{cor:sw-complete} that also $\SW(\cC)$
    is limit complete. Since here the connectivity structure on $\SW(\cC)$
    comes from a $t$-structure, we get (d) by Lemma \ref{lem:left-t-comp-is-conn-comp}.
    Given left completeness of $\SW(\cC)$,
    we again obtain limit-completeness of $\cC$ and hence that $\cC_{\geq \infty} = 0$, which by (1) shows that $\infty$-connective objects vanish.
    Moreover, for $n \geq 0$, the $t$-truncation $\tau_{\leq n} \colon \SW(\cC)_{t \leq n+1} \to \SW(\cC)_{t \leq n}$
    restricts to the categorical truncation $\tau_{\leq n} \colon \tau_{\leq n+1}\cC \to \tau_{\leq n}\cC$.
    We obtain a commutative square
    \[\begin{tikzcd}[ampersand replacement=\&]
        {\SW(\cC)} \& {\lim_n \SW(\cC)_{t \leq n}} \\
        \cC \& {\lim_n \tau_{\leq n}\cC}
        \arrow["{(\tau_{\leq n})_n}", from=1-1, to=1-2]
        \arrow[hook, from=2-1, to=1-1]
        \arrow["{(\tau_{\leq n})_n}"', from=2-1, to=2-2]
        \arrow[hook, from=2-2, to=1-2]
    \end{tikzcd}\]
    which is vertically right adjointable since truncations
    commute with taking connective covers, see \cite[Proposition 1.2.1.10]{HA}.
    Since the top map in the above square is an equivalence by assumption,
    we see that the bottom map is fully faithful.
    The vertical right adjoints are localizations hence essentially surjective,
    and it follows that the bottom map is also essentially surjective,
    hence an equivalence, giving (c).
    Finally, for (c) $\Rightarrow$ (b) we can argue as in
    Lemma \ref{lem:left-t-comp-is-conn-comp}.
\end{proof}

As a next step, we want to define and construct colimit completions of prestable categories.

\begin{definition}\label{def:colim-comp}
    A right exact functor of prestable categories $\eta \colon \cC \to \lc{\cC}$
    exhibits $\lc{\cC}$ as the colimit completion of $\cC$
    if $\lc{\cC}$ is colimit complete,
    and for any other colimit complete prestable category $\cD$,
    precomposition by $\eta$ induces an equivalence
    \[
        \eta^* \colon \Fun^\rex(\lc{\cC},\cD) \xto{\simeq} \Fun^\rex(\cC,\cD).
    \]
\end{definition}

From the above results it should be clear that a colimit completion of $\cC$
should be a kind of colimit completion of the stable connectivity structure $(\SW(\cC),\cC)$.
In particular, if the latter underlies a weak weight structure,
we can use Theorem \ref{thm:upw-comp} to achieve this via left completion of the weak weight structure.

\begin{observation}\label{obs:sw-weight}
    Let $\SW(\cC)_{w \leq 0} = \lperp{(\cC_{\geq 1})}$
    and $\cC_{w \leq n} \coloneqq \cC \cap \SW(\cC)_{w \leq n}$ for $n \geq 0$.
    Since $\cC \subseteq \SW(\cC)$ is closed under retracts (Lemma \ref{lem:stable-conn-struc}),
    we see that $(\SW(\cC),\cC,\SW(\cC)_{w \leq 0})$ satisfies
    all the axioms of a weak weight structure except the existence of weight decompositions
    (cf.~Definition \ref{def:weak-w}).
\end{observation}

\begin{example}
    Note that  $\cC_{w \le 0}$ could easily be zero. For example, every stable category is prestable and in this case $\cC_{w \le 0}$ is zero.
    For another example more in the spirit of this section we let $\cC = \Sp^\fg_{t[0,\infty)}$
    be the category of bounded above, connective spectra with finitely generated homotopy groups, then $\cC_{w \le n} = 0$ for every $n$.
    To see this, we will use Anderson duality (cf.~Section \ref{sec:anderson}).
    Recall that $I_\Z$ is a coconnective spectrum with finitely generated homotopy
    groups, and for any $X \in \Sp^\fg$ the canonical map
    $X \to \hom(\hom(X,I_\Z),I_\Z)$ is an equivalence.
    Now if $X \in \cC_{w \leq 0}$, then
    \[
        0
        = \pi_{-1}\hom(X,\Sigma^n\tau_{\geq -n}I_\Z)
        = \pi_{-(n+1)}\hom(X, \tau_{\geq -n}I_\Z)
    \]
    for $n \geq 0$. Since $X$ is connective,
    the last term surjects onto $\pi_{-(n+1)}\hom(X,I_\Z)$,
    and so we get that $\hom(X,I_\Z)$ is connective.
    Then $X \simeq \hom(\hom(X,I_\Z),I_\Z)$ is coconnective,
    hence concentrated in degree 0.
    In fact, using Lemma \ref{lem:sum-iz} we see that $\pi_0X$ must be free,
    so $X = \Z^n$ for some $n \geq 0$.
    But then $\hom(X,\Z)$ cannot be connective unless $n = 0$,
    in which case $X=0$.
\end{example}

\begin{definition}\label{def:weak-pre-weight}
    For $0 \leq \ell \leq \infty$, we say that a prestable category $\cC$ is weakly $\ell$-pre-weighted if
    for each $X \in \cC$ and $n \geq 0$ there exists a cofiber sequence (``weight decomposition'')
    \[
        X_{< n+\ell} \to X \to X_{\geq n}
    \]
    with $X_{< n+\ell} \in \cC_{w< n+\ell}$ and $X_{\geq n} \in \cC_{\geq n}$.
    A functor of weakly preweighted categories is weight exact if it is right exact
    and preserves weight coconnectives.
\end{definition}

\begin{lemma}
    Let $\cC$ be prestable and $0 \leq \ell \leq \infty$. The following are equivalent:
    \begin{enumerate}
        \item $\cC$ is weakly $\ell$-pre-weighted.

        \item $(\SW(\cC),\cC, \SW(\cC)_{w \leq 0})$ is weakly $\ell$-weighted.

        \item There exists a weakly $\ell$-weighted category $\cD$
            and an equivalence $\cC \simeq \cD_{\geq 0}$.
    \end{enumerate}
\end{lemma}
\begin{proof}
    The equivalence of (1) and (2) follows from Observation \ref{obs:sw-weight}
    and shifting the weight decompositions.
    Clearly (2) implies (3), and (3) implies (1) because an equivalence $\cC \simeq \cD_{\geq 0}$
    restricts to equivalences $\cC_{\geq n} \simeq \cD_{\geq n}$ for all $n \geq 0$
    (since both sides are the essential images of $\Sigma^n$ on $\cC$ resp.~$\cD_{\geq 0}$).
\end{proof}

\begin{example}\label{ex:prestable-weight-bounded}
    Say that $\cC$ is weight bounded above if for each $X \in \cC$ there is some $n \geq 0$
    so that all maps from $X$ to any $Y \in \cC_{\geq n}$ vanish.
    Equivalently, if $\cC_{w <\infty} = \cC$.
    In this case $\cC$ is weakly $\infty$-preweighted with trivial weight decompositions
    by Remark \ref{rem:bounded-is-weakly-weighted}.

    Another example of similar fashion is the case where $\cC$ is a stable category
    considered as a prestable one. Then $\cC_{w \leq 0} = 0$ and $\cC_{w \geq \infty} = \cC$,
    so we also trivially have weight decompositions.
\end{example}

\begin{observation}
    If $\cC$ is a colimit complete pre-weighted category ($\ell = 0$),
    then by Proposition \ref{prop:weight-pdelta} we have $\cC \simeq \cP^{\Delta^\op}(\cC_{w = 0})$
    for $\cC_{w=0} = \cC_{w \leq 0} = \cC \cap \SW(\cC)_{w \leq 0}$.
\end{observation}

The following result is then essentially a consequence of Theorem \ref{thm:upw-comp}.
Given a weakly preweighted $\cC$, we endow $\Ind(\SW(\cC))$ with the weight structure
of Proposition \ref{prop:ind-weight}, which allows us to also talk about (weight) (co)connectivity
in $\Ind(\cC) \subseteq \Ind(\SW(\cC))$.

\begin{proposition}\label{prop:prestable-colim-comp}
    A weakly $\ell$-pre-weighted category $\cC$ admits a colimit completion.
    Concretely,
    \begin{align*}
        \lc{\cC} \coloneqq \ &\{X \in \Ind(\cC_{w< \infty}) \mid \text{$\forall n \geq 0$ there is $c \in \cC_{w<\infty}$ and a map
        $jc \to X$ with $n$-connective cofiber}\}\\
                =\ &\{\indcolim{n} c_n \in \Ind(\cC_{w<\infty}) \mid c_\bullet \colon \N \to \cC_{w < \infty} \text{ growing in connectivity}\}
    \end{align*}
    is the smallest subcategory of $\Ind(\cC_{w<\infty})$ containing
    $\cC_{w < \infty}$ and closed under sequential colimits growing in connectivity.
    Moreover:
    \begin{enumerate}
        \item $\lc{\cC}$ is the connective part of the left weight completion
            of $(\SW(\cC),\cC, \SW(\cC)_{w \leq 0})$.
            In fact, there is an equivalence of categories $\SW(\lc{\cC}) \simeq \lc{\SW(\cC)}$ under $\lc{\cC}$.\footnote{The left completion $\lc{\SW(\cC)}$ seems to generally not be left saturated, even if $\SW(\cC)$ is.
                Hence the equivalence under $\lc{\cC}$ need not be compatible with the coconnectives on both sides.}

        \item The Yoneda embedding $\cC_{w < \infty} \hookrightarrow \lc{\cC}$
            left Kan extends to a right exact functor $\eta \colon \cC \to \lc{\cC}$
            which exhibits $\lc{\cC}$ as the colimit completion of $\cC$.

        \item Let $\Cat_{\prest}$ be the category of prestable categories and right exact functors
            and $\Cat_{\prest}^{\uparrow}$ the full subcategory on the colimit complete ones.
            The above construction defines a partial left adjoint
            \[
                \lc{(-)} \colon \Cat_{\prest}^{\textsf{wpw}} \to \Cat_\prest^{\uparrow}
            \]
            to $\Cat_{\prest}^{\uparrow} \subseteq \Cat_{\prest}$,
            defined on the full subcategory $\Cat_{\prest}^{\textsf{wpw}}$
            of weakly pre-weighted categories.
    \end{enumerate}
\end{proposition}
\begin{proof}
    We apply Theorem \ref{thm:upw-comp} to the weakly $\ell$-weighted
    category $(\SW(\cC),\cC, \SW(\cC)_{w \leq 0})$
    and recognize $\lc{\cC} = \lc{\SW(\cC)}_{w \geq 0}$,
    which shows the other descriptions of $\lc{\cC}$.
    For (1), note that by Lemma \ref{lem:stable-conn-struc},
    we have an equivalence $\SW(\lc{\cC}) = \SW(\lc{\SW(\cC)}_{\geq 0}) \simeq \lc{\SW(\cC)}_{>-\infty}$
    under $\lc{\cC}$, but since $\lc{\SW(\cC)}$ is still right bounded, the latter agrees with $\lc{\SW(\cC)}$.
    Clearly this equivalence identifies the two embeddings of $\lc{\cC}$.

    Since (3) follows formally from (2), it remains to show the latter.
    As $\lc{\cC}$ is the connective part of a left complete weak weight structure,
    it is colimit complete.
    To show the remaining property of Definition \ref{def:colim-comp},
    let $\cD$ be a colimit complete prestable category.
    We know from Corollary \ref{cor:sw-complete} that $(\SW(\cD),\cD)$ is a colimit complete stable connectivity structure.
    Now there is a commutative square
    \[\begin{tikzcd}[ampersand replacement=\&]
        {\Fun^\Ex_{\geq 0}(\SW(\lc{\cC}), \SW(\cD))} \& {\Fun^\Ex_{\geq 0}(\SW(\cC), \SW(\cD))} \\
        {\Fun^\rex(\lc{\cC}, \cD)} \& {\Fun^\rex(\cC,\cD)}
        \arrow["{\SW(j)^*}", from=1-1, to=1-2]
        \arrow[from=1-1, to=2-1]
        \arrow[from=1-2, to=2-2]
        \arrow["{j^*}"', from=2-1, to=2-2]
    \end{tikzcd}\]
    where the vertical maps restrict along the canonical inclusions $\cC \subseteq \SW(\cC)$
    respectively $\lc{\cC} \subseteq \SW(\lc{\cC})$.
    We claim these vertical restrictions are equivalences.
    More generally, let $\cA$ be prestable and $(\cB,\cB_{\geq 0})$ be a stable connectivity structure.
    By \cite[Proposition C.1.1.7]{SAG} the restriction $\Fun^\Ex(\SW(\cA),\cB) \to \Fun^\rex(\cA,\cB)$
    is an equivalence.
    It is straightforward to verify that this restricts to an equivalence between the full subcategories
    $\Fun^\Ex_{\geq 0}(\SW(\cA),\cB)$ (of functors preserving connectives, i.e.~restricting to $\cA \to \cB_{\geq 0}$)
    and $\Fun^\rex(\cA,\cB_{\geq 0})$.
    In particular, we can pick $\cA = \cC$ and $\cB = (\SW(\cD),\cD)$, which yields the claim.

    So to show (2), it remains to see that also the top horizontal morphism is an equivalence.
    This follows by noting that both $\SW(\cC_{w < \infty}) \subseteq \SW(\cC)$
    and $\SW(\cC_{w < \infty}) \subseteq \SW(\lc{\cC})$ are left dense,
    so we can use Lemma \ref{lem:weight-comp-lkan}
    and a 2-out-of-3 argument as in the proof of Corollary \ref{cor:left-comp-via-embed}.
\end{proof}

By Example \ref{ex:prestable-weight-bounded} the above Proposition gives us a functorial
colimit completion for weight bounded above prestable categories.
In analogy with the weighted case,
one might now be tempted to simply define $\lc{\cC}$ as $\lc{(\cC_{w < \infty})}$
for a general prestable category $\cC$.
Via left Kan extension along the inclusion $\cC_{w<\infty} \subseteq \cC$
the Yoneda embedding then induces a map $\cC \to \Ind(\cC_{w<\infty})$.
However, there is now in general no reason for this to factor through $\lc{(\cC_{w<\infty})}$.
In fact, the natural condition to guarantee this is that every object in $\cC$
admits $\cC_{w < \infty}$-weight complexes, i.e.~that $\SW(\cC_{w < \infty}) \subseteq \SW(\cC)$
is left dense. But this is precisely the condition that $\cC$ is weakly pre-weighted.
The following Corollary makes a version of this statement precise.

\begin{corollary}\label{cor:conn-complete-needs-w}
    For a prestable category $\cC$, the following are equivalent:
    \begin{enumerate}
        \item $\cC$ is colimit complete,
            and the colimit functor $\lc{(\cC_{w<\infty})} \to \cC$
            (left Kan extended from the inclusion $\cC_{w<\infty} \subseteq \cC$) is an equivalence.

        \item $\cC$ is weakly pre-weighted,
            and $\SW(\cC)$ is a left complete weakly weighted category.
    \end{enumerate}
    In this case, the inverse of the colimit
    functor is the restricted weight complex functor
    $w \colon \cC \to \lc{(\cC_{w<\infty})}$.
\end{corollary}
\begin{proof}
    Let us first argue that if $\cC$ is colimit complete,
    then the colimit functor exists.

    Since $\cC_{w < \infty} = \cC \cap \SW(\cC)_{w < \infty}$
    we have $\SW(\cC_{w < \infty}) = \SW(\cC)_{w < \infty}$ as full subcategories of $\SW(\cC)$.
    This subcategory is trivially weakly weighted,
    and the inclusion $\SW(\cC_{w < \infty}) \subseteq \lc{\SW(\cC_{w < \infty})}$ is left dense
    by Theorem \ref{thm:upw-comp}.
    Now if $\cC$ is colimit complete, then so is $(\SW(\cC),\cC)$,
    and hence the left Kan extension of $\SW(\cC_{w < \infty}) \to \SW(\cC)$
    along $\SW(\cC_{w < \infty}) \subseteq \lc{\SW(\cC_{w < \infty})}$
    exists by Lemma \ref{lem:weight-comp-lkan},
    which also shows that the resulting functor restricts to
    $\lc{(\cC_{w < \infty})} = \lc{\SW(\cC_{w < \infty})}_{\geq 0} \to \SW(\cC)_{\geq 0} = \cC$.
    It is straightforward to check that this is still left Kan extended from $\cC_{w < \infty}$.
    Hence the colimit functor exists as claimed.

    Now the implication (1) $\Rightarrow$ (2)
    is clear from the assumption that the colimit functor is an equivalence,
    and that it identifies the two evident inclusions of $\cC_{w < \infty}$.

    Conversely, given (2) we know that $\cC$ is colimit complete,
    and as we saw above this yields a colimit functor
    $f \colon \lc{\SW(\cC_{w < \infty})} \to \SW(\cC)$ which is left Kan extended
    from its restriction to $\SW(\cC_{w < \infty}) = \SW(\cC)_{w < \infty}$.
    On the other hand, Theorem \ref{thm:upw-comp}
    also yields a weight complex functor $w \colon \SW(\cC) \to \lc{\SW(\cC_{w < \infty})}$
    which is also left Kan extended from the inclusion $\SW(\cC_{w < \infty}) \subseteq \lc{\SW(\cC_{w < \infty})}$.
    Since both $f$ and $w$ are left Kan extended from $\SW(\cC_{w < \infty})$,
    it follows that they are mutually inverse,
    and by Lemma \ref{lem:weight-comp-lkan} both functors preserve connectives.
    In particular, the restricted colimit functor $\lc{(\cC_{w < \infty})} \to \cC$ is an equivalence,
    showing (1).
\end{proof}

\subsection{Spectrum objects}\label{sec:prestable-spectrum}

Given a prestable category $\cC$, there are generally many stable categories $\cD$ and fully faithful
embeddings $\cC \hookrightarrow \cD$ for which $\cC$ defines a stable connectivity structure on $\cD$.
The minimal choice $\SW(\cC)$ was investigated above,
but as Lurie details in \cite[Remark C.1.2.10]{SAG},
under the additional assumption that $\cC$ admits finite limits,
there is also a ``maximal'', or better, ``terminal'' choice for such a $\cD$;
this is the category
\[
    \Sp(\cC) = \lim( \cdots \xto{\Omega}  \cC \xto{\Omega}  \cC \xto{\Omega} \cC)
\]
of spectrum objects in $\cC$, which we can also identify as the right completion of the $t$-structure on $\SW(\cC)$.

We shall now generalize the construction of spectrum objects to prestable categories which do not necessarily admit finite limits.
More precisely, we give a construction of $\Sp(\cC)$ for an arbitrary prestable category $\cC$
which agrees with the old construction whenever $\cC$ admits finite limits.
To state the definition, we let $\SW(\cC)_{t \leq 0} = \rperp{(\cC_{\geq 1})} \subseteq \SW(\cC)$
and $\cC_{\tau \leq n} \coloneqq \SW(\cC)_{t[0,n]} = \cC \cap \SW(\cC)_{t \leq n}$ for $n \geq 0$.

\begin{observation}
    Since $(\SW(\cC),\cC)$ is a bounded below stable connectivity structure
    and $\cC \subseteq \cC_{\geq 0}$ is closed under retracts,
    it follows from Remark \ref{rem:bounded-is-weak-t}
    that $(\SW(\cC),\cC,\SW(\cC)_{t \leq 0})$ is always a right bounded weak $t$-structure
    with trivial $t$-decompositions.
\end{observation}

\begin{lemma}\label{lem:truncated}
    For $X \in \cC$ and $n \geq 0$, the following are equivalent:
    \begin{enumerate}
        \item $\Omega^n X \in \SW(\cC)_{t \leq 0}$.
        \item $X \in \cC_{\tau \leq n}$.
        \item $X$ is $n$-truncated in $\cC$.
    \end{enumerate}
\end{lemma}
\begin{proof}
    The equivalence of the first two is clear by shifting and definition.
    For $X \in \cC$, we have that $X \in \cC_{t \leq n}$ if and only if for all $Y \in \cC$
    we have that $\hom(Y,X) \in \Sp_{t \leq n}$.
    But this is precisely the case if $\map(Y,X) = \tau_{\geq 0}\hom(Y,X) \in \tau_{\leq n}\An$,
    which is the definition of $X$ being $n$-truncated.
\end{proof}

\begin{definition}
    A functor of prestable categories $\cC \to \cD$ is said to have truncation amplitude $\leq n$
    if it sends $\cC_{\tau \leq k}$ to $\cD_{\tau \leq k + n}$ for $k \gg 0$.
    It has bounded truncation amplitude if there is such an $n \geq 0$.
\end{definition}

The following is then an immediate consequence of Lemma \ref{lem:truncated}.

\begin{lemma}
    $f \colon \cC \to \cD$ has truncation amplitude $\leq n$
    if and only if $\SW(f)$ has $t$-amplitude $[0,n]$.
    In particular, if $f$ is even an exact functor of prestable categories admitting finite limits,
    then $\SW(f)$ is $t$-exact.
\end{lemma}

\begin{definition}\label{def:prestable-spectrum}
    Let $\cC$ be a prestable category.
    We define $\Sp(\cC) \coloneqq \trc{\SW(\cC)}$
    as the right completion of the above weak $t$-structure on $\SW(\cC)$.
\end{definition}

In particular, all the results of Theorem \ref{thm:t-dw-comp} apply.
We will specialize some into the following corollary.
As usual, we equip $\Ind(\SW(\cC))$ with the induced $t$-structure
of Proposition \ref{prop:ind-t}.

\begin{theorem}\label{thm:prestable-spectrum}
    For a prestable $\cC$,
    there is a $t$-exact inclusion $\Sp(\cC) \subseteq \Ind(\SW(\cC))$
    which identifies
    \begin{align*}
        \Sp(\cC) =\ &\{X \in \Ind(\SW(\cC)) \mid \text{$\forall n \geq 0$ there is $c \in \cC$ and a map
        $c \to X$ with $n$-coconnective cofiber}\}\\
        =\ &\{X \in \Ind(\SW(\cC)) \mid X \text{ admits a $\SW(\cC)$-Whitehead tower}\}\\
        =\ &\{\indcolim{n} c_n \in \Ind(\SW(\cC)) \mid c_\bullet \colon \N \to \SW(\cC) \text{ growing in coconnectivity}\}
    \end{align*}
    as the largest stable subcategory of $\Ind(\SW(\cC))$ in which $\SW(\cC)$ is right dense,
    and also the smallest stable subcategory of $\Ind(\SW(\cC))$ containing
    $\cC$ and closed under sequential colimits growing in coconnectivity.
    Moreover:
    \begin{enumerate}
        \item There is a fully faithful, $t$-exact and right dense inclusion $j \colon \SW(\cC) \subseteq \Sp(\cC)$.
            Moreover, the inclusion $\cC \subseteq \Sp(\cC)_{t \geq 0}$ is dense.

        \item The stable connectivity structure $(\Sp(\cC),\cC)$ is colimit complete\footnote{Recall that left completeness of the weak $t$-structure corresponds to \emph{limit} completeness of the underlying connectivity
            structure.}
            if and only if $\cC$ is colimit complete,
            in which case $\cC = \Sp(\cC)_{t \geq 0}$ by idempotent completeness.

        \item For every right complete weak $t$-category $\cD$
            and $0 \leq b < \infty$, restriction induces an equivalence
            \[
                \Fun^\Ex_{[0,b]}(\Sp(\cC), \cD)
                \xto{\simeq} \Fun^\rex_{\tau \leq b}(\cC,\cD_{\geq 0})
            \]
            where $\Fun^\rex_{\tau \leq b} \subseteq \Fun^\rex$
            is the full subcategory on functors of truncation amplitude $\leq b$.

        \item Let $\Cat_\prest^{\tau < \infty}$ denote the category of prestable categories
            and right exact functors of bounded truncation amplitude.
            The above construction defines a functor
            \[
                \Sp(-) \colon \Cat_\prest^{\tau < \infty} \to \wTCatr
            \]
            which fits into a commutative diagram
            \[\begin{tikzcd}[ampersand replacement=\&]
                {\Cat_\prest^{\tau <\infty}} \&\& \wTCatr \\
                {\Cat_\prest^\lex} \& {\Cat^\lex} \& {\Cat^\Ex}
                \arrow["{\Sp(-)}", from=1-1, to=1-3]
                \arrow[from=1-3, to=2-3]
                \arrow[hook', from=2-1, to=1-1]
                \arrow[from=2-1, to=2-2]
                \arrow["{\Sp(-)}", from=2-2, to=2-3]
            \end{tikzcd}\]
            where the unlabeled arrows are the evident forgetful functors / inclusions,
            and the bottom $\Sp(-)$ is the usual spectrum objects construction.
            The analogous statements hold if we restrict to truncation amplitude $\tau \leq 0$,
            which lets us land in $\wTCatr_{=0}$ ($t$-amplitude $[0,0]$).
    \end{enumerate}
\end{theorem}
\begin{proof}
    The equivalent descriptions and results of (1) are immediate
    from the definition of $\Sp(\cC)$ as right $t$-completion
    and Theorem \ref{thm:t-dw-comp}.

    For (2), note first that if $\cC$ is colimit complete,
    then it is idempotent complete by Corollary \ref{cor:prest-idem} so by (1) we have $\cC = \Sp(\cC)_{\geq 0}$.
    Now by Lemma \ref{lem:conn-comp-prestable}
    and Corollary \ref{cor:sw-complete}
    it remains to show that if $\cC$ is colimit complete,
    then the inclusion $\SW(\cC) \subseteq \Sp(\cC)$ preserves sequential colimits growing in connectivity.
    Since $\SW(\cC)$ is idempotent complete by Corollary \ref{cor:grow-conn-idem},
    we have $\SW(\cC) = \Sp(\cC)_{>-\infty}$ by the right density in (1).
    Now suppose that $X_\bullet \colon \N \to \SW(\cC)$ is a sequential diagram growing in connectivity and $X = \colim_n X_n \in \SW(\cC)$.
    Since any $Y \in \Sp(\cC)$ admits a cofiber sequence $C \to Y \to Y_{\leq 0}$
    where $C \in \SW(\cC)$ and $Y_{\leq 0} \in \Sp(\cC)_{\leq 0}$,
    it suffices to show that the canonical map $\hom(X,Y_{\leq 0}) \to \lim_n \hom(X_n, Y_{\leq 0})$
    is an equivalence for all $Y_{\leq 0} \in \Sp(\cC)_{t \leq 0}$.
    Analogously to Example \ref{ex:corep-bdd-t}
    we see that $\hom(-,Y_{\leq 0})$ sends connective objects to coconnective spectra,
    so by right completeness of $\Sp$ sends the sequential cone growing in connectivity
    $X_\bullet \Rightarrow \const X$ to a limit cone.
    This concludes the proof of (2).

    For (3), we already know from Theorem \ref{thm:t-dw-comp} that restriction along $\SW(\cC) \subseteq \Sp(\cC)$
    induces an equivalence
    \[
        \Fun^\Ex_{[0,b]}(\Sp(\cC),\cD) \xto{\simeq} \Fun^\Ex_{[0,b]}(\SW(\cC), \cD).
    \]
    At the end of the proof of Proposition \ref{prop:prestable-colim-comp}, we saw
    that restricting along $\cC \subseteq \SW(\cC)$ gives an equivalence
    $\Fun^\Ex_{[0,\infty]}(\SW(\cC), \cD) \simeq \Fun^\rex(\cC,\cD_{\geq 0})$.
    It is then straightforward to verify that this further restricts to an equivalence on full subcategories
    of $t$-amplitude $[0,b]$ respectively truncation amplitude $\tau \leq b$:
    \[
        \Fun^\Ex_{[0,b]}(\SW(\cC),\cD) \xto{\simeq} \Fun^\rex_{\tau \leq b}(\cC,\cD_{\geq 0}).
    \]
    By composing, this proves (3).

    Finally, to see (4), we recognize the claimed commutative diagram as the outer edge
    of the following commutative diagram
    \[\begin{tikzcd}[ampersand replacement=\&]
        {\Cat_\prest^{\tau<\infty}} \& \wTCat \& \wTCatr \& {\Cat^\Ex} \\
        {\Cat_\prest^{\lex}} \& \TCat \& \TCatr \& {\Cat^\lex}
        \arrow["\SW"', from=1-1, to=1-2]
        \arrow["{\Sp(-)}", curve={height=-12pt}, from=1-1, to=1-3]
        \arrow["{(-)^{\downarrow t}}"', from=1-2, to=1-3]
        \arrow[from=1-3, to=1-4]
        \arrow[hook', from=2-1, to=1-1]
        \arrow["\SW", from=2-1, to=2-2]
        \arrow[curve={height=18pt}, from=2-1, to=2-4]
        \arrow[hook', from=2-2, to=1-2]
        \arrow["{(-)^{\downarrow t}}", from=2-2, to=2-3]
        \arrow[hook', from=2-3, to=1-3]
        \arrow[from=2-3, to=1-4]
        \arrow["{\Sp(-)}"', from=2-4, to=1-4]
    \end{tikzcd}\]
    where all unlabeled functor are forgetful functors / inclusions.
    Note that $\Cat^\lex_\prest$ is indeed a full subcategory of $\Cat^{\tau < \infty}_{\prest}$
    since equivalences automatically preserve finite limits,
    and exact functors of prestable categories admitting finite limits
    have truncation amplitude $\tau \leq 0$.
    Similarly, since $t$-categories are saturated, the analogue of Lemma \ref{lem:saturated}
    shows that $\TCat$ is actually a full subcategory of $\wTCat$
    (where in the latter we consider exact functors of bounded above $t$-amplitude),
    and analogously for $\TCatr \subseteq \wTCatr$.
    Now commutativity of the left and middle squares as well as the right upper triangle is clear.
    The first two top horizontal functors are our definition of $\Sp(-)$ for prestable categories.

    Now note that by Corollary \ref{cor:t-comp-compare} the functor $\trc{(-)} \colon \TCat \to \TCatr$
    agrees with the usual definition of right $t$-completion of $t$-categories,
    so it follows from \cite[Remark C.1.2.10]{SAG} that $\trc{\SW(-)} \simeq \Sp(-)$
    as functors $\Cat^\lex_\prest \to \Cat^\Ex$, i.e.~that the bottom face of the diagram commutes as well.
\end{proof}

\begin{remark}
    More generally, one could define $\Sp(\cC)$ for any category $\cC$ with finite colimits,
    without assuming $\cC$ is prestable.
    To this end, one first passes to the \emph{prestable envelope} of $\cC$,
    defined as the full subcategory of $\SW(\cC)$ generated under finite colimits and extensions by the essential image of the functor $\cC \to \SW(\cC)$ (which need not be fully faithful).
    One then applies the above construction to this prestable envelope.
\end{remark}

\begin{example}
    It can happen that $\cC$ does not have any truncated objects.
    For example, we can take the category $\cC = \Sp^\omega_{\geq 0}$ of connective, finite spectra.
    Then $\SW(\cC) = \Sp^\omega$, and hence $\cC_{t \leq n} = 0$ for all $n$,
    since non-zero finite spectra are never bounded above in the standard $t$-structure on spectra.\footnote{If $X$ is a finite $t$-bounded above spectrum, then so is $X/p$ for any prime $p$.
        It follows that $X/p$ lies in the localizing subcategory generated by $\F_p$.
        By the chromatic convergence theorem, we have $X/p = \lim_n L_n X/p$.
        But $L_n \F_p = 0$ for all $n$, which gives $X/p = 0$ for all $p$,
        i.e.~that $X$ is rational. But finite spectra have finitely generated homotopy groups,
        so this yields $X =0$.}
    It follows that every sequential diagram $c_\bullet \colon \N \to \SW(\cC) = \Sp^\omega$
    growing in $t$-coconnectivity must be eventually constant, and hence $\Sp(\cC) = \SW(\cC)$.
    Even more extreme, if $\cC$ is already stable,
    then it also doesn't have any truncated objects and we find that $\Sp(\cC) = \SW(\cC) = \cC$.
\end{example}

\begin{example}
    Let $\cC = \Sp^{\fg}_{\geq 0}$ be the category of connective spectra with finitely generated homotopy groups.
    This is the connective part of the standard $t$-structure on spectra restricted to the full subcategory
    $\Sp^\fg$ of spectra with finitely generated homotopy groups.
    In particular, $\cC$ is a prestable category admitting finite limits,
    with $\SW(\cC) = \Sp^{\fg}_{>-\infty}$ and $\Sp(\cC) = \Sp^\fg$.
\end{example}

\begin{warning}
    Even if $\cC$ is pre-weighted,
    so that $(\SW(\cC),\cC)$ is part of a weight structure,
    the connectivity structure $(\Sp(\cC),\cC)$ (or $(\Sp(\cC),\Sp(\cC)_{t \geq 0})$)
    can generally fail to be part of a weak weight structure.
    A concrete counterexample is given by $\cC = \Sp^\fg_{\geq 0}$, see Example \ref{ex:no-weight-on-t-comp}.
\end{warning}

\section{Adjacency of weak weight and weak $t$-structures}\label{sec:adj}

In practice, weight structures often come together with an adjacent $t$-structure,
i.e.~one having the same category of connective objects.
In this section we study the analogous interaction between weak weight- and weak $t$-structures, showing in particular that the existence of an adjacent weak $t$-structure is preserved under left weight completion,
but that conversely weight structures are in general not preserved by right $t$-completion.
We introduce the mixed weight- and $t$-completion process $\cA \mapsto \K(\cA)$ for additive categories $\cA$,
and deduce a reconstruction result (Theorem \ref{thm_recovery} / Theorem \ref{thm:k}) from our previous results.

\begin{definition}\label{def:adj}
    Let $\cC$ be a weakly weighted category. We say that a weak $t$-structure $(\cC_{t \geq 0},\cC_{t \leq 0})$ on $\cC$ is adjacent to the weak weight structure if
    \begin{enumerate}
        \item $\cC_{w \geq 0} = \cC_{t \geq 0}$, in which case we simply denote this by $\cC_{\geq 0}$, and
        \item the weak $t$-structure is left saturated: $\cC_{t \leq 0} = \rperp{(\cC_{\geq 1})}$.
    \end{enumerate}
    Dually, a weak $t$-structure on $\cC$ is co-adjacent if $\cC_{w \leq 0} = \cC_{t \leq 0}$ and $\cC_{t \geq 0} = \lperp{(\cC_{\leq -1})}$.
\end{definition}

\begin{remark}
    If we are considering (non-weak) $t$-structures, condition (2) follows from (1).
\end{remark}

\begin{lemma}\label{lem:adj-t-on-upw-comp}
    Let $\cC$ be a weakly weighted category with adjacent weak $t$-structure
    of defect $m$.
    Then also the weakly weighted category $\wlc{\cC}$
    admits an adjacent weak $t$-structure of defect $m$
    and the map $\eta \colon \cC \to \wlc{\cC}$ is both weight and $t$-exact.
\end{lemma}
\begin{proof}
    Let $0 \leq \ell \leq \infty$ be so that $\cC$ and hence $\wlc{\cC}$ are weakly $\ell$-weighted.
    Since $\cC \to \wlc{\cC}$ is weight exact and the (putative) weak $t$-structures
    are adjacent, for $t$-exactness it suffices to check that $t$-coconnective
    objects are preserved.
    So given $Y \in \cC_{t \leq 0}$ and $X \in (\wlc{\cC})_{w \geq 0}$,
    we need to show that $\hom_{\wlc{\cC}}(X,\eta Y) \in \Sp_{t \leq 0}$.
    Using the explicit description of $(\wlc{\cC})_{w \geq 0}$ from Theorem \ref{thm:upw-comp}
    and the fact that $\Sp_{t \leq 0}$ is closed under limits,
    we reduce to the case that $X \in \cC_{w[0,n]}$ for some $n < \infty$.
    But then Theorem \ref{thm:upw-comp}(2b) yields that
    \[
        \hom_{\cC}(X,Y) \simeq \hom_{\wlc{\cC}}(\eta X, \eta Y),
    \]
    which is $t$-coconnective since $X \in \cC_{w \geq 0} = \cC_{t \geq 0}$ and $Y \in \cC_{t \leq 0}$.

    It remains to construct a $t$-decomposition for any $X \in \wlc{\cC}$.
    By left density of $\cC_{<\infty}$ in $\wlc{\cC}$
    we find a cofiber sequence $c \to X \to X_{w \geq 2}$
    with $c \in \cC_{<\infty}$ and $X_{w \geq 2} \in (\wlc{\cC})_{w \geq 2}$.
    Picking a $t$-decomposition $c_{w > -m} \to c \to c_{t \leq 0}$,
    we see that the composite $\Omega X_{w \geq 2} \to c_{t \leq 0}$
    vanishes, and hence we can form the diagram of cofiber sequences
    \[\begin{tikzcd}
        {\Omega X_{w \geq 2}} & {c_{w > -m}} & F \\
        {\Omega X_{w \geq 2}} & c & X \\
        0 & {c_{t \leq 0}} & {\eta c_{t \leq 0}}
        \arrow[dashed, from=1-1, to=1-2]
        \arrow[Rightarrow, no head, from=1-1, to=2-1]
        \arrow[from=1-2, to=1-3]
        \arrow[from=1-2, to=2-2]
        \arrow[from=1-3, to=2-3]
        \arrow[from=2-1, to=2-2]
        \arrow[from=2-1, to=3-1]
        \arrow[from=2-2, to=2-3]
        \arrow[from=2-2, to=3-2]
        \arrow[from=2-3, to=3-3]
        \arrow[from=3-1, to=3-2]
        \arrow[Rightarrow, no head, from=3-2, to=3-3]
    \end{tikzcd}\]
    where $F \in (\wlc{\cC})_{w > -m}$,
    and so $F \to X \to \eta c_{t \leq 0}$ is a $t$-decomposition of defect $m$ for $X$.
\end{proof}

Unfortunately, the dual assertion, that the weak weight structure is preserved by right $t$-completion, fails completely, as the following example shows.

\begin{example}\label{ex:no-weight-on-t-comp}
    Let $\Sp^{\fg} \subseteq \Sp$ be the category of spectra with finitely generated homotopy groups,
    and $\Sp^\ft = \Sp^{\fg}_{>-\infty}$ the category of finite type spectra,
    i.e.~bounded below spectra with finitely generated homotopy groups.
    We have previously seen that the standard weight structure on $\Sp$ restricts to $\Sp^{\ft}$
    (and exhibits the latter as the left weight completion of $\Sp^\omega$).
    However, we claim that it does not restrict to a weight structure on $\trc{(\Sp^\ft)} = \Sp^\fg$,
    or even to a weak one.
    Indeed, note that $\KU \in \Sp^\fg$. Suppose we have a weight decomposition
    $A_{w \leq n} \to \KU \to B_{w \geq 0}$ for some $0 \leq n < \infty$.
    Tensoring with $\Z$, we see that $\Z \tensor A_{w \leq n}$ is $n$-coconnective,
    whereas $\Z \tensor B_{w \geq 0} \in \Sp^\fg_{\geq 0}$ is connective and still has finitely generated
    homotopy groups. However, this means $\pi_{2n+2}(\Z \tensor B_{w \geq 0}) = \pi_{2n+2}(\Z \tensor \KU) = \Q$
    is finitely generated, which is absurd.
    It seems that the next best category onto which the weight and $t$-structures restrict
    is $\Sp^{\omega_1}$, the category of spectra with countable homotopy groups.
\end{example}

Nevertheless, we can note that in the above example,
the $t$-structure on $\Sp^\fg$ is so that the bounded below part
$\Sp^\fg_{>-\infty} = \Sp^\ft$ admits a left complete weight structure.
We will see below that this is generally enough to determine the entire category purely from the weight heart
of the bounded below part, see Theorem \ref{thm:k}.
The dual notion of a weakly weighted category
where $\cC_{w<\infty}$ admits an adjacent weak $t$-structure
in fact already guarantees the existence of the adjacent weak $t$-structure on the entire category.

\begin{lemma}\label{lem:adj-t-on-wt-cocon}
    Let $\cC$ be weakly weighted and let $\cC_{t \leq 0} \coloneqq \rperp{(\cC_{\geq 1})}$.
    If there exists some $0 \leq m \leq \infty$ such that every $X \in \cC_{w < \infty}$
    admits an adjacent $t$-decomposition $X_{w > -m} \to X \to X_{t \leq 0}$
    of defect $m$, then every $X \in \cC$ does, so $\cC$ has an adjacent weak $t$-structure of defect $m$.
\end{lemma}
\begin{proof}
    For $X \in \cC$ there is some $0 \leq n < \infty$, a weight decomposition
    $X_{w \leq n} \to X \to X_{w \geq 1}$ and a $t$-decomposition
    $(X_{w \leq n})_{w > -m} \to X_{w \leq n} \to (X_{w \leq n})_{t \leq 0}$.
    Since the map
    $\Omega X_{w \geq 1} \to (X_{w \leq n})_{t \leq 0}$
    vanishes, we can build the diagram of cofiber sequences
    \[\begin{tikzcd}
        {\Omega X_{w \geq 1}} & {(X_{w \leq n})_{> -m}} & {X_{>-m}} \\
        {\Omega X_{w \geq 1}} & {X_{w \leq n}} & X \\
        0 & {(X_{w \leq n})_{t \leq 0}} & {(X_{w \leq n})_{t \leq 0}}
        \arrow[dashed, from=1-1, to=1-2]
        \arrow[Rightarrow, no head, from=1-1, to=2-1]
        \arrow[from=1-2, to=1-3]
        \arrow[from=1-2, to=2-2]
        \arrow[from=1-3, to=2-3]
        \arrow[from=2-1, to=2-2]
        \arrow[from=2-1, to=3-1]
        \arrow[from=2-2, to=2-3]
        \arrow[from=2-2, to=3-2]
        \arrow[from=2-3, to=3-3]
        \arrow[from=3-1, to=3-2]
        \arrow[Rightarrow, no head, from=3-2, to=3-3]
    \end{tikzcd}\]
    The right vertical cofiber sequence yields the desired $t$-decomposition of defect $m$ for $X$.
\end{proof}

In the remainder of this section, we investigate stable categories which admit adjacent
weight and weak $t$-structures and are colimit complete for both,
i.e.~left complete for the weight structure and right complete for the (weak) $t$-structure.
One of the canonical examples is the unbounded derived category of a Grothendieck abelian category
with enough projectives, but also examples we have already seen above, such as $\LMod(R)$ for a connective $\E_1$-ring $R$.

\begin{definition}\label{def:k}
    Let $\cA$ be an additive category.
    Since $\wlc{\Stab(\cA)}$ is right bounded, it trivially admits a weak $t$-structure
    by Remark \ref{rem:bounded-is-weak-t}.
    This allows us to define
    \[
        \K(\cA) \coloneqq \wtc{\Stab(\cA)} \coloneqq \trc{(\wlc{\Stab(\cA)})} \simeq \Sp(\cP^{\Delta^\op}(\cA))
    \]
    For the last equivalence,
    we used that $\wlc{\Stab(\cA)} \simeq \SW(\cP^{\Delta^\op}(\cA))$ and the definition
    of $\Sp(-)$ from Definition \ref{def:prestable-spectrum}.
    We also define $\K^\op(\cA) \coloneqq \K(\cA^\op)^\op$.
\end{definition}

\begin{observation}\label{obs:k-basic}
    Note that:
    \begin{enumerate}
        \item $\K(\cA)$ admits a right complete weak $t$-structure,
            and $\K(\cA)_{> -\infty} \simeq \wlc{\Stab(\cA)} \simeq \SW(\cP^{\Delta^\op}(\cA))$
            (cf.~Proposition \ref{prop:weight-pdelta}) admits an adjacent (in that weight- and $t$-connectives agree)
            left complete weight structure with weight heart $\cA^\idem$.

        \item If the left complete weighted category $\K(\cA)_{>-\infty}$
            happens to admit an adjacent $t$-structure, then $\K(\cA)$ agrees with the usual
            right completion of said $t$-structure.
            In fact, in this case it follows that $\cP^{\Delta^\op}(\cA)$ is a prestable category
            admitting finite limits and $\K(\cA) \simeq \Sp(\cP^{\Delta^\op}(\cA))$
            (cf.~\cite[Proposition C.1.2.9, Remark C.1.2.10]{SAG}).

        \item Using the definition of Spectrum objects of a prestable category
            from section \ref{sec:prestable-spectrum}, we can even write
            $\K(\cA) = \Sp(\cP^{\Delta^\op}(\cA))$ in full generality.
    \end{enumerate}
\end{observation}

\begin{remark}\label{rem:k-vs-kw}
    Recall the notation $\Kw(\cA)$ from Definition \ref{def:kw}
    and note that
    \[
        \Kw(\cA)_{w > -\infty} \simeq \wlc{\Stab(\cA)} \simeq \K(\cA)_{t>-\infty}
    \]
    or equivalently $\Kw(\cA)_{w \geq 0} \simeq \cP^{\Delta^\op}(\cA) \simeq \K(\cA)_{t \geq 0}$.
    The definitions of both $\Kw$ and $\K$ then let us write
    \[
        \Kw(\cA) \simeq \wrc{(\K(\cA)_{t>-\infty})}
        \quad\text{and}\quad
        \K(\cA) \simeq \trc{(\Kw(\cA)_{w > -\infty})} \simeq \Sp(\Kw(\cA)_{w \geq 0}).
    \]
    In this sense, we can think of $\K(\cA)$ as the right $t$-completion of $\Kw(\cA)$,
    and dually of $\Kw(\cA)$ as the right weight completion of $\K(\cA)$.
    Let us note however that for reasons similar
    to the discussion around Corollary \ref{cor:conn-complete-needs-w}
    there will generally not exist functors $\K(\cA) \to \Kw(\cA)$
    or $\Kw(\cA) \to \K(\cA)$ that are completions.
    For one, $\K(\cA)$ need not even admit a weak weight structure
    as the example of $\cA = \{\bigoplus_\fin \S\}$ from Example \ref{ex:no-weight-on-t-comp} shows
    (note $\K(\cA) = \Sp^\fg$).
    Similarly, $\Kw(\cA)$ need not admit an adjacent weak $t$-structure.
    Nevertheless, we will see in Section \ref{sec:anderson}
    that it can happen that we do have functors in both directions.
\end{remark}

\begin{theorem}\label{thm:k}
    Let $(\cC,\cC_{\geq 0})$ be a stable connectivity structure
    and define $\cC_{w \leq 0} = \lperp{(\cC_{\geq 1})}$
    as well as $\cC_{t \leq 0} = \rperp{(\cC_{\geq 1})}$.
    There is an equivalence $\cC \simeq \K(\cC_{w=0})$
    restricting to $\cC_{\geq 0} \simeq \K(\cC_{w=0})_{\geq 0}$
    if and only if
    \begin{enumerate}
        \item $(\cC,\cC_{\geq 0}, \cC_{t \leq 0})$ is a right complete
            weak $t$-structure,
        \item $(\cC_{>-\infty}, \cC_{\geq 0}, \cC_{>-\infty} \cap \cC_{w \leq 0})$
            is a left complete weight structure.
    \end{enumerate}
\end{theorem}
\begin{proof}
    Given an equivalence $f \colon \cC \simeq \K(\cC_{w=0})$
    restricting to an equivalence $\cC_{\geq 0} \simeq \K(\cC_{w=0})_{\geq 0}$,
    we can simply transfer all the structure from $\K(\cC_{w=0})$
    through this equivalence.

    Conversely, suppose that (1) and (2) are satisfied.
    By Theorem \ref{thm:vova} the evident map $\cC_{w=0} \to \cC$
    uniquely extends to a weight exact map $\Stab(\cC_{w=0}) \to \cC$
    which is an equivalence onto $\cC_{wb}$.
    By Theorem \ref{thm:upw-comp},
    this uniquely extends to a weight exact equivalence
    $\wlc{\Stab(\cC_{w=0})} \simeq \wlc{\cC_{wb}} \simeq \cC_{>-\infty}$.
    Finally, viewing $\wlc{\Stab(\cC_{w=0})}$ as
    a trivial weak $t$-category
    via Remark \ref{rem:bounded-is-weak-t} then lets us use Theorem \ref{thm:t-dw-comp}
    to obtain a $t$-exact equivalence $\K(\cC_{w=0}) = \trc{(\wlc{\Stab(\cC_{w=0})})} \simeq \trc{(\cC_{>-\infty})} \simeq \cC$, as desired.
\end{proof}

\begin{example}\label{example_modules}
    By Proposition \ref{prop:lmod-weight}, for a connective $\E_1$-ring $R$
    there is a $t$-exact equivalence
    \[
        \K(\{R\}^\oplus) \xto{\simeq} \LMod(R)
    \]
    which restricts to a weight exact equivalence $\wlc{\Stab(\{R\}^\oplus)} \simeq \SW(\cP^{\Delta^\op}(\{R\}^\oplus)) \simeq \LMod(R)_{>-\infty}$.
\end{example}

\begin{example}
    There is a $t$-exact equivalence $\Sp^\fg \simeq \K(\{\bigoplus_\fin \S\})$
    and similarly $\cD(\Z)^\fg \simeq \K(\{\bigoplus_\fin \Z\})$.
\end{example}

\begin{example}
    We will see in Proposition \ref{prop:anderson-weight} that there is also a $t$-exact equivalence
    \[
        \K^\op(\{I_\Z\}^\oplus) \simeq \Sp
    \]
    restricting to a weight exact equivalence $\wrc{\Stab(\{I_\Z\}^\oplus)} \simeq \Sp_{t < \infty}$
    where $\Sp_{t < \infty} = \Sp_{w' < \infty}$ is equipped with the Anderson weight structure
    from Proposition \ref{prop:anderson-weight}.
\end{example}

Recall from Corollary \ref{cor:d-} that if $\cA$ is an abelian category with enough projectives,
we can understand the bounded below derived category of $\cA$
as the left weight completion of $\Stab(\cA_\proj)$.
Moreover, $\cD^-(\cA)$ admits a left complete and right bounded $t$-structure.
Now if $\cA$ is instead a Grothendieck abelian category,
then there is a model structure on $\Ch(\cA)$ constructed in
\cite[Proposition 1.3.5.3]{HA} with underlying equivalences the quasi-isomorphisms
and whose underlying $\infty$-category is the unbounded derived category $\cD(\cA)$.
The fibrant replacement functor defines a left Bousfield localization $\Ch_\infty(\cA) \to \cD(\cA)$
which precisely inverts the quasi-isomorphisms.
By \cite[Proposition 1.3.5.21]{HA} the derived category $\cD(\cA)$ admits a right complete $t$-structure,
which is generally not left complete,
see e.g.~\cite{Neeman-Derived}.
However, if $\cA$ furthermore admits enough projectives, we have the following result.

\begin{proposition}\label{prop:da-as-k}
    Let $\cA$ be a Grothendieck abelian category with enough projectives.
    Then we have a $t$-exact equivalence $\cD^-(\cA) \simeq \cD(\cA)_{>-\infty}$
    which exhibits $\cD(\cA)$ as the right $t$-completion of $\cD^-(\cA)$.
    In particular, we have an equivalence $\K(\cA_\proj) \simeq \cD(\cA)$,
    and $\cD(\cA)$ is $t$-complete and left weight complete.
\end{proposition}
\begin{proof}
    By the above discussion,
    it only remains to establish the $t$-exact equivalence
    $\cD^-(\cA) \simeq \cD(\cA)_{>-\infty}$.
    The proof is essentially unwinding the definitions of the categories
    and their $t$-structures as given in \cite[Section 1.3.2,1.3.5]{HA}.

    Namely, recall that $\cD(\cA) = N_{dg}(\Ch(\cA)^\circ) \subseteq N_{dg}(\Ch(\cA)) \eqqcolon \Ch_\infty(\cA)$
    where $\Ch(\cA)^\circ \subseteq \Ch(\cA)$ is the full subcategory
    on (bi-)fibrant objects in the model structure constructed
    in \cite[Proposition 1.3.5.3]{HA}.
    The fibrant replacement yields a left Bousfield localization
    $L \colon \Ch_\infty(\cA) \to \cD(\cA)$
    with fully faithful right adjoint $i$.
    In particular, the unit of this adjunction is a pointwise quasi-isomorphism
    (these are the weak equivalences in the model structure).

    Now by \cite[Proposition 1.3.5.18]{HA} there is a $t$-structure
    on $\Ch_\infty(\cA)$ with $\Ch_\infty(\cA)_{t \geq 0}$
    consisting of those complexes with homology concentrated in non-negative
    degrees, and dually for $\Ch_\infty(\cA)_{t \leq 0}$.
    Since the unit of the adjunction $L \dashv i$
    is a pointwise quasi-isomorphism,
    it follows that the composite $iL$ is $t$-exact.

    The $t$-structure on $\cD(\cA)$ is descended from $\Ch_\infty(\cA)$
    in the sense that $X \in \cD(\cA)$ is (co)connective
    if and only if $iX$ is (co)connective.
    By the above, it then follows that both $L$ and $i$ are $t$-exact.
    Moreover, the induced $t$-structure on $\cD(\cA)$
    is right complete by \cite[Proposition 1.3.5.21]{HA}.

    We also have an obvious inclusion
    $j \colon \cD^-(\cA) \coloneqq \Ch_\infty^-(\cA_\proj) \subseteq \Ch_\infty(\cA)$, and the $t$-structure on $\cD^-(\cA)$ is defined analogously as for $\Ch_\infty(\cA)$
    so that $j$ is $t$-exact, and it is left complete
    by \cite[Proposition 1.3.3.16]{HA}.

    Finally, \cite[Proposition 1.3.5.24]{HA} shows
    that the composite $Lj \colon \cD^-(\cA) \to \cD(\cA)$
    is fully faithful with image $\cD(\cA)_{>-\infty}$.
    By the above, it is also $t$-exact,
    and thus induces a $t$-exact equivalence
    $\cD^-(\cA) \simeq \cD(\cA)_{>-\infty}$.
    Since $\cD(\cA)$ is right complete, we get
    \[
        \cD(\cA)
        \simeq \trc{(\cD(\cA)_{>-\infty})}
        \simeq \trc{\cD^-(\cA)}
    \]
    so that $\cD(\cA)$ is the right $t$-completion of a left $t$-complete category,
    so both left and right $t$-complete.
\end{proof}

\begin{corollary}\label{abelian}
    Let $\cA$ be a Grothendieck abelian category with enough projectives.
    Then
    \[
        \Kw(\cA_\proj) \simeq \Ch_\infty(\cA_\proj)
        \quad\text{and}\quad
        \K(\cA_\proj) \simeq \cD(\cA)
    \]
\end{corollary}

\section{Weight structures on Spectra}\label{sec:anderson}

We saw in Example \ref{spectra} that the standard weight structure on Spectra is left weight complete.
One of the main results of this section is the identification of its (right) weight completion as modules over the Steenrod algebra $A$.
In fact, we will see in Proposition \ref{prop:right-comp-steenrod}
that right weight completion lets us pass from $\Sp$ to $\LMod(A)$,
whereas right $t$-completion lets us pass back from $\LMod(A)$ to $\Sp$.
Along the way, it will be convenient to introduce another more exotic weight structure on spectra
which is, in a suitable sense, dual to the standard one.
It is $\omega_1$-compactly generated by the single object $\Z$ in the sense of Theorem \ref{thm:gen-weight}.
We call it the Anderson weight structure because its weight heart identifies with the full subcategory $\{I_\Z\}^\oplus \subseteq \Sp$ on arbitrary sums of the Anderson dual spectrum $I_\Z$.
Let us begin by recalling the construction of $I_{\mathbb{Z}}$ and its relevant properties.

\begin{construction}
    If $D$ is a divisible abelian group
    then by the Brown Representability Theorem there exists a spectrum $I_D \in \Sp$
    so that we have natural isomorphisms $\pi_n\hom(X,I_D) = \Hom(\pi_{-n}X,D)$ for all $n \in \Z$.
    Moreover, this is weakly functorial in that a map of divisible abelian
    groups $f \colon D \to D'$ lifts to a map of spectra $F \colon I_D \to I_{D'}$
    so that the map $\pi_n\hom(-,F)$
    corresponds to $\Hom(\pi_{-n}(-),f)$ under the above natural isomorphisms, for any $n \in \Z$.

    Despite its notation, the Anderson dual spectrum $I_\Z$ does not arise via this construction, as $\Z$ is not divisible.
    Instead, the above yields a map of spectra $I_\Q \to I_{\Q/\Z}$ which on $\pi_0$ is the canonical quotient map $\Q \to \Q/\Z$,
    and we define $I_\Z \coloneqq \fib(I_\Q \to I_{\Q/\Z})$.
    Here $I_{\Q/\Z}$ is the Brown-Comenetz dual spectrum introduced in \cite{Brown-Comenetz}
    and is infamous for its many curious properties.
    On the other hand, $I_\Q$ is simply the Eilenberg MacLane spectrum of $\Q$.
    The Anderson dual spectrum was introduced in \cite{Anderson} to deduce a universal coefficient theorem
    for complex $K$-theory.
\end{construction}

We will also need the following Lemma.

\begin{lemma}\label{lem:sum-iz}
    Let $J$ be a set and $n \in \Z$.
    There is a natural short exact sequence of abelian groups
    \[
        0 \to \Ext^1(\pi_{-(n+1)}X, \bigoplus_J \Z)
        \to \pi_n\hom(X, \bigoplus_J I_\Z)
        \to \Hom(\pi_{-n}X, \bigoplus_J \Z)
        \to 0
    \]
    which splits unnaturally. Moreover:
    \begin{enumerate}
        \item Under the isomorphism $\pi_n\hom(X,\bigoplus_J I_\Z) \cong \pi_0\hom(X, \Omega^n\bigoplus_J I_\Z)$, the second map takes $\pi_{-n}$.

        \item Let $\{\S\}^\oplus,\{I_\Z\}^\oplus \subseteq \Sp$ denote the full subcategories of $\Sp$
            on sums of spheres respectively sums of $I_\Z$'s.
            Then the adjunction $I_\Z \tensor - \dashv \hom(I_\Z,-)$ restricts to an adjoint equivalence of categories
            $\Thick(\{\S\}^\oplus) \simeq \Thick(\{I_\Z\}^\oplus)$.
    \end{enumerate}
\end{lemma}
\begin{proof}
    For $J$ a singleton, this is the usual UCT sequence for the Anderson
    dual spectrum, see e.g.~\cite[Lemma 2.4]{Thomas-Anderson}.
    The case of general $J$ is shown in the same way once we know
    the natural isomorphisms
    \[
        \pi_n\hom(X,\bigoplus_J \Q) \cong \Hom(\pi_{-n}X, \bigoplus_J \Q)
        \quad\text{and}\quad
        \pi_n\hom(X,\bigoplus_J I_{\Q/\Z}) \cong \Hom(\pi_{-n}X, \bigoplus_J \Q/\Z).
    \]
    Consider a collection of divisible abelian groups
    $(D_j)_{j \in J}$, and lift each of the maps $D_j \to \bigoplus_j D_j$
    to a map $I_{D_j} \to I_{\bigoplus_j D_j}$.
    We claim that together these induce an equivalence
    $\bigoplus_j I_{D_j} \xto{\simeq} I_{\bigoplus_j D_j}$.
    Indeed, we can check this after applying $\pi_0\hom(X,-)$ for an arbitrary
    finite spectrum $X$, where it follows from the following commutative diagram,
    where we also use compactness of $\pi_0X \in \Ab$:
    \[\begin{tikzcd}
        {\bigoplus_j \pi_0\hom(X,I_{D_j})} & {\pi_0\hom(X, \bigoplus_j I_{D_j})} & {\pi_0\hom(X,I_{\bigoplus_jD_j})} \\
        {\bigoplus_j \Hom(\pi_0X, D_j)} && {\Hom(\pi_0X, \bigoplus_j D_j)}
        \arrow["\cong", from=1-1, to=1-2]
        \arrow["\cong"', from=1-1, to=2-1]
        \arrow[from=1-2, to=1-3]
        \arrow["\cong", from=1-3, to=2-3]
        \arrow["\cong", from=2-1, to=2-3]
    \end{tikzcd}\]
    The addendum (1) is clear from the construction.
    Finally, to see (2), it suffices to check that the adjunction restricts to an adjoint equivalence
    $\{\S\}^\oplus \simeq \{I_\Z\}^\oplus$.
    Note that by the above, for any set $J$ and integer $n$
    we obtain a map of short exact sequences
    \[\begin{tikzcd}
        0 & {\bigoplus_J \Ext(\pi_{-(n+1)}I_\Z, \Z)} & {\bigoplus_J \pi_n\hom(I_\Z,I_\Z)} & {\bigoplus_J \Hom(\pi_{-n}I_\Z, \Z)} & 0 \\
        0 & {\Ext(\pi_{-(n+1)}I_\Z,\bigoplus_J \Z)} & {\pi_n\hom(I_\Z,\bigoplus_J I_\Z)} & {\Hom(\pi_{-n}I_\Z, \bigoplus_J \Z)} & 0
        \arrow[from=1-1, to=1-2]
        \arrow[from=1-2, to=1-3]
        \arrow["\cong"', from=1-2, to=2-2]
        \arrow[from=1-3, to=1-4]
        \arrow[from=1-3, to=2-3]
        \arrow[from=1-4, to=1-5]
        \arrow["\cong"', from=1-4, to=2-4]
        \arrow[from=2-1, to=2-2]
        \arrow[from=2-2, to=2-3]
        \arrow[from=2-3, to=2-4]
        \arrow[from=2-4, to=2-5]
    \end{tikzcd}\]
    where the outer vertical maps are isomorphisms since $I_\Z$ has
    finitely generated homotopy groups.
    We conclude that the canonical map $\bigoplus_J \hom(I_\Z,I_\Z) \to \hom(I_\Z,\bigoplus_J I_\Z)$
    is an equivalence.
    It follows that the adjunction $I_\Z \tensor - \dashv \hom(I_\Z,-)$ restricts to the full subcategories
    $\{\S\}^\oplus \rightleftarrows \{I_\Z\}^\oplus$.
    Together with the fact that tensoring with $I_\Z$ induces the equivalence $\hom(\S,\S) \xto{\simeq} \hom(I_\Z,I_\Z)$,
    we then conclude that $I_\Z \tensor -\colon \{\S\}^\oplus \to \{I_\Z\}^\oplus$ is fully faithful and hence an equivalence.
\end{proof}

Using the above, it is not hard to show that the functor $\hom(-,I_\Z) \colon \Sp^\op \to \Sp$
restricts to an involution on the full subcategory $\Sp^\fg \subseteq \Sp$ of spectra with finitely generated
homotopy groups, i.e.~we have an equivalence $\hom(-,I_\Z) \colon (\Sp^\fg)^\op \simeq \Sp^\fg$,
see e.g.~\cite[Theorem 2.9]{Thomas-Anderson}.\footnote{In fact, in forthcoming work \cite{CondDual} we show that for $X \in \Sp$, the canonical map $X \to \hom(\hom(X,I_\Z),I_\Z)$ into the double dual is an equivalence if and \emph{only if} $X \in \Sp^\fg$.}
We can now construct the Anderson weight structure on Spectra
mentioned in the introduction,
which uses the case $\kappa = \omega_1$ of Theorem \ref{thm:gen-weight}.

\begin{proposition}\label{prop:anderson-weight}
    There is a weight structure $w'$ on spectra $\omega_1$-compactly generated by the single object $\Z$, satisfying
    \begin{enumerate}
        \item $\Sp_{w' \geq 0} = \{X \in \Sp \mid \hom(\Z,X) \in \Sp_{\geq 0}\}$.

        \item $\Sp_{w' \leq 0} = \{X \in \Sp \mid \text{$\pi_k(X) = 0$ for $k>0$ and $\pi_0(X)$ free}\}$.

        \item $\Sp_{w'=0} = \{I_\Z\}^\oplus$ is the full subcategory of $\Sp$ on direct sums of $I_\Z$.

        \item The weight structure is right complete.

        \item $\bigcap_{n \geq 0} \Sp_{w' \geq n} = \{X \mid \hom(\Z,X) = 0\}$.
    \end{enumerate}
    We will refer to this weight structure as the Anderson weight structure on spectra.
\end{proposition}
\begin{proof}
    Applying Theorem \ref{thm:gen-weight}, point (1) is clear.
    Let $\cC \subseteq \Sp$ be the full subcategory on $t$-coconnective
    spectra with free $\pi_0$.
    Note that this is closed under finite limits,
    extensions, retracts, and sums.
    The fact that it is also closed under colimits of ordinal-indexed diagrams
    as in Lemma \ref{lem:weight-closure}(4) follows from Lemma \ref{lem:free-ab}.
    Since $\Z \in \cC$, we thus get $\Sp_{w' \leq 0} \subseteq \cC$.
    For the reverse inclusion,
    we simply note that if $A$ is an abelian group,
    then $A[-1]$ lies in $\Sp_{w' \leq 0}$ as fiber of a map
    between sums of $\Z$'s.
    By induction, we see that any $(-1)$-t-coconnective
    and bounded below spectrum lies in $\Sp_{w' \leq 0}$.
    Since $\Sp_{w' \leq 0}$ is also closed under sequential colimits
    of maps whose cofiber lies in $\Sp_{w' \leq 0}$,
    we conclude that if $X$ is $(-1)$-t-coconnective,
    then $X = \colim_n \tau_{\geq -n}X \in \Sp_{w' \leq 0}$.
    From this we deduce $\cC \subseteq \Sp_{w' \leq 0}$,
    and thus (2). This immediately yields (5),
    and (4) by combining Remark \ref{rem:conn-comp-inv-under-eqv}
    and (the dual of) Lemma \ref{lem:left-t-comp-is-conn-comp}.

    It remains to show (3).
    Note that the standard $t$-structure on spectra lets us conclude
    that $\hom(\Z,X)$ is $t$-coconnective whenever $X$ is $t$-coconnective.
    But now if $X \in \Sp_{w'=0}$, then $\hom(\Z,X)$ is concentrated in a single degree,
    where it is given by the free abelian group
    \[
        \pi_0\hom(\Z,X) \simeq \pi_0\map(\Z,X) \simeq \pi_0\map(\Z,\pi_0X) \cong \pi_0X.
    \]
    Next, we can deduce from Lemma \ref{lem:sum-iz} that $\bigoplus_J I_\Z \in \Sp_{w' = 0}$.
    Now pick an isomorphism $\pi_0X \cong \bigoplus_J \Z$.
    By Lemma \ref{lem:sum-iz} we have a split short exact sequence
    \[
        0 \to \Ext(\pi_{-1}(X),\bigoplus_J \Z) \to \pi_0\hom(X,\bigoplus_J I_\Z) \to \Hom(\pi_0X,\bigoplus_J \Z) \to 0.
    \]
    We then find a map of spectra $f \colon X \to \bigoplus_J I_\Z$
    which is a $\pi_0$-isomorphism.
    We have a commutative diagram
    \[\begin{tikzcd}
        {\pi_0\hom(\Z,\pi_0X)} & {\pi_0\hom(\Z,\pi_0\bigoplus_JI_\Z)} \\
        {\pi_0\hom(\Z,X)} & {\pi_0\hom(\Z,\bigoplus_J I_\Z)}
        \arrow["{(\pi_0f)_*,\cong}", from=1-1, to=1-2]
        \arrow["\cong"', from=1-1, to=2-1]
        \arrow["\cong", from=1-2, to=2-2]
        \arrow["{f_*}"', from=2-1, to=2-2]
    \end{tikzcd}\]
    showing that $\hom(\Z,f)$ is a $\pi_0$-isomorphism.
    But by the above argument, both source and target of this map
    are concentrated in degree $0$, so $\hom(\Z,\fib(f)) = 0$.
    But then $\fib(f) \in \Sp_{w' \leq 0} \cap \Sp_{w' \geq 1} = 0$,
    so $f$ is an equivalence, as desired.
\end{proof}

We note that $\bigcap_{n \geq 0} \Sp_{w' \geq n} = \{X \mid \hom(\Z,X) = 0\}$ contains for example all finite $p$-complete spectra\footnote{Suppose $X$ is finite $p$-complete. By chromatic convergence,
    we have $\hom(\F_p,X) = \lim_n \hom(\F_p,L_n X) = \lim_n \hom(L_n \F_p,L_n X) = 0$, since $L_n \F_p =0$ for all $n$.
Since $\hom(\Z,X)$ is $p$-complete, this shows it vanishes.}.
So this weight structure is very different from the standard one.

\begin{corollary}
    The Anderson weight structure restricts to the full subcategory $\Sp^\fg_{t < \infty} \subseteq \Sp$
    of bounded above spectra with finitely generated homotopy groups.
    Taking Anderson duals gives a weight-exact equivalence $(\Sp^\fg_{>-\infty})^\op \simeq \Sp^\fg_{t < \infty}$,
    and this allows us to exhibit $\Thick(I_\Z) \subseteq \Sp^\fg_{t < \infty}$
    and hence also the composite
    $I_\Z \tensor - \colon \Sp^\omega = \Thick(\S) \simeq \Thick(I_\Z) \subseteq \Sp^\fg_{t < \infty}$
    as right completions.
    Dually, the composite $\hom(I_\Z,-) \colon \Thick(I_\Z) \simeq \Thick(\S) \subseteq \Sp^\fg_{>-\infty}$
    is a left completion.
\end{corollary}
\begin{proof}
    We saw in Example \ref{ex:left-completion}
    that $\Thick(\S) = \Sp^\omega \subseteq \Sp^\fg_{>-\infty}$ is a left completion.
    Since Anderson duality restricts to an involution $(\Sp^{\fg})^\op \simeq \Sp^\fg$
    of $t$-amplitude $[-1,0]$, we obtain compatible equivalences
    \[\begin{tikzcd}[ampersand replacement=\&]
        {\Thick(\S)^\op} \& {\Thick(I_\Z)} \\
        {(\Sp^\fg_{>-\infty})^\op} \& {\Sp^\fg_{t < \infty}}
        \arrow["\simeq", from=1-1, to=1-2]
        \arrow[hook, from=1-1, to=2-1]
        \arrow[hook', from=1-2, to=2-2]
        \arrow["\simeq", from=2-1, to=2-2]
    \end{tikzcd}\]
    Now $\Thick(I_\Z)$ agrees with the stable envelope generated by the additive category
    consisting of finite sums of $I_\Z$, and hence the inclusion $\Thick(I_\Z) \subseteq \Sp$ is weight exact.
    The top horizontal map is induced by the functoriality of $\Stab(-)$ hence also weight exact.
    It follows that if we define a weight structure on $\Sp^\fg_{t < \infty}$
    via transferring the one on $(\Sp^\fg_{>-\infty})^\op$,
    then the right vertical map is a right completion.
    Since $\Sp$ with the Anderson structure is already right complete,
    we obtain an induced weight exact $f \colon \Sp^\fg_{t < \infty} \to \Sp$
    which restricts to the inclusion on $\Thick(I_\Z)$.
    By Corollary \ref{cor:left-comp-included}
    also $f$ is fully faithful and clearly identifies with the obvious inclusion.
    In particular, the ``transferred'' weight structure on $\Sp^\fg_{t < \infty}$
    is restricted from the Anderson weight structure,
    so all the claims follow from the above.
\end{proof}

\begin{corollary}
    Let $X \in \Sp$ be $b$-coconnective and suppose that $\hom(\Z,X)$ is
    $a$-connective for integers $a \leq b$.
    Then $X \in \Sp_{a \leq w' \leq b+1}$, so $X$ admits a finite filtration
    with subquotients equivalent to shifts of direct sums of $I_\Z$'s.
    If $X$ has moreover finitely generated homotopy groups,
    we can pick these sums to be finite,
    and in particular $X$ lies in the thick subcategory generated by $I_\Z$.
\end{corollary}

\begin{corollary}\label{cor:sp-st-vs-and}
    The functor $I_\Z \tensor -\colon \Sp \to \Sp$
    restricts to $\Sp_{w \leq 0} \to \Sp_{w' \leq 0}$,
    and its right adjoint $\hom(I_\Z,-)$
    restricts to $\Sp_{w' \geq 0} \to \Sp_{\geq 0}$.
    Both further restrict to mutually inverse weight exact equivalences
    $I_\Z \tensor - \colon \Sp_{-\infty < w < \infty} \simeq \Sp_{-\infty <w'<\infty} \noloc \hom(I_\Z,-)$.
\end{corollary}
\begin{proof}
    The last point is immediate from Lemma \ref{lem:sum-iz}(2)
    and the characterisations of the weight hearts of the standard
    respectively Anderson weight structure as $\{\S\}^\oplus$
    respectively $\{I_\Z\}^\oplus$.

    Suppose $X \in \Sp_{w \leq 0}$, so that $\Z \tensor X$ is coconnective
    with free $\pi_0$.
    Then $A \tensor X$ is $1$-$t$-coconnective for any $A \in \Ab$,
    and hence if $Y$ has homotopy groups concentrated in degrees $[-n,0]$,
    also $Y \tensor X$ is 1-$t$-coconnective.
    We conclude that $I_\Z \tensor X = \colim_n \tau_{\geq -n}I_\Z \tensor X$
    is a sequential colimit growing in $t$-coconnectivity
    of 1-coconnective spectra. Thus $I_\Z \tensor X$ is 1-coconnective,
    and $\tau_{[0,1]}(I_\Z \tensor X) = \tau_{[0,1]}(\tau_{\geq -1}I_\Z \tensor X)$.
    However, the defining fiber sequence $I_\Z \to \Q \to I_{\Q/\Z}$
    shows that $\tau_{\geq-1} I_\Z = \Z$, and so $I_\Z \tensor X$ is indeed $t$-coconnective
    with free $\pi_0$, hence Anderson-coconnective.

    Now let $X \in \Sp_{w' \geq 0}$, so that $\hom(\Z,X)$ is connective.
    Then $\hom(I_\Z, X) = \lim_n \hom(\tau_{\geq -n}I_\Z, X)$
    and an analogous analysis as in the first point shows that this is a sequential
    limit growing in connectivity of connective spectra, hence connective.
\end{proof}

\begin{remark}\label{rem:sp-and-not-upw-comp}
    As in the standard weight structure on $\Sp$ (cf.~Remark \ref{rem:sp-not-dw-comp}), point (5) of Proposition \ref{prop:anderson-weight} is not the only reason that the Anderson weight structure is not left complete,
    so we cannot just quotient out $\Sp_{w' \geq \infty}$
    to guarantee left completeness by Lemma \ref{lem:upw-comp-quot}.

    To see this, recall that $I_\Z \tensor \F_p = 0$.
    Indeed, by the defining fiber sequence $I_\Z \to \Q \to I_{\Q/\Z}$,
    this follows from $I_{\Q/\Z} \tensor \F_p = 0$, which was shown in \cite[Lemma 7.1]{Hovey-Palmieri}.
    Since $\F_p$ is of finite type, we can pick a weight complex
    for $\F_p$ in the standard weight structure of $\Sp$
    which exhibits $\F_p = \colim_n X_n$
    for finite spectra $X_n$ where $\cofib(X_n \to X_{n+1})$ is a finite sum
    of $\S^{n+1}$.
    But then $I_\Z \tensor X_\bullet$ is a sequential diagram
    of Anderson-connective spectra growing in Anderson-connectivity
    whose colimit is 0.
    In particular, for any $k \geq 1$, we have that
    $(I_\Z \tensor X_n/X_k)_{n \geq k}$ is a sequential diagram
    in $\Sp_{w' \geq k+1}$ which is growing in Anderson-connectivity,
    but whose colimit is $I_\Z \tensor \Sigma X_k$.
    Since $I_\Z \tensor -$ is a weight exact equivalence
    from standard-weight-bounded spectra to Anderson-weight-bounded spectra,
    and $\Sigma X_k$ is 1-connective but not 2-connective,
    we see that also $I_\Z \tensor \Sigma X_k$ lies in $\Sp_{w' \geq 1}$
    but not in $\Sp_{w' \geq 2}$ by Corollary \ref{cor:sp-st-vs-and}, since $\Sigma X_k$ is weight bounded.
    Thus for each $k \geq 0$,
    the sequential diagram $(\Sigma^{-(k+1)}I_\Z \tensor X_n /X_k)_{n \geq k}$
    is pointwise Anderson connective and also growing in Anderson connectivity,
    but its colimit $\Sigma^{-k}I_\Z \tensor X_k$
    is not Anderson $-(k-1)$-connective.
\end{remark}

\begin{remark}\label{remark_bounded}
    We can use Remarks \ref{rem:sp-not-dw-comp} and \ref{rem:sp-and-not-upw-comp}
    to see that Corollary \ref{cor:sp-st-vs-and} is in some sense optimal;
    $I_\Z \tensor -$ does not restrict to $\Sp_{w \geq 0} \to \Sp_{w' \geq 0}$
    and $\hom(I_\Z,-)$ does not restrict to $\Sp_{w' \leq 0} \to \Sp_{w \leq 0}$.
    For the former, using the notation of Remark \ref{rem:sp-and-not-upw-comp}
    we see that $\F_p/X_k$ is $(k+1)$-connective, but $I_\Z \tensor \F_p/X_k \simeq I_\Z \tensor \Sigma X_k$
    is only Anderson 1-connective, so $I_\Z \tensor -$ does not restrict to $\Sp_{w\geq 0} \to \Sp_{w' \geq -k}$ for any $k$.
    For the latter, we note that $X= \prod_{n \geq 0} I_\Z[-n] \in \Sp_{w' \leq 0}$ but $\hom(I_\Z,X) = \prod_{n \geq 0}\S[-n]$ is not even weight bounded above in the standard weight structure on spectra by the other remark.
\end{remark}

We thank Vova Sosnilo for a fruitful discussion from which the following result emerged.
Let $\Coh(\Sp) \subseteq \Sp$ denote the full subcategory consisting of those spectra
with only finitely many nontrivial homotopy groups, each of which is finitely generated.
Equivalently, this is the thick subcategory of $\Sp$ generated by $\Z$.
Since $\Ind(\Coh(\Sp))$ is compactly generated, we can use Theorem \ref{thm:gen-weight}
to endow it with the weight structure compactly generated by the single object $j\Z$.

\begin{proposition}\label{prop_andersoncompletion}
    The inclusion $\ell \colon \Coh(\Sp) \subseteq \Sp$ $\Ind$-extends to a functor $L \colon \Ind\Coh(\Sp) \to \Sp$
    whose right adjoint $R \colon \Sp \to \Ind\Coh(\Sp)$ is a left completion for the Anderson weight structure on spectra.
    Under the equivalence $\Ind\Coh(\Sp) \simeq \Fun^{\Ex}(\Coh(\Sp)^\op,\Sp)$
    we have $RX = \hom(\ell(-),X)$.
\end{proposition}
\begin{proof}
    As a first goal, we show that $R$ is fully faithful
    on the full subcategory $\Sp_{w' < \infty} = \Sp_{t < \infty}$.

    By definition there is an equivalence $\alpha \colon Lj \simeq \ell$,
    where $j \colon \Coh(\Sp) \subseteq \Ind\Coh(\Sp)$
    denotes the Yoneda embedding.
    We claim that the mate $BC(\alpha) \colon j \xto{\eta j} RLj \xto[\simeq]{R\alpha} \ell$
    of $\alpha$ is again an equivalence,
    where $\eta \colon \id \to RL$ is the unit.
    By the Yoneda lemma and the fact that $j\Z \in \Ind\Coh(\Sp)$
    generates under colimits,
    it suffices to check that for all $C \in \Coh(\Sp)$ the composite
    \[
        \hom_{\Ind\Coh(\Sp)}(j\Z,jC) \xto{(\eta_{jC})_*} \hom_{\Ind\Coh(\Sp)}(j\Z, RLjC)
    \]
    is an equivalence, which follows by 2-out-of-3
    from the following commutative diagram:
    \[\begin{tikzcd}[ampersand replacement=\&]
        {\hom(\Z,C)} \& {\hom(j\Z,jC)} \& {\hom(j\Z, RLjC)} \& {\hom(Lj\Z, LjC)} \\
        {\hom(\ell \Z, \ell C)} \&\&\& {\hom(Lj\Z, \ell C)}
        \arrow["\simeq", from=1-1, to=1-2]
        \arrow["j"', from=1-1, to=1-2]
        \arrow["\simeq"', from=1-1, to=2-1]
        \arrow["\ell", draw=none, from=1-1, to=2-1]
        \arrow["{(\eta_{jC})_*}", from=1-2, to=1-3]
        \arrow["L"', curve={height=12pt}, from=1-2, to=1-4]
        \arrow["\simeq", from=1-3, to=1-4]
        \arrow["{(\alpha_C)_*}"', from=1-4, to=2-4]
        \arrow["\simeq", draw=none, from=1-4, to=2-4]
        \arrow["{\alpha_\Z^*}", from=2-1, to=2-4]
        \arrow["\simeq"', draw=none, from=2-1, to=2-4]
    \end{tikzcd}\]
    Thus the Beck--Chevalley map $BC(\alpha) \colon j \to R\ell$ is an equivalence.
    Now the following left square induces by the general theory
    of mates (cf.~\cite[Lemma C.2]{CLL}) the right commutative square of functors
    \[\begin{tikzcd}[ampersand replacement=\&]
        \cC \& {\Ind(\cC)} \&\& Lj \&\& {LR\ell} \\
        \cC \& \cD \&\& \ell \&\& \ell
        \arrow["j", from=1-1, to=1-2]
        \arrow[Rightarrow, no head, from=1-1, to=2-1]
        \arrow[""{name=0, anchor=center, inner sep=0}, "L", curve={height=-6pt}, from=1-2, to=2-2]
        \arrow["{LBC(\alpha)j}", from=1-4, to=1-6]
        \arrow["\simeq"', draw=none, from=1-4, to=1-6]
        \arrow["\alpha"', from=1-4, to=2-4]
        \arrow["\simeq", draw=none, from=1-4, to=2-4]
        \arrow["{\eps\ell}", from=1-6, to=2-6]
        \arrow["\ell"', from=2-1, to=2-2]
        \arrow[""{name=1, anchor=center, inner sep=0}, "R", curve={height=-6pt}, from=2-2, to=1-2]
        \arrow[Rightarrow, no head, from=2-4, to=2-6]
        \arrow["\dashv"{anchor=center, rotate=-180}, draw=none, from=0, to=1]
    \end{tikzcd}\]
    which shows that the counit $\eps \colon LR \to \id$
    of the adjunction is an equivalence on $\Coh(\Sp)$.

    Next, we claim that $R$ and hence both source and target of
    the counit $\eps \colon LR \to \id$
    preserve uniformly bounded above filtered colimits.
    Using the explicit description of $R$,
    this reduces to showing that $\hom_{\Sp}(\Z,-)$ preserves uniformly bounded above filtered colimits,
    which is an instance of Lemma \ref{lem:ex-tensor}.

    Now $\Sp_{t < \infty}$ is generated as a stable category
    under uniformly bounded above filtered colimits by $\Coh(\Sp)$;
    one first obtains arbitrary Eilenberg--Mac Lane spectra
    by writing an abelian group as filtered colimit of its finitely
    generated subgroups,
    and then uses the Whitehead tower to obtain any bounded above spectrum.
    This proves that $\eps \colon LR \to \id$
    is an equivalence at all $X \in \Sp_{t<\infty}$.
    Given another $Y \in \Sp$, the map $\hom_\Sp(X,Y) \to \hom_{\Ind\Coh(\Sp)}(RX,RY)$ is then an equivalence,
    since postcomposing with the adjunction equivalence,
    the composite identifies with precomposition by the equivalence $\eps_X$.
    In particular, $R$ is fully faithful on $\Sp_{t<\infty}$.

    As a next step,
    we note that since $R\Z = j\Z$ and $Lj\Z = \Z$
    and both $L$ and $R$ preserve all the operations
    under which the weight coconnectives are generated
    from $\Z$ resp.~$j\Z$ (Theorem \ref{thm:gen-weight}(2)),
    it follows that both $L$ and $R$ preserve weight coconnectives.
    We now deduce via Lemma \ref{lem:saturated}(2) that $R$ also preserves
    connectives, so $R$ is weight exact.

    Since $R \colon \Sp_{w' \leq 0} \hookrightarrow \Ind\Coh(\Sp)_{w \leq 0}$,
    hits the generator, and preserves the relevant operations,
    we even get that it is an equivalence $\Sp_{w' \leq 0} \simeq \Ind\Coh(\Sp)_{w \leq 0}$.

    In view of Corollary \ref{cor:easy-left-comp} it remains to see that $\Ind\Coh(\Sp)$
    is left complete, but this is an instance of Theorem \ref{thm:gen-weight}(5),
    as $\mathbb{Z}$ is a compact generator of $\Ind\mathrm{Coh}(\Sp)$.
\end{proof}

\begin{corollary}
    The above weight structure on $\Ind\Coh(\Sp)$
    is both left and right complete.
\end{corollary}
\begin{proof}
    It is the left completion of a right complete
    weight structure, so we can use Proposition \ref{prop:lr-complete}.
\end{proof}

Since the Anderson weight structure is already right complete,
the last result shows that the composite
\[
    \Stab(\{I_\Z\}^\oplus) \simeq \Sp_{w'b} \subseteq \Sp \xto{R} \Ind\Coh(\Sp)
\]
is a (left and right) weight completion
of the weight-bounded part of the Anderson weight structure.
This composite sends $I_\Z \mapsto I_\bullet \coloneqq \indcolim{n} \tau_{\geq -n} I_\Z$, as one checks that the equivalence $\Ind\Coh(\Sp) \simeq \Fun^\Ex(\Coh(\Sp)^\op,\Sp)$ identifies $I_\bullet$ with $\hom(-,I_\Z)|_{\Coh(\Sp)^\op}$.

By Corollary \ref{cor:sp-st-vs-and} there is a weight exact equivalence
$I_\Z \tensor - \colon \Sp_{wb} \simeq \Sp_{w'b}$,
so the composite functor $\Sp_{wb} \to \Ind\Coh(\Sp)$ sending
$\S \mapsto I_\bullet$
is also a weight
completion of the weight-bounded part of the standard weight structure on spectra.

\begin{corollary}\label{co_functor}
    The colimit preserving functor
    \[
        I_\bullet \otimes - \colon \Sp \to \Ind\Coh(\Sp)
    \]
    exhibits $\Ind\mathrm{Coh}(\Sp)$ as the right weight completion of the standard weight structure on $\Sp$.
    Moreover:
    \begin{enumerate}
        \item As any right weight completion
            (cf.~the dual of Theorem \ref{thm:upw-comp}(2b))
            it is an equivalence on bounded below objects:
            \[
                \Sp_{>-\infty} \simeq \Ind\Coh(\Sp)_{w >-\infty}.
            \]

        \item We have an identification $I_{\bullet+1} \simeq \hom(\tau_{\leq \bullet}\S, I_\Z)$ which allows us to identify
            the unit map of the adjunction $I_\bullet \tensor - \dashv \hom_{\Ind\Coh(\Sp)}(I_\bullet,-)$
            with the canonical map
            \[
                \eta_X \colon X \to \lim_n (\tau_{\leq n}\S \tensor X).
            \]
            This map is an equivalence for bounded below spectra $X$.
    \end{enumerate}
\end{corollary}
\begin{proof}
    By the above observations we know that the restriction to $\Sp_{wb}$
    is a weight completion.
    In particular, $I_\bullet$ lies in the weight heart of $\Ind\Coh(\Sp)$.
    Thus $I_\bullet \tensor -$ preserves all operations
    under which $\Sp_{w \leq 0}$ is generated by $\S$,
    and sends $\S$ into $\Ind\Coh(\Sp)_{w \leq 0}$,
    so preserves weight-coconnectives.
    Similarly, since both sides are left complete,
    both $\Sp_{w \geq 0}$ and $\Ind\Coh(\Sp)$ are generated
    under colimits by the respective weight hearts,
    so the fact that $I_\bullet \tensor -$ is an equivalence on weight hearts
    and preserves colimits shows that it also preserves weight connectives.
    Thus $I_\bullet \tensor -$ is weight exact
    and an equivalence on weight-bounded objects.
    Lemma \ref{lem:upw-comp-eqv}(2) shows that the map is then also an equivalence on weight bounded below objects,
    and we deduce from the dual of Corollary \ref{cor:easy-left-comp}
    that $I_\bullet \tensor -$ is a right weight completion.

    The proof also verified (1).
    For (2), note that the special case $J=*$
    of Lemma \ref{lem:sum-iz} and the fact
    that the higher homotopy groups of $\S$ are torsion
    allows us to deduce that
    $\hom(\tau_{\geq n+1}\S,I_\Z) \in \Sp_{t \leq -(n+2)}$,
    and hence the canonical map
    \[
        \tau_{\geq -(n+1)}\hom(\tau_{\leq n}\S,I_\Z) \to \tau_{\geq -(n+1)}\hom(\S,I_\Z)
        \simeq \tau_{\geq -n}I_\Z
    \]
    is an equivalence.
    Moreover, we can also deduce from said Lemma
    that $\hom(\tau_{\leq n}\S,I_\Z) \in \Sp_{\geq -(n+1)}$,
    which yields a natural identification
    $\hom(\tau_{\leq n}\S, I_\Z) \simeq \tau_{\geq -(n+1)}I_\Z$.
    In particular, we obtain an equivalence of $\Ind$-systems
    $I_\bullet \simeq \indcolim{n} \hom(\tau_{\leq n}\S,I_\Z)$.
    This yields for any $X \in \Sp$
    \begin{align*}
        \hom_{\Ind\Coh(\Sp)}(I_\bullet, I_\bullet \tensor X)
        &\simeq \lim_n \colim_k \hom_{\Ind\Coh(\Sp)}(j\hom(\tau_{\leq n}\S,I_\Z), (j\tau_{\geq -k}I_\Z) \tensor X)\\
        &\simeq \lim_n \colim_k \hom_{\Sp}(\hom(\tau_{\leq n}\S,I_\Z), (\tau_{\geq -k}I_\Z)) \tensor X\\
        &\simeq \lim_n \hom_{\Sp}(\hom(\tau_{\leq n}\S, I_\Z), I_\Z) \tensor X\\
        &\simeq \lim_n (\tau_{\leq n}\S \tensor X)
    \end{align*}
    where we use that $\hom(j\hom(\tau_{\leq n}\S,I_\Z),-)$
    preserves colimits and hence the $\Sp$-tensoring,
    and that $\hom(\tau_{\leq n}\S,I_\Z) \in \Sp_{\geq -(n+1)}$
    in combination with Example \ref{ex:corep-bdd-t}
    to pull the colimit over $k$ back into the hom.
    Finally, the last equivalence uses that $\tau_{\leq n}\S$
    has finitely generated homotopy groups in each degree,
    so that the canonical map $\tau_{\leq n}\S \to \hom(\hom(\tau_{\leq n}\S,I_\Z),I_\Z)$ is an equivalence.
\end{proof}

Anderson Duality induces an involution $\hom(-,I_\Z) \colon \Coh(\Sp)^\op \simeq \Coh(\Sp)$, which in turn induces
\[
    \Phi \colon \Ind(\Coh(\Sp)^\op) \simeq \Ind\Coh(\Sp)
\]
which sends $j\Z \mapsto j\Z$. In particular, if we also endow
$\Ind(\Coh(\Sp)^\op)$ with the weight structure compactly
generated by $j\Z$, then $\Phi$ is weight exact.
In particular, also the composite $\Phi^{-1} \circ (I_\bullet \tensor -)$
is a right weight completion of the standard weight structure on $\Sp$.
We can also describe the composite functor $\Phi^{-1} \circ (I_\bullet \tensor -)$ as
\[
    \Sp
    \xto{I_\bullet \tensor -} \Ind\Coh(\Sp)
    \xto[\simeq]{\Phi^{-1}} \Ind(\Coh(\Sp)^\op)
    \xto[\simeq]{} \Fun^\Ex(\Coh(\Sp),\Sp),\ X \mapsto (C \mapsto C \tensor X)
\]
Upon identifying $\Fun^\Ex(\Coh(\Sp), \Sp)$ with left modules over the ``integral Steenrod algebra'' $A \coloneqq \hom(\Z,\Z) \in \Alg(\Sp)$
by restriction to the generator $\Z \in \Coh(\Sp)$,
the above composite functor becomes taking homology.

\begin{proposition}\label{prop:right-comp-steenrod}
    The category $\LMod(A)$ of left modules over the integral Steenrod
    algebra admits a weight structure such that
    \begin{enumerate}
        \item The equivalence $\Ind(\Coh(\Sp)) \simeq \LMod(A)$
            is weight exact.
            In other words, the weight structure on $\LMod(A)$
            is compactly generated by $A$.
            In particular, $A \in \LMod(A)_{w \leq 0}$.

        \item $\LMod(A)_{w \geq 0}$ consists of the $A$-modules
            which are underlying connectives.

        \item The functor $\Z \tensor -\colon \Sp \to \LMod(A)$
            is a right weight completion,
            and in particular restricts to an equivalence
            $\Sp_{w>-\infty} \simeq \LMod(A)_{w>-\infty}$
            with inverse given by the right adjoint $\hom_A(\Z,-)$.
            The unit of this adjunction can be identified with
            \[
                X \to \hom_A(\Z, \Z \tensor X) \simeq \lim_n (\tau_{\leq n}\S \tensor X)
            \]
            and is an equivalence on bounded below spectra $X$.

        \item Since the weight structure is compactly generated, it admits an adjacent $t$-structure.
            With respect to this, the right adjoint $\hom_A(\Z,-) \colon \LMod(A) \to \Sp$
            is a right $t$-completion. Hence
            \[
                \wrc{\Sp} \simeq \LMod(A)
                \quad\text{and}\quad
                \trc{\LMod(A)} \simeq \Sp.
            \]
            In particular, we also have $\K(\{\S\}^\oplus) \simeq \Sp$
            and $\Kw(\{\S\}^\oplus) \simeq \LMod(A)$,
            and the above equivalences are an instance of Remark \ref{rem:k-vs-kw} where we do have comparison functors in both directions.
    \end{enumerate}
\end{proposition}
\begin{proof}
    The image of $\S$ under the composite map $\Sp \to \Fun^\Ex(\Coh(\Sp),\Sp) \simeq \LMod(A)$ is precisely $\Z$,
    and the image of $j\Z$ under the equivalence
    $\Ind(\Coh(\Sp)) \simeq \LMod(A)$ is precisely $A$.
    Hence, transferring the weight structure through these equivalences,
    we recognize it as being compactly generated by $A$.
    From this points (1)-(3) follow from Corollary \ref{co_functor}.
    To see (4), we need to argue that $\hom_A(\Z,-)$ is $t$-exact.
    By adjacency of the $t$-structures and the weight-exact equivalence from (3)
    we see that both $\hom_A(\Z,-)$ and its left adjoint $\Z \tensor -$ preserve $t$-connectives.
    The analogue of Lemma \ref{lem:saturated}(2)
    shows that $\hom_A(\Z,-)$ is then actually $t$-exact,
    and we deduce from Corollary \ref{cor:right-comp-via-embed} that it is a right $t$-completion.
\end{proof}

One can also formulate Proposition \ref{prop_andersoncompletion} in a similar manner by identifying $\Ind\mathrm{Coh}(\Sp)$ with right modules $\RMod(A) \simeq \LMod(A^\op)$.
Chasing through the identifications, left completion for the Anderson weight structure described there becomes
\[
    \Sp \to \RMod(A) \qquad X \mapsto \hom(\Z, X).
\]
The Anderson duality equivalence $\Phi$ then becomes the observation that we have an equivalence of algebras
$A \simeq A^\op$ induced by applying the functoriality of Anderson duality to $A \coloneqq \hom(\Z,\Z)$
using that $\Z$ is Anderson selfdual.

\begin{remark}
    Following \cite[Section 5.2]{Carlsson},
    the unit map $X \to \hom_A(\Z, \Z \tensor X)$
    of the above adjunction
    can be viewed as a completion at $\Z$ inspired by Morita Theory.
    More generally, given a map $B \to C$ in $\CAlg(\Sp)$
    we consider the endomorphism algebra $E \coloneqq \hom_B(C,C)$.
    Then for any $B$-module $M$, the basechange $C \tensor_B M$
    is also canonically an $E$-module,
    and we may consider the canonical map
    $M \to \hom_E(C, C \tensor_B M)$.
    Endofunctors of this form are studied in \cite{DGI}
    and Carlsson improves on some of their results
    to deduce that this map
    identifies with the Bousfield--Kan / Adams style nilpotent completion
    $M \to \Tot(C^{\tensor_B \bullet+1} \tensor_B M)$
    under various finiteness hypotheses on the map $B \to C$.

    To the best of our knowledge, none of their results apply
    to our setting where $B \to C$ is $\S \to \Z$
    and $E = A$ is the integral Steenrod algebra.
    Nevertheless, there is a natural comparison
    map $\Tot(\Z^{\tensor \bullet+1} \tensor X) \to \hom_A(\Z, \Z \tensor X)$,
    from the nilpotent completion of $X$ at $\Z$
    to this Morita-inspired completion of $X$ at $\Z$.
    By the above results, it follows that this map
    is an equivalence for bounded below $X$,
    where the nilpotent completion at $\Z$ just recovers $X$ itself.

    This suggests that it might always be an equivalence,
    but we expect this to be false;
    one can reduce this to checking that our unit map above
    $X \to \hom_A(\Z, \Z \tensor X)$ is always an equivalence
    for $X$ of the form $X = \Z \tensor Y$.
    Since the bounded below case is already known,
    a crucial test case is $X = \bigoplus_{n \geq 0}\Z[-n]$.
    Based on an explicit computation
    of an algebraic analogue of this question
    (considering maps of graded modules over the mod $p$ Steenrod algebra)
    we expect this to fail.

    By the identification of the unit map part (3)
    of the above Proposition,
    the question is equivalent to the canonical map
    \[
        \Z \tensor Y \to \lim_n (\tau_{\leq n}\S \tensor \Z \tensor Y)
    \]
    being an equivalence for all $Y$.
    Again this is easily deduced for bounded below $Y$,
    but the case $Y = \bigoplus_{n \geq 0}\S[-n]$ has no reason to be true.
\end{remark}

\begin{remark}
    For any ring spectrum $R$, there is a natural weight structure on $\LMod(R)$ whose connective objects are those whose underlying spectra are connective; this structure is compactly generated by $R$ itself. If $R$ is coconnective, then the weight coconnectives are contained in the underlying coconnectives, but are not so easy to describe explicitly (apart from being generated from $R$ under operations). Moreover, this weight structure admits an adjacent $t$-structure, since the connective objects are closed under colimits and extensions. Note that this is in a sense dual to the more commonly used $t$-structure in which coconnective objects are those whose underlying spectra are coconnective (see e.g.~\cite{Burklund-Levy}).

    However, it is in general difficult to describe the $t$-coconnective objects or the associated connective cover functor. In some cases, the situation even degenerates: if $R$ admits a unit in positive degree, then the only connective object is zero. On the other hand, if $R$ is connective, this recovers the standard $t$-structure.

    We now consider the situation of Proposition~\ref{prop:right-comp-steenrod}, where $R = A$ is the integral Steenrod algebra, which is coconnective. In this case the connective cover functor admits a concrete description. Namely, it is given by
    \[
        M \;\mapsto\; \mathbb{Z} \otimes \bigl(\tau_{\geq 0}\hom_A(\mathbb{Z}, M)\bigr).
    \]
    In particular we find that the $t$-coconnective objects are those $M$ such that $\hom_A(\Z,M)$ is connective.

    To see this formula, one uses the adjunction $\Z \tensor - \colon \Sp \rightleftarrows \LMod(A) \noloc \hom_A(\Z,-)$
    together with the fact that every connective object in $\LMod(A)$ is of the form $\mathbb{Z} \otimes X$ for some connective spectrum $X$.
\end{remark}

\begin{remark}\label{remark_weight}
In Corollary \ref{cor:sp-st-vs-and} we saw that the functor
\[
I_\Z \otimes - : \Sp \to \Sp
\]
restricts to an equivalence on weight bounded objects, where the source is equipped with the
standard weight structure and the target with the Anderson weight structure. However, this functor is not weight exact in general (see Remark \ref{remark_bounded}). If it were weight exact, it would induce an equivalence on weight completions fitting into a commutative square
\[
\begin{tikzcd}
\mathrm{Sp} \arrow[r, "I_{\mathbb{Z}} \otimes -"] \arrow[d, "\mathbb{Z} \otimes -"'] & \mathrm{Sp} \arrow[d, "\mathrm{hom}(\Z {,} -)"] \\
\mathrm{LMod}(A) \arrow[r, dashed, "?"', "\simeq"] & \mathrm{RMod}(A).
\end{tikzcd}
\]
However the dashed arrow does not exist (see the computation for $\Z \in \Sp$ below).
On the other hand, there is an equivalence of weight completions
\[
    \Psi \colon \LMod(A) \simeq \RMod(A),
\]
induced by the equivalence $A \simeq A^\op$, which acts as the identity on underlying spectra. In fact we obtain a lax commuting diagram
\[
\begin{tikzcd}
\mathrm{Sp}
  \arrow[r, "I_{\mathbb{Z}} \otimes -"]
  \arrow[d, "\mathbb{Z} \otimes -"']
& \mathrm{Sp}
  \arrow[d, "\mathrm{hom}(\mathbb{Z}{,}{-})"] \\
\mathrm{LMod}(A)
  \arrow[r, "\Psi"', "\simeq"{above}]
& \mathrm{RMod}(A)
\arrow[Rightarrow, from=2-1, to=1-2, shorten <=10pt, shorten >=10pt, "\not\simeq"]
\end{tikzcd}
\]
which strongly commutes on weight bounded objects, but not in general.
Namely, for $\mathbb{Z} \in \Sp$ we have
\[
\Psi(\mathbb{Z} \otimes \mathbb{Z}) \not\simeq \mathrm{hom}(\mathbb{Z}, \mathbb{Z} \otimes I_\Z),
\]
since the target is equivalent to $\mathbb{Q}$ whereas the source is not. Still we want to think of the equivalence $\Psi$ as the `correct' extension of Corollary \ref{cor:sp-st-vs-and}.
It is natural to ask whether this square can nevertheless be understood as arising from the functor $I_\Z \otimes -$, despite its failure to be weight exact--more precisely, whether non-weight-exact (or more generally non-weight-bounded) functors can still induce functors on weight completions in some weaker sense. At present, we do not know a satisfactory answer to this question.
\end{remark}

\section{Weight and $t$-structures on Tate categories}\label{sec:tate}

In this section we apply our results to Tate categories.
We show that if $\cC$ is a weakly weighted category,
then $\Tate(\cC)$ always admits a weight structure
whose heart recovers Drinfeld's Tate vector spaces for $\cC = \Perf(k)$
where $k$ is a field, cf.~\cite{Drinfeld}.
Moreover, we give sufficient conditions
under which there exists an adjacent weak $t$-structure
on $\Tate(\cC)$ or its left completion,
as well as a concrete example showing that generally
one can only hope for a \emph{weak} adjacent $t$-structure
(Example \ref{ex:only-weak-t}).
The motivation is that $\wtc{\Tate(\cC)} = \K(\Tate(\cC)_{w=0})$
yields an interesting category of Tate objects which is entirely
determined by the weight heart, as we saw in Section \ref{sec:adj}.
In forthcoming work \cite{CondDual}, for certain algebraic examples of $\cC$ such as $\cD(\Z)^\omega$,
we relate this to a full subcategory of condensed $\Ind$-objects in $\cC$.
Let us begin by recalling the notion of Tate objects and Tate categories.

\begin{definition}
    Let $\cC$ be a small stable category.
    We let $\Tate(\cC) \subseteq \Ind(\Pro(\cC))$
    denote the smallest stable category spanned by $\Pro(\cC)$ and $\Ind(\cC)$.
\end{definition}

This was investigated by Hennion \cite{Hennion}, although we warn
the reader that he worked in an idempotent complete context,
and he would refer to our Tate objects as elementary Tate objects.

\begin{proposition}[{{\cite[Theorem 3]{Hennion}}}]
    The category $\Tate(\cC)$ admits an equivalent description
    as the full subcategory of $\Ind(\Pro(\cC))$
    on those objects $X$ which sit in an extension
    $X^p \to X \to X^i$, where $X^p \in \Pro(\cC)$ and $X^i \in \Ind(\cC)$.
\end{proposition}

Let $\cC$ be a weakly weighted category.
Recall from Proposition \ref{prop:ind-weight} that $\Ind(\cC)$
attains an induced weight structure compactly generated by $\cC_{\leq 0}$.
Dually, $\Pro(\cC) = \Ind(\cC^\op)^\op$ attains a weight structure
cocompactly cogenerated by $\cC_{\geq 0}$, in that
\[
    \Pro(\cC)_{\leq 0} = \{X \in \Pro(\cC) \mid \hom(X,c) \in \Sp_{\geq 0}\text{ for }c \in \cC_{\geq 0}\}
\]
and $\Pro(\cC)_{\geq 0}$ is the smallest class containing $\cC_{\geq 0}$
and closed under finite colimits, extensions, retracts, products and sequential
limits along maps whose fiber already lies in the class.

\begin{theorem}\label{thm:tate-weight}
    For a weakly $\ell$-weighted $\cC$,
    there exists a weight structure on $\Tate(\cC)$ such that:
    \begin{enumerate}
        \item The inclusions $\cC \subseteq \Ind(\cC),\Pro(\cC) \subseteq \Tate(\cC)$ are weight exact.

        \item For $-\infty \leq a \leq b \leq \infty$ where at least one of $a,b$ is finite,
            the category $\Tate(\cC)_{[a,b]}$ is the retract-closure\footnote{This retract-closure
                is taken in $\Tate(\cC)$, which generally does not have many retracts.
            } of the full subcategory
            on those Tate objects $X^p \to X \to X^i$ with $X^p \in \Pro(\cC)_{[a,b]}$
            and $X^i \in \Ind(\cC)_{[a,b]}$.

        \item For every $X \in \Tate(\cC)$ there is a weight decomposition $X_{\leq 0} \to X \to X_{\geq 1}$
            where both $X_{\leq 0}$ and $X_{\geq 1}$ are actually (not just retracts of) extensions
            $X_{\leq 0}^p \to X_{\leq 0} \to X_{\leq 0}^i$ and $X_{\geq 1}^p \to X_{\geq 1} \to X_{\geq 1}^i$.

        \item If $\ell = 0$, then $\Tate(\cC)_{=0}$
            is the retract-closure of the full subcategory
            on Tate-objects of the form $\bigoplus_i c_i \oplus \prod_j d_j$
            with $c_i,d_j \in \cC_{=0}$ (where we view $\bigoplus_i c_i$ as $\Ind$- and $\prod_j d_j$ as $\Pro$-object).

        \item If $\cD$ is another weakly weighted category
            and $f \colon \cC \to \cD$ an exact functor of weight amplitude $[a,b]$,
            then $\Tate(f) \colon \Tate(\cC) \to \Tate(\cD)$
            is of weight amplitude $[a-\ell,b+\ell]$,
            and similarly for the (half)-open intervals.
    \end{enumerate}
\end{theorem}
\begin{proof}
    We define $\Tate(\cC)_{\leq 0}$ and $\Tate(\cC)_{\geq 0}$ so that (2) is satisfied
    when $a$ or $b$ is infinite,
    which automatically guarantees (1).
    Now let $X_{\leq 0} \in \Tate(\cC)_{\leq 0}$ and $Y_{\geq 0} \in \Tate(\cC)_{\geq 0}$. We need to show that $\hom(X_{\leq 0},Y_{\geq 0})$ is connective.
    Since this is closed under retracts and extensions in both variables,
    we may reduce to the cases where $X_{\leq 0}$ and $Y_{\geq 0}$
    are pure $\Pro$- or $\Ind$-objects. The two cases where they are both $\Pro$-
    or both $\Ind$-objects follow from the weight structures on $\Pro(\cC)$ and $\Ind(\cC)$.
    Writing $X^p_{\leq 0} = \prolim{i} c_{i,\leq 0}$ and $Y^i_{\geq 0} = \indcolim{j} d_{j,\geq 0}$, we have
    \[
        \hom_{\Tate(\cC)}(X^p_{\leq 0}, Y_{\geq 0}^i)
        \simeq \colim_{i,j}\hom_\cC(c_{i,\leq 0}, d_{j,\geq 0})
    \]
    which is a filtered colimit of connective spectra, hence connective.
    The remaining case is $\hom(X^i_{\leq 0}, Y^p_{\geq 0})$.
    But the class of $X^i_{\leq 0} \in \Ind(\cC)_{\leq 0}$
    for which this hom spectrum is connective
    is closed under the operations by which $\Ind(\cC)_{\leq 0}$
    is generated from $\cC_{\leq 0}$ (cf.~Proposition \ref{prop:ind-weight}).
    Since $\cC_{\leq 0}$ is also contained in this class
    by the weight structure on $\Pro(\cC)$,
    we see that the class agrees with all of $\Ind(\cC)_{\leq 0}$.
    This concludes the proof of orthogonality of $\Tate(\cC)_{\leq 0}$
    and $\Tate(\cC)_{\geq 0}$.

    Next, we need to verify that weight decompositions exist.
    So let $X \in \Tate(\cC)$ sit in an extension $X^p \to X \to X^i$.
    Using the weight structures on $\Pro(\cC)$ and $\Ind(\cC)$,
    we pick weight decompositions $X^i_{\leq 0} \to X^i \to X^i_{\geq 1}$
    and $X^p_{\leq 0} \to X^p \to X^p_{\geq 1}$.
    By the orthogonality we just verified,
    the composite map $X^i_{\leq 0} \to \Sigma X^p_{\geq 1}$ vanishes,
    so that we obtain factorizations as indicated in the following diagram
    \[\begin{tikzcd}
        {X_{\leq 0}} & {X^i_{\leq 0}} & {\Sigma X^p_{\leq 0}} \\
        X & {X^i} & {\Sigma X^p} \\
        {X_{\geq 1}} & {X^i_{\geq 1}} & {\Sigma X^p_{\geq 1}}
        \arrow[from=1-1, to=1-2]
        \arrow[from=1-1, to=2-1]
        \arrow[dashed, from=1-2, to=1-3]
        \arrow[from=1-2, to=2-2]
        \arrow[from=1-3, to=2-3]
        \arrow[from=2-1, to=2-2]
        \arrow[from=2-1, to=3-1]
        \arrow[from=2-2, to=2-3]
        \arrow[from=2-2, to=3-2]
        \arrow[from=2-3, to=3-3]
        \arrow[from=3-1, to=3-2]
        \arrow[dashed, from=3-2, to=3-3]
    \end{tikzcd}\]
    where all sequences are (defined to be) fiber sequences.
    This gives extensions $X^p_{\leq 0} \to X_{\leq 0} \to X^i_{\leq 0}$
    showing $X_{\leq 0} \in \Tate(\cC)_{\leq 0}$,
    and dually $X_{\geq 1} \in \Tate(\cC)_{\geq 1}$,
    so that $X_{\leq0} \to X \to X_{\geq 1}$ is a weight decomposition
    in $\Tate(\cC)$. This also proves (3).

    To check (2) for finite $a$ and $b$
    (the case where one is infinite being true by definition),
    it suffices to consider the case
    $[a,b] = [0,n]$ by shifting.
    Then given $X_{[0,n]} \in \Tate(\cC)_{[0,n]}$ we can assume (up to retracts)
    that $X_{[0,n]}$ admits an extension
    $X^p_{\geq 0} \to X_{[0,n]} \to X^i_{\geq 0}$.
    As above we can pick weight decompositions
    $X^p_{[0,n]} \to X^p_{\geq 0} \to X^p_{\geq n+1}$
    and analogously for $X^i_{\geq 0}$,
    and form the diagram of fiber sequences
    \[\begin{tikzcd}
        F & {X^i_{[0,n]}} & {\Sigma X^p_{[0,n]}} \\
        {X_{[0,n]}} & {X^i_{\geq 0}} & {\Sigma X^p_{\geq 0}} \\
        {X_{\geq n+1}} & {X^i_{\geq n+1}} & {\Sigma X^p_{\geq n+1}}
        \arrow[from=1-1, to=1-2]
        \arrow[from=1-1, to=2-1]
        \arrow[dashed, from=1-2, to=1-3]
        \arrow[from=1-2, to=2-2]
        \arrow[from=1-3, to=2-3]
        \arrow[from=2-1, to=2-2]
        \arrow[from=2-1, to=3-1]
        \arrow[from=2-2, to=2-3]
        \arrow[from=2-2, to=3-2]
        \arrow[from=2-3, to=3-3]
        \arrow[from=3-1, to=3-2]
        \arrow[dashed, from=3-2, to=3-3]
    \end{tikzcd}\]
    The bottom map in the left vertical fiber sequence vanishes
    by orthogonality, which exhibits $X_{[0,n]}$
    as retract of $F$, which in turn sits in an extension
    $X^p_{[0,n]} \to F \to X^i_{[0,n]}$. This proves (2).

    To see (4), by (2) we know that if $X_{=0} \in \Tate(\cC)_{=0}$
    then it is a retract of some $Y_{=0}$ which admits
    an extension $Y^p_{=0} \to Y_{=0} \to Y^i_{=0}$.
    But since $Y^i_{=0} \to \Sigma Y^p_{=0}$ necessarily vanishes,
    this fiber sequence splits, and we get $Y_{=0} \simeq Y^p_{=0} \oplus Y^i_{=0}$. We conclude by the descriptions of $\Ind(\cC)_{=0}$ and $\Pro(\cC)_{=0}$ in the case of $\ell = 0$, using Proposition \ref{prop:ind-weight}(4) and its dual.

    Finally, (5) follows from (2), the fact that $\Tate(f)$ is exact,
    and $\Ind(f)$ respectively $\Pro(f)$ having weight amplitude
    $[a-\ell,b]$ respectively $[a,b+\ell]$ by Proposition \ref{prop:ind-weight}(6) and its dual. The argument for (half-)open intervals is analogous.
\end{proof}

\begin{corollary}
    If $\cC$ is a weighted category equipped with a weight exact equivalence $D \colon \cC^\op \xto{\simeq} \cC$,
    then this extends to a weight exact equivalence $\Tate(D) \colon \Tate(\cC)^\op \xto{\simeq} \Tate(\cC)$
    which yields a commutative diagram of weight exact functors
    \[\begin{tikzcd}
        {\Ind(\cC)^\op} & {\Tate(\cC)^\op} & {\Pro(\cC)^\op} \\
        {\Pro(\cC)} & {\Tate(\cC)} & {\Ind(\cC)}
        \arrow[hook', from=1-1, to=1-2]
        \arrow["\simeq"', from=1-1, to=2-1]
        \arrow["\simeq", from=1-2, to=2-2]
        \arrow[hook, from=1-3, to=1-2]
        \arrow["\simeq", from=1-3, to=2-3]
        \arrow[hook, from=2-1, to=2-2]
        \arrow[hook', from=2-3, to=2-2]
    \end{tikzcd}\]
    In other words, the extended equivalence acts on extensions as
    \[
        \prolim{i} c_i \to X \to \indcolim{j} d_j \qquad\mapsto\qquad \prolim{j} Dd_j \to DX \to \indcolim{i} Dc_i.
    \]
\end{corollary}

\begin{lemma}\label{lem:tate-upw-comp}
    Let $\cC$ be a weakly $\ell$-weighted category for some $\ell < \infty$.
    Then the map $\eta \colon \cC \to \lc{\cC}$
    induces equivalences
    $\lc{\Pro(\cC)} \simeq \lc{\Pro(\lc{\cC})}$ of amplitude $[0,\ell]$
    and
    $\lc{\Tate(\cC)} \simeq \lc{\Tate(\lc{\cC})}$ of amplitude $[-\ell,\ell]$.\footnote{The analogous statement that there is an amplitude $[-\ell,0]$ equivalence $\lc{\Ind(\cC)} \simeq \lc{\Ind(\lc{\cC})}$ is less interesting but also true by Proposition \ref{prop:ind-weight}.}
    In particular, the left completion of $\Pro(\cC)$ resp.~$\Tate(\cC)$,
    only depends on $\Pro(\lc{\cC})$ resp.~$\Tate(\lc{\cC})$ up to equivalence in $\wWCat$ (or $\wWCat_{=0}$ in the case $\ell = 0$).
\end{lemma}
\begin{proof}
    Recall from Theorem \ref{thm:upw-comp}(2) that $\eta_{<\infty}
    \colon \cC_{<\infty} \hookrightarrow (\lc{\cC})_{<\infty}$
    is a fully faithful weight exact functor with dense image.
    It follows that the inclusion induces an equivalence
    $\Pro(\cC_{\leq n}) \simeq \Pro((\lc{\cC})_{\leq n})$ for every $n$.
    From this we obtain a commutative diagram
    \[\begin{tikzcd}
        {\Pro(\cC_{\leq 0})} & {\Pro(\cC)_{\leq 0}} & {\Pro(\cC_{\leq \ell})} & {\Pro(\cC)_{\leq \ell}} & {\Pro(\cC_{\leq 2\ell})} & \cdots \\
        {\Pro((\lc{\cC})_{\leq 0})} & {\Pro(\lc{\cC})_{\leq 0}} & {\Pro((\lc{\cC})_{\leq \ell})} & {\Pro(\lc{\cC})_{\leq \ell}} & {\Pro((\lc{\cC})_{\leq 2\ell})} & \cdots
        \arrow[hook', from=1-1, to=1-2]
        \arrow["\simeq"', from=1-1, to=2-1]
        \arrow[hook', from=1-2, to=1-3]
        \arrow[hook', from=1-3, to=1-4]
        \arrow["\simeq"', from=1-3, to=2-3]
        \arrow[hook', from=1-4, to=1-5]
        \arrow[hook', from=1-5, to=1-6]
        \arrow["\simeq"', from=1-5, to=2-5]
        \arrow[hook, from=2-1, to=2-2]
        \arrow[hook, from=2-2, to=2-3]
        \arrow[hook, from=2-3, to=2-4]
        \arrow[hook, from=2-4, to=2-5]
        \arrow[hook, from=2-5, to=2-6]
    \end{tikzcd}\]
    where the vertical maps are induced by $\eta$,
    and the horizontal inclusions come from the dual
    of Proposition \ref{prop:ind-weight}(3).
    So the functor $\Pro(\eta)$ is of amplitude $[0,\ell]$ (by the dual of Proposition \ref{prop:ind-weight}(6))
    and restricts to an equivalence $\Pro(\cC)_{< \infty} \xto{\simeq} \Pro(\lc{\cC})_{< \infty}$.
    By Corollary \ref{cor:left-comp-inverts-idem} this yields an equivalence
    $\lc{\Pro(\eta)} \colon \lc{\Pro(\cC)} \simeq \lc{\Pro(\lc{\cC})}$ of amplitude $[0,\ell]$.

    Similarly, since $\cC_{\leq 0} \subseteq (\lc{\cC})_{\leq 0}$
    is a dense full subcategory, the concrete description of $\Ind(\cC)_{\leq 0}$
    as generated by $\cC_{\leq 0}$
    immediately yields that the induced map
    $\Ind(\cC)_{\leq 0} \to \Ind(\lc{\cC})_{\leq 0}$
    is an equivalence.
    Shifting and taking the union, this yields $\Ind(\cC)_{<\infty} \simeq \Ind(\lc{\cC})_{<\infty}$.

    By exactness of $\Tate(\eta)$, the description of $\Tate(\cC)_{< \infty}$
    from Theorem \ref{thm:tate-weight}(2), and the cases
    of $\Pro$ and $\Ind$ above, $\eta$ induces
    a fully faithful dense inclusion $\Tate(\cC)_{< \infty} \hookrightarrow \Tate(\lc{\cC})_{< \infty}$.
    Since $\Tate(\eta)$ also has amplitude $[-\ell,\ell]$ by Theorem \ref{thm:tate-weight},
    Corollary \ref{cor:left-comp-inverts-idem}
    shows that it remains to check that $\Tate(\cC)_{<\infty} \subseteq \Tate(\lc{\cC})_{< \infty}$
    is also left dense.
    To this end, let $X \in \Tate(\lc{\cC})$ and $n \geq 0$.
    By Theorem \ref{thm:tate-weight}
    we can construct a weight decomposition
    $X_{\leq n} \to X \to X_{\geq n+1}$
    with an extension $X^p_{\leq n} \to X_{\leq n} \to X^i_{\leq n}$.
    Now $\Tate(\eta)$ restricts to an equivalence on weight bounded above
    pure $\Ind$- or $\Pro$-objects, and hence $X^p_{\leq n},X^i_{\leq n}$
    lie in its essential image. Since it is also fully faithful (and exact)
    on this subcategory, also $X_{\leq n}$ lies in the image of $\Tate(\eta)$.
    Thus $\Tate(\cC)_{<\infty}$ is left dense in $\Tate(\lc{\cC})$,
    and we conclude as mentioned above.
    Finally, the last statement follows from Lemma \ref{lem:saturation}(5),
    where in the case $\ell = 0$ we note that all the weight structures are saturated
    anyways so we can instead use Lemma \ref{lem:saturated}(3)
    to deduce that the inverse equivalences are also weight exact.
\end{proof}

Next, we investigate under what conditions one can expect an adjacent
weak $t$-structure on $\Tate(\cC)$.

\begin{theorem}[{{\cite[Lemma 4.1, Theorem 5.1]{Neeman}}}]\label{thm:neeman}
    Let $\cC$ be a countable stable category,
    meaning $\cC$ has a countable set of isomorphism classes of objects,
    and for any two objects there are only countably many homotopy
    classes of maps between them.
    Then every object $X \in \Ind(\cC)$ sits in a cofiber sequence
    \[
        \bigoplus_j d_j \to \bigoplus_i c_i \to X
    \]
    where $d_j,c_i \in \cC$.
\end{theorem}

\begin{remark}
    If $R$ is an $\E_1$-ring such that $\pi_*R$ is countable, then $\Perf(R)$ is countable stable.
\end{remark}

\begin{proposition}\label{prop:adj-t-on-pro}
    Let $\cC$ be a weakly weighted category with adjacent weak $t$-structure
    of defect $m < \infty$.
    Suppose that $\cC_{w < \infty}$ is countable.
    Then $\Pro(\cC)$ admits an adjacent weak $t$-structure of defect $2m+1$.
    Moreover, we have inclusions $\Pro(\cC_{t \leq 0}) \subseteq \Pro(\cC)_{t \leq 0} \subseteq \Pro(\cC_{t \leq 2m+1})$ and the inclusion $\cC \subseteq \Pro(\cC)$ is weight- and $t$-exact.
\end{proposition}
\begin{proof}
    Recall that by Lemma \ref{lem:adj-t-on-wt-cocon} it is enough to produce
    $t$-decompositions of the desired defect for objects in $\Pro(\cC)_{w < \infty} \subseteq \Pro(\cC_{w<\infty})$ (the inclusion holds by the dual of Proposition \ref{prop:ind-weight}(3)).
    Given such an $X \in \Pro(\cC_{w < \infty})$,
    we note that by the dual of Theorem
    \ref{thm:neeman} there exists a fiber sequence
    $X \to \prod_i c^i \to \prod_j d^j$
    with $c^i,d^j \in \cC_{w < \infty}$.
    Pick $t$-decompositions $c^i_{w > -m} \to c^i \to c^i_{t \leq 0}$
    for all $i$ and $d^j_{w > -2m} \to d^j \to d^j_{t \leq -m}$ for all $j$.
    Now
    \[
        \hom_{\Pro(\cC)}(\prod_i c^i_{w > -m}, \prod_j d^j_{t \leq -m})
        = \prod_j \bigoplus_i \hom_\cC(c^i_{w > -m}, d^j_{t \leq -m})
    \]
    is $(-1)$-coconnective since $\cC_{w > -m} = \cC_{t > -m}$
    and coconnective spectra are closed under products and sums.
    Hence we may construct the diagram of fiber sequences
    \[\begin{tikzcd}
        {X_{w > -(2m+1)}} & {\prod_i c^i_{w > -m}} & {\prod_j d^j_{w > -2m}} \\
        X & {\prod_i c^i} & {\prod_j d^j} \\
        {X_{t \leq 0}} & {\prod_i c^i_{t \leq 0}} & {\prod_j d^j_{t \leq -m}}
        \arrow[from=1-1, to=1-2]
        \arrow[from=1-1, to=2-1]
        \arrow[dashed, from=1-2, to=1-3]
        \arrow[from=1-2, to=2-2]
        \arrow[from=1-3, to=2-3]
        \arrow[from=2-1, to=2-2]
        \arrow[from=2-1, to=3-1]
        \arrow[from=2-2, to=2-3]
        \arrow[from=2-2, to=3-2]
        \arrow[from=2-3, to=3-3]
        \arrow[from=3-1, to=3-2]
        \arrow[dashed, from=3-2, to=3-3]
    \end{tikzcd}\]
    Once we show that $\Pro(\cC_{t \leq 0}) \subseteq \Pro(\cC)_{t \leq 0} = \rperp{(\Pro(\cC)_{\geq 1})}$
    it will follow that the left vertical fiber sequence
    is a $t$-decomposition of defect $2m+1$ for $X$,
    establishing the claimed adjacent weak $t$-structure on $\Pro(\cC)$.
    Since $t$-coconnective spectra are closed under limits,
    it suffices to show that $\cC_{t \leq 0} \subseteq \Pro(\cC)_{t \leq 0}$.
    For $c \in \cC_{t \leq 0}$,
    the class of $Z \in \Pro(\cC)_{w \geq 0}$ for which
    $
        \hom_{\Pro(\cC)}(Z,c)
    $
    is $t$-coconnective is closed under extensions, retracts, products
    and sequential limits of maps with fiber already in the class.
    It also contains $\cC_{w \geq 0} = \cC_{t \geq 0}$,
    and thus consists of all $Z \in \Pro(\cC)_{w \geq 0}$,
    showing that $c \in \Pro(\cC)_{t \leq 0}$.
    Note that this also shows that $\cC \to \Pro(\cC)$
    is $t$-exact, by adjacency and weight exactness.

    Moreover, if $X \in \Pro(\cC)_{t \leq 0}$,
    then by shifting, the above constructs a $t$-decomposition
    $X_{w > 0} \to X \to X_{t \leq 2m+1}$ where $X_{t \leq 2m+1} \in \Pro(\cC_{t \leq 2m+1})$
    (since this category is closed under finite limits in $\Pro(\cC)$).
    The map $X_{w > 0} \to X$ vanishes, exhibiting $X$ as retract
    of $X_{t \leq 2m+1}$.
    But $\Pro(\cC_{t \leq 2m+1})$ is already idempotent complete,
    so this shows $\Pro(\cC)_{t \leq 0} \subseteq \Pro(\cC_{t \leq 2m+1})$.
\end{proof}

\begin{proposition}\label{prop:adj-t-on-tate}
    Let $\cC$ be a weakly $\ell$-weighted category with adjacent weak $t$-structure of defect $m$,
    where $\ell,m < \infty$.
    If $\cC_{w < \infty}$ is countable,
    then $\Tate(\cC)$ admits an adjacent weak $t$-structure of defect $2m+2+\ell$ such that:
    \begin{enumerate}
        \item $\Tate(\cC)_{t \leq 0}$ is contained in the retract-closure
            of those Tate-objects sitting in an extension of $\Pro(\cC)_{t \leq 2m+1}$
            and $\Ind(\cC)_{t \leq 2m+2+\ell}$.

        \item The inclusion $\Ind(\cC) \subseteq \Tate(\cC)$ is weight- and $t$-exact,
            whereas the inclusion $\Pro(\cC) \subseteq \Tate(\cC)$ is weight exact
            but of $t$-amplitude $[0,\ell]$.
    \end{enumerate}
\end{proposition}
\begin{proof}
    Let us begin by showing (2).
    Since all the (putative) weak $t$-structures are adjacent and we already know weight exactness
    of both inclusions by Theorem \ref{thm:tate-weight},
    we only need to show the statement for $t$-coconnectives.
    By the description of $\Tate(\cC)_{\geq 0}$
    and the fact that coconnective spectra are closed under extensions and retracts,
    this reduces to checking that $\hom(X^p_{w \geq 0}, X^i_{t \leq 0})$
    and $\hom(X^i_{w \geq \ell}, X^p_{t \leq 0})$ are coconnective spectra.
    For the former, we use that the class of $Y \in \Pro(\cC)$ for which
    $\hom(Y,X^i_{t \leq 0})$ is coconnective contains $\cC_{\geq 0}$
    and is closed under extensions, retracts, finite colimits, products, and sequential limits
    along maps whose fiber already lies in the class.
    It thus contains all of $\Pro(\cC)_{w \geq 0}$, as desired.
    The latter claim follows directly from the fact that $\Ind(\cC)_{\geq \ell} \subseteq \Ind(\cC_{\geq 0})$
    by Proposition \ref{prop:ind-weight}(3).

    Next, we provide $t$-decompositions of Tate objects.
    Recall that $\Pro(\cC)$ admits an adjacent weak $t$-structure of defect $2m+1$
    by Proposition \ref{prop:adj-t-on-pro}.
    Given $X \in \Tate(\cC)$ with extension $X^p \to X \to X^i$,
    we have $t$-decompositions $X^i_{w \geq \ell+2} \to X^i \to X^i_{t \leq \ell+1}$ in $\Ind(\cC)$
    and $X^p_{w > -(2m+1)} \to X^p \to X^p_{t \leq 0}$ in $\Pro(\cC)$.
    By (2) we get that the map $X^i_{w \geq \ell+2} \to \Sigma X^p_{t \leq 0}$
    vanishes, which allows us to consider the following diagram of fiber sequences
    \[\begin{tikzcd}[ampersand replacement=\&]
        {X_{w > -(2m+1)}} \& {X^i_{w \geq \ell+2}} \& {\Sigma X^p_{w > -(2m+1)}} \\
        X \& {X^i} \& {\Sigma X^p} \\
        {X_{t \leq \ell+1}} \& {X^i_{t \leq \ell+1}} \& {\Sigma X^p_{t \leq 0}}
        \arrow[from=1-1, to=1-2]
        \arrow[from=1-1, to=2-1]
        \arrow[dashed, from=1-2, to=1-3]
        \arrow[from=1-2, to=2-2]
        \arrow[from=1-3, to=2-3]
        \arrow[from=2-1, to=2-2]
        \arrow[from=2-1, to=3-1]
        \arrow[from=2-2, to=2-3]
        \arrow[from=2-2, to=3-2]
        \arrow[from=2-3, to=3-3]
        \arrow[from=3-1, to=3-2]
        \arrow[dashed, from=3-2, to=3-3]
    \end{tikzcd}\]
    The left vertical fiber sequence is then a $t$-decomposition for $X$ of defect $2m+2+\ell$.
    This establishes the adjacent weak $t$-structure of defect $2m+2+\ell$ on $\Tate(\cC)$.
    Note also that if $X \in \Tate(\cC)_{t \leq 0}$,
    then by shifting, the above constructs a $t$-decomposition
    $X_{w > 0} \to X \to X_{t \leq 2m+2+\ell}$
    which exhibits $X$ as a retract of $X_{t \leq 2m+2+\ell}$,
    which in turn sits in an extension of an object in $\Ind(\cC)_{t \leq 2m+2+\ell}$
    and an object in $\Pro(\cC)_{t \leq 2m+1}$.
\end{proof}

\begin{corollary}\label{coro_weak}
    Let $\cC$ be a countable weighted category
    such that $\lc{\cC}$ admits an adjacent weak $t$-structure
    of defect $m$, and $m < \infty$.
    Then $\lc{\Pro(\cC)}$ respectively $\lc{\Tate(\cC)}$ admit adjacent
    weak $t$-structures of defect $2m+1$ respectively $2m+2$.
\end{corollary}
\begin{proof}
    Since $\cC_{w <\infty } \hookrightarrow (\lc{\cC})_{w < \infty}$
    is fully faithful with dense image,
    it follows that also $(\lc{\cC})_{w <\infty}$ is countable,
    and hence $\lc{\cC}$ satisfies the assumptions of Propositions \ref{prop:adj-t-on-pro} and \ref{prop:adj-t-on-tate},
    so that $\Pro(\lc{\cC})$ and $\Tate(\lc{\cC})$ admit adjacent weak $t$-structures
    of defects $2m+1$ respectively $2m+2$.
    By Lemma \ref{lem:tate-upw-comp} we have equivalences $\lc{\Pro(\cC)} \simeq \lc{\Pro(\lc{\cC})}$
    and $\lc{\Tate(\cC)} \simeq \lc{\Tate(\lc{\cC})}$ in $\wWCat_{=0}$,
    so we conclude by Lemma \ref{lem:adj-t-on-upw-comp}.
\end{proof}

\begin{example}\label{ex:only-weak-t}
    $\Sp^\omega$ is a countable weighted category
    and $\Sp^{\omega,\uparrow} = \Sp^\ft = \Sp^\fg_{>-\infty}$
    admits an adjacent $t$-structure.
    Thus $\lc{\Pro(\Sp^\omega)}$ respectively $\lc{\Tate(\Sp^\omega)}$
    admit adjacent weak $t$-structures of defect $1$ respectively $2$.
    Here, a weak $t$-structure on the left completion is really the best one can
    hope for: an adjacent weak $t$-structure on $\Pro(\Sp^\omega)$
    corresponds under the equivalence $\Pro(\Sp^\omega) \simeq \Sp^\op$
    induced by Spanier-Whitehead duality to a co-adjacent weak $t$-structure
    on $\Sp$. However, if $X$ is a spectrum such that $\hom(X,Y) \in \Sp_{t \leq 0}$
    for all $Y \in \Sp_{w \leq 0}$, then already $X = 0$!
    Indeed, this is because $\Z \tensor I_{\Q/\Z} = 0$ (\cite[Lemma 7.1]{Hovey-Palmieri})
    so that $\Sigma^nI_{\Q/\Z} \in \Sp_{w \leq 0}$ for all $n$,
    and $0 = \pi_0\hom(X,\Sigma^nI_{\Q/\Z}) = \Hom(\pi_n X, \Q/\Z)$
    for all $n$ implies $X = 0$.
\end{example}

\appendix

\section{Mittag-Leffler conditions for ordinal-indexed inverse limits}\label{sec:app}

In this appendix we give an analogue of a Mittag--Leffler condition
which for example guarantees that an ordinal indexed inverse diagram
of connective spectra has a connective limit.
This is crucial for proving Theorem \ref{thm:gen-weight} in the case $\kappa > \omega$.

Recall the notion of connectivity of anima and maps of anima
from Example \ref{ex:conn-topos}.
Concretely, an anima is $0$-connective if it is nonempty,
and a $0$-connective map is a $\pi_0$-surjection.
For $n \geq 1$, an $n$-connective anima satisfies $\pi_k(X) \cong \pi_k(*)$
for $k \leq n-1$, and a map of anima is $n$-connective
if its fiber is $n$-connective for any choice of base point.
Note that if $X \to Y$ is an $n$-connective map and $Y$ is $n$-connective,
then also $X$ is $n$-connective.

\begin{proposition}\label{prop:mittag-leffler-an}
    Let $\gamma$ be an ordinal, $n \geq 0$, and $X_\bullet \colon \gamma^\op \to \An$ a diagram satisfying
    \begin{enumerate}
        \item $X_0$ is $n$-connective.
        \item For each successor ordinal $\alpha+1 < \gamma$,
            the map $X_{\alpha+1} \to X_\alpha$
            is $n$-connective.
        \item For each limit ordinal $\lambda < \gamma$,
            the map $X_\lambda \to \lim_{\alpha < \lambda} X_\alpha$
            is $n$-connective.
    \end{enumerate}
    Then all the $X_\alpha$ for $\alpha < \gamma$
    and also $X_\gamma \coloneqq \lim_{\alpha < \gamma} X_\alpha$ are $n$-connective.
\end{proposition}

Let us first show that we can reduce this proposition to the following
special case, which we learned from Maxime Ramzi, see \cite[Theorem 1.1]{Ramzi-Ordinals}.
For the convenience of the reader we reproduce and elaborate on its proof below.

\begin{lemma}\label{lem:mittag-leffler-an}
     Let $\gamma$ be an ordinal, $n \geq 0$, and $X_\bullet \colon \gamma^\op \to \An$ a diagram satisfying
    \begin{enumerate}
        \item $X_\alpha$ is nonempty for all $\alpha < \gamma$.
        \item For each successor ordinal $\alpha+1 < \gamma$,
            the map $X_{\alpha+1} \to X_\alpha$
            is a $\pi_0$-surjection.
        \item For each limit ordinal $\lambda < \gamma$,
            the map $X_\lambda \to \lim_{\alpha < \lambda}X_\alpha$
            is a $\pi_0$-surjection.
    \end{enumerate}
    Then $X_\gamma \coloneqq \lim_{\alpha < \gamma} X_\alpha$ is nonempty.
\end{lemma}

\begin{proof}[Proof of Prop.~\ref{prop:mittag-leffler-an} assuming Lemma \ref{lem:mittag-leffler-an}]
    Note that Lemma \ref{lem:mittag-leffler-an}
    is just the $n=0$ case of the Proposition,
    with the following stronger assumption in place of (1):
    \begin{enumerate}
        \item[(1')] $X_\alpha$ is $n$-connective for all $\alpha < \gamma$.
    \end{enumerate}
    Let us first show how this implies the cases for $n \geq 1$
    where we also assume (1') instead of (1).

    For the case $n=1$, the $n=0$-case immediately yields that $\lim X_\bullet$ is $0$-connective, hence nonempty.
    To show it is 1-connective, we need to argue that $\pi_0\lim X_\bullet = *$.
    Given two points in $x,y \colon * \to \lim X_\bullet$, let $P_{x,y}\lim X_\bullet$ be the space of paths,
    which we need to prove is nonempty.
    The base points $x,y$ induce two base points at each level $X_\alpha$ which are preserved by the structure maps.
    We will also denote these by $x,y$.
    Then $P_{x,y}\lim_{\alpha < \gamma}X_\alpha \simeq \lim_{\alpha < \gamma}P_{x,y}X_\alpha$,
    and the diagram $P_{x,y}X_\bullet$ satisfies the assumptions of the $n=0$ case,
    so that its limit is non-empty. This settles the case $n=1$.
    For $n \geq 2$, we can do a similar reduction to $n=1$ by taking iterated loop spaces,
    where we now already know that $\pi_0\lim X_\bullet = *$ so there is only one choice of basepoint.

    Thus the assumptions (1')+(2)+(3) imply the main claim for all $n \geq 0$.
    Let us now show that we can replace (1') with (1),
    for which we need to show that under assumptions (1)-(3),
    all the $X_\alpha$ are $n$-connective.
    We do this by transfinite induction.
    Namely, condition (2) guarantees that if $X_\alpha$ is $n$-connective
    then also $X_{\alpha+1}$ is. Moreover, if $\lambda < \gamma$
    is a limit ordinal so that for all $\alpha < \lambda$
    we have already shown that $X_\alpha$ is $n$-connective,
    then $X_\bullet|_\lambda$ satisfies (1')+(2)+(3)
    so that $\lim_{\alpha < \lambda} X_\alpha$ is $n$-connective.
    By (3) we see that $X_\lambda$ sits in a fiber sequence
    between two $n$-connective anima, and is thus itself $n$-connective
    (for $n\geq 1$ there is no issue with base points,
    and for $n=0$ (3) precisely tells us that we have $\pi_0$-surjections,
    so $X_\lambda$ is nonempty).
    By transfinite induction, this shows that for all $\alpha < \gamma$,
    the anima $X_\alpha$ is $n$-connective, as desired.
\end{proof}

\begin{proof}[Proof of Lemma \ref{lem:mittag-leffler-an}]
    Note that if $\gamma = \beta+1$ is not a limit ordinal,
    then $\lim_{\alpha < \gamma}X_\alpha = X_\beta$, and we are done by (1).
    So suppose from now on that $\gamma$ is a limit ordinal.

    By straightening-unstraightening,
    the diagram $X_\bullet$ induces a right fibration of quasi-categories
    $p \colon \int X_\bullet \to \gamma$, and we can identify
    $\lim_{\alpha < \gamma} X_\alpha$ with the Kan-complex of sections
    $\Fun_{/\gamma}(\gamma,\int X_\bullet)$,
    which we have to prove is nonempty.
    More generally, for each $\beta \leq \gamma$
    we can identify the Kan complex $\Fun_{/\gamma}(\beta, \int X_\bullet)$
    with the limit $\lim_{\alpha < \beta}X_\alpha$.
    Here we implicitly identify an ordinal $\alpha$
    with the nerve $N(\alpha)$ of the 1-category given by the poset $\alpha$.

    Because the transition maps $\alpha \to \beta$
    are monomorphisms of simplicial sets and hence cofibrations in Quillen model structure whenever $\alpha \leq \beta$,
    the expression $\beta = \colim_{\alpha < \beta} \alpha$
    is true both strictly (i.e.~in simplicial sets) and homotopically (i.e.~in $\Cat_\infty$).
    Since $\Fun_{/\gamma}(-,\int X_\bullet)$ also sends colimits
    of simplicial sets to limits,
    we conclude that for each $\beta \leq \gamma$,
    the strict limit $\lim_{\alpha < \beta}\Fun_{/\gamma}(\alpha,\int X_\bullet)$
    is a Kan-complex representing the limit $\lim_{\alpha < \beta}X_\alpha$.

    For $\alpha < \beta < \gamma$, consider the following diagram of pullbacks of
    simplicial sets
    \[\begin{tikzcd}
        {\Fun_{/\gamma}(\beta,\int X_\bullet)} & {\Fun(\beta,\int X_\bullet)} \\
        {\Fun_{/\gamma}(\alpha, \int X_\bullet)} & P & {\Fun(\alpha,\int X_\bullet)} \\
        {*} & {\Fun(\beta,\gamma)} & {\Fun(\alpha,\gamma)}
        \arrow[from=1-1, to=1-2]
        \arrow["{R'}"', from=1-1, to=2-1]
        \arrow["\lrcorner"{anchor=center, pos=0.125}, draw=none, from=1-1, to=2-2]
        \arrow["R", from=1-2, to=2-2]
        \arrow[from=2-1, to=2-2]
        \arrow[from=2-1, to=3-1]
        \arrow["\lrcorner"{anchor=center, pos=0.125}, draw=none, from=2-1, to=3-2]
        \arrow[from=2-2, to=2-3]
        \arrow[from=2-2, to=3-2]
        \arrow["\lrcorner"{anchor=center, pos=0.125}, draw=none, from=2-2, to=3-3]
        \arrow["{p_*}", from=2-3, to=3-3]
        \arrow["\inc", from=3-1, to=3-2]
        \arrow[from=3-2, to=3-3]
    \end{tikzcd}\]
    Note that since $p$ is a right fibration,
    also $p_*$ is a right fibration by \cite[\href{https://kerodon.net/tag/00TQ}{Tag 00TQ}]{Kerodon}.
    Since $N(\alpha) \to N(\beta)$ is a monomorphism
    of simplicial sets, \cite[\href{https://kerodon.net/tag/00TP}{Tag 00TP}]{Kerodon} tells us that also $R$ is a right fibration.
    Since right fibrations are stable under pullback,
    we conclude that all vertical maps in this diagram are right fibrations,
    and thus that all the strict pullbacks present the corresponding
    (homotopy-)pullbacks taken in $\Cat_\infty$.

    Assumption (2) now tells us that the right fibration
    $\Fun_{/\gamma}(\beta+1,\int X_\bullet) \to \Fun_{/\gamma}(\beta, \int X_\bullet)$ is surjective on $\pi_0$, and thus surjective on 0-simplices
    by the defining lifting property against $\{1\} \to \Delta^1$.
    Similarly, assumption (3) tells us that $\Fun_{/\gamma}(\lambda+1, \int X_\bullet) \to \Fun_{/\gamma}(\lambda, \int X_\bullet)$
    is surjective on 0-simplices.
    Hence, passing to $0$-simplices in the $\gamma^\op$-indexed diagram
    $\beta \mapsto \Fun_{/\gamma}(\beta,\int X_\bullet)$
    thus yields a diagram $Y_\bullet \colon \gamma^\op \to \Set$ satisfying:
    \begin{enumerate}
        \item Each $Y_{\alpha}$ is nonempty for $\alpha < \gamma$.
            Indeed, since $\gamma$ is a limit ordinal,
            we know that $\alpha+1 < \gamma$.
            Now $Y_{\alpha+1}$ is the set of 0-simplices
            of a Kan complex representing the nonempty anima $X_\alpha$
            and is thus nonempty.
            Since we have a map $Y_{\alpha+1} \to Y_\alpha$,
            also the latter is nonempty.

        \item For any $\alpha < \gamma$,
            the map $Y_{\alpha+1} \to Y_\alpha$ is surjective.
            Note that here we use both assumptions (2) and (3)
            (the latter in the case $\alpha=\lambda$).

        \item For any limit-ordinal $\lambda < \gamma$,
            we have $Y_\lambda \xto{\cong} \lim_{\alpha < \lambda}Y_\alpha$,
            since $\Fun_{/\gamma}(\lambda, \int X_\bullet) \cong \lim_{\alpha < \lambda} \Fun_{/\gamma}(\alpha,\int X_\bullet)$ is a strict limit
            in simplicial sets, which passes to 0-simplices.
    \end{enumerate}
    We need to show that $\lim_{\alpha < \gamma} Y_\alpha \neq \empty$.
    We define a compatible system of elements $y_\alpha \in Y_\alpha$
    via transfinite induction.
    Pick any $y_0 \in Y_0$.
    Having defined $y_\alpha \in Y_\alpha$, we let $y_{\alpha+1} \in Y_{\alpha+1}$ be any preimage of $y_\alpha$ under the surjection $Y_{\alpha+1} \to Y_\alpha$.
    If $\lambda \leq \gamma$ is a limit ordinal so that we have defined
    the compatible family $y_\alpha$ for all $\alpha < \lambda$,
    then this precisely defines us a point in $\lim_{\alpha < \lambda}Y_\alpha \cong Y_\lambda$.
    By transfinite induction, we obtain a point in $Y_\gamma \coloneqq \lim_{\alpha < \gamma} Y_\alpha$.
\end{proof}

\begin{corollary}\label{cor:mittag-leffler-sp}
    Let $R$ be a connective $\E_1$-ring,
    $\gamma$ an ordinal, $n \in \Z$, and $X_\bullet \colon \gamma^\op \to \LMod(R)$ a diagram satisfying
    \begin{enumerate}
        \item $X_{0}$ is $n$-connective (in the canonical $t$-structure on $\Mod(R)$).

        \item For each successor ordinal $\alpha+1 < \gamma$,
            the map $X_{\alpha+1} \to X_\alpha$ has $n$-connective fiber.

        \item For each limit ordinal $\lambda < \gamma$,
            the map $X_\lambda \to \lim_{\alpha < \lambda}X_\alpha$
            has $n$-connective fiber.
    \end{enumerate}
    Then all the $X_\alpha$ for $\alpha < \gamma$
    and also $\lim_{\alpha < \gamma} X_\alpha$ are $n$-connective.
\end{corollary}
\begin{proof}
    Since connectivity, limits and colimits are detected and preserved by the forgetful functor $U \colon \LMod(R) \to \Sp$,
    we may as well assume that $R = \S$ so that $\LMod(R) = \Sp$.
    But then we can reduce to Proposition \ref{prop:mittag-leffler-an}
    by considering $\Omega^{\infty-k}$ for $k \gg 0$.
\end{proof}

\begin{remark}
    In the case where the diagram $X_\bullet$ takes values in $\Sp_{t=0} \simeq \Ab$ (though the limits are still taken in $\Sp$)
    the second condition simply becomes that the relevant maps are all surjections.
    That such inverse systems of abelian groups have vanishing $\lim^{(n)}$-terms for $n > 0$
    was shown in \cite[Corollary A.3.14]{Neeman-Triangulated}.
\end{remark}

Recall the orthogonal complement notation from Section \ref{sec:notation}.

\begin{lemma}\label{lem:orth-closure}
    Let $\cC$ be a stable category and $\cD \subseteq \cC$ a full subcategory.
    Then:
    \begin{enumerate}
        \item ${}^\perp\cD$ and $\cD^\perp$
            are closed under retracts and extensions.

        \item If $\cD$ is closed under looping \alt{suspensions}
            then ${}^\perp \cD$ \alt{$\cD^\perp$} is closed under all colimits
            \alt{all limits}
            and $\cD^\perp$ \alt{${}^\perp\cD$}
            is closed under finite limits \alt{finite colimits}
            and colimits \alt{limits}
            of diagrams $X_\bullet \colon \gamma \to \cC$ \alt{$X_\bullet \colon \gamma^\op \to \cC$} for some ordinal $\gamma$ satisfying
            \begin{enumerate}
                \item For a successor ordinal $\alpha+1 < \gamma$,
                    the map $X_{\alpha} \to X_{\alpha+1}$ has cofiber in $\cD$
                    \alt{the map $X_{\alpha+1} \to X_\alpha$ has fiber in $\cD$}.
                \item For a limit ordinal $\lambda < \gamma$,
                    the map $\colim_{\alpha < \lambda} X_\alpha \to X_\lambda$
                    has cofiber in $\cD$
                    \alt{the map $X_\lambda \to \lim_{\alpha < \lambda} X_\alpha$
                    has fiber in $\cD$}.
            \end{enumerate}
    \end{enumerate}
\end{lemma}
\begin{proof}
    This follows from the analogous closure properties of
    $\Sp_{\geq 0} \subseteq \Sp$
    and $\Sp_{t \leq 0} \subseteq \Sp$.
    For the ordinal-indexed (co)limits,
    this is Corollary \ref{cor:mittag-leffler-sp}.
\end{proof}

The following lemma is needed to identify the weight coconnectives
in the Anderson weight structure on spectra, see Proposition \ref{prop:anderson-weight}.

\begin{lemma}\label{lem:free-ab}
    Let $\gamma$ be an ordinal and $A_\bullet \colon \gamma \to \Ab$
    a diagram of abelian groups. Suppose that
    \begin{enumerate}
        \item Each $A_\alpha$ is a free abelian group.
        \item For a successor ordinal $\alpha+1 < \gamma$,
            the map $A_{\alpha} \to A_{\alpha+1}$ is injective
            with free abelian cokernel.
        \item For each limit ordinal $\lambda < \gamma$,
            the map $\colim_{\alpha < \lambda} A_\alpha \to A_\lambda$
            is injective with free abelian cokernel.
    \end{enumerate}
    Then $A_\gamma \coloneqq \colim_{\alpha < \gamma} A_\alpha$
    is also free abelian.
\end{lemma}
\begin{proof}
    If $\gamma$ is a successor ordinal, we are done by (1),
    so assume from now on that it is a limit ordinal.
    Extend $A_\bullet$ to a diagram $\gamma+1 \to \Ab$
    by taking the colimit.
    We will construct a diagram $X_\bullet \colon \gamma+1 \to \Set$
    where each map $X_{\alpha} \to X_{\beta}$ is an injection
    for $\alpha < \beta \leq \gamma$,
    together with a natural isomorphism $\Z[X_\bullet] \cong A_\bullet$.

    Since $A_0$ is free, we have $A_0 \cong \Z[X_0]$ for some set $X_0$.
    Suppose we have constructed a diagram $X_\bullet \colon \alpha+1 \to \Set$
    together with a natural isomorphism $\Z[X_\bullet] \cong (A_\bullet)_{\alpha+1}$. To extend this to $\alpha+2$,
    we note that there is a short exact sequence $A_\alpha \to A_{\alpha+1} \to C$
    where $C$ is a free abelian group. Thus the sequence splits,
    which allows us to extend the basis $X_{\alpha}$ of $A_\alpha$
    to a basis $X_{\alpha+1}$ of $A_{\alpha+1}$, giving an extension
    of $X_\bullet$ and the natural isomorphism to $\alpha+2$.

    Now if $\lambda < \gamma$ is a limit ordinal
    so that we have constructed the
    lift up to stage $\alpha$ for all $\alpha < \lambda$
    then we can use the fact that
    $\Fun(\lambda,\Set) \cong \lim_{\alpha < \lambda} \Fun(\alpha,\Set)$
    and the concrete description of limits in the 1-category
    of 1-categories to obtain the extension $X_\bullet \colon \lambda \to \Set$
    of injections with a natural isomorphism $\Z[X_\bullet] \cong A_\bullet|_{\lambda}$.
    Now assumption (3) guarantees by the same argument
    as above that we can extend the basis $\colim_{\alpha < \lambda}X_\alpha$
    of $\colim_{\alpha < \lambda}A_\alpha$
    to a basis of $A_\lambda$, and thus obtain an extension
    of $X_\bullet$ and the natural isomorphism to $\lambda+1$.
    This concludes the transfinite induction.
\end{proof}

\bibliography{reference}
\end{document}